\documentclass[11pt]{amsart}
\title{Parametrized spectra, a low-tech approach}
\author{Cary Malkiewich}
\address{Department of Mathematics, Binghamton University}
\email{malkiewich@math.binghamton.edu}

\usepackage{amsmath,amssymb,amsthm,array,color,multirow,pdflscape,mathrsfs,subcaption}
\usepackage[all]{xypic}
\usepackage{tikz,tikz-cd}
\usepackage{imakeidx}
\makeindex[columns=2, title=Index] 
\newdir{ >}{{}*!/-7pt/@{>}}

\usepackage[margin=1.25in]{geometry}
\newcommand{\beforesubsection}{\vspace{1em}}
\newcommand{\aftersubsection}{ \mbox{}\nopagebreak \vspace{0.5em} }

\makeatletter
\def\l@section{\@tocline{1}{0pt}{1pc}{}{}}
\def\l@subsection{\@tocline{2}{0pt}{1pc}{4.6em}{}}
\def\l@subsubsection{\@tocline{3}{0pt}{1pc}{7.6em}{}}
\renewcommand{\tocsection}[3]{%
	\indentlabel{\@ifnotempty{#2}{\makebox[2.3em][l]{%
				\ignorespaces#1 #2.\hfill}}}#3}
\renewcommand{\tocsubsection}[3]{%
	\indentlabel{\@ifnotempty{#2}{\hspace*{2.3em}\makebox[2.3em][l]{%
				\ignorespaces#1 #2.\hfill}}}#3}
\renewcommand{\tocsubsubsection}[3]{%
	\indentlabel{\@ifnotempty{#2}{\hspace*{4.6em}\makebox[3em][l]{%
				\ignorespaces#1 #2.\hfill}}}#3}
\makeatother

\usepackage{xcolor}
\definecolor{darkgreen}{rgb}{0,0.30,0} 
\definecolor{darkred}{rgb}{0.75,0,0}
\definecolor{darkblue}{rgb}{0,0,0.6} 
\usepackage[pdfborder=0,pagebackref,colorlinks,citecolor=darkgreen,linkcolor=darkgreen,urlcolor=darkblue]{hyperref}
\renewcommand*{\backref}[1]{}
\renewcommand*{\backrefalt}[4]{({%
    \ifcase #1 Not cited.%
          \or On p.~#2%
          \else On pp.~#2%
    \fi%
    })}

\def\makeautorefname#1#2{\expandafter\def\csname#1autorefname\endcsname{#2}}
\makeautorefname{equation}{Equation}%
\makeautorefname{footnote}{footnote}%
\makeautorefname{item}{item}%
\makeautorefname{figure}{Figure}%
\makeautorefname{table}{Table}%
\makeautorefname{part}{Part}%
\makeautorefname{appendix}{Appendix}%
\makeautorefname{chapter}{Chapter}%
\makeautorefname{section}{Section}%
\makeautorefname{subsection}{Section}%
\makeautorefname{subsubsection}{Section}%
\makeautorefname{paragraph}{Paragraph}%
\makeautorefname{subparagraph}{Paragraph}%
\makeautorefname{theorem}{Theorem}%
\makeautorefname{thm}{Theorem}%
\makeautorefname{addm}{Addendum}%
\makeautorefname{mainthm}{Main theorem}%
\makeautorefname{corollary}{Corollary}%
\makeautorefname{cor}{Corollary}%
\makeautorefname{lemma}{Lemma}%
\makeautorefname{lem}{Lemma}%
\makeautorefname{sublemma}{Sublemma}%
\makeautorefname{sublem}{Sublemma}%
\makeautorefname{subl}{Sublemma}%
\makeautorefname{prop}{Proposition}%
\makeautorefname{property}{Property}
\makeautorefname{pro}{Property}
\makeautorefname{sch}{Scholium}%
\makeautorefname{step}{Step}%
\makeautorefname{conject}{Conjecture}%
\makeautorefname{conj}{Conjecture}%
\makeautorefname{questn}{Question}
\makeautorefname{quest}{Question}
\makeautorefname{qn}{Question}
\makeautorefname{definition}{Definition}%
\makeautorefname{defn}{Definition}%
\makeautorefname{defi}{Definition}%
\makeautorefname{def}{Definition}%
\makeautorefname{dfn}{Definition}%
\makeautorefname{df}{Definition}%
\makeautorefname{notation}{Notation}
\makeautorefname{notn}{Notation}
\makeautorefname{rem}{Remark}%
\makeautorefname{rems}{Remarks}%
\makeautorefname{rmk}{Remark}%
\makeautorefname{rk}{Remark}%
\makeautorefname{remarks}{Remarks}%
\makeautorefname{rems}{Remarks}%
\makeautorefname{rmks}{Remarks}%
\makeautorefname{rks}{Remarks}%
\makeautorefname{example}{Example}%
\makeautorefname{examp}{Example}%
\makeautorefname{exmp}{Example}%
\makeautorefname{exam}{Example}%
\makeautorefname{exa}{Example}%
\makeautorefname{ex}{Example}%
\makeautorefname{axiom}{Axiom}%
\makeautorefname{axi}{Axiom}%
\makeautorefname{ax}{Axiom}%
\makeautorefname{case}{Case}%
\makeautorefname{claim}{Claim}%
\makeautorefname{clm}{Claim}%
\makeautorefname{assumpt}{Assumption}%
\makeautorefname{asses}{Assumptions}%
\makeautorefname{conclusion}{Conclusion}%
\makeautorefname{concl}{Conclusion}%
\makeautorefname{conc}{Conclusion}%
\makeautorefname{cond}{Condition}%
\makeautorefname{const}{Construction}%
\makeautorefname{con}{Construction}%
\makeautorefname{criterion}{Criterion}%
\makeautorefname{criter}{Criterion}%
\makeautorefname{crit}{Criterion}%
\makeautorefname{exercise}{Exercise}%
\makeautorefname{exer}{Exercise}%
\makeautorefname{exe}{Exercise}%
\makeautorefname{warn}{Warning}%
\newtheorem{thm}{Theorem}[subsection]
\newtheorem{cor}{Corollary}[subsection]
\newtheorem{lem}{Lemma}[subsection]
\newtheorem{prop}{Proposition}[subsection]
\theoremstyle{definition}
\newtheorem{df}{Definition}[subsection]
\newtheorem{ex}{Example}[subsection]
\newtheorem{warn}{Warning}[subsection]
\newtheorem{rmk}{Remark}[subsection]

\numberwithin{equation}{subsection}
\numberwithin{figure}{subsection}

\makeatletter
\let\c@cor=\c@thm
\let\c@prop=\c@thm
\let\c@lem=\c@thm
\let\c@df=\c@thm
\let\c@ex=\c@thm
\let\c@warn=\c@thm
\let\c@rmk=\c@thm
\let\c@notn=\c@thm
\let\c@equation\c@thm
\let\c@figure\c@thm
\let\c@table\c@thm
\makeatother

\hypersetup{linktoc=all}

\newcommand{\sB}{\mathscr{B}}
\newcommand{\sC}{\mathscr{C}}
\newcommand{\sD}{\mathscr{D}}

\newcommand{\bH}{\mathbf{H}}

\newcommand{\bS}{\mathbf{S}}
\newcommand{\bT}{\mathbf{T}}

\newcommand{\D}{\mathbb D}
\renewcommand{\L}{\mathbb L}
\newcommand{\M}{\mathbb M}
\newcommand{\N}{\mathbb N}
\newcommand{\R}{\mathbb R}

\newcommand{\Z}{\mathbb Z}
\newcommand{\Q}{\mathbb Q}
\newcommand{\Sph}{\mathbb S}
\newcommand{\mc}{\mathcal}
\newcommand{\cat}[1]{\textup{\textbf{{#1}}}}

\newcommand{\ra}{\longrightarrow}

\newcommand{\simar}{\overset\sim\longrightarrow}
\newcommand{\congar}{\overset\cong\longrightarrow}

\newcommand{\Hom}{\textup{Hom}}
\newcommand{\Map}{\textup{Map}}

\newcommand{\id}{\textup{id}}
\newcommand{\colim}{\textup{colim}\,}
\newcommand{\hocolim}{\textup{hocolim}\,}

\newcommand{\Cyl}{\textup{Cyl}}
\newcommand{\sma}{\wedge}
\newcommand{\barsmash}{\,\overline\wedge\,}
\newcommand{\barsma}[1]{\,\overline\wedge_{#1}\,}
\DeclareMathOperator{\barmap}{\underline{\smash{\textup{Map}}}}
\DeclareMathOperator{\barF}{\underline{\smash{F}}}
\newcommand{\ti}{\widetilde}
\newcommand{\op}{\textup{op}}

\newcommand{\sk}{\textup{Sk}}
\newcommand{\tr}{\textup{tr}\,}
\newcommand{\THH}{\textup{THH}}

\newcommand{\sh}{\textup{sh}}

\DeclareMathOperator{\ho}{\textup{Ho}}
\DeclareMathOperator{\Fun}{\textup{Fun}}
\newcommand{\Ex}{\mathcal{E}x}
\newcommand{\bcr}[3]{\left[{#1}\xrightarrow{#2}{#3}\right]}
\newcommand{\bcl}[3]{\left[{#3}\xleftarrow{#2}{#1}\right]}

\def\calBi#1#2{\ensuremath{\mathscr{#1}\!/\!_{\mathbf{#2}}}}

\newcommand{\Fix}{\textup{Fix}}
\newcommand{\ind}{\textup{ind}}
\newcommand{\cyc}{\textup{cyc}}

\newcommand{\shad}[1]{{\ensuremath{\hspace{1mm}\makebox[-1mm]{$\langle$}\makebox[0mm]{$\langle$}\hspace{1mm}{#1}\makebox[1mm]{$\rangle$}\makebox[0mm]{$\rangle$}}}}
\DeclareMathOperator{\ob}{\textup{ob}}
\newcommand{\Osp}{\mathcal{OS}}
\newcommand{\Psp}{\mathcal{PS}}
\newcommand{\adj}{\dashv}
\newcommand{\Th}{\textup{Th}}
\newcommand{\non}{\textup{non}}
\newcommand{\po}[2]{\ar@{}@<{#2}>[rd]|({#1})*\txt{\Large $\ulcorner$}}
\newcommand{\pb}[2]{\ar@{}@<{#2}>[rd]|({#1})*\txt{\Large $\lrcorner$}}

\begin{document}

\maketitle

\begin{abstract}
	We give an alternative treatment of the foundations of parametrized spectra, with an eye toward applications in fixed-point theory. We cover most of the central results from the book of May and Sigurdsson, sometimes with weaker hypotheses, and give a new construction of the bicategory $\Ex$ of parametrized spectra. We also give a careful account of coherence results at the level of homotopy categories.
	
	This is an expository work, akin to a set of lecture notes, encompassing and extending the more novel material in the paper ``A convenient category of parametrized spectra.'' Our primary goal is to give careful and explicit explanations for the standard facts surrounding the bicategory of parametrized spectra and the Reidemeister trace. In particular, our proof of Costenoble-Waner duality gives a more explicit account of how the concrete index formulas of Dold are related to the abstract theory of bicategorical duality from May and Sigurdsson, Ponto, and Shulman.
\end{abstract}

\setcounter{tocdepth}{2}
\tableofcontents

\parskip 2ex

\section{Introduction}
Parametrized spectra are a ``fiberwise'' or ``parametrized'' refinement of spectra. They were first introduced in order to systematically study the Becker-Gottlieb transfer \cite{becker1975transfer,becker1976transfer,clapp1984homotopy}, but they also occur naturally in fixed-point theory and intersection theory \cite{ponto_asterisque,klein_williams}, twisted $K$-theory \cite{atiyah_segal,abg_ktheory,fht_1,hebestreit2019homotopical}, differential cohomology \cite{braunack2018rational}, computations with Thom spectra, e.g. \cite{mahowald1977new,units_of_ring_spectra,hahn2018eilenberg,klang2018factorization}, the functor calculus of Goodwillie \cite{calc3}, the algebraic $K$-theory of spaces of Waldhausen \cite{waldhausen_1,waldhausen_2}, symplectic geometry and string topology \cite{kragh_parametrized_nearby}, and the index theory of Dwyer, Weiss, and Williams \cite{dww,klein_williams_bundle}. The framework of parametrized spectra allows us to formulate a Poincar\' e duality theorem with extraordinary, twisted coefficients (see \autoref{thm:cw_duality_1}).

Despite the ubiquity of these applications, the theory has had a tendency to make technical demands that seem excessive, even by homotopy theorists' standards. While there are several classical texts on the subject, e.g. \cite{clapp1981duality,clapp1984homotopy,crabb_james}, May and Sigurdsson's book \cite{ms} represents the first generation of results that integrate the classical techniques with the modern theory of diagram spectra from \cite{mmss}. It is hard to overstate their achievement -- they encounter pernicious technical issues that have no parallel in the non-parametrized theory, and introduce elaborate innovations to work around these problems. Shortly after their work, a second group of authors found a different approach based on quasicategories ($\infty$-categories) that avoided these problems. This led to considerable further development of parametrized homotopy theory and closely related topics, see e.g. \cite{abg,units_of_ring_spectra,infinity_cat_line_bundles,ayala2015factorization,barwick2016parametrized,harpaz_nuiten_prasma_parametrized}. In a sense, the main insight is that by Lurie's straightening/unstraightening \cite[2.2.1]{htt}, parametrized spectra are equivalent to diagrams of spectra, and diagrams are simpler to work with on a technical level.\footnote{It is not even mandatory to pass to $\infty$-categories to use this insight, see for instance \cite{vegetables}.}

This document, together with \cite{braunack2018rational,hebestreit_sagave_schlichtkrull}, might be regarded as a third generation. We are beginning to recognize a renewed need for good point-set models of parametrized spectra, in order to build examples, perform computations, and more directly relate the homotopy theory of spectra to the geometry of fiber bundles. This leads us back to the approach of \cite{ms}, with a desire to simplify and streamline so that we can hold onto as much geometry as possible. We use the insights gained in the intervening years to give more transparent proofs of the central results from the first 19 chapters of \cite{ms}, in many cases holding under weaker hypotheses.

We have not attempted to make these foundations complete. We don't discuss rationalization or $E_\infty$ structures, though the interested reader can find such a discussion in \cite{braunack2018rational} and \cite{hebestreit_sagave_schlichtkrull}, respectively. Instead we focus on those aspects of the theory that are most relevant in applications to classical topology and fixed-point theory, as developed for instance in \cite{clapp1981duality,ponto_asterisque,ponto_shulman_general}. We prove that parametrized spectra are related by a list of operations $f_!$, $f^*$, $f_*$, $\barsmash$, $\sma_B$, $\barsma{B}$, $\odot$, $\shad{}$, ... that are appropriately compatible with each other, both on the nose and in the homotopy category. We use this to relate the classical definition of the Lefschetz number $L(f)$ and Reidemeister trace $R(f)$ to the formal definition from \cite{ponto_asterisque}, along the way obtaining a convenient point-set formula for the fiberwise Reidemeister trace (see \autoref{thm:intro_reid_formula}, \autoref{cor:fiberwise_reidemeister_formula}).

Because of this focus, we can get away without a lot of technology. The author hopes this will make the subject accessible to a wider audience than is typical. In particular, we don't require the reader to know about $\infty$-categories. We don't even use model categories all that much -- classical arguments with cofibrations and weak equivalences get us most of the way there. We do assume that the reader has seen diagram spectra before, see \cite{mmss,schwede_symmetric_spectra}. Of course, once you have the foundations of parametrized spectra in place, you can always study them using any homotopy-theoretic technology you like, including model categories, homotopical categories, quasicategories, etc.

So what is a parametrized spectrum? Intuitively, it is an object over some base space $B$, which associates to each $b \in B$ a spectrum $X_b$ in a ``continuous'' way. Often, but not always, we insist that different points $b$ and $b'$ in the same path component of $B$ have fiber spectra $X_b, X_{b'}$ that are stably equivalent. Alternatively, a parametrized spectrum is a parametrized space over $B$ that has been ``stabilized.''

For a precise definition, recall that a spectrum $X$ is given by a sequence of based spaces $X_0, X_1, \dots, $ with bonding maps $\Sigma X_n \to X_{n+1}$. To make $X$ a parametrized spectrum over $B$, we instead ask that each space $X_n$ is a retractive space over $B$. This means there are maps $B \overset{i_n}\to X_n \overset{p_n}\to B$ composing to the identity of $B$. Therefore, for each $b \in B$, the fiber $p_n^{-1}(b) = (X_n)_b$ is a space with basepoint $i_n(b)$. We replace the bonding maps by maps of retractive spaces over $B$
\[ \Sigma_B {X}_n \ra {X}_{n+1}. \]
Here $\Sigma_B$ is an operation that takes the reduced suspension $\Sigma(-)$ of each of the fibers $(X_n)_b$. More precisely, it multiplies $X_n$ by the unit interval $I$, and then quotients out by the usual equivalence relation for $\Sigma(-)$ on each fiber separately. For each $b \in B$ we therefore get a map
\[ \Sigma (X_n)_b \to (X_{n+1})_b, \]
making the fibers into a spectrum $X_b := \{ (X_n)_b : n \geq 0 \}$. We have therefore achieved the philosophy that $X$ should consist of ordinary spectra $X_b$ that ``vary continuously in $b$.'' We can define parametrized versions of symmetric and orthogonal spectra just as easily.

In this document we give the details of an entire approach to the foundations of parametrized spectra. Since large sections are devoted to giving a cleaner treatment of existing results, we summarize our most significant innovations here.
\begin{thm}\label{thm:intro_q}(\autoref{thm:stable_model_structure})
	The ``$q$-model structure'' on parametrized orthogonal spectra (see \cite{ms}) exists, and is left and right proper.
\end{thm}
Furthermore on the $q$-cofibrant spectra, the reduced mapping cone has the right homotopy type, hence all of the expected results for cofiber sequences and colimits can be obtained from this model structure. In the language of \cite{ms}, it is ``well-grounded,'' though using a different ground structure than the one used in \cite{ms}. For a quick explanation of why \autoref{thm:intro_q} is possible, see \autoref{fixed_may_sigurdsson} and \autoref{fixed_may_sigurdsson_2}. We actually regard \autoref{thm:intro_q} as a new proof of a known result, because \cite{hebestreit_sagave_schlichtkrull} uses a more formal approach to establish a similar model structure for parametrized symmetric spectra.

The next result is a technical cornerstone of the present work. Its relies on several point-set topology results recalled in \autoref{sec:technical_lemmas} and a sequence of arguments in \autoref{sec:reedy} that generalizes earlier work from \cite[\S 4]{malkiewich2017coassembly}.
\begin{thm}\label{thm:intro_cof_fib}
	There is a notion of ``cofibrant'' and ``fibrant'' for parametrized orthogonal spectra with the following properties.
	\begin{itemize}
		\item Every spectrum is equivalent by a zig-zag to a spectrum that is cofibrant and fibrant.
		\item Every cofibrant spectrum $X$ has a natural cofibrant-and-fibrant replacement $PX$. Furthermore $P$ commutes with the external smash product, $PX \barsmash PY \cong P(X \barsmash Y)$.
		\item The external smash product $\barsmash$ preserves cofibrant spectra, equivalences of cofibrant spectra, and spectra that are both cofibrant and fibrant.
		\item For a map $f\colon A \to B$, the pullback $f^*$ preserves cofibrant spectra, fibrant spectra, and equivalences of fibrant spectra.
		\item The pushforward $f_!$ preserves cofibrant spectra and equivalences between them. It also preserves fibrant spectra if $f$ is a Hurewicz fibration.
	\end{itemize}
\end{thm}
This is proven in \autoref{prop:spectra_pushout_product} and \autoref{prop:stably_derived} -- the two notions are called ``freely $f$-cofibrant'' and ``level $h$-fibrant.''

\autoref{thm:intro_cof_fib} affords a new construction of the homotopy bicategory of parametrized spectra $\Ex$. Namely, we build a point-set version of $\Ex$, restrict to the cofibrant-fibrant spectra, and observe that the bicategory operations preserve both the cofibrant-fibrant spectra and the equivalences between them. Therefore they define a coherent set of operations on the homotopy category as well. This definition of $\Ex$ is significantly shorter and simpler than existing definitions (we show they are all canonically isomorphic in \autoref{thm:four_bicategories_of_spectra}). This part of the foundations is especially important because of its applications to fixed-point theory and algebraic $K$-theory, see for instance \cite{ponto_asterisque,ponto_shulman_indexed,ponto_shulman_general,ponto_shulman_mult,lind2019transfer,campbell_ponto}.

The following rigidity result is helpful when constructing $\Ex$ because it makes the coherence axioms trivial to check on the point-set level. The argument is a generalization of the one used in the author's thesis \cite[1.2, 3.17]{malkiewich_cyclotomic_dx}.
\begin{thm}\label{thm:intro_rigidity}(\autoref{thm:spectra_rigidity})
	Any functor of the form $f_!g^*(X_1 \barsmash \ldots \barsmash X_n)$ is rigid, i.e. the only automorphism is the identity, provided that $f$ is injective on each fiber of $g$.
\end{thm}

The above notions of cofibrant and fibrant do not fit into a model structure. So we can use them to derive the operations $\barsmash$, $f^*$ and $f_!$, but the commutation of the derived functors cannot be proven by restricting to cofibrant-fibrant objects and using Whitehead's theorem. Instead, we use a new framework for passing an isomorphism of composites of point-set functors to an isomorphism of their derived functors, in a way that respects composition (see \autoref{sec:composing_comparing}). This gives us another way to construct $\Ex$, which we find useful for the proof of Costenoble-Waner duality. Our framework differs from the one in \cite{shulman_comparing} in that it is more ad-hoc but can accommodate longer composites such as left- then right- then left-deformable functors.

Finally, we generalize and expand on the proof of Costenoble-Waner duality from \cite[18.5]{ms}. We prove two versions, one with neighborhoods and one with mapping cones, that hold under different assumptions. By playing these two versions off of each other, we obtain a point-set formula for the fiberwise Reidemeister trace in the sense of \cite{klein_williams,ponto_asterisque} that holds under broad assumptions.
\begin{thm}\label{thm:intro_reid_formula}(\autoref{cor:reidemeister_formula}, \autoref{cor:fiberwise_reidemeister_formula})
	When $X$ is a compact ENR, the Reidemeister trace $R(f)$ of a map $f\colon X \to X$ is the desuspension of the map of spaces
\[ \xymatrix @R=0em{
	S^n \ar[r] & S^n_\epsilon \sma (\Lambda^f X)_+ \\
	v \ar@{|->}[r] & \left(v - f(p(v))\right) \sma \left\{ \ t \mapsto p((1-t)f(p(v)) + tv) \ \right\}.
} \]
	Here $p$ is the retract onto $X$ of a neighborhood in $\R^n$, and $\Lambda^f X$ is the twisted free loop space
	\[ \Lambda^f X = \{ \ \gamma: [0,1] \to X : \ \gamma(0) = f\gamma(1) \ \}. \]
	More generally, when $q\colon E \to B$ is a fiber bundle with fiber and base both compact ENRs, the fiberwise Reidemeister trace $R_B(f)$ of a fiberwise map $f\colon E \to E$ is a map of spaces over $B$ given by the same formula, which varies continuously in $B$.
\end{thm}

\begin{figure}[h]
	\centering
	\def\svgwidth{0.85\columnwidth}
	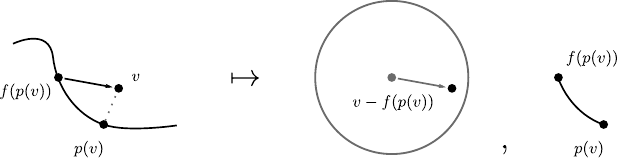
	\caption{A point-set formula for the Reidemeister trace.}\label{fig:reidemeister_formula}
\end{figure}
See also \autoref{fig:reidemeister_formula}. This is a fixed-point invariant that greatly generalizes the one-parameter Reidemeister trace of \cite{geoghegan_forum,gn_one_parameter}. This point-set formula was known to Klein, Ponto, and Williams around the time that \cite{ponto_asterisque} and \cite{klein_williams} were written, for smooth fiber bundles with cell complex base. But as far as we can tell, \autoref{thm:intro_reid_formula} seems to be its first appearance in the literature.

Just about every result generalizes to $G$-equivariant parametrized orthogonal spectra when $G$ is a finite group. With the exception of the technical lemmas at the beginning of \autoref{sec:G_reedy}, the arguments are essentially the same. We have therefore chosen to write the entire text nonequivariantly, and summarize the changes needed for the equivariant theory in \autoref{sec:G}.

\textbf{Acknowledgements.}
The author would like to give enormous thanks to John Lind, Peter May, and Kate Ponto for explaining the inside knowledge of how this theory is really supposed to go. Let's spread this knowledge more widely now. He thanks Mike Shulman for writing \cite{shulman_framed_monoidal} because we rely heavily on his perspective. He thanks Steve Ferry and Shmuel Weinberger for help and guidance concerning mapping cylinder neighborhoods and compact ANRs. Thanks also to Fabian Hebestreit and Steffen Sagave, who presciently suggested that the external smash products are really the more important ones in applications, and who were generous in discussing and sharing \cite{hebestreit_sagave_schlichtkrull} with the author while it was still being written. The author thanks the Hausdorff Institute in Bonn for their hospitality in July of 2017 when he started working on the $q$-model structure, and the Max Planck Institute in Bonn for their hospitality during the calendar year of 2018, during which the first six sections of this document were written.

The author would also like to thank David Carchedi, Madgalena Kedziorek, Inbar Klang, Mona Merling, Sean Tilson, Sarah Yeakel, and everyone else who showed him support and encouragement during the writing of this paper. This has been a large undertaking, motivated by the needs of a certain community, and it has not always been easy to keep going. The author hopes that this groundwork will support a diverse array of future projects by young mathematicians at the intersection of homotopy theory and fixed-point theory.

\beforesubsection
\subsection{How much category theory do we really need?}\aftersubsection

Here we briefly address concerns from two potential segments of our audience: those who worry that we are using too much category theory, and those who worry that we are using too little.

We begin with the first concern. Parametrized spectra are related by a lot of operations: external, internal, and relative smash products, pullbacks, pushforwards, composition products, shadows, base-change objects, $\ldots$. Each of these operations can also be \emph{derived}, meaning that it is modified so that it preserves weak equivalences (or gives an operation on the homotopy category). This is a lot of information to keep track of, before we even get to the relationships between these operations. To avoid getting lost, we need some kind of language that helps us stay organized. We use the notion of a ``bifibration'' and a ``bicategory,'' which come from fibered category theory and 2-category theory, respectively. Though abstract, they turn out to be just what we need to collect and organize these operations in a sensible way.

These notions are not familiar to everyone, so we have tried our best to make them accessible. We assure the reader interested in fixed-point theory that these subjects do not really have to be mastered getting started. For now, you can get along just fine without sections \autoref{sec:SMBF}, \autoref{sec:deriving_SMBFs}, or any of the sections that use model categories. The category theory contained in sections \autoref{sec:just_derived_functors} and \autoref{sec:duality} should be sufficient if all you want to do is start seriously working with traces in parametrized spectra.

Now for the second concern. In this paper, we focus a lot of our attention on homotopy categories. In particular, we build a homotopy bicategory of parametrized spectra, which is a direct generalization of the homotopy category of spectra with its symmetric monoidal structure. From the point of view of pure homotopy theory, this means that our treatment is incomplete. Passing to homotopy categories throws away the ``higher'' structure that one needs to take pushouts, homotopy colimits, bar constructions, etc.

Our primary justification for doing this is that duality theory and traces already work in a satisfactory way on the level of the homotopy category, and these are the notions we need to form the applications to fixed-point theory. Given the scope of the audience we hope to attract, it makes sense to keep the technology at the minimum level needed to make the applications work.

Of course, we cannot anticipate all future applications. So we have taken care to record how each of our homotopy categories comes from a point-set category with weak equivalences, and each operation comes from a point-set functor that preserves the weak equivalences. In principle, this is all the information one needs to build homotopy-coherent refinements of the constructions in this paper. Though, this may require significant technical work, and perhaps also some new insights. 
Even at the level of homotopy categories, there is a non-trivial amount of coherence machinery for understanding how we form the bicategory of parametrized spectra from other related structures \cite{mp2}.

The foundations in this paper are specifically adapted for applications beyond homotopy theory -- the goal is to pass as easily as possible between the homotopy theory and explicit point-set models, because we care about questions that concern the point-set models. Philosophically, the program of reformulating homotopy theory using $\infty$-categories travels in a different direction than this, away from the use of point-set models and towards a language in which every noun and verb is a homotopy-invariant notion. There is a tension between these two points of view, but it is a pleasant tension and it need not be resolved. It is healthy for a field to grow outwards in different directions at the same time.

\newpage

\section{Parametrized spaces}\label{sec:spaces}

\subsection{Basic definitions}\aftersubsection

Let $B$ be a topological space. A \textbf{retractive space}\index{retractive space} or ``ex-space'' over $B$ consists of a space $X$, called the total space, along with an inclusion and a projection map $B \overset{i}\to X \overset{p}\to B$ that compose to the identity of $B$. It is helpful to think of $X$ as a family of spaces $X_b = p^{-1}(b)$, each of which has a basepoint $i(b) \in X_b$. For this reason we often call $i$ the \textbf{basepoint section}\index{basepoint section} of $X$.

If $X \to B$ is a map with no basepoint section, we define $X_{+B}$ to be the disjoint union $X \amalg B$, with projection given by $p$ and $\id_B$, and with the canonical inclusion of $B$.

As usual, $B$ and $X$ are not plain-vanilla topological spaces. Recall that $X$ is \textbf{compactly generated} when its closed sets are determined by their preimages in every compact Hausdorff space $K$ mapping into $X$. Furthermore $X$ is \textbf{weak Hausdorff} if the image of every such $K \to X$ is closed. See \cite[App A]{lewis_appendix}, \cite{strickland2009category}, and \cite[App A]{schwede_global} for basic properties of such spaces. In this paper we adopt the following two conventions.
\begin{itemize}
	\item \textbf{(CGWH)}\index{CGWH} Every base space $B$ and total space $X$ is compactly generated weak Hausdorff.
	\item \textbf{(CG)}\index{CG} Every base space $B$ is compactly generated weak Hausdorff, but the total space $X$ only has to be compactly generated.
\end{itemize}
Most of the theory is exactly the same in both cases. At the points where it is different, we will use the tags (CGWH) and (CG) to distinguish what happens in each case. The reader can therefore pick either convention and learn the entire theory following just that convention.\footnote{At the end, the two model categories parametrized spectra we get are Quillen equivalent, so the difference does not matter for the purposes of homotopy theory.} Note that \cite{ms} is written in (CG).

\begin{ex}\label{ex:trivial}
	If $K$ is any based space then $B \times K$ is a retractive space over $B$. If $K$ is an unbased space then we have a homeomorphism $B \times (K_+) \cong (B \times K)_{+B}$.
\end{ex}

\begin{ex}\label{ex:fiberwise_thom_space}
	Let $V \to B$ be an $n$-dimensional real vector bundle with an inner product, and $S(V)$ and $D(V)$ its associated unit sphere and disc bundles. The \textbf{Thom space} $\Th(V)$ is the quotient $D(V)/S(V)$. One disadvantage of this classical construction is that it loses the information of the parametrization over $B$. To correct for this, define the \textbf{fiberwise Thom space}\index{fiberwise!Thom space $\Th_B(V)$} $\Th_B(V)$ to be the fiberwise quotient, i.e. the pushout
	\[ \xymatrix @R=2em{
		S(V) \ar[r] \ar[d] & D(V) \ar[d] \\
		B \ar[r] & \Th_B(V).
	} \]
	Then $\Th_B(V)$ is a retractive space over $B$.
\end{ex}

Retractive spaces over $B$ form a category $\mc R(B)$. The morphisms or ``ex-maps'' from $(B,X,i_X,p_X)$ to $(B,Y,i_Y,p_Y)$ are the continuous maps $f\colon X \to Y$ commuting with the inclusion maps $i$ and projection maps $p$. In other words, $f \circ i_X = i_Y$ and $p_Y \circ f = p_X$, or the following diagram commutes.\index{retractive space!morphisms of}
\[ \xymatrix @R=1.2em{
	& B \ar[ld]_-{i_X} \ar[rd]^-{i_Y} & \\
	X \ar[rr]^-f \ar[rd]_-{p_X} && Y \ar[ld]^-{p_Y} \\
	& B &
} \]

The category $\mc R(B)$ has a zero object, and all colimits and limits. The zero object is $B$, as a space over itself along the identity map. We form the colimit of a diagram $\cat I \to \mc R(B)$ by taking the colimit of the total spaces, then further quotienting their basepoint sections together.\index{colimit} Dually, the limit of the diagram maps to a product of copies of $B$, and we pull back to $B$ along the diagonal map.\index{limit} If the diagram is ``connected,'' this extra step of quotienting the basepoint sections or pulling back to $B$ is not necessary. So in particular: \\

\centerline{
\begin{tabular}{lcr}
	The coproduct of $X$ and $Y$ && is the union $X \cup_B Y$. \\
	The product of $X$ and $Y$ && is the fiber product $X \times_B Y$. \\
	The pushout of $Y \leftarrow X \to Z$ && is the usual pushout $Y \cup_X Z$. \\
	The pullback of $Y \to W \leftarrow Z$ && is the usual pullback $Y \times_W Z$.
\end{tabular}
}

A map of retractive spaces $f\colon X \to Y$ is a \textbf{fiberwise (based) homotopy equivalence}\index{fiberwise!homotopy equivalence} if it has a homotopy inverse in $\mc R(B)$, meaning that the compositions are homotopic to the identity along homotopies of maps of retractive spaces. We say that $f$ is a \textbf{weak equivalence}\index{weak equivalence} (or $q$-equivalence) if the map of total spaces $X \to Y$ is a weak homotopy equivalence. So it induces isomorphisms on homotopy groups at all basepoints.\footnote{Other kinds of weak equivalences are possible. One studied in \cite{intermont_johnson} measures the equivalences by fiberwise maps out of $S^n \times U$ where $U \subseteq B$ is an open subset of the base.}

A map $f\colon X \to Y$ of retractive spaces over $B$ is a fibration if the map of total spaces $X \to Y$ is a fibration. This can mean several things:
\begin{itemize}
	\item A \textbf{Hurewicz fibration} or \textbf{$h$-fibration}\index{$h$-fibration} $X \to Y$ is a map satisfying the homotopy lifting property. In other words the projection map $X^I \to X \times_Y Y^I$ has a section. We call this a path-lifting function.
	\item A \textbf{Serre fibration} or \textbf{$q$-fibration}\index{$q$-fibration} $X \to Y$ is a map satisfying the usual lifting property with respect to cylinders on discs:
	\[ \xymatrix @R=1.7em{
		D^n \times \{0\} \ar[r] \ar[d] & X \ar[d] \\
		D^n \times I \ar[r] \ar@{-->}[ur] & Y
	}\]
	\item A \textbf{quasifibration}\index{quasifibration} is a map $X \to Y$ such that for each $y \in Y$, the inclusion of the fiber $X_y$ into the homotopy fiber $X \times_Y (Y,\{y\})^{(I,\{1\})}$ is a weak equivalence.
\end{itemize}

We say the retractive space $X$ is \textbf{$h$-fibrant} when the projection $p\colon X \to B$ is an $h$-fibration. We can always replace $X$ by the $h$-fibrant space $X \times_B B^I$.

If $X$ and $Y$ are fibrant in any of the above senses, a map $X \to Y$ is a weak equivalence precisely when the map of fibers $X_b \to Y_b$ is a weak equivalence for every $b \in B$. It suffices to check this for one $b$ in every path component of $B$.

Cofibrations come next. A \textbf{closed inclusion} or \textbf{$i$-cofibration}\index{$i$-cofibration} is any map $X \to Y$ that is a homeomorphism onto its image, which is closed in $Y$. In (CGWH), the basepoint section is always a closed inclusion, so every space is $i$-cofibrant.

A map of retractive spaces $X \to Y$ is an \textbf{$h$-cofibration}\index{$h$-cofibration} if it is a closed inclusion, and the map of total spaces has the homotopy extension property. In other words, the inclusion
\[ (X \times I) \cup_{(X \times \{0\})} (Y \times \{0\}) \to (Y \times I) \]
has a retract. In (CGWH) we are allowed to drop ``closed inclusion'' from the definition because it follows from the homotopy extension property.

As usual, the map $X \to Y$ is an $h$-cofibration iff we can identify $X$ as a closed subspace of $Y$ and give $(Y,X)$ the structure of an NDR-pair (see e.g. \cite[\S 6.4]{concise}). This is a function $u: Y \to [0,1]$ with $X = u^{-1}(0)$, and a homotopy $h\colon Y \times I \to Y$, constant on $X$, from the identity map of $Y$ to a map that retracts the open neighborhood $U = u^{-1}([0,1))$ back onto $X$. Without loss of generality, we may choose $h$ so that it brings each point $y \in U$ to $X$ at time $u(y)$, after which it stays stationary \cite[Lem 4]{strom_2}.

The next definition will be different because it is not enough to look at the total spaces. The map $X \to Y$ of retractive spaces is an \textbf{$f$-cofibration}\index{$f$-cofibration} if it is a closed inclusion and has the fiberwise homotopy extension property.\footnote{In \cite{ms} these are called $\bar{f}$-cofibrations; in \cite{crabb_james} they are called ``closed fibrewise cofibrations.''} This means the inclusion
\[ (X \times I) \cup_{(X \times \{0\})} (Y \times \{0\}) \to (Y \times I) \]
has a retract \emph{that respects} the projection to $B$. Of course, this is stronger than being an $h$-cofibration. And it is equivalent to finding an NDR-pair structure on $(Y,X)$ such that the homotopy $h$ respects the projection to $B$.

For future reference, we will use the term \textbf{$q$-cofibration}\index{$q$-cofibration} for a map in $\mc R(B)$ that is a retract of a relative cell complex. However, these won't appear until the section on cofibrantly-generated model categories, \autoref{sec:cof_gen_model_cats}. The cells $D^n$ are equipped with maps to $B$, and the inclusions $S^{n-1} \to D^n$ are certainly $h$-cofibrations, but they do not have to be $f$-cofibrations in general. 

The following chart summarizes all of the classes of maps we have defined.
\[ \xymatrix @R=.5em @C=2em{
	\textup{fiberwise htpy equivalences} \ar@{=>}[r] & \textup{weak equivalences} \\
	h\textup{-fibrations} \ar@{=>}[r] & q\textup{-fibrations} \ar@{=>}[r] & \textup{quasifibrations} \\
	f\textup{-cofibrations} \ar@{=>}[r] & h\textup{-cofibrations} \ar@{=>}[r] & \textup{closed inclusions} \\
	& q\textup{-cofibrations} \ar@{=>}[u] &
} \]

\begin{ex}
	If $K$ is a well-based space then the retractive space $B \times K$ from \autoref{ex:trivial} is an $f$-cofibrant, $h$-fibrant space over $B$. The fiberwise Thom space $\Th_B(V)$ from \autoref{ex:fiberwise_thom_space} is always $f$-cofibrant and $h$-fibrant, using \autoref{prop:clapp} below.
\end{ex}

\beforesubsection
\subsection{Technical lemmas}\label{sec:technical_lemmas}\aftersubsection

\begin{prop}\label{prop:technical_cofibrations}\hfill
	\vspace{-1em}
	
	\begin{enumerate}
		\item Each notion of ``cofibration'' in the category $\mc R(B)$ is closed under pushout, transfinite compositions (therefore also coproducts), and retracts.
		\item (Gluing lemma) Given a weak equivalence of pushout diagrams
		\[ \xymatrix @R=1.7em{
			Y \ar[d]^-\sim & X \ar[l] \ar[r] \ar[d]^-\sim & Z \ar[d]^-\sim \\
			Y' & X' \ar[l] \ar[r] & Z',
		} \]
		if each one has one leg that is an $h$-cofibration, the map of pushouts
		\[ \xymatrix{ Y \cup_X Z \ar[r]^-\sim & Y' \cup_{X'} Z' } \]
		is also a weak equivalence.
		\item (Colimit lemma) Given a weak equivalence of sequential colimit diagrams
		\[ \xymatrix @R=1.7em{
			X_0 \ar[d]^-\sim \ar[r] & X_1 \ar[d]^-\sim \ar[r] & X_2 \ar[d]^-\sim \ar[r] & \ldots \\
			X_0' \ar[r] & X_1' \ar[r] & X_2' \ar[r] & \ldots, \\
		} \]
		if each map in each colimit diagram is an $h$-cofibration,\footnote{In (CGWH) the maps of the colimit system only have to be closed inclusions.} the map of colimits
		\[ \xymatrix{ \colim_n X_n \ar[r]^-\sim & \colim_n X_n' } \]
		is also a weak equivalence.
	\end{enumerate}
\end{prop}

\begin{proof}
	Reduces to the non-parametrized version by looking at the total space. Note that the colimits here, when computed in $\mc R(B)$, are the same as the corresponding colimits in the category of spaces.
\end{proof}

The dual of this proposition also holds with $h$-fibrations or with $q$-fibrations, but not with quasifibrations.\footnote{The pullback of a quasifibration is not always a quasifibration. Equivalently, a pullback square with one leg a quasifibration may not be a homotopy pullback. There's an excellent discussion of this in \texttt{http://www.lehigh.edu/{~}dmd1/tg516.txt}, which gives sufficient conditions for a quasifibration to be preserved under pullback.}

\begin{cor}\label{cofibration_of_pushouts}
	For each notion of ``cofibration,'' a map of pushout diagrams
	\[ \xymatrix @R=1.7em{
		C \ar[d] & A \ar[l] \ar[r] \ar[d] & B \ar[d] \\
		C' & A' \ar[l] \ar[r] & B' } \]
	induces a cofibration on the pushouts, provided that both $B \to B'$ and $C \cup_A A' \to C'$ are cofibrations.
\end{cor}
\begin{proof}
	The desired map factors as
	\[ \xymatrix @R=1.7em{
		B \cup_A C \ar@{ >->}[r] & B' \cup_A C \ar@{ >->}[r] & B' \cup_{A'} C',
	} \]
	the first leg is a pushout of $B \to B'$ and the second is a pushout of $C \cup_A A' \to C'$.
\end{proof}

\begin{lem}\label{lem:product_ndr}
	If $A \subseteq X$ and $A' \subseteq X'$ are closed NDR-pairs then the map
	\[ \xymatrix{ (A \times X') \cup_{A \times A'} (X \times A) \ar[r] & X \times X' } \]
	also has the structure of a closed NDR-pair. The same is true for fiberwise closed NDR-pairs over $B$, or for just closed inclusions.
\end{lem}

\begin{proof}
	Follows because a product of closed inclusions is a closed inclusion \cite[A.7.1(a)]{lewis_appendix},
%
	together with the usual formulas for an NDR-pair structure on a product, \cite[\S 6.4]{concise}, \cite[Thm 6]{strom_2}.
\end{proof}

\begin{lem}\label{lem:dold}
	If $X \to Y$ is an $f$-cofibration of $h$-fibrant spaces over $B$, then the path-lifting function for $Y$ can be chosen so that it restricts to a path-lifting function of $X$.
\end{lem}
\begin{proof}
	It suffices to produce a lift
	\[ \xymatrix @R=1.7em{
		Y \times_B B^I \times \{0\} \cup_{X \times_B B^I \times \{0\}} X \times_B B^I \times I \ar[r] \ar[d] & Y \ar[d] \\
		Y \times_B B^I \times I \ar[r] \ar@{-->}[ur] & B,
	} \]
	where the bottom horizontal sends $\gamma \in B^I$ and $t \in I$ to $\gamma(t)$, and the top horizontal is the path-lifting function of $X$ and the identity of $Y$. To get the lift, compose the fiber-preserving retract of the left-hand vertical map with the top horizontal map.
\end{proof}

The following proposition is a technical cornerstone for us. It is classical by now, but not terrifically well-known, and so we reproduce the idea of the proof here. In essence it says that fibrations can be glued together to form new fibrations.

\begin{prop}[Clapp]\cite[Prop 1.3]{clapp1981duality}\label{prop:clapp}
	Given a pushout square of spaces over $B$
	\[ \xymatrix @R=1.7em{
		X \ar[r]^-i \ar[d]^-f & Y \ar[d] \\
		Z \ar[r] & Y \cup_X Z
	}, \]
	if $X$, $Y$, and $Z$ are all $h$-fibrant and $i$ is an $f$-cofibration then the pushout $Y \cup_X Z$ is also $h$-fibrant.
\end{prop}

\begin{proof}
	(Sketch) By \autoref{lem:dold}, we can take the path-lifting functions $L_X: X \times_B B^I \to X^I$ and $L_Y: Y \times_B B^I \to Y^I$ so that they agree along $i$. We also have a third, unrelated path-lifting function $L_Z: Z \times_B B^I \to Z^I$.  To glue these together to a path-lifting function for $Y \cup_X Z$, we use $L_Z$ on $Z$, and a combination of $L_Y$ and $L_Z$ on $Y \setminus X$.
	
	Specifically, we use a presentation of $X \to Y$ as an NDR-pair to assign every point $y \in Y$ to a path $h(t,y)$, such that the points close to $X$ are sent to a path that goes into $X$ and then remains constant in $X$ for some time (more time as we get closer to $X$). We then define the path lifting on $Y \setminus X$ as follows: first apply $L_Y$ to get a path $\sigma$ in $Y$. Then apply $h(t,-)$ to $\sigma$ for varying $t$ to get a square in $Y$. Traverse this square diagonally. If we are far from $X$ there is nothing more to say. If we are close to $X$ then at some point in the middle, our lifted path lands in $X$. When this happens, we switch over and lift the remainder of our path in $B$ using $L_Z$ instead. (The point at which we switch over is given by a formula involving $u$.) This recipe actually makes sense for all of $Y$, and on the subspace $X$ it returns $L_Z$, so our choices of path-lifting now glue together to give a single path-lifting function on $Y \cup_X Z$.
\end{proof}

The remaining results in this section are more standard, so we omit the proofs. It turns out that in the above proposition we can weaken the assumption of an $f$-cofibration to an $h$-cofibration. For this we recall the following two results. Recall that a DR-pair $(Y,X)$ is an NDR-pair in which the neighborhood $U = u^{-1}([0,1))$ is in fact all of $Y$.
\begin{prop}[Str\o m]\label{prop:strom_lift}\cite[Thm 3]{strom_1}
	\begin{enumerate}
		\item Given a square
		\[ \xymatrix @R=1.7em{
			K \times I \cup_{K \times 0} L \times 0 \ar[r] \ar[d]^-i & E \ar[d] \\
			L \times I \ar[r] \ar@{-->}[ur] & B
		}\]
		in which $K \to L$ is an NDR-pair and $E \to B$ is an $h$-fibration, a lift exists.
		\item Given a square
		\[ \xymatrix @R=1.7em @C=3em{
			A \ar[r] \ar[d]^-i & E \ar[d] \\
			X \ar[r] \ar@{-->}[ur] & B
		}\]
		in which $A \to X$ is a DR-pair and $E \to B$ is an $h$-fibration, a lift exists.
	\end{enumerate}
\end{prop}

%

\begin{lem}[Heath and Kamps]\cite[1.3]{heath_kamps}\label{lem:heath_kamps}
	An $h$-cofibration $A \to X$ between $h$-fibrant spaces over $B$ is an $f$-cofibration.
\end{lem}

Therefore \autoref{prop:clapp} applies with ``$f$-cofibration'' replaced by ``$h$-cofibration.'' Of course the set of maps to which the result applies has not been enlarged. It is merely easier to check that a given map is in that class.

We'll also use the fact that sometimes a pullback of an $h$-cofibration is still an $h$-cofibration.

\begin{prop}[Str\o m]\cite[Thm 12]{strom_2}\label{prop:strom_pullback}
	In a pullback square
	\[ \xymatrix @R=1.7em{
		E|_A \ar[r] \ar[d] & E \ar[d] \\
		A \ar[r] & B
	} \]
	if $A \to B$ is an $h$-cofibration and $E \to B$ is an $h$-fibration then $E|_A \to E$ is also an $h$-cofibration.
\end{prop}

%

Finally we lift a standard result about closed inclusions to $\mc R(B)$.
\begin{lem}\label{lem:closed_inclusion_equalizer}
	(CGWH) The map $f\colon X \to Y$ in $\mc R(B)$ is a closed inclusion iff it is the equalizer of some pair of maps $Y \rightrightarrows Z$.\footnote{In (CG), the equalizer of two maps is only an inclusion. Ordinary inclusions are not as well-behaved as closed inclusions, for instance inclusions into $X$ are not preserved by pushouts along their intersection.}
\end{lem}
\begin{proof}
	An equalizer in unbased spaces is always a subspace of $Y$, and the weak Hausdorff condition guarantees that it is closed \cite[7.6]{lewis_appendix}. Conversely, if $f$ is a closed inclusion, take $Z = Y \cup_X B$. Since this pushout respects the underlying sets, as a set $Z$ is the disjoint union of $Y \setminus X$ and $B$. Then the subspace of $Y$ equalizing the quotient map $Y \to Z$ and the composite $Y \to B \to Z$ is precisely $X$.
\end{proof}

For completeness we recall the fiberwise Whitehead theorem, though our treatment of the foundations does not use it.
\begin{thm}\cite[2.7]{clapp1981duality}\label{thm:fiberwise_whitehead}
	If $f\colon X \to Y$ is a map of $h$-cofibrant, $h$-fibrant retractive spaces over $B$, that is a homotopy equivalence on the total spaces, then $f$ is a fiberwise homotopy equivalence.
\end{thm}

\beforesubsection
\subsection{Pullback and pushforward}\aftersubsection

If $f\colon A \to B$ is any map of spaces, the \textbf{pullback functor}\index{pullback $f^*$} $f^*\colon \mc R(B) \to \mc R(A)$ sends each retractive space $X$ over $B$ to the pullback $f^*X = A \times_B X$:
\[ \xymatrix @R=1.7em @C=3em{
	f^*X \ar[r] \ar[d] & X \ar[d]^-{p_X} \\
	A \ar[r]^-f & B
} \]
When $f$ is the inclusion of a subspace, $f^*$ restricts each retractive space to the subspace that lies over $A$.

To make $f^*X$ a retractive space we have to choose inclusion and projection maps. The projection map is obvious, and we define the inclusion by $a \mapsto (a,i_X(f(a)))$. In other words, we apply the universal property of the pullback to the following commuting square.
\[ \xymatrix @R=1.7em @C=3em{
	A \ar@{-->}[r]^-{i_X \circ f} \ar@{-->}[d]_-{\id_A} & X \ar[d]^-{p_X} \\
	A \ar[r]^-f & B
} \]
The universal property also assigns to each map $X \to Y$ of retractive spaces over $B$ a map $f^*X \to f^*Y$ of retractive spaces over $A$. This makes $f^*$ into a functor.
\begin{lem}\label{lem:f_star_preserves}
	The pullback functor $f^*$ preserves \vspace{-.5em}
	\begin{itemize}
		\item $h$-fibrations and $q$-fibrations,
		\item weak equivalences between $q$-fibrant (or $h$-fibrant) spaces,
		\item $f$-cofibrations, and closed inclusions.
	\end{itemize}
	Moreover if $f$ itself is a Hurewicz fibration ($h$-fibration), then $f^*$ preserves $h$-cofibrations, and all weak equivalences.
\end{lem}

\begin{proof}
	Most of these are straightforward exercises. The part with closed inclusions uses \cite[A.10.1]{lewis_appendix} and the part with $h$-cofibrations uses \autoref{prop:strom_pullback}.
\end{proof}

\begin{rmk}
	Since $f^*$ is defined by a universal property, it is not unique, only unique up to canonical isomorphism. So technically, there are many pullback functors for each map $f$. 
\end{rmk}

Continuing to let $f\colon A \to B$ be any map of spaces, the \textbf{pushforward functor}\index{pushforward $f_{^^21}$} $f_!\colon \mc R(A) \to \mc R(B)$ sends each retractive space $X$ over $A$ to the pushout $f_!X = B \cup_A X$:
\[ \xymatrix @R=1.7em @C=3em{
	A \ar[r]^-f \ar[d]_-{i_X} & B \ar[d] \\
	X \ar[r] & f_!X
} \]
When $f$ is the inclusion of a subspace, $f_!$ takes a retractive space over $A$ and tacks on the complement $B \setminus A$, by gluing on $B$ along $A$.

\begin{rmk}
	In both (CG) and (CGWH), the underlying set of $f_! X$ is the pushout of the underlying sets of $B$ and $X$ along $A$. Similarly, the pullback $f^* X$ is also the pullback on the underlying sets. 
\end{rmk}

To make $f_!X$ a retractive space we take the obvious inclusion of $B$, and define the projection to $B$ by $b \mapsto b$, $x \mapsto f(p_X(x))$. In other words, we apply the universal property of the pushout to the following commuting square.
\[ \xymatrix @R=1.7em @C=3em{
	A \ar[r]^-f \ar[d]_-{i_X} & B \ar@{-->}[d]^-{\id_B} \\
	X \ar@{-->}[r]_-{f \circ p_X} & B
} \]
\begin{lem}\label{lem:f_shriek_preserves}
	The pushforward functor $f_!$ preserves\vspace{-.5em}
	\begin{itemize}
		\item $f$-cofibrations, $h$-cofibrations, closed inclusions, and
		\item weak equivalences between $h$-cofibrant (or $f$-cofibrant) spaces.
	\end{itemize}
	Moreover if $f$ itself is a Hurewicz fibration ($h$-fibration), then $f_!$ preserves spaces that are both $h$-fibrant and $f$-cofibrant.
\end{lem}

\begin{proof}
	Again this is straightforward. The last claim follows from \autoref{prop:clapp}.
\end{proof}

\begin{ex}
	In \autoref{ex:trivial}, the space $B \times K$ is isomorphic to the pullback $r^*K$ along the unique map $r\colon B \to *$. Its pushforward to a point is a half-smash product
	\[ r_!r^*K \cong r_!(B \times K) \cong (B_+) \sma K. \]
\end{ex}
\begin{ex}\label{ex:thom_base_change}
	In \autoref{ex:fiberwise_thom_space}, the ordinary Thom space $\Th(V)$ can be recovered as a pushforward of the fiberwise Thom space $\Th_B(V)$\index{fiberwise!Thom space}
	\[ r_!\Th_B(V) \cong \Th(V). \]
	Moreover if $f^*V$ denotes the usual pullback of a vector bundle $V$ along $f\colon A \to B$, we have an isomorphism
	\[ \Th_A(f^*V) \cong f^*\Th_B(V). \]
\end{ex}

\begin{lem}
	There is an adjunction $(f_! \adj f^*)$.
\end{lem}
\begin{proof}
	If $X \in \mc R(A)$ and $Y \in \mc R(B)$, the data of a map $f_!X \to Y$ in $\mc R(A)$ or a map $X \to f^*Y$ in $\mc R(B)$ rearranges into a choice of map $X \to Y$ making the following diagram commute.
	\[ \xymatrix @R=1.7em @C=3em{
		A \ar[r]^-f \ar[d]_-{i_X} & B \ar[d]^-{i_Y} \\
		X \ar[d]_-{p_X} \ar@{-->}[r] & Y \ar[d]^-{p_Y} \\
		A \ar[r]^-f & B
	} \]
	Some diagram-juggling is then required to check that this respects maps in $X$ and $Y$, and therefore gives an adjunction.
\end{proof}

\begin{rmk}
	The above proof illustrates how natural it is to consider the category $\mc R$ of all retractive spaces \emph{over all possible base spaces}, with morphisms defined by the diagram above. The functor $\mc R \to \cat{Top}$ that forgets down to the base space is an example of a ``bifibration,'' see \autoref{sec:SMBF} or \cite{ponto_shulman_indexed}.
\end{rmk}


Since $(f_! \adj f^*)$ is an adjoint pair, $f_!$ preserves all colimits and $f^*$ preserves all limits. In addition, $f^*$ preserves many of the colimits we actually care about.
\begin{prop}\label{prop:f_star_colimits}\hfill
	\vspace{-.5em}
	
	\begin{itemize}
		\item (CGWH) $f^*$ preserves the following colimits: coproducts, pushouts along a closed inclusion, sequential colimits along closed inclusions, and quotients by fiberwise actions of compact Hausdorff topological groups.
		\item (CG) $f^*$ preserves all colimits.
	\end{itemize}
\end{prop}
\begin{proof}
	The (CG) statement follows by analyzing the underlying sets, see also \cite[10.3]{lewis_appendix}. In (CGWH) it suffices to prove this statement in the category of spaces over $B$, then pass to retractive spaces (where the colimits change, by a pushout along $\amalg B \to B$). For the first step, the colimits listed above as constructed in weak Hausdorff spaces agree with the colimit in ordinary (or $k$-)spaces \cite{strickland2009category}, and the result is still weak Hausdorff, so it follows that they commute with $f^*$. For the second step, the pushout has one leg a closed inclusion because every (CGWH) retractive space is $i$-cofibrant. Therefore this follows from the first step.
\end{proof}

Finally, we use this adjunction to form the Beck-Chevalley isomorphism, which allows us to switch pullbacks and pushforwards past each other.
\begin{prop}[Beck-Chevalley]\label{prop:beck_chevalley_spaces}
	Given a commuting square of spaces
	\begin{equation}\label{diagonal_square}
	\xymatrix@R=1em @C=1em{
		& A \ar[ld]_-g \ar[rd]^-f & \\
		B \ar[rd]_-p && C \ar[ld]^-q \\
		& D &
	}
	\end{equation}
	There is a canonical natural transformation of functors $\mc R(C) \to \mc R(B)$ from the ``top route'' to the ``bottom route''
	\[ g_!f^* \Rightarrow p^*q_! \]
	given by pasting together units and counits of adjunctions, according to either of the three recipes illustrated below.\index{Beck-Chevalley isomorphism}
	\[ \xymatrix@R=1em @C=1em{
		&& A \ar[lld]_-{g_!} \ar@{<-}[ld]^-{g^*} \ar@{<-}[rd]_-{f^*} & \\
		B \ar@{=}[r] &B \ar@{<-}[rd]^-{p^*} && C \ar@{<-}[ld]_-{q^*} \ar@{=}[r] & C \ar[lld]^-{q_!} \\
		&& D & & 
	}
	\quad
	\xymatrix@R=1em @C=1em{
		&& A \ar[ld]_-{g_!} \ar@{<-}[rd]^-{f^*} & \\
		B \ar@{=}[r] \ar@{<-}[rd]_-{p^*} &B \ar[d]^-{p_!} && C \ar@{<-}[d]_-{q^*} \ar@{=}[r] & C \ar[ld]^-{q_!} \\
		& D \ar@{=}[rr] & & D & 
	}
	\quad
	\xymatrix@R=1em @C=1em{
		&& A \ar@{<-}[rrd]^-{f^*} \ar[ld]^-{g_!} \ar[rd]_-{f_!} && \\
		B \ar@{<-}[rrd]_-{p^*} &B \ar[rd]^-{p_!} \ar@{=}[l] && C \ar[ld]_-{q_!} &  C \ar@{=}[l] \\
		&& D & &
	} \]
	If \eqref{diagonal_square} is a pullback square, this transformation is an isomorphism.
\end{prop}

\begin{proof}
	The proof that the two non-symmetric recipes both give the same map as the symmetric one is an exercise with the triangle identities for an adjunction, and uses the above-mentioned fact that our adjunctions are compatible with compositions. 
	
	To prove it's an isomorphism, take any $X \in \mc R(C)$. If we think of $X$ as a space over $D$ (not a retractive space, just a space), $p^*$ takes it to a space over $B$, but the result is canonically isomorphic to $g^*X$ because $g$ is the pullback of $p$. Furthermore, $q_!X = D \cup_C X$ is a pushout along the closed inclusion $i_X\colon C \to X$. Therefore $p^*$ takes it to the pushout
	\[ p^*D \cup_{p^* C} p^*X \cong B \cup_A g^*X = f_!g^*X. \]
	This proves that some isomorphism $g_!f^* \cong p^*q_!$ exists. Put this isomorphism into the left-vertical in the square below, and fill out the rest of the square using the Beck-Chevalley recipe where the square in the middle consists entirely of pushforwards:
	\[ \xymatrix{
		B \cup_A (A \times_C X) \ar[d]_-\cong \ar[r] & B \times_D (D \cup_B B \cup_A (A \times_C X)) \ar[d]^-\cong \\
		B \times_D (D \cup_C X) & B \times_D (D \cup_C C \cup_A (A \times_C X)) \ar[l]
	} \]
	It suffices to check this commutes on $B$ and on $A \times_C X$. On $B$ it is automatic because all maps are maps of retractive spaces. On $A \times_C X$ the left-vertical is given by the formula $(a,x) \mapsto (g(a),x) \in B \times_D X$. The others are given by
	\[ \xymatrix{
		(a,x) \in A \times_C X \ar[r] & (g(a),a,x) \in B \times_D A \times_C X \ar[d]^-\cong \\
		(g(a),x) \in B \times_D X & (g(a),a,x) \in B \times_D A \times_C X \ar[l]
	} \]
	which finishes the proof. By the way, we will see later (\autoref{prop:spaces_rigidity}) that this is the only natural isomorphism $g_!f^* \cong p^*q_!$ that exists.
\end{proof}

\beforesubsection
\subsection{Sections}\label{sec:sections}\aftersubsection

There is a third way to change base spaces. If $X \in \mc R(B)$, its \textbf{space of sections}\index{sections $\Gamma_B(-)$} $\Gamma_B(X)$ consists of all the maps $s\colon B \to X$ such that the composite $p_X \circ s\colon B \to X \to B$ is the identity. This is a based space with basepoint given by $i_X$. In other words $\Gamma_B(X) \in \mc R(*)$.

More generally, along any fiber bundle $f\colon A \to B$, we can send a parametrized space $X \in \mc R(A)$ to the space $f_*X \in \mc R(B)$ whose fiber over $b \in B$ is the space of sections of $X$ over the subspace $f^{-1}(b)$ of $A$. The details of this are in the next proposition.

\begin{prop}\label{prop:sheafy_pushforward}
	If $f\colon A \to B$ is a fiber bundle and $B$ is a cell complex then $f^*$ has a right adjoint $f_*$.\index{sheafy pushforward $f_*$}
\end{prop}

\begin{proof}
	Cover $B$ by neighborhoods $U_\alpha \subseteq B$ on which the bundle $A \to B$ is trivial, and pick fiber-preserving homeomorphisms $\phi_\alpha\colon U_\alpha \times F_\alpha \cong f^{-1}(U_\alpha)$. (The fibers may change because $B$ may not be connected.) For each $U_\alpha$, define $(f_*X)_\alpha$ as the pullback
	\[ \xymatrix @R=1.7em{
		(f_*X)_\alpha \ar[d] \ar[r] & \Map(F_\alpha,X) \ar[d]^-{\Map(\id,p_X)} \\
		U_\alpha \ar[r] & \Map(F_\alpha,A)
	} \]
	where the bottom map is the adjoint of the composite of $U_\alpha \times F_\alpha \cong f^{-1}(U_\alpha) \to A$. More simply, it sends each point $b \in U_\alpha$ to the inclusion of $F_\alpha$ into $A$ whose image is $f^{-1}(b)$. So each fiber of $(f_* X)_\alpha$ is a space of sections of $X$ over one of the subspaces $f^{-1}(b)$.
	
	For any pair $\alpha, \beta$, let $U_{\alpha\beta} = U_\alpha \cap U_\beta$ and let $(f_*X)_{\alpha\beta}$ and $(f_*X)_{\beta\alpha}$ denote the pullbacks
	\[ \xymatrix @R=1.7em{
		(f_*X)_{\alpha\beta} \ar[d] \ar[r] & (f_* X)_\alpha \ar[d] \\
		U_{\alpha\beta} \ar[r] & U_\alpha
	}
	\quad
	\xymatrix @R=1.7em{
		(f_*X)_{\beta\alpha} \ar[d] \ar[r] & (f_* X)_\beta \ar[d] \\
		U_{\alpha\beta} \ar[r] & U_\beta.
	}
	\]
	We form a homeomorphism $h_{\alpha\beta}\colon (f_*X)_{\alpha\beta} \cong (f_*X)_{\beta\alpha}$ as follows. Take the adjoint of the composite
	\[ \xymatrix @R=1.7em{
		U_{\alpha\beta} \times F_\alpha \ar[r]^-{\phi_\beta^{-1} \circ \phi_\alpha}_-\cong & U_{\alpha\beta} \times F_\beta \ar[r] & F_\beta
	} \]
	to get a function $\tau\colon U_{\alpha\beta} \to \Map(F_\alpha,F_\beta)$. Then form the map of products
	\[ \xymatrix @R=1.7em{
		U_{\alpha\beta} \times \Map(F_\beta,X) \ar[r]^-{(1,\tau) \times 1} &
		U_{\alpha\beta} \times \Map(F_\alpha,F_\beta) \times \Map(F_\beta,X) \ar[r]^-{1 \times \circ} &
		U_{\alpha\beta} \times \Map(F_\alpha,X).
	} \]
	This respects the subspaces
	\[ \xymatrix @R=1.7em{
		U_{\alpha\beta} \times_{\Map(F_\beta,A)} \Map(F_\beta,X) \ar[r] &
		U_{\alpha\beta} \times_{\Map(F_\alpha,A)} \Map(F_\alpha,X)
	} \]
	that define the pullbacks, and hence give a map $(f_*X)_{\beta\alpha} \to (f_*X)_{\alpha\beta}$. The same recipe in reverse gives the inverse of this map, hence it is a homeomorphism.
	
	The cocycle condition for the trivializations $\phi_\alpha$ imply that for any triple $\alpha, \beta, \gamma$ the composite $h_{\gamma\alpha} \circ h_{\beta\gamma} \circ h_{\alpha\beta}$ is the identity wherever it is defined. Hence the spaces $(f_* X)_\alpha$ glue together to give a single space $f_* X$ over $B$ whose restriction to $U_\alpha$ is canonically homeomorphic to $(f_* X)_\alpha$; this is elementary to argue in (CG).

	To establish the same claim in (CGWH), we need to check that this construction produces a weak Hausdorff space. For this we use the assumption that $B$ is a cell complex and choose the original cover so that every cell $e_i$ is contained in a $U_{\alpha(i)}$. Then any map into $f_*X$ from a compact Hausdorff space $K$ has image in $B$ landing in a finite complex, hence it suffices to prove that the image of $K$ when restricted to each closed cell $e_i$ in the base gives a closed subset of $(f_*X)_{e_i}$. This follows because each of the spaces $(f_*X)_\alpha$ is weak Hausdorff, being a pullback of weak Hausdorff spaces.
	
	To build the adjunction $(f^* \adj f_*)$, we first describe each $Y \in \mc R(B)$ as a collection of spaces $Y_\alpha \in \mc R(U_\alpha)$ glued together with homeomorphisms $Y_{\alpha\beta} \cong Y_{\beta\alpha}$ over $U_{\alpha\beta}$. The pullback $f^*Y = Y \times_B A \in \mc R(A)$ is therefore a collection of spaces $Y_\alpha \times F_\alpha \in \mc R(U_\alpha \times F_\alpha)$ glued together along the product of the gluing maps for the $Y_\alpha$ with those for the $F_\alpha$.
	
	Each map of topological spaces $f^*Y \to X$ is given by the data of maps $Y_\alpha \times F_\alpha \to X$ such that the following diagram commutes.
	\begin{equation}\label{eq:product_transition}
	\xymatrix @R=1.7em{
		Y_{\alpha\beta} \times F_\alpha \ar[d] \ar[r] & Y_{\beta\alpha} \times U_{\alpha\beta} \times F_\alpha \ar[r] & Y_{\beta\alpha} \times \Map(F_\alpha,F_\beta) \times F_\alpha \ar[r] & Y_{\beta\alpha} \times F_\beta \ar[d] \\
		Y_\alpha \times F_\alpha \ar[rr] && X & \ar[l] Y_\beta \times F_\beta
	}
	\end{equation}
	If the map $f^*Y \to X$ is over $A$, then the square below on the left commutes, hence its adjoint on the right commutes, so the adjoint maps $Y_\alpha \to \Map(F_\alpha,X)$ pass to maps $Y_\alpha \to (f_*X)_\alpha$ over $U_\alpha \subseteq B$.
	\[ \xymatrix @R=1.7em{
		Y_\alpha \times F_\alpha \ar[d] \ar[r] & X \ar[d]^-{p_X} \\
		U_\alpha \times F_\alpha \ar[r] & A
	}
	\quad
	\xymatrix @R=1.7em{
		Y_\alpha \ar[d] \ar[r] & \Map(F_\alpha,X) \ar[d]^-{\Map(\id,p_X)} \\
		U_\alpha \ar[r] & \Map(F_\alpha,A)
	} \]
	The overlap condition from \eqref{eq:product_transition} rearranges to the following diagram, which shows that for each pair $\alpha, \beta$ the two resulting maps $Y_{\alpha\beta} \to \Map(F_\alpha,X)$ are identical. Hence we get a map $Y \to f_*X$.
	\[ \xymatrix @R=1.7em{
		Y_{\alpha\beta} \ar[dr] \ar[r] & Y_\beta \times U_{\alpha\beta} \ar[r] & \Map(F_\beta,X) \times \Map(F_\alpha,F_\beta) \ar[d] \\
		& Y_\alpha \ar[r] & \Map(F_\alpha,X)
	} \]
	Furthermore if $f^*Y \to X$ is under $A$ then this map $Y \to f_*X$ is under $A$. This assignment is natural in $Y$ and $X$, and reversing the above steps gives its inverse, hence it is a bijection.
\end{proof}

As a result $f^*$ preserves all colimits so long as $f$ is a fiber bundle with base a cell complex. It is a classical fact that every map is weakly equivalent to such a fiber bundle, therefore $f^*$ is always a ``homotopical'' left adjoint, meaning in particular that it is a left adjoint if both $\mc R(B)$ and $\mc R(A)$ are localized by inverting the weak equivalences. This is similar to the behavior of the fixed points $(-)^G$ in equivariant homotopy theory. As in that case, $f^*$ will always commute with homotopy colimits up to equivalence, even though in (CGWH) it doesn't always strictly commute with strict colimits.

\begin{prop}\label{prop:sheafy_pushforward_preserves}
	The functor $f_*$ of \autoref{prop:sheafy_pushforward} preserves $h$-fibrations. If the fibers of $f$ are cell complexes, then $f_*$ preserves $q$-fibrations and weak equivalences between $q$-fibrant spaces.
\end{prop}

\begin{proof}
	The first part is a standard argument that re-arranges a lift in the square below on the right to a lift on the left. Note that in both squares any lift will automatically be a map over $A$ or $B$.
	\[ \xymatrix @R=1.7em @C=3em{
		f^*K \times \{0\} \ar[r] \ar[d] & X \ar[d] \\
		f^*K \times I \ar@{-->}[ur] \ar[r] & Y
	}
	\qquad
	\xymatrix @R=1.7em @C=3em{
		K \times \{0\} \ar[r] \ar[d] & f_*X \ar[d] \\
		K \times I \ar@{-->}[ur] \ar[r] & f_*Y
	}
	 \]
	Similarly, if $K = D^n$ and the fiber of $f$ is a cell complex then $f^*K$ is a cell complex. So if $X \to Y$ is a $q$-fibration, then the left-hand lift exists, hence the right-hand lift exists too.
	
	By Ken Brown's lemma \cite[1.1.12]{hovey_model_cats} it remains to show that $f_*$ preserves acyclic $q$-fibrations, i.e. maps that are both a weak equivalence and a $q$-fibration. To do this we rearrange the square below on the right to the one on the left. In the square on the left, the left-vertical is a relative cell complex and the right-vertical is a Serre fibration and a weak equivalence, hence the lift exists.
	\[ \xymatrix @R=1.7em @C=3em{
		f^*S^{n-1} \ar[r] \ar[d] & X \ar[d] \\
		f^*D^n \ar@{-->}[ur] \ar[r] & Y
	}
	\qquad
	\xymatrix @R=1.7em @C=3em{
		S^{n-1} \ar[r] \ar[d] & f_*X \ar[d] \\
		D^n \ar@{-->}[ur] \ar[r] & f_*Y.
	}
	 \]
\end{proof}

\begin{rmk}
	Under the (CG) assumptions, $f^*$ always has a right adjoint $f_*$ \cite[2.1]{ms}. However, $f_*$ does not seem to be right-deformable (see \autoref{sec:just_derived_functors}) for arbitrary maps $f$. This makes it difficult to use for homotopy theory. We won't define or use $f_*$ at this level of generality, because it appears that any time we are tempted to use it, there is some other workaround that ends up being more comprehensible (such as passing to modules over ring spectra, see \autoref{sec:equivalent_to_modules}).
\end{rmk}

\beforesubsection
\subsection{Whiskering and monoidal fibrant replacement}\aftersubsection

In this section we define a whiskering functor $W$ and a fibrant replacement functor $P$. Our definition of $P$ comes from \cite{malkiewich2017coassembly}, and has different properties than the fibrant replacement $L$ defined in \cite[8.3.1]{ms}. It is only a fibrant replacement when the input space is $h$-cofibrant. However, it is strong symmetric monoidal, which leads to major simplifications when we set up the bicategory of parametrized spectra.

For each real number $t \in [0,1]$, let $i_t\colon B \to B \times I$ be the map $b \mapsto (b,t)$. The \textbf{whiskering functor}\index{whiskering functor $WX$} $W\colon \mc R(B) \to \mc R(B)$ takes each retractive space $X$ to the retractive space whose total space is the pushout
\[ \xymatrix @R=1.7em @C=3em{
	B \ar[r]^-{i_X} \ar[d]_-{i_1} & X \ar[d] \\
	B \times I \ar[r] & WX.
} \]
The inclusion of the basepoint section is along $i_0\colon B \to B \times I$, and the projection is by $p_X$ and the projection map $B \times I \to B$. We define a natural map $WX \to X$ of retractive spaces over $B$, by the identity of $X$ and the projection $B \times I \to B$.
\begin{prop}\label{prop:whiskering_properties}\hfill
	\vspace{-1em}
	
	\begin{itemize}
		\item The map $WX \to X$ is always a weak equivalence.
		\item $WX$ is always $f$-cofibrant.
		\item $W$ preserves all $f$-cofibrations and $h$-cofibrations.
		\item If $X$ is $h$-fibrant then so is $WX$.
	\end{itemize}
\end{prop}
\begin{proof}
	The first claim follows because $WX$ is defined by a homotopy pushout along an equivalence $B \to B \times I$. The second claim is proven by a direct construction of a fiberwise retract of the map
	\[ (B \times I) \cup_{(B \times \{0\})} (WX \times \{0\}) \to (WX \times I). \]
	The third claim follows from \autoref{cofibration_of_pushouts} and the fourth from \autoref{prop:clapp}.
\end{proof}


For each real number $t \in [0,1]$, let $p_t\colon B^I \to B$ be the map that evaluates a path in $B$ at the point $t$. We define the \textbf{monoidal fibrant replacement functor}\index{fibrant replacement!$P$ for spaces} $P\colon \mc R(B) \to \mc R(B)$ by the formula $(p_1)_!(p_0)^*$. More explicitly, when $X$ is a retractive space over $B$, $PX$ is the retractive space whose total space is the pushout
\[ \xymatrix @R=1.7em @C=6em{
	B^I \ar[r]^-{i_X \times_B \ \id_{B^I}} \ar[d]_-{p_1} & X \times_B B^I \ar[d] \\
	B \ar[r] & PX
} \]
We also get a natural map $X \to PX$ of retractive spaces over $B$, by including $X$ into $X \times_B B^I$ along the inclusion of constant paths into $B^I$.

\begin{prop}\label{prop:px_properties}\hfill
	\vspace{-1em}
	
	\begin{itemize}
		\item $P$ preserves all $f$-cofibrations and $h$-cofibrations.
		\item If $X$ is $h$-cofibrant then $PX$ is $h$-fibrant and $X \to PX$ is a weak equivalence.
	\end{itemize}
\end{prop}

As a result of this and \autoref{lem:heath_kamps}, every $h$-cofibration $X \to Y$ of $h$-cofibrant spaces gets ``upgraded'' by $P$ to an $f$-cofibration $PX \to PY$ of $f$-cofibrant spaces.

\begin{proof}
\autoref{lem:f_star_preserves} and \autoref{lem:f_shriek_preserves} tell us that the $f$-cofibrations and $h$-cofibrations both survive the journey through the definition of $P$. If $X$ is $h$-cofibrant, \autoref{lem:f_star_preserves} tells us that the top horizontal in the above square is an $h$-cofibration. Since the left vertical is a weak equivalence, therefore so is the right vertical. The first three vertices are all $h$-fibrations over $B$, so by \autoref{prop:clapp} the fourth vertex is also a fibration over $B$.
\end{proof}

In summary, each space $X$ can be replaced by an equivalent $f$-cofibrant space $WX$, and this in turn can be replaced by an equivalent space $PWX$ that is both $f$-cofibrant and $h$-fibrant.

\begin{rmk}
	There is a more obvious fibrant replacement functor $RX = X \times_B B^I$. The composite $WRX$ is another way of replacing $X$ by an equivalent space that is both $f$-cofibrant and $h$-fibrant. However $R$ does not commute with smash products. Later on, this will prevent us from using it to take fibrant replacement of spectra.
\end{rmk}

\newpage
\section{Smash products}\label{sec:products}

\subsection{External smash products and mapping spaces}\label{sec:external_smash}\aftersubsection

If $X$ is a retractive space over $A$ and $Y$ is a retractive space over $B$, the \textbf{external smash product}\index{external smash product $X \barsmash Y$!of retractive spaces} $X \barsmash Y$ is a retractive space over $A \times B$ whose fiber over $(a,b) \in A \times B$ is the smash product of the fibers, $X_a \sma Y_b$.

Of course, that's just a property, not a full definition. The full definition is that the total space of $X \barsmash Y$ is the pushout
\begin{equation}\label{external_smash}
\xymatrix @R=2em{
	(X \times B) \cup_{A \times B} (A \times Y) \ar[r] \ar[d] & X \times Y \ar[d] \\
	A \times B \ar[r] & X \barsmash Y }
\end{equation}
where the left vertical map arises from $p_X$ and $p_Y$, and the top horizontal arises from $i_X$ and $i_Y$. The inclusion of $A \times B$ into $X \barsmash Y$ is the obvious one, and the projection to $A \times B$ is by the square
\[
\xymatrix @R=2em{
	(X \times B) \cup_{A \times B} (A \times Y) \ar[r] \ar[d] & X \times Y \ar@{-->}[d]^-{p_X \times p_Y} \\
	A \times B \ar@{-->}[r]^-\id & A \times B. }
\]
This is yet another example where the pushout respects the underlying sets, because in (CG) that is automatic and in (CGWH) the top horizontal map is a closed inclusion. So the underlying set of $X \barsmash Y$ really is $X_a \sma Y_b$ on each fiber.

In the special case where we take $X_{+A} = X \amalg A$, with $X$ an unbased space over $A$, the external smash product $X_{+A} \barsmash Y$ is called a ``half-smash product'' and its total space is given more simply as the pushout
\begin{equation}\label{external_half_smash}
\xymatrix @R=1.7em{
	X \times B \ar[r] \ar[d] & X \times Y \ar[d] \\
	A \times B \ar[r] & X_{+A} \barsmash Y. }
\end{equation}
In the further special case of $X_{+A} \barsmash Y_{+B}$ we just get the disjoint union of $X \times Y$ and $A \times B$, in other words
\[ X_{+A} \barsmash Y_{+B} \cong (X \times Y)_{+(A \times B)}. \]

\begin{ex}\label{ex:fiberwise_suspension}
	Taking $A = *$ and $X = S^1$, we get the definition of the \textbf{fiberwise reduced suspension}\index{fiberwise!reduced suspension $\Sigma_B$} functor on spaces $Y$ over $B$:
	\[ \Sigma_B Y := S^1 \barsmash Y \]
\end{ex}

\begin{ex}\label{ex:thom_smash_product}
	If we take two vector bundles $V$ and $W$ over $A$ and $B$, respectively, the fiberwise Thom space $\Th_{A \times B}(V \times W)$ can be expressed by the formula
	\[ \Th_{A \times B}(V \times W) \cong \Th_A(V) \barsmash \Th_B(W). \]
	Combining with the previous example, we get
	\[ \Th_B(\R^n \times W) \cong S^n \barsmash \Th_B(W) = \Sigma^n_B \Th_B(W). \]
	This is the fiberwise variant of the isomorphism $\Th(\R^n \times W) \cong \Sigma^n \Th(W)$.
\end{ex}

\begin{ex}
	If $X$ is an unbased space over $B$ with projection $p\colon X \to B$, and $Y$ is a retractive space over $B$, we can take the half-smash product \eqref{external_half_smash} and pull back along $\Delta_B$, which preserves the pushout in both (CG) and (CGWH). This gives the square
	\begin{equation*}
	\xymatrix @R=1.7em{
		X \times_B B \ar[r] \ar[d] & X \times_B Y \ar[d] \\
		B \times_B B \ar[r] & X_{+b} \sma_B Y }
	\end{equation*}
	and hence an isomorphism
	\begin{equation}\label{half_smash_internal}
		X_{+B} \sma_B Y \cong p_!p^* Y.
	\end{equation}
\end{ex}


The external smash product has a right adjoint in each variable, the \textbf{external mapping space}\index{external!mapping space $\barmap_B(Y,Z)$} $\barmap_B(Y,Z)$.\footnote{The internal smash product, on the other hand, only has right adjoints if we adopt (CG) assumptions \cite{ms}.} For $Z \in \mc R(A \times B)$ and $Y \in \mc R(B)$, this is defined as the following pullback.
\begin{equation}\label{external_map}
	\xymatrix @R=2em{
		\barmap_B(Y,Z) \ar[r] \ar[d] & \Map(Y,Z) \ar[d] \\
		A \times \{{*}\} \times \{{*}\} \ar[r] & \Map(Y,A) \times \Map(Y,B) \times \Map(B,Z)
	}
\end{equation}
Here $\Map$ refers to the usual space of unbased maps, the right-vertical composes with the projections $Z \to A$, $Z \to B$, and the section $B \to Y$, and the bottom-horizontal assigns to each $a \in A$ the constant map $Y \to \{a\}$, the projection $Y \to B$, and the section $B \cong \{a\} \times B \to Z$.

When $Y$ is nonempty the bottom horizontal map is split,\footnote{When $Y$ is empty, both $B$ and $Z$ are empty too, and the horizontal maps become $A \to \{*\}$.} hence in (CGWH) both horizontal maps are closed inclusions, while in (CG) they are just inclusions. So $\barmap_B(Y,Z)$ is precisely the subspace of all maps $Y \to Z$ that fit into three commuting diagrams
\[ \xymatrix @R=2em{
	Y \ar[d] \ar[r] & Z \ar[d] \\
	\{{*}\} \ar[r]^-a & A
}
\qquad
\xymatrix @R=2em{
	Y \ar[dr] \ar[r] & Z \ar[d] \\
	 & B
}
\qquad
\xymatrix @R=2em{
B \ar[d] \ar[r]^-{a \times \id} & A \times B \ar[d] \\
Y \ar[r] & Z
}
\]
or equivalently the single commuting diagram
\[ \xymatrix @C=3em @R=1.7em{
	B \ar[d] \ar[r]^-{a \times \id} & A \times B \ar[d] \\
	Y \ar[d] \ar[r] & Z \ar[d] \\
	B \ar[r]^-{a \times \id} & A \times B
}
\]
for some value of $a \in A$.

From this we deduce that $\barmap_B(Y,Z)$ is a retractive space over $A$ whose fiber over $a \in A$ is the space of maps of retractive spaces from $Y$ into $Z_a$ over $B$. The usual rules define adjunctions
\[ X \barsmash Y \to Z \quad \textup{ over }A \times B \quad \longleftrightarrow \quad X \to \barmap_B(Y,Z) \quad \textup{ over }A, \]
\[ X \barsmash Y \to Z \quad \textup{ over }A \times B \quad \longleftrightarrow \quad Y \to \barmap_A(X,Z) \quad \textup{ over }B. \]

\begin{ex}\label{ex:fiberwise_loops}\hfill
	\vspace{-1em}
	
	\begin{itemize}
		\item Setting $A = *$, the based space $\barmap_B(Y,Z)$ is just the space of maps $Y \to Z$ that agree with the inclusions and projections to $B$, in other words the space of maps of retractive spaces from $Y$ to $Z$. This gives an enrichment of the category $\mc R(B)$ in based spaces, in other words it gives a natural topology on the sets $\mc R(B)(Y,Z)$.\index{retractive space!enrichment}
		\item Setting $A = *$ and $X = S^1$ defines the \textbf{fiberwise based loops}\index{fiberwise!based loops $\Omega_B$} of $Y \in \mc R(B)$,
		\[ \Omega_B Y := \barmap_*(S^1,Y). \]
		Every point of $\Omega_B Y$ is a map $S^1 \to Y$ that lands entirely in one fiber $Y_b$, sending the basepoint of $S^1$ to the basepoint of $Y_b$. We also get an adjunction between $\Sigma_B$ and $\Omega_B$ that agrees with the usual one over every point of $B$.
		\item Setting $A = *$ and $Y = B_{+B}$, we get the space of sections from \autoref{sec:sections}\index{sections $\Gamma_B(-)$},
		\[ \barmap_B(B_{+B},Z) \cong \Gamma_B(Z). \]
		More generally, if $Y = E_{+B}$ for some space $E$ with projection $p\colon E \to B$, juggling the above adjunctions and identifications gives
		\begin{equation}\label{half_map_sections}
			\barmap_B(E_{+B},Z) \cong \Gamma_E(p^* Z).
		\end{equation}
	\end{itemize}
\end{ex}

\beforesubsection
\subsection{Pushout-products and pullback-homs}\label{pushout_product_lemmas}\aftersubsection

To discuss how $\barsmash$ handles cofibrations, fibrations, and weak equivalences, we will use the language of pushout-products. Suppose that $\mc C$, $\mc D$, and $\mc E$ are categories that have all small colimits, and
\[ \otimes\colon \mc C \times \mc D \to \mc E \]
is a colimit-preserving functor. Then the \textbf{pushout-product}\index{pushout-product $\square$} of a map $f: K \to X$ in $\mc C$ and $g: L \to Y$ in $\mc D$ is the map $f \square g$ in $\mc E$ from the pushout of the first three terms in the following square to the final vertex.
\[ \xymatrix @R=2em @C=3em{
	K \otimes L \ar[r]^-{\id_K \otimes g} \ar[d]_-{f \otimes \id_L} & K \otimes Y \ar[d]^-{f \otimes \id_Y} \\
	X \otimes L \ar[r]_-{\id_X \otimes g} & X \otimes Y
}\]

\[ f \square g\colon X \otimes L \cup_{K \otimes L} K \otimes Y \to X \otimes Y \]

These have already appeared. \autoref{lem:product_ndr} is a statement about pushout-products of cofibrations. In \eqref{external_smash}, the top horizontal map is the pushout-product $i_X \square i_Y$ with $\otimes$ set to the Cartesian product $\times$.

If $K = \emptyset$ is an initial object, the pushout-product $f \square g$ is just $\id_X \otimes g$. So tensoring an object with a map is a special case of the pushout-product.

Let $a\mc C$ be the category of arrows, whose objects are morphisms of $\mc C$ and whose morphisms are commuting squares. We say a morphism in $a\mc C$ is a ``pushout morphism'' if this square is a pushout square in $\mc C$. The following is a straightforward exercise.
\begin{lem}\label{pushout_product_pushouts}\hfill
	\vspace{-1em}
	
	\begin{itemize}
		\item The operation $- \square -$ defines a functor $a\mc C \times a\mc D \to a\mc E$.
		\item For fixed $g \in a\mc D$, $- \square g$ sends pushout morphisms to pushout morphisms.
	\end{itemize}
\end{lem}

Composition of arrows (i.e. objects of $a\mc C$) interacts with $\square$ in a way that is somewhat trickier to state. For starters, notice that this is a ``horizontal'' composition that should not be confused with ``vertical'' composition of morphisms in $a\mc C$.
\begin{lem}\label{pushout_product_composition}
	If $f: K \to X$ and $f': X \to Z$ are two morphisms in $\mc C$, then the pushout-product $(f' \circ f) \square g$ is a composition of $f' \square g$ and a pushout of $f \square g$.
\end{lem}
\begin{proof}
	Examine the following diagram in which all squares are pushouts.
	\[ \xymatrix @R=1.8em @!C{
		K \otimes L \ar[r]^-{\id_K \otimes g} \ar[d]^-{f \otimes \id_L} & K \otimes Y \ar[d] \ar[dr]^-{f \otimes \id_Y} \\
		X \otimes L \ar[r] \ar[d]^-{f' \otimes \id_L} \ar@/_1em/[rr]_(.4){\id_X \otimes g} & X \otimes L \cup_{K \otimes L} K \otimes Y \ar[r]^-{f \square g} \ar[d] & X \otimes Y \ar[d] \ar[dr]^-{f' \otimes \id_Y} \\
		Z \otimes L \ar[r] \ar@/_3em/[rrr]_-{\id_Z \otimes g} & Z \otimes L \cup_{K \otimes L} K \otimes Y \ar[r] \ar@/_1em/[rr]_-{(f' \circ f)\square g} & Z \otimes L \cup_{X \otimes L} X \otimes Y \ar[r]^-{f' \square g} & Z \otimes Y \\
	}\]
\end{proof} 

The above argument generalizes to countable sequential colimits as follows. If $f: X^{(0)} \to X^{(\infty)}$ is the composition of the maps $f_{m-1}: X^{(m-1)} \to X^{(m)}$ in $\mc C$, meaning that $X^{(\infty)}$ is identified with the colimit of the $X^{(m)}$, and $g: L \to Y$ is an arrow of $\mc D$, then we have the larger diagram in which the squares are pushout squares:
\[ \resizebox{\textwidth}{!}{
	\xymatrix @R=1.7em @C=1em{
	X^{(m-1)} \otimes L \ar[r] \ar[d] & \ldots \ar[r] & X^{(m-1)} \otimes Y \ar[d] \ar[rrd] && \\
	X^{(m)} \otimes L \ar[r] \ar[d] & \ldots \ar[r] & X^{(m)} \otimes L \cup_{X^{(m-1)} \otimes L} X^{(m-1)} \otimes Y \ar[rr] \ar[d] && X^{(m)} \otimes Y \ar[d] \ar[rd] \\
	\vdots \ar[d] && \vdots \ar[d] && \vdots \ar[d] & \ddots \ar[rd] \\
	X^{(\infty)} \otimes L \ar[r] & \ldots \ar[r] & X^{(\infty)} \otimes L \cup_{X^{(m-1)} \otimes L} X^{(m-1)} \otimes Y \ar[rr] && X^{(\infty)} \otimes L \cup_{X^{(m)} \otimes L} X^{(m)} \otimes Y \ar[r] & \ldots \ar[r] & X^{(\infty)} \otimes Y
}
}\]
From this diagram one can conclude that the pushout-product
\[ f \square g : (X^{(\infty)} \otimes L) \cup_{X^{(0)} \otimes L} (X^{(0)} \otimes Y) \to X^{(\infty)} \otimes Y \]
is the countable composition of pushouts of the maps $f_{m-1} \square g$.\footnote{This argument also generalizes to general transfinite compositions as in \cite[2.1.1]{hovey_model_cats}, by factoring the pushout-product into maps indexed by the same ordinal $\lambda$ as the $X^{(m)}$, using the above construction for the successor ordinals in $\lambda$, and using the commutation of $\otimes$ and $\cup$ with colimits for the colimit ordinals in $\lambda$.}

Next we define pullback-homs, the duals of pushout-products. We assume that $\mc C$ has all small limits and that there is a functor
\[ \Hom_{\mc D}\colon \mc D^{op} \times \mc E \to \mc C \]
and an adjunction between $- \otimes d$ and $\Hom_{\mc D}(d,-)$ for every $d \in \mc D$. In other words, there are bijections natural in $c$, $d$, and $e$,
\[ c \otimes d \to e \quad \textup{ in }\mc E \quad \longleftrightarrow \qquad c \to \Hom_{\mc D}(d,e) \quad \textup{ in }\mc C. \]
For every $g: L \to Y$ in $\mc D$ and $h\colon M \to Z$ in $\mc E$, the \textbf{pullback-hom}\index{pullback-hom} $\Hom_\square(g,h)$ is the map in $\mc C$ that takes the first vertex of the square below to the pullback of the other three vertices:
\[ \xymatrix @R=2em{
	\Hom_{\mc D}(Y,M) \ar[r] \ar[d] & \Hom_{\mc D}(L,M) \ar[d] \\
	\Hom_{\mc D}(Y,Z) \ar[r] & \Hom_{\mc D}(L,Z)
} \]

\[ \Hom_\square(g,h)\colon \Hom_{\mc D}(Y,M) \to \Hom_{\mc D}(L,M) \times_{\Hom_{\mc D}(L,Z)} \Hom_{\mc D}(Y,Z). \]

This can be visualized as the space of all choices of diagonal in the square below, mapping to the space of all choices of two horizontal maps making the square commute.
\[ \xymatrix @R=1.7em @C=3em{
	L \ar@{-->}[r] \ar[d]_-g & M \ar[d]^-h \\
	Y \ar@{-->}[r] \ar@{-->}[ur] & Z
} \]
If $L = \emptyset$ is the initial object then the pullback-hom is just $\Hom_{\mc D}(Y,h)$. $Z = *$ is the final object then the pullback-hom is just $\Hom_{\mc D}(g,M)$.

\begin{lem}\label{lem:pullback_hom_adjunction}
	For a fixed arrow $g \in a\mc D$, the functors $- \square g$ and $\Hom_\square(g,-)$ form an adjunction between $a\mc C$ and $a\mc E$.
\end{lem}
\begin{proof}
	A choice of dotted maps making the square
	\[ \xymatrix @R=1.7em{
		K \ar@{-->}[r] \ar[d]^-f & \Hom_{\mc D}(Y,M) \ar[d]^-{\Hom_\square(g,h)} \\
		X \ar@{-->}[r] & \Hom_{\mc D}(L,M) \times_{\Hom_{\mc D}(L,Z)} \Hom_{\mc D}(Y,Z)
	} \]
	commute in $\mc C$ can be expressed uniquely by the data of three maps $K \otimes Y \to M$, $X \otimes L \to M$, and $X \otimes Y \to Z$ in $\mc E$, subject to three compatibility conditions, two that say the above square commutes, and one that says that the maps to the lower-right land not just in the product but in the fiber product. These three compatibility conditions correspond to the three regions in the diagram below, and they hold if the maps around these regions commute.
	\[ \xymatrix @R=1.7em{
		& K \otimes L \ar[ld]_-{f \otimes \id} \ar[rd]^-{\id \otimes g} & \\
		X \otimes L \ar@{-->}[r] \ar[rdd]_-{\id \otimes g} & M \ar[d]^-h & K \otimes Y \ar@{-->}[l] \ar[ldd]^-{f \otimes \id} \\
		& Z & \\
		& X \otimes Y \ar@{-->}[u] & 
	} \]
	This rearranges to the statement that the square
	\[ \xymatrix @R=1.7em{
		X \otimes L \cup_{K \otimes L} K \otimes Y \ar@{-->}[r] \ar[d]^-{f \square g} & M \ar[d]^-{h} \\
		X \otimes Y \ar@{-->}[r] & Z
	} \]
	commutes (two of the conditions) and that the top map is in fact a map out of the pushout and not just the coproduct (the last condition). In this way the adjunction between $\otimes$ and $\Hom_{\mc D}$ specifies the data of an adjunction between $- \square g$ and $\Hom_\square(g,-)$. It is straightforward to check these identifications are natural in $f$ and $h$ (and therefore in $g$).
\end{proof}

\beforesubsection
\subsection{Smashing cofibrations, fibrations, and weak equivalences}\aftersubsection

\begin{prop}\label{prop:h_cofibrations_pushout_product}\hfill
	\vspace{-1em}
	
	\begin{itemize}
		\item Let $f: K \to X$ and $g: L \to Y$ be $h$-cofibrations of retractive spaces over $A$ and $B$, respectively. Then $f \square g$, constructed using $\barsmash$, is an $h$-cofibration.
		\item The same is true for $f$-cofibrations or closed inclusions.
		\item If $X$ and $Y$ are $h$-cofibrant and $h$-fibrant then $X \barsmash Y$ is $h$-fibrant.
		\item If $X$ is an $h$-cofibrant space over $A$ and $g: Y \to Y'$ is a weak equivalence of $h$-cofibrant spaces over $B$ then $\id_X \barsmash g$ is a weak equivalence.
	\end{itemize}
\end{prop}

Note that the third of these also appears in \cite[8.2.4]{ms} and \cite[3.6]{may1975classifying}.
 
\begin{rmk}\hfill
	\vspace{-1em}
	
	\begin{itemize}
		\item As a corollary, $\Sigma_B(-)$ preserves $f$-cofibrations, $h$-cofibrations, and weak equivalences between $h$-cofibrant spaces. 
		\item Conspicuously, we do not prove that a pushout-product of two $h$-cofibrations, one of which is a weak equivalence, gives a weak equivalence. We do not know if this is true.\footnote{It is true if we also assume that all of the spaces involved are $h$-cofibrant. That can be deduced from the statement that we did prove.} This would be a problem if we were trying to construct a model structure (see \autoref{prop:quillen_bifunctor_defn}), but fortunately, we are not.
	\end{itemize}
\end{rmk}

\begin{proof}
	For the first two claims, we compare universal properties and write the desired pushout-product as the pushout of the map of spans
	\[ \xymatrix @R=1.7em{
		A \times B \ar@{=}[d] & (X \times B) \cup_{A \times B} (A \times Y) \ar[l] \ar[r] \ar@{=}[d] & (X \times L) \cup_{K \times L} (K \times Y) \ar[d] \\
		A \times B & (X \times B) \cup_{A \times B} (A \times Y) \ar[l] \ar[r] & X \times Y
	} \]
	By \autoref{cofibration_of_pushouts}, the claim reduces to the fact that $f \square g$ computed with the Cartesian product is a cofibration, which follows from \autoref{lem:product_ndr}.
	
	For the third claim, in the square \eqref{external_smash}, the top-left term is a pushout along $h$-cofibrations of spaces that are fibrant over $A \times B$. By \autoref{prop:clapp} and its extension with $h$-cofibrations, this top-left term of \eqref{external_smash} is therefore fibrant. The top map of \eqref{external_smash} itself is also an $h$-cofibration, and the other terms are clearly fibrant over $A \times B$, so again their pushout $X \barsmash Y$ is also fibrant.
	For the final claim, the same analysis shows that for a weak equivalence $Y \to Y'$ we get a weak equivalence on the top-left term of the square \eqref{external_smash}, as well as on the other two terms, and therefore gives a weak equivalence on the pushouts $X \barsmash Y \to X \barsmash Y'$.
\end{proof}

\begin{prop}\label{h_fibrations_pullback_hom}\hfill
	\vspace{-1em}
	
	\begin{itemize}
		\item Let $g: L \to Y$ be an $h$-cofibration over $B$ and $h: M \to Z$ be an $h$-fibration over $A \times B$. Then $\Hom_\square(g,h)$, constructed using $\barmap_B$, is an $h$-fibration.
		\item (CGWH) If $h\colon M \to Z$ is a closed inclusion then so is $\barmap_B(Y,h)$.
		\item (CGWH) If $Y$ is compact and $h\colon M \to Z$ is an $f$-cofibration then $\barmap_B(Y,h)$ is an $f$-cofibration.
		\item If $Y$ is $h$-cofibrant and has total space homotopy equivalent to a cell complex, and $Z$ is $h$-fibrant, then $\barmap_B(Y,Z)$ preserves weak equivalences in both variables.
	\end{itemize}
\end{prop}

As a consequence, $\Omega_B(-)$ preserves $h$-fibrations and weak equivalences between $h$-fibrant spaces. Curiously, it also preserves $f$-cofibrations, at least in (CGWH).

\begin{proof}
	For the first claim, examine the proof of \autoref{lem:pullback_hom_adjunction}. A dotted map lifting the bottom square corresponds to a dotted map lifting the top square. Taking $X = K \times I$, we observe that $K \to K \times I$ is an $h$-cofibration and the left-vertical in the last square is clearly a homotopy equivalence. Therefore we get a lift by \autoref{prop:strom_lift}, hence the original right-vertical $\Hom_\square(g,h)$ was an $h$-fibration.
	
	The second claim follows from \autoref{lem:closed_inclusion_equalizer} because $\barmap_B(Y,-)$ is a right adjoint and therefore preserves equalizers. 
	
	For the third claim, given a retract
	\[ \xymatrix{
		M \times I \cup_{M \times 0} Z \times 0 \ar[r] & Z \times I \ar@/_1em/[l]
	} \]
	over $A \times B$, form the following square.
	\[ \xymatrix{
		\Map(Y,M \times I \cup_{M \times 0} Z \times 0) \ar[r] \ar@/_/[d] & \Map(Y,Z \times I) \ar@/^/[d] \ar@/_/[l] \\
		(\Map(Y,M) \times I) \cup_{\Map(Y,M) \times 0} (\Map(Y,Z) \times 0) \ar[r] \ar[u] & \Map(Y,Z) \times I \ar[u]
	} \]
	The left-pointing map is $\Map(Y,-)$ applied to our retract. The maps pointing to the right are closed inclusions, so we regard the left-hand column as subspaces of the right-hand column, and define all the vertical maps by focusing on the right-hand column and checking our rule preserves the given subspaces. We define the upward vertical map by assembly, sending $(f,t)$ to the function $(y \mapsto (f(y),t))$. We define the first factor of the downward vertical map by projecting $Z \times I \to Z$. The second factor takes the minimum over $Y$ of the $I$-coordinate of $f$.
	
	To prove this last map is continuous, it suffices to show the preimage of an open interval is open in the compact-open topology, for then it must also be open in the compactly-generated compact-open topology. The preimage of $(a,b) \subseteq I$ is the basic open set of all functions sending $Y$ to $Z \times (a,1]$, minus the set sending $Y \to Z \times [b,1]$. This latter set is closed because each point in its complement can be separated from it by a basic open set of the form $y \to (0,b)$.
	
	By a diagram chase, we can now retract the lower-right term onto the lower-left term by mapping up, left, and then down.\footnote{The author learned this argument from Irakli Patchkoria, who in turn learned it from Thomas Nikolaus and Wolfgang L\" uck.} It is straightforward to check that this retraction respects the three conditions that define the closed subspaces $\barmap_B(Y,M) \subseteq \Map(Y,M)$ and $\barmap_B(Y,Z) \subseteq \Map(Y,Z)$, and also the projection map from these closed subspaces to $A$.
	
	For the final claim, by Ken Brown's lemma it suffices to assume the equivalence $Y \to Y'$ is also an $h$-cofibration, or $Z \to Z'$ is also an $h$-fibration. In both cases we take the proof of \autoref{lem:pullback_hom_adjunction} and plug in $A = *$, $K = S^{n-1}_+$, $X = D^n_+$. It suffices to take a square defining an element of relative $\pi_n$ and find a lift up to homotopy, and this rearranges into a square of one of the two forms
	\[ \xymatrix @R=1.7em{
		D^n_+ \barsmash Y \cup_{S^{n-1}_+ \barsmash Y} S^{n-1}_+ \barsmash Y' \ar@{-->}[r] \ar[d] & Z \ar[d] \\
		D^n_+ \barsmash Y' \ar@{-->}[r] & B
	},
	\quad
	\xymatrix @R=1.7em{
		D^n_+ \barsmash A \cup_{S^{n-1}_+ \barsmash A} S^{n-1}_+ \barsmash Y \ar[r]^-\cong & S^{n-1}_+ \barsmash Y \ar@{-->}[r] \ar[d] & Z \ar[d] \\
		& D^n_+ \barsmash Y \ar@{-->}[r] & Z'
	} \]
	In both squares the left-vertical is an $h$-cofibration by \autoref{prop:h_cofibrations_pushout_product}, and the right-vertical is an $h$-fibration. In the first square the left-hand map is also a homotopy equivalence (using that $Y$ is a cell complex), hence the lift exists by \autoref{prop:strom_lift}. In the second square we may without loss of generality replace $Y$ by an actual cell complex (since we only want a lift up to homotopy). Then since $Z \to Z'$ is a Serre fibration and a weak equivalence, the existence of this lift is a standard property of the $q$-model structure, see \autoref{sec:model_cats}.
\end{proof}

We will need the following corollary when we prove Costenoble-Waner duality. It holds with $S^1$ replaced by any compact space, but we will not need that level of generality.

For any space $X$, let $\Lambda X = \Map(S^1,X)$ denote its free loop space. We think of it as a retractive space over $X$ along the inclusion of constant loops and the evaluation at a fixed point in $S^1$.
\begin{lem}\label{lem:awesome}
	(CGWH) If the diagonal map of $X$ is an $h$-cofibration, then the inclusion of constant loops $X \to \Lambda X$ is an $f$-cofibration over $X$.
\end{lem}

The hypothesis is true if $X$ is a retract of a cell complex, because it is preserved under retracts, and if $X$ is a cell complex then its diagonal can be made a subcomplex.

\begin{proof}
	Let $X \times X$ be a retractive space over $X$, along the diagonal map
	\[ \Delta_X\colon X \to X \times X \]
	and the projection onto the first factor $X$. By assumption, $\Delta_X$ is an $h$-cofibration. Since both sides are $h$-fibrant over $X$, it is an $f$-cofibration by \autoref{lem:heath_kamps}. Therefore by \autoref{h_fibrations_pullback_hom},  $\barmap_*(S^1,\Delta_X)$ is an $f$-cofibration. But this map is homeomorphic as a map over $X$ to the inclusion of constant loops $X \to \Lambda X$, by writing out the definition of $\barmap_*$ and comparing universal properties.
\end{proof}

\beforesubsection
\subsection{Rigidity and interaction with base-change}\aftersubsection

The following result says that the external smash product commutes with pullbacks and pushforwards.
\begin{lem}\label{lem:external_smash_and_base_change}
	Given maps
\[ \xymatrix @R=.2em{
	X \in \mc R(A), & f\colon A \to A', & W \in \mc R(A'), \\
	Y \in \mc R(B), & g\colon B \to B', & Z \in \mc R(B'),
} \]
there are canonical isomorphisms
\begin{equation}\label{eq:external_smash_and_base_change}
\begin{array}{c}
f_!X \barsmash g_!Y \cong (f \times g)_!(X \barsmash Y) \\
f^*W \barsmash g^*Z \cong (f \times g)^*(W \barsmash Z).
\end{array}
\end{equation}
\end{lem}
By ``canonical,'' we mean that $f_!X \barsmash g_!Y$ fits into the pushout diagram that defines $(f \times g)_!(X \barsmash Y)$ in an obvious way, and so receives a map, and this map is an isomorphism.

\begin{proof}
	For the isomorphism with the pullback functors, it suffices to assume one of the two maps, say $g$, is the identity. In the following cube, the first three vertical maps are pullbacks of $f \times \id_B$, and the top and bottom faces are pushouts (along closed inclusions in (CGWH)). Since $(f \times \id_B)^*$ preserves such pushouts, the final vertical map is also a pullback of $f \times \id_B$.
	\[ \xymatrix@!C=3em @R=1.2em{
			& A \times B \ar'[d][dd]^-(.4){}\ar[rrr]^-{i_{f^*W} \times 1}\ar[ld]_-{1 \times i_Z}
			&&& f^*W \times B \ar[dd]^-{}\ar[ld]_-{} \\
			A \times Z \ar[dd]_-{}\ar[rrr]^-{}
			&&& A \times Z \cup_{A \times B} f^*W \times B \ar[dd]^-(.3){}\ar[rrd] \\
			& A' \times B \ar'[rr]^-(.6){i_W \times 1}[rrr]\ar[ld]_-{1 \times i_Z}
			&&& W \times B\ar[ld]_-{}
			& A \times B \ar[dd]^-{f \times \id_B} \\
			A' \times Z \ar[rrr]^-{}
			&&& A' \times Z \cup_{A' \times B} W \times B \ar[rrd] \\
			&&&&& A' \times B
	} \]
	Repeating this argument for the cube
	\[ \xymatrix@!C=3em @R=1.2em{
		& A \times Z \cup_{A \times B} f^*W \times B \ar'[d][dd]^-(.4){}\ar[rrr]^-{}\ar[ld]_-{}
		&&& f^*W \times Z \ar[dd]^-{}\ar[ld]_-{} \\
		A \times B \ar[dd]_-{}\ar[rrr]^-{}
		&&& f^*W \barsmash Z \ar[dd]^-(.3){}\ar[rrd] \\
		& A' \times Z \cup_{A' \times B} W \times B \ar'[rr]^-(.6){}[rrr]\ar[ld]_-{}
		&&& W \times Z\ar[ld]_-{}
		& A \times B \ar[dd]^-{f \times \id_B} \\
		A' \times B \ar[rrr]^-{}
		&&& W \barsmash Z \ar[rrd] \\
		&&&&& A' \times B
	} \]
	we get a pullback diagram demonstrating that $f^*W \barsmash Z$ is a pullback of $W \barsmash Z$.
	
	For pushouts the proof is similar, but some of the vertical maps are pushouts of $f \times \id_B$, others of $f \times \id_Y$. For the first half, we observe that in the diagram below the top and bottom faces are pushouts by definition. The back face is a pushout because $- \times B$ preserves colimits. It follows that the front face is also a pushout.
	\[ \xymatrix@!C=3em @R=1.2em{
		& A \times B \ar'[d][dd]^-(.4){}\ar[rrr]^-{i_{X} \times 1}\ar[ld]_-{1 \times i_Y}
		&&& X \times B \ar[dd]^-{}\ar[ld]_-{} \\
		A \times Y \ar[dd]_-{}\ar[rrr]^-{}
		&&& A \times Y \cup_{A \times B} X \times B \ar[dd]^-(.3){} \\
		& A' \times B \ar'[rr]^-(.6){i_{f_!X} \times 1}[rrr]\ar[ld]_-{1 \times i_Y}
		&&& f_!X \times B\ar[ld]_-{} \\
		A' \times Y \ar[rrr]^-{}
		&&& A' \times Y \cup_{A' \times B} f_!X \times B
	} \]
	Then in the diagram below, the left half of the back face is a pushout. The entire back face is a pushout since $- \times Y$ preserves colimits, so the right half of the back face is also a pushout. Again since the top and bottom faces are pushouts, this implies the front face is a pushout, and we are done.
	\[ \xymatrix@!C=3em @R=1.2em{
		& A \times Y \ar[dd]^-(.4){f \times \id_Y}\ar[rrr]^-{}
		&&& A \times Y \cup_{A \times B} X \times B \ar'[d][dd]^-(.4){}\ar[rrr]^-{}\ar[ld]_-{}
		&&& X \times Y \ar[dd]^-{}\ar[ld]_-{} \\
		&&& A \times B \ar[dd]_-{}\ar[rrr]^-{}
		&&& X \barsmash Y \ar[dd]^-(.3){} \\
		& A' \times Y \ar[rrr]
		&&& A' \times Y \cup_{A' \times B} f_!X \times B \ar'[rr]^-(.6){}[rrr]\ar[ld]_-{}
		&&& f_!X \times Y\ar[ld]_-{} \\
		&&& A' \times B \ar[rrr]^-{}
		&&& f_!X \barsmash Y
	} \]
\end{proof}

This proof actually demonstrates that in the larger category $\mc R$ of all parametrized spaces, the external smash product of any two ``pushforward maps'' $X \to f_!X$ is isomorphic to a pushforward map, and similarly the external smash product of two pullback maps $f^*X \to X$ is isomorphic to a pullback map, see \autoref{prop:spaces_SMBF} below. The isomorphisms \eqref{eq:external_smash_and_base_change} then follow from this fact and the universal property of the pushout/pullback.

The isomorphisms of \autoref{lem:external_smash_and_base_change} and the Beck-Chevalley isomorphism of \autoref{prop:beck_chevalley_spaces} allow us to commute pullbacks, pushforwards, and smash products with each other up to isomorphism. In order to calculate traces later, we will need to have control over the isomorphism thus produced. Is it possible to compose strings of these isomorphisms and end up with two different isomorphisms between the same two functors? 

This is a coherence problem. There are general methods for solving it (see \autoref{sec:SMBF} and \cite{mp2}); however, in this category, we have a cheap shortcut. Any time a functor is isomorphic to $(f \times g)^*(X \barsmash Y)$ or $(f \times g)_!(X \barsmash Y)$, there is actually only one isomorphism.
\begin{prop}[Rigidity]\label{prop:spaces_rigidity}
	Suppose $n \geq 0$ and we have maps of spaces
	\[ \xymatrix{ B & \ar[l]_-f A \ar[r]^-g & C_1 \times \ldots \times C_n } \]
	such that $(f,g)\colon A \to B \times C_1 \times \ldots \times C_n$ is injective.
	
	Then any functor $\mc R(C_1) \times \ldots \times \mc R(C_n) \to \mc R(B)$ isomorphic to
	\[ \Phi\colon (X_1,\ldots,X_n) \leadsto f_!g^*(X_1 \barsmash \ldots \barsmash X_n) \]
	is in fact uniquely isomorphic to $\Phi$. In other words, $\Phi$ is \textbf{rigid}.\index{rigidity!for spaces}
\end{prop}

\begin{proof}
	It suffices to prove that $\Phi$ has a unique natural automorphism. Given any natural automorphism $\eta\colon \Phi \Rightarrow \Phi$, examine $\eta$ on an $n$-tuple of spaces of the form $(*_{+C_1},\ldots,*_{+C_n})$ where the free points map to $(c_1,\ldots,c_n) \in C_1 \times \ldots \times C_n$. The external smash product of these spaces is $*_{+(C_1 \times ... \times C_n)}$, with the free point mapping to $(c_1,\ldots,c_n)$. Pulling back gives $g^{-1}(c_1,\ldots,c_n)_{+A}$, which is a single fiber of $g$ regarded as a space over $A$, with a disjoint copy of $A$. Pushing forward gives a space over $B$. By assumption, every fiber of this space is either $*_+$ or $\emptyset_+$. Therefore the automorphism of this space provided by $\eta$ must be the identity map.
	
	Now suppose $(X_1,\ldots,X_n)$ is any other collection of retractive spaces in $\mc R(C_1) \times \ldots \times \mc R(C_n)$, and $x$ is some point in the total space of $f_!g^*(X_1 \barsmash \ldots \barsmash X_n)$. We will calculate the action of $\eta$ on $x$. Of course if $x$ is in the basepoint section then this action is automatically trivial, so we'll assume we aren't in that case. Then $x$ can be lifted along the following zig-zag of maps of total spaces, because the first backwards (vertical) map is surjective away from the basepoint section, and the second one is surjective everywhere.
	\[ \xymatrix @R=1.2em{
		& X_1 \times \ldots \times X_n \ar[d] \\
		g^*(X_1 \barsmash \ldots \barsmash X_n) \ar[d] \ar[r] & X_1 \barsmash \ldots \barsmash X_n \\
		f_!g^*(X_1 \barsmash \ldots \barsmash X_n) &
	} \]
	The resulting lift $(x_1,\ldots,x_n)$ corresponds to an $n$-tuple of maps $m_i\colon *_{+C_i} \to X_i$ over $C_i$. The natural transformation $\eta$ assigns to this $n$-tuple the commuting square
	\[ \xymatrix @C=6em{
		f_!g^*(*_{+C_1} \barsmash \ldots \barsmash *_{+C_n}) \ar[r]^-{\eta = \id} \ar[d]^-{f_!g^*(m_1 \barsmash \ldots \barsmash m_n)} & f_!g^*(*_{+C_1} \barsmash \ldots \barsmash *_{+C_n}) \ar[d]^-{f_!g^*(m_1 \barsmash \ldots \barsmash m_n)} \\
		f_!g^*(X_1 \barsmash \ldots \barsmash X_n) \ar[r]^-\eta & f_!g^*(X_1 \barsmash \ldots \barsmash X_n).
	} \]
	By construction, the vertical map $f_!g^*(m_1 \barsmash \ldots \barsmash m_n)$ contains the chosen point $x$ in its image.\footnote{The non-obvious part of that claim is that some point $z \in g^*(*_{+C_1} \ldots)$ is sent to our chosen lift $y$ of $x$ living in $g^*(X_1 \ldots)$. To find $z$, you have project $y$ to $A$, giving $a \in g^{-1}(c_1,\ldots)$. Then take $z$ to be the unique nontrivial point of $g^*(*_{+C_1} \ldots)$ lying over $a$. The universal property of the pullback then tells you $z$ is mapped to $y$, so the image of $z$ in $f_!g^*(*_{+C_1} \ldots)$ is sent to $x$.} Since the top horizontal map is the identity, we conclude the bottom horizontal map must send $x$ to $x$. But $x$ was arbitrary, therefore $\eta$ acts as the identity.
\end{proof}

\begin{rmk}
	The same conclusion applies if we restrict the domain of $\Phi$ to any full subcategory that contains all the $n$-tuples of spaces of the form $(*_{+C_1},\ldots,*_{+C_n})$.
\end{rmk}

\begin{cor}
	\begin{itemize}
		\item The functor
		\[ \Phi\colon (X_1,\ldots,X_n) \leadsto h^*k_!(X_1 \barsmash \ldots \barsmash X_n) \]
		is also rigid, with no restrictions on $h$ and $k$.
		\item The adjunction between the base-change functors $(f_! \adj f^*)$ is unique. Any model for $f_!$ and $f^*$ comes with one and only one adjunction between them.
	\end{itemize}
\end{cor}

\begin{proof}
	\begin{itemize}
		\item Define $f$ and $g$ by pulling back $h$ and $k$. The resulting pullback square of spaces is also a pullback square on the underlying sets, hence $f$ is injective on every fiber of $g$. Along the Beck-Chevalley isomorphism, $\Phi$ is isomorphic to the functor in \autoref{prop:spaces_rigidity}, and is therefore also rigid.
		\item It is standard that any two adjunctions are isomorphic along the identity of $f_!$ and some automorphism of $f^*$. But $f^*$ has no nontrivial automorphisms, hence any two adjunctions must be equal.
	\end{itemize}
\end{proof}

The monoidal fibrant replacement functor $P$ is not rigid, because we could for instance re-parametrize the path in $B^I$. However, any automorphism of $P$ that induces the identity on $P(B_{+B}) \cong (B^I)_{+B}$ will induce the identity on $P(*_{+B})$, and by the proof of \autoref{prop:spaces_rigidity} will therefore induce the identity on $P(X)$ for all $X \in \mc R(B)$.

\begin{prop}\label{prop:P_strong_monoidal}\cite[4.15]{malkiewich2017coassembly}
	There is a canonical homeomorphism $PX \barsmash PY \cong P(X \barsmash Y)$ of functors $\mc R(A) \times \mc R(B) \to \mc R(A \times B)$.
\end{prop}
\begin{proof}
	By ``canonical'' we mean that both functors come with an identification $F(A_{+A},B_{+B}) \cong (A^I \times B^I)_{+(A \times B)}$, and the map respects that identification. We define the canonical isomorphism by composing the isomorphisms of \autoref{lem:external_smash_and_base_change}, and tracing through the diagrams to check it respects the above identifications.
	
	Alternatively, we prove these maps are also natural with respect to replacing $A^I$ and $B^I$ with other spaces over $A$ and $B$, in which case the resulting self-isomorphism $(A^I \times B^I)_{+(A \times B)}$ is natural in $A^I$ and $B^I$ and so must be the identity by a rigidity argument.
\end{proof}

\beforesubsection
\subsection{Homology, cofiber sequences, homotopy, and fiber sequences}\label{space_cofiber_fiber}\aftersubsection

Now we can re-create the rest of classical homotopy and homology theory in the parametrized setting. The process is fairly straightforward, but it is useful to have a few guidelines and warnings.

\begin{itemize}
	\item The \textbf{strict cofiber}\index{cofiber} of a map $f\colon X \to Y$ of retractive spaces is the pushout
	\[ Y/_B X := B \cup_X Y. \]
	Over each point of $B$, this is the usual cofiber.
	\item The mapping cone or \textbf{uncorrected homotopy cofiber}\index{uncorrected homotopy cofiber}\index{mapping cone} $C_B f$ is the pushout
	\[ C_B f = X \barsmash I \cup_{X \barsmash S^0} Y \]
	where $S^0 \to I$ is regarded as a map of based spaces (retractive spaces over $*$). Over each point $b \in B$, this is the usual (uncorrected) homotopy cofiber. We also have an isomorphism $PC_B f \cong C_B Pf$.\footnote{For these last two claims, recall that in (CG) the pullback functor commutes with all colimits, and in (CGWH) the same is true for the colimit forming $C_B f$ because $B \to X$ and $S^0 \to I$ are closed inclusions.} As in the case of based spaces, $C_B f$ doesn't always have the homotopy type we want. It does if $X$ is $h$-cofibrant, using \autoref{prop:h_cofibrations_pushout_product}, or if $f$ is an $h$-cofibration. In these cases we call it the \textbf{homotopy cofiber}\index{homotopy!cofiber}. If $f$ is an $h$-cofibration, the strict and homotopy cofibers are equivalent.
	\item The \textbf{reduced homology}\index{homology} of $X$ is the collection of relative singular homology groups $H_*(X,B)$. This functor always respects weak equivalences, though it is easier to prove this when $X$ is $h$-cofibrant. One sample result is that a homotopy cofiber sequence gives a long exact sequence on homology groups
	\[ \xymatrix{ \ldots \ar[r] & H_n(X,B) \ar[r] & H_n(Y,B) \ar[r] & H_n(C_B f,B) \cong H_n(Y,X) \ar[r] & H_{n-1}(X,B) \ar[r] & \ldots } \]
	\item We form a \textbf{homotopy colimit}\index{homotopy!colimit} of a diagram of retractive spaces, or a \textbf{realization} of any simplicial retractive space, by replacing all instances of $\sma \Delta^n_+$ in the usual formula by $\barsmash \Delta^n_+$, thinking of $\Delta^n_+$ as a retractive space over $*$. The homotopy colimit preserves equivalences of diagrams of $h$-cofibrant retractive spaces, while the realization preserves equivalences of simplicial retractive spaces whose latching maps are $h$-cofibrations. 
	\item The \textbf{strict fiber}\index{fiber (strict)} of a map $f\colon X \to Y$ is the preimage of the basepoint section $B$, or equivalently the pullback
	\[ X \times_Y B. \]
	Over each point of $B$, this is the usual fiber.
	\item The \textbf{uncorrected homotopy fiber}\index{uncorrected homotopy fiber} $F_B f$ is the pullback
	\[ F_B f = X \times_{\barmap_*(S^0,Y)} \barmap_*(I,Y) \]
	where $\barmap_*$ is our right adjoint to $\barsmash$ that outputs a space over $B$. Over each point of $B$, this is the usual homotopy fiber. It preserves equivalences as soon as $Y$ is $h$-fibrant or $X \to Y$ is an $h$-fibration; in these cases we call it the \textbf{homotopy fiber}\index{homotopy!fiber}. (Using \autoref{ex:smashing_spaces_is_left_Quillen} below, this is true more generally when $Y$ is $q$-fibrant or $f$ is a $q$-fibration.)
	
	\textbf{Warning:} The uncorrected homotopy fiber does not always preserve equivalences. Take $X = B = I$ and $Y$ to be the cone on a circle, with projection to $B$ by the cone coordinate, joined to an extra copy of $B$ along a single point of the base circle. Since $Y$ is contractible, the (corrected) homotopy fiber of $X \to Y$ is contractible, but the uncorrected homotopy fiber has total space equivalent to $\Z$.\footnote{This indicates a (fixable) error in the proof of \cite[12.4.2]{ms}. The map labeled $F_B(I,Y) \to Y$, which in our language is $\barmap_*(I,Y) \to Y$, is only a Hurewicz fibration once we also assume that $Y$ is $h$-fibrant.}
	\item The \textbf{homotopy groups}\index{homotopy!groups} of $X$ are the homotopy groups of the fibers $X_b$. These only respect weak equivalences if $X$ is quasifibrant. One sample result is that a homotopy fiber sequence of fibrant spaces gives a long exact sequence on homotopy groups
	\[ \xymatrix{ \ldots \ar[r] & \pi_{n+1,b}(Y) \ar[r] & \pi_{n,b}(F_B f) \cong \pi_{n+1,b}(Y;X) \ar[r] & \pi_{n,b}(X) \ar[r] & \pi_{n,b}(Y) \ar[r] & \ldots } \]
	\item We can form a \textbf{homotopy limit}\index{homotopy!limit} of a diagram or the \textbf{totalization} of a cosimplicial retractive space by replacing all instances of $F(\Delta^n_+,-)$ in the usual formulas by $\barmap_*(\Delta^n_+,-)$. The homotopy limit preserves equivalences as soon as the retractive spaces in our diagram are all $q$-fibrant, and the totalization appears to also preserve equivalences of cosimplicial objects whose matching maps are $q$-fibrations.
	\item Define the homotopy category $\ho\mc R(B)$ by inverting the weak equivalences. Taking maps in this category, maps out of a homotopy cofiber sequence or into a homotopy fiber sequence form a long exact sequence. Maps out of a sequential direct limit (colimit) of cofibrations or into a sequential inverse limit of fibrations give a $\lim^1$ exact sequence.\footnote{The classical argument gives you these statements for maps up to fiberwise based homotopy. Using the fundamental theorem of model categories (\autoref{sec:model_cats}), this is the same as maps in the homotopy category, assuming the sources are made cofibrant and the targets are made fibrant first. See \cite[5.6]{ms} for more details, and for simplicity replace every instance of ``well-grounded'' with ``cofibrant.''}
\end{itemize}


\beforesubsection
\subsection{Symmetric monoidal categories and bifibrations}\label{sec:SMBF}\aftersubsection

The external smash product $\barsmash$ makes the category $\mc R$ of all retractive spaces into a \textbf{symmetric monoidal category}\index{symmetric monoidal!category}. This means there is a product functor and unit object
\[ \xymatrix{ \mc R \times \mc R \ar[r]^-{\barsmash} & \mc R & {*} \ar[r]^-{S^0} & \mc R } \]
such that the product is associative, commutative, and unital. This does not have to happen strictly, so we ask for four natural isomorphisms $\alpha, \lambda, \rho, \gamma$ making the four diagrams of functors below commute.
\[ \begin{array}{ccc}
\xymatrix{
	\mc R \times \mc R \times \mc R \ar[d]_-{\barsmash \times 1} \ar[r]^-{1 \times \barsmash} & \mc R \times \mc R \ar[d]^-{\barsmash} \\
	\mc R \times \mc R \ar[r]_-{\barsmash} & \mc R
}
&
\xymatrix{
	\mc R \ar[d]_-{S^0 \times 1} \ar@{=}[dr] \ar[r]^-{1 \times S^0} & \mc R \times \mc R \ar[d]^-{\barsmash} \\
	\mc R \times \mc R \ar[r]_-{\barsmash} & \mc R
}
&
\xymatrix{
	\mc R \times \mc R \ar[dr]_-{\barsmash} \ar[r]^-{\textup{swap}} & \mc R \times \mc R \ar[d]^-{\barsmash} \\
	& \mc R
}
\\[5em]
\alpha\colon (X \barsmash Y) \barsmash Z \cong X \barsmash (Y \barsmash Z) &
\lambda\colon S^0 \barsmash X \cong X &
\gamma\colon X \barsmash Y \cong Y \barsmash X \\
& \rho\colon X \barsmash S^0 \cong X &
\end{array}
\]
Finally, these four isomorphisms have to be coherent. This means that any time we have an expression for the product of $X_1$, $\ldots$, $X_n$, such as
\[ (X_2 \barsmash (S^0 \barsmash X_3)) \barsmash (S^0 \barsmash X_1), \]
if we use the isomorphisms $\alpha, \lambda, \rho, \gamma$ to re-arrange the terms, and at some point arrive at the same expression for the product again, the composite of all the isomorphisms we have applied must be the identity map. In fact, it is enough to check this condition for just four particular cases, and then every other case follows automatically. This is a non-trivial theorem whose details and proof can be found in \cite[VII,XI]{maclane}.

The practical consequence of this coherence theorem is that we can take an $n$-fold product of a set of objects $\{X_a\}_{a \in A}$, $|A| = n$, without having to worry about which order we put them in or how we group them with parentheses. Any two expressions for this product admit a canonical isomorphism between them arising from $\alpha, \lambda, \rho, \gamma$, so we can treat them as the same without getting into trouble.

\begin{prop}
	There is a unique symmetric monoidal structure on $\mc R$ whose product is $\barsmash$ and unit is $S^0 \in \mc R(*)$.
\end{prop}

\begin{proof}
	Once we fix a model for $\barsmash$ and $S^0$, it is straightforward to construct natural isomorphisms
	\[ (X \barsmash Y) \barsmash Z \cong X \barsmash (Y \barsmash Z), \quad
	S^0 \barsmash X \cong X, \quad
	X \barsmash Y \cong Y \barsmash X.
	\]
	By the rigidity result \autoref{prop:spaces_rigidity}, these isomorphisms are then unique, so there is only one possible choice for $\alpha, \lambda, \rho, \gamma$. Their coherence also follows immediately from \autoref{prop:spaces_rigidity}.
\end{proof}

\begin{ex}
	The category $\cat{Top}$ of (CG or CGWH) spaces is a symmetric monoidal category with product $\times$, unit $*$, and isomorphisms arising from the universal property of the product. The universal property makes coherence easy to check. This works for any category with a product; the symmetric monoidal categories that arise in this way are called \textbf{cartesian monoidal categories}.
\end{ex}

If $(\mc C,\otimes,I_C)$ and $(\mc D,\boxtimes,I_D)$ are symmetric monoidal categories, a \textbf{(strong) symmetric monoidal functor}\index{symmetric monoidal!functor} is a functor $F\colon \mc C \to \mc D$ and two natural isomorphisms $m_F$, $i_F$ making these squares commute:
\[ \begin{array}{ccc}
\xymatrix @R=1.7em @C=3em{
	\mc C \times \mc C \ar[d]_-{F \times F} \ar[r]^-{\otimes} & \mc C \ar[d]^-F \\
	\mc D \times \mc D \ar[r]^-{\boxtimes} & \mc D
}
&&
\xymatrix @R=1.7em{
	{*} \ar@{=}[d] \ar[r]^-{I_C} & \mc C \ar[d]^-F \\
	{*} \ar[r]^-{I_D} & \mc D
}
\\
\\
F(X) \boxtimes F(Y) \cong F(X \otimes Y)
&&
I_D \cong F(I_C)
\end{array} \]
Furthermore, each of the four isomorphisms $\alpha, \lambda, \rho, \gamma$ in $\mc D$ corresponds to the same isomorphism in $\mc C$ along these maps. For instance, this condition for $\alpha$ means that the following hexagon must commute.
\[ \xymatrix @R=1.7em @C=5em{
	(F(X) \boxtimes F(Y)) \boxtimes F(Z) \ar[d]_-{\alpha} \ar[r]^-{m \boxtimes \id}
	& F(X \otimes Y) \boxtimes F(Z) \ar[r]^-{m}
	& F((X \otimes Y) \otimes Z) \ar[d]^-{F(\alpha)} \\
	F(X) \boxtimes (F(Y) \boxtimes F(Z)) \ar[r]^-{\id \boxtimes m}
	& F(X) \boxtimes F(Y \otimes Z) \ar[r]^-{m}
	& F(X \otimes (Y \otimes Z)) \\
} \]
One can visualize these six maps as the faces of a cube, obtained by taking the square of functors that represents $\alpha$ and multiplying by an interval representing $F$. In other words, the definition of a symmetric monoidal \emph{functor} is obtained from the definition of a symmetric monoidal \emph{category} by making the replacements
\begin{align*}
	\textup{category} &\to \textup{functor} \\
	\textup{functor} &\to \textup{natural isomorphism} \\
	\textup{natural isomorphism} &\to \textup{coherence condition} \\
	\textup{coherence condition} &\to \textup{nothing}.
\end{align*}
\begin{ex}
	The canonical homeomorphism from \autoref{prop:P_strong_monoidal} and the unique homeomorphism $P(S^0) \cong S^0$ make $P$ into a strong symmetric monoidal functor. 
\end{ex}

We have now given the category $\mc R$ three operations, $\barsmash$, $f^*$, and $f_!$. Loosely, these operations always commute with each other, because of \autoref{lem:external_smash_and_base_change} and \autoref{prop:beck_chevalley_spaces}. A category with this calculus of three operations is called a \textbf{symmetric monoidal bifibration (SMBF)}\index{symmetric monoidal!bifibration}\index{SMBF}. In more detail, an SMBF consists of
\begin{itemize}
	\item a symmetric monoidal category $(\mc C,\boxtimes,I)$,
	\item a cartesian monoidal category $(\bS,\times,*)$,
	\item a functor $\Phi\colon \mc C \to \bS$,
	\item and a class of commuting squares (\textbf{Beck-Chevalley squares})\index{Beck-Chevalley square} in $\bS$, containing at least the squares listed in \cite[5.3]{mp2},
\end{itemize}
with the following properties.
\begin{itemize}
	\item The functor $\Phi$ is strict symmetric monoidal, meaning it preserves the product and unit on the nose, and carries the isomorphisms $\alpha, \lambda, \rho, \gamma$ in $\mc C$ to the corresponding isomorphisms in $\bS$. (In other words, when we define the symmetric monoidal functor structure on $\Phi$, the two isomorphisms we pick are actually identity maps.)
	\item For every arrow $f\colon A \to B$ in $\bS$ and object $X$ in the fiber category $\mc C^B = \Phi^{-1}(B)$, there is a \textbf{cartesian arrow}\index{cartesian/co-cartesian arrow} $f^*X \to X$ in $\mc C$ with the universal property illustrated below.
	\[ \xymatrix @R=1em{
		Y \ar@{~>}[dd]^-\Phi \ar[rrd] \ar@{-->}[rd]_-{\exists !} & & \\
		& f^*X \ar[r] \ar@{~>}[d]^-\Phi & X \ar@{~>}[d]^-\Phi \\
		\Phi(Y) \ar[r] & A \ar[r]^-f & B
	} \]
	\item For every $f\colon A \to B$ and $X$ in $\mc C^A = \Phi^{-1}(A)$, there is a \textbf{co-cartesian arrow} $X \to f_!X$ with the universal property illustrated below.
	\[ \xymatrix @R=1em{
		& & Y \ar@{~>}[dd]^-\Phi \\
		X \ar[r] \ar[rru] \ar@{~>}[d]^-\Phi & f_!X \ar@{~>}[d]^-\Phi \ar@{-->}[ru]_-{\exists !} & \\
		A \ar[r]^-f & B \ar[r] & \Phi(Y)
	} \]
	\item The product $\boxtimes$ preserves cartesian arrows and co-cartesian arrows.
	\item In any Beck-Chevalley square, the Beck-Chevalley map from \autoref{prop:beck_chevalley_spaces} is an isomorphism.
\end{itemize}
See \cite{shulman_framed_monoidal,mp2}. The proof of the following mostly amounts to translating the above universal properties back into the more familiar ones we used to define $f^*$ and $f_!$ earlier, and then re-interpreting the proof of \autoref{lem:external_smash_and_base_change}.

\begin{prop}\label{prop:spaces_SMBF}
	The category $\mc R$ of all retractive spaces, with the unique symmetric monoidal structure coming from $\barsmash$ and $S^0$, is a symmetric monoidal bifibration.
\end{prop}

\begin{rmk}
	If we drop every reference to the product $\barsmash$, we get the notion of a \textbf{bifibration}\index{bifibration} of categories. In other words, a bifibration is a map of categories $\Phi\colon \sC \to \bS$, with pushforwards and pullbacks that ``commute'' along Beck-Chevalley isomorphisms.
\end{rmk}

As we have seen in the case of $\mc R$, the nice thing about bifibrations is that the definitions of $f_!$, $f^*$, their adjunction, and the canonical isomorphisms between their composites, are all rolled into the definition of the category $\mc R$ itself. For instance we can reconstruct $f_!$ by choosing a co-cartesian arrow over $f$ for each $X \in \mc C^A$, and similarly for the pullback $f^*$. This point of view often leads to major simplifications.

For instance, in what sense does $\barsmash$ commute with pullbacks? Back in \autoref{lem:external_smash_and_base_change} we had to give an explicit description of the commutation isomorphisms, but now we can just say that $\barsmash$ preserves cartesian and co-cartesian arrows. We are lucky that in this case the commutation isomorphism happens to be unique. As soon as we pass to the homotopy category, that is no longer necessarily the case (and is certainly false in the homotopy category of spectra). The structure of a bifibration helps us keep track of which isomorphism is the \emph{canonical} one.

\beforesubsection
\subsection{Internal smash products}\label{sec:internal}\aftersubsection

We have seen that $\mc R$ is symmetric monoidal under the external smash product $\barsmash$. There is another product that makes $\mc R(B)$ into a symmetric monoidal category for any fixed base space $B$. Define the \textbf{internal smash product}\index{internal smash product $X \sma_B Y$} by
\[ \sma_B\colon \mc R(B) \times \mc R(B) \to \mc R(B)
\qquad X \sma_B Y = \Delta_B^*(X \barsmash Y), \]
where $\Delta_B\colon B \to B \times B$ is the diagonal map. Applying $\Delta_B^*$ to the definition of the external smash product gives the pushout diagram
\begin{equation}\label{eq:internal_smash}
\xymatrix @R=1.7em{
	X \cup_B Y \ar[d] \ar[r] & X \times_B Y \ar[d] \\
	B \ar[r] & X \sma_B Y.
}
\end{equation}
Informally, $X \sma_B Y$ is a space over $B$ whose fiber over $b$ is the smash product $X_b \sma Y_b$.

\begin{ex}\label{ex:internal_smash_thom_spectra}
	If $V$ and $W$ are vector bundles over $B$ and $V \times_B W$ is their fiber product, there is an isomorphism
	\[ \Th_B(V \times_B W) \cong \Th_B(V) \sma_B \Th_B(W) \]
	obtained by combining the isomorphisms from \autoref{ex:thom_smash_product} and \autoref{ex:thom_base_change}.
\end{ex}

By \autoref{prop:spaces_rigidity}, the internal smash product $\sma_B$ is a rigid functor. The commutation of pullbacks $f^*$ with $\barsmash$ allows us to verify that it is associative, unital, and commutative up to some natural isomorphisms. It follows immediately that those natural isomorphisms are unique and coherent, making $\mc R(B)$ into a symmetric monoidal category. The unit is the 0-sphere over $B$,
\[ S^0_B := S^0 \times B \cong r_B^* S^0. \]

The categories $\mc R(B)$, together with the products $\sma_B$ and pullbacks $f^*$, form an \textbf{indexed symmetric monoidal category with coproducts}\index{indexed category}. In broad strokes, this structure says that the pullbacks compose in an associative way and commute with the internal smash products. The pushforwards $f_!$ do not commute with the internal smash products, and instead satisfy a projection formula. We will not define this structure precisely because it is more complicated than that of a symmetric monoidal bifibration, and because these two formal structures are actually equivalent to each other \cite{shulman_framed_monoidal}.

Concretely, these structures are equivalent because $\sma_B$ and $\barsmash$ can be recovered from each other. We already defined $\sma_B$ using $\barsmash$ and $\Delta_B^*$. Going the other way, we reconstruct $\barsmash$ from the internal smash products by the formula
\[ X \barsmash Y = (\pi_A^* X) \sma_{A \times B} (\pi_B^* Y) \]
where $\pi_A\colon A \times B \to A$ and $\pi_B\colon A \times B \to B$ are the projections.

The internal smash product $\sma_B$ is easier to contemplate than $\barsmash$, but it does not interact as well with cofibrations, fibrations, and weak equivalences.
\begin{cor}\label{cor:internal_smash_properties}\hfill
	\vspace{-1em}
	
	\begin{itemize}
		\item Pushout-products of $f$-cofibrations defined using $\sma_B$ are $f$-cofibrations.
		\item $\sma_B$ preserves spaces that are both $f$-cofibrant and $h$-fibrant, and also weak equivalences between such spaces.
		\item More generally, $\sma_B$ preserves equivalences when both spaces are $f$-cofibrant and at least one of them is $q$-fibrant.
	\end{itemize}
\end{cor}

\begin{proof}
	The first two points follow from \autoref{lem:f_star_preserves} and \autoref{prop:h_cofibrations_pushout_product}. For the last point, assume $X$ and $Y$ are $f$-cofibrant, and $X$ is $q$-fibrant. Since pullbacks and pushout-products preserve $f$-cofibrations, $X \cup_B Y \to X \times_B Y$ is an $f$-cofibration of spaces over $B$, so \eqref{eq:internal_smash} is a homotopy pushout square. The term $X \cup_B Y$ preserves equivalences because $X$ and $Y$ are cofibrant, and $X \times_B Y$ preserves equivalences because it is a pullback along a $q$-fibration. Therefore $X \sma_B Y$ preserves equivalences under these assumptions.
\end{proof}

The internal smash product does not preserve weak equivalences if both spaces are not fibrant. Consider the following counterexample. Over a connected space $B$, all retractive spaces of the form $*_{+B}$ are equivalent. But their internal smash products are either the zero object or $*_{+B}$ again, according to whether the two points lie over different points of $B$ or the same point of $B$.

\begin{rmk}\label{fixed_may_sigurdsson}
	The earlier work \cite{ms} regards $\sma_B$ as the most fundamental smash product, and defines $\barsmash$ in terms of it. This follows the philosophy of \cite{clapp1981duality,crabb_james} that every construction and definition should be a fiberwise version of the classical construction. In hindsight, however, it turns out to be better to regard the external smash product $\barsmash$ as the more fundamental notion, and to define the internal smash product $\sma_B$ in terms of it.
	
	One might expect this does not matter, because $\barsmash$ and $\sma_B$ can each be defined in terms of the other.  But $\sma_B$ does not preserve equivalences nearly as often as $\barsmash$, as we have seen above. As a result, using $\sma_B$ suggests methods of argument that require stronger hypotheses to work (essentially $f$-cofibrations wherever $h$-cofibrations would suffice). This eventually leads to the well-grounded objects and the $qf$-model structure of \cite{ms}. It is surprising, but true, that we can avoid these difficulties simply by starting with $\barsmash$ and waiting until the end to use $\sma_B$. We will discuss this further in \autoref{fixed_may_sigurdsson_2}.
\end{rmk}

We also briefly mention that $\sma_B$ generalizes to something called the \textbf{external smash product rel $B$}\index{external!smash rel $B$ $X \barsma{B} Y$}, $\barsma{B}$. When $X \in \mc R(D)$ and $Y \in \mc R(E)$, and $D$ and $E$ are equipped with maps to $B$, this relative external smash product is defined as
\[ X \barsma{B} Y := \Delta_{D,E}^*(X \barsmash Y), \]
\[ \Delta_{D,E}\colon D \times_B E \to D \times E. \]
It is essentially the external smash product, but carried out separately over every point of $B$. This makes the category of retractive spaces $X$ over spaces $E$ over $B$ into a symmetric monoidal category. The symmetric monoidal category $\mc R(B)$ that we constructed above is contained inside, as the subcategory on which $E \to B$ is the identity map of $B$.

\newpage
\section{Derived functors}

Up to this point, we have been informally talking about studying retractive spaces ``up to weak equivalence.'' We now recall how to formalize this using homotopical categories. This leads to a derived pushforward $\L f_!$, derived pullback $\R f^*$, and derived external smash product $\barsmash^\L$, that have the same relationship up to weak equivalence that $f_!$, $f^*$ and $\barsmash$ have on the nose.

\beforesubsection
\subsection{Derived functors in homotopical categories}\label{sec:just_derived_functors}\aftersubsection

Recall the following definition from \cite{dhks}. A \textbf{homotopical category}\index{homotopical!category} $(\cat C,W)$ is a category $\cat C$ and a class of morphisms $W$ in $\cat C$ satisfying the following 2-out-of-6 property. Given three composable maps $f, g, h$, if $g \circ f$ and $h \circ g$ are in $W$ as shown below, then $f$, $g$, $h$, and $h \circ g \circ f$ are also in $W$.
\[ \xymatrix @C=3em @R=1.7em{
	A \ar[r]^-f \ar[rd]_-\sim & B \ar[d]_-g \ar[rd]^-\sim & \\
	& C \ar[r]^-h & D
} \]
In addition, every isomorphism of $\cat C$ must be in $W$. We call the maps in $W$ the \textbf{weak equivalences} of $\cat C$.

\begin{ex}
	The category $\mc R(B)$ of retractive spaces over $B$, together with the weak homotopy equivalences, is a homotopical category. Any finite product $\mc R(B_1) \times \ldots \times \mc R(B_n)$ is also a homotopical category, using the maps that are $n$-tuples of weak equivalences.
\end{ex}

The two conditions on $W$ are automatically satisfied if there is a functor $\delta\colon \cat C \to \cat D$ and $W$ consists of those maps $f$ for which $\delta(f)$ is an isomorphism. The converse of this is not quite true \cite[27.5]{dhks}. But when $W$ satisfies the above two conditions, one can find a functor $\delta$ that turns the maps in $W$ into isomorphisms and does as little else as possible.
\begin{prop}\cite[1.1]{gabriel_zisman}
	For any homotopical category $(\cat C,W)$, there is a \textbf{homotopy category}\index{homotopy!category} $\ho\cat C = \cat C[W^{-1}]$ and a functor $\delta\colon \cat C \to \ho\cat C$, that is initial among all functors out of $\cat C$ sending weak equivalences to isomorphisms.
\end{prop}

We may take the objects of $\ho \cat C$ to be the objects of $\cat C$, and we do so throughout this text. The set of maps from $X$ to $Y$ is the set of zig-zags
\[ \xymatrix{
	X \ar[r] &  X_1 & \ar[l]_-\sim X_2 & \ar[l]_-\sim X_3 \ar[r] &  X_4 & \ar[l]_-\sim X_5 \ar[r] &  X_6 \ar[r] &  X_7 \ar[r] &  \ldots \ar[r] &  Y
} \]
in which every backwards map is a weak equivalence, subject to the relation that we may compose (or un-compose) two maps in the same direction, and cancel out (or introduce) a pair of maps of the form $\overset{w}\rightarrow \overset{w}\leftarrow$ or $\overset{w}\leftarrow \overset{w}\rightarrow$ for any weak equivalence $w$.\footnote{Technically, this forms a class in general and not necessarily a set. We are guaranteed it is a set when $\cat C$ is a model category, see \cite{hovey_model_cats}.}

A \textbf{homotopical functor}\index{homotopical!functor} $F\colon \cat C \to \cat D$ is a functor that sends every weak equivalence to a weak equivalence. By the universal property of $\ho \cat C$, for each homotopical functor $F$ there is a unique functor $\ho F$, the \textbf{total derived functor}\index{derived functor} of $F$, making this square of functors strictly commute.
\[ \xymatrix @R=1.7em @C=3em{
	\cat C \ar[d]_-\delta \ar[r]^-F & \cat D \ar[d]^-\delta \\
	\ho\cat C \ar@{-->}[r]_-{\ho F} & \ho\cat D
} \]
We sometimes say ``$F$ on the homotopy category'' instead of $\ho F$, since $F$ determines $\ho F$ on the nose.

Any natural transformation $\eta\colon F \Rightarrow G$ of homotopical functors induces a natural transformation $\ho\eta\colon \ho F \Rightarrow \ho G$.
Passing to the homotopy category preserves composition, both of functors and of natural transformations. This can be restated in the language of 2-categories as follows.

\begin{prop}
	The above operations specify a 2-functor from the 2-category of homotopical categories, homotopical functors, and natural transformations to the 2-category of categories, functors, and natural transformations.
\end{prop}
\begin{cor}
	Any adjunction of homotopical functors $(F \adj G)$ gives an adjunction on their images in the homotopy category $(\ho F \adj  \ho G)$.
\end{cor}

Next recall that a \textbf{left deformation retract} or \textbf{cofibrant replacement} is any functor $Q\colon \cat C \to \cat C$ that admits an equivalence to the identity, $QX \overset\sim\to X$. The functor $F\colon \cat C \to \cat D$ is \textbf{left-deformable}\index{left-deformable!functor} if we can find a full subcategory $\cat A \subset \cat C$ on which $F$ preserves weak equivalences, and a cofibrant replacement functor $Q$ whose image lies in $\cat A$.

If $F$ is left-deformable, its \textbf{left-derived functor}\index{derived functor} is the homotopical functor
\[ \L F = F \circ Q. \]
%
The left-derived functor satisfies two universal properties, which imply that it does not actually depend on the choice of retraction functor $Q$. To state these properties, consider the category of all functors $\Fun(\cat C, \cat D)$, with weak equivalences given by natural transformations that are a weak equivalence at each object. The homotopical functors form a full subcategory $\Fun^h(\cat C,\cat D)$.

\begin{prop}\label{prop:derived_functors_universal_property}
	If one such $\cat A$ and $Q$ exist, then $\ho\L F$ is terminal among all those functors $\ho\cat C \to \ho\cat D$ admitting a map to $F$ as functors $\cat C \to \ho\cat D$. Likewise $\L F$ is terminal in comma category $\ho(\Fun^{h}(\cat C,\cat D) \downarrow F)$ of homotopy functors mapping to $F$, with the weak equivalences inverted.
\end{prop}

\begin{proof}
	The first part is classical and a proof can be found for instance in \cite[6.4]{riehl_basic}. The second part is perhaps less classical but it is proven by exactly the same argument.
\end{proof}

So $\L F$ satisfies two universal properties: one in the homotopy category of functors $\ho\Fun(\cat C,\cat D)$, and one as functors on the homotopy category $\Fun(\ho\cat C,\ho\cat D)$.\footnote{In fact, the left-derived functor satisfies at least four more universal properties, before we even begin to consider the ``homotopical'' universal properties that state that a certain space of maps is weakly contractible. The two that we have listed will be the most directly useful to us.}
Therefore any two models of $\L F$ are \emph{canonically isomorphic} in these categories. In fact, any map between two models of $\L F$, in either of these two settings, must be equal to the canonical isomorphism, as soon as we know that it commutes with the map back to $F$.\footnote{If $F$ is not left-deformable, the situation gets more complicated. There might be a functor satisfying only one of the above two universal properties. Or there may be two different ones, and the image of one in the homotopy category is not equivalent to the other. Our discussion here is condensed from the more general discussion found in \cite{dhks} and \cite{shulman2006homotopy}.}

If instead $Q$ admits a natural equivalence \emph{from} the identity, we call it a \textbf{right deformation retract} or \textbf{fibrant replacement}, denoting it by $R$ instead of $Q$. Dualizing the above discussion gives the \textbf{right derived functor}\index{derived functor}\index{right-deformable!functor} $\R F = F \circ R$. This functor, and its image in the homotopy category, satisfy the duals of the above universal properties. So they are also unique up to canonical isomorphism.

\begin{ex}\label{ex:strict_and_uncorrected}\hfill
	\vspace{-1em}
	
	\begin{itemize}
		\item Every homotopical functor $F$ has both a left- and a right-derived functor, both of which are equivalent to $F$ itself. 
		\item The pushforward $f_!\colon \mc R(A) \to \mc R(B)$ is left-deformable. It is homotopical on the subcategory of $h$-cofibrant spaces, and the whiskering functor $W$ from \autoref{prop:whiskering_properties} gives a left retraction to this category. 
		\item The pullback $f^*\colon \mc R(B) \to \mc R(A)$ is right-deformable. It is homotopical on the subcategory of $q$-fibrant spaces, and the functor $RX = X \times_B B^I$, gives a right deformation onto this subcategory. 
		\item The external smash product $\barsmash\colon \mc R(A) \times \mc R(B) \to \mc R(A \times B)$ is left-deformable. It is homotopical on the subcategory of pairs of $h$-cofibrant spaces.
		\item The external mapping space $\barmap_B\colon \mc R(B)^{\op} \times \mc R(A \times B) \to \mc R(A)$ is right-deformable. It is homotopical on the subcategory of pairs where the first space is $h$-cofibrant and homotopy equivalent to a cell complex, and the second is $h$-fibrant.
		\item If $f$ is a fiber bundle with cell complex fiber, then $f_*$ is right-deformable by \autoref{prop:sheafy_pushforward_preserves}. In (CG), $f_*$ is defined for every map $f$, but this is of limited use because it does not appear to be right-deformable in general.
		\item The strict cofiber $SC(-)$ and uncorrected homotopy cofiber $C(-)$ are both left-deformable functors $a\mc R(B) \to \mc R(B)$. 
		The collapse map $C(-) \to SC(-)$ induces an equivalence $\L C(-) \overset\sim\to \L SC(-)$, and we refer to both left-derived functors as the homotopy cofiber. This all follows directly from the gluing lemma, \autoref{prop:technical_cofibrations}.
		\item Dual to the above example, the strict fiber and uncorrected homotopy fiber are both right-deformable, their right-derived functors are equivalent, and we call them the homotopy fiber.
		\item The homotopy colimit of a diagram of spaces $X\colon \cat I \to \cat{Top}$ is the left-derived functor of the colimit, applied to $X$.\index{homotopy!colimit} As a functor over $\colim$, it is (canonically) equivalent to the functor given by the usual Bousfield-Kan formula \cite[XII]{bousfield_kan}.
		\item If $R$ is a ring, the Hochschild complex $B^{\cyc}_*(R;M)$ for $R$-$R$ bimodules $M$ is the left-derived functor of the abelianization functor
		\[ M \leadsto M/(rm = mr, r \in R), \]
		in the homotopical category of $R$-chain complexes and quasi-isomorphisms.\index{Hochschild homology} There are two common models for this, one by a cyclic bar construction and one by a derived tensor of $R$ and $M$ over $R \otimes R^{\op}$. However these two models must agree, up to a canonical quasi-isomorphism of functors lying over the abelianization.
	\end{itemize}
\end{ex}

\begin{rmk}
	If $F$ has neither a left nor a right derived functor, it may have a ``middle derived functor'' $\mathbb M F$\index{middle-deformable}\index{derived functor} obtained by applying a sequence of left and right deformation retracts. For instance, the internal smash product $\sma_B$ is middle-deformable. Unfortunately, middle-derived functors are not unique in general, they depend explicitly on which retractions we choose. 
\end{rmk}
	

\begin{warn}\label{warn:changing_equivs}
	If we change the class of weak equivalences $W$, the derived functors $\L F$ or $\R F$ can change. For instance let $G$ be a finite group, $\cat C$ the category of $G$-spaces, and $F(X) = X^G$, the $G$-fixed points of $X$. If we derive $F$ using the maps inducing equivalences on all fixed-point subspaces $X^H \overset\sim\to Y^H$, we just get $X^G$ again. But if we derive $F$ using the maps inducing equivalences on the underlying space $X \overset\sim\to Y$, we get the homotopy fixed points $X^{hG}$.
\end{warn}

\beforesubsection
\subsection{Composing and comparing derived functors}\label{sec:composing_comparing}\aftersubsection

We need to lift the isomorphisms that commute $f_!$, $f^*$, and $\barsmash$ to weak equivalences that commute the derived functors $\L f_!$, $\R f^*$, and $\barsmash^{\L}$, in a canonical way. We will describe a new theoretical framework for doing this, since the approach of \cite{ms} relies on choosing a class of cofibrant-fibrant objects, but we don't wish to have the theory depend on such choices. We take motivation from \cite{shulman_comparing} and adapt the idea to accommodate longer strings of functors.

Suppose we have a composable list of functors of homotopical categories
\[ \xymatrix{
	\cat C \ar@{=}[r] & \cat C_0 \ar[r]^-{F_1} & \cat C_1 \ar[r]^-{F_2} & \ldots \ar[r]^-{F_n} & \cat C_n \ar@{=}[r] & \cat C'
} \]
and each functor $F_i$ is left or right deformable.\footnote{If $F_i$ is both left and right deformable and the canonical map $\L F_i \to F_i \to \R F_i$ is \emph{not} an equivalence, fix a choice of whether to look at $\L F_i$ or $\R F_i$. If the canonical map is an equivalence (for instance if $F_i$ is already a homotopy functor) then the choice does not matter for the discussion that follows.} Let $\D F_i$ refer to the left or right derived functor of $F_i$, and $\D F_i \leftrightarrow F_i$ the canonical map in the homotopy category of functors.

Our essential task is to take an isomorphism between two such composites, and deduce that the composites of the derived functors $\D F_n \circ \ldots \circ \D F_2 \circ \D F_1$ are also equivalent. This is not possible in general, even if all the functors are left-deformable. For instance, the colimit functor is a composite of a coproduct and a coequalizer:
\[ \colim \cong \textup{coeq} \circ \amalg. \]
But this isomorphism of point-set functors does not become an equivalence on the composites of left-derived functors:
\[ \L \colim \not\simeq \L(\textup{coeq}) \circ \L(\amalg). \]
In other words, the homotopy colimit is not the homotopy coequalizer of the coproduct.

So our task is not possible unless we impose some kind of condition on the list $(F_1,\ldots,F_n)$. We say the list is \textbf{coherently deformable}\index{coherently deformable} if there exists a full subcategory $\cat A_{i-1} \subseteq \cat C_{i-1}$ for each $1 \leq i \leq n$ such that
\begin{itemize}
	\item $\cat A_{i-1}$ contains the image of some composite of left and right retraction functors applied to $\cat C_{i-1}$,
	\item on $\cat A_{i-1}$, the canonical map $\D F_i \leftrightarrow F_i$ is an equivalence,\footnote{Notice this implies that $\D F_i$ satisfies the same universal property on $\cat A_{i-1}$ that it satisfies on $\cat C_{i-1}$, and also that $F_i$ is homotopical on $\cat A_{i-1}$.}
	\item and $F_i(\cat A_{i-1}) \subseteq \cat A_i$.
\end{itemize}
This condition is preserved if we restrict to a subset of the above composite, or if we replace any functor $F_i$ by $\D F_i$. Given another such list
\[ \xymatrix{
	\cat C \ar@{=}[r] & \cat D_0 \ar[r]^-{G_1} & \cat D_1 \ar[r]^-{G_2} & \ldots \ar[r]^-{G_k} & \cat D_k \ar@{=}[r] & \cat C'
} \]
we say that the \textbf{pair is coherently deformable} if there exists some coherent deformation for each one with the same choice of subcategory $\cat A_0 \subseteq \cat C$.
\begin{prop}[Canonical recipe for deriving a natural transformation]\label{prop:passing_natural_trans_to_derived_functors}\hfill
	Given a pair of coherently deformable composites as above, for each natural transformation
	\[ \eta\colon F_n \circ \ldots \circ F_1 \Rightarrow G_k \circ \ldots \circ G_1 \]
	there is a unique morphism in the homotopy category of homotopy functors, or of functors on the homotopy category,
	\[ \D \eta\colon \D F_n \circ \ldots \circ \D F_1 \Rightarrow \D G_k \circ \ldots \circ \D G_1 \]
	that when restricted to $\cat A_0$ agrees with
	\begin{align*}
	& \D F_n \circ \ldots \circ \D F_2 \circ \D F_1 \\
	\simeq\ & \D F_n \circ \ldots \circ \D F_2 \circ F_1 \\
	\simeq\ & \D F_n \circ \ldots \circ F_2 \circ F_1 \\
	\simeq\ & F_n \circ \ldots \circ F_2 \circ F_1 \\
	\overset\eta\to\ & G_k \circ \ldots \circ G_2 \circ G_1 \\
	\simeq\ & \D G_k \circ \ldots \circ G_2 \circ G_1 \\
	\simeq\ & \D G_k \circ \ldots \circ \D G_2 \circ G_1 \\
	\simeq\ & \D G_k \circ \ldots \circ \D G_2 \circ \D G_1.
	\end{align*}
	Furthermore $\D \eta$ does not depend on the choice of $\cat A_0$.
\end{prop}

In particular, if $\eta$ is an equivalence on some (therefore on all) choices of $\cat A_0$, $\D\eta$ is an isomorphism. The defining property actually forces $\D$ to respect vertical compositions and units, so that $\D(\eta_2 \circ \eta_1) = \D \eta_2 \circ \D \eta_1$ (whenever the three strings of functors are jointly coherently deformable) and $\D(\id) = \id$. It also respects horizontal compositions, i.e. if $\eta$ is a horizontal composite of two smaller transformations $\alpha * \beta$ along two sub-composites of $F_n \circ \ldots \circ F_1$, then $\D\eta = \D\alpha * \D\beta$.\footnote{We can't say that $\D$ is a 2-functor because the set of all transformations $\eta$ satisfying the assumptions isn't the set of 2-cells in a 2-category. The horizontal and vertical compositions are not always defined.}

\begin{proof}
	We just need to know that restricting to $\cat A_0$ induces an equivalence on the homotopy category of homotopy functors to $\cat D$. This is true because the given deformation retracts provide the inverse equivalence. Since $\D \eta$ is a morphism, and we have already picked its source and target, it is therefore defined uniquely by its restriction to $\cat A_0$. We may furthermore pick any natural transformation we like, such as the one given in the statement of the proposition. For the last statement, if $\cat A_0$ and $\cat A_0'$ both occur as subcategories of $\cat C$ for which the above conditions hold, then $\cat A_0 \cup \cat A_0'$ also occurs, and $\D \eta$ as calculated on $\cat A_0 \cup \cat A_0'$ clearly agrees with $\D\eta$ as calculated on $\cat A_0$ or $\cat A_0'$.
\end{proof}

\begin{ex}\label{ex:derived_external_smash_and_base_change}
	Given maps $f: A \to A'$ and $g: B \to B'$, the two composites
	\[ \xymatrix{
		\mc R(A) \times \mc R(B) \ar[r]^-\barsmash \ar[d]_-{f_! \times g_!} & \mc R(A \times B) \ar[d]^-{(f \times g)_!} \\
		\mc R(A') \times \mc R(B') \ar[r]^-\barsmash & \mc R(A' \times B')
	} \]
	are coherently deformable, using for example the $h$-cofibrant spaces for all four categories. Therefore the canonical interchange isomorphism from \autoref{lem:external_smash_and_base_change} induces an equivalence of derived functors that is canonical in the homotopy category,
	\[ \L f_!X \barsmash^{\L} \L g_!Y \simeq \L (f \times g)_!(X \barsmash^{\L} Y). \]
	Similarly the two composites
	\[ \xymatrix{
		\mc R(A') \times \mc R(B') \ar[r]^-\barsmash \ar[d]_-{f^* \times g^*} & \mc R(A' \times B') \ar[d]^-{(f \times g)^*} \\
		\mc R(A) \times \mc R(B) \ar[r]^-\barsmash & \mc R(A \times B)
	} \]
	are coherently deformable, using for example the $f$-cofibrant and $h$-fibrant spaces for all four categories. We therefore also get a canonical equivalence
	\[ \R f^*X \barsmash^{\L} \R g^*Y \simeq \R(f \times g)^*(X \barsmash^{\L} Y). \]
	Both of these equivalences pass to canonical natural isomorphisms of functors on the homotopy categories $\ho\mc R(-)$, where they agree with the maps defined in \cite[9.4.1]{ms}.
\end{ex}

\begin{ex}\label{ex:derived_beck_chevalley}
	Given a commuting square of unbased spaces and ``rotated'' square of functors
	\[ \xymatrix@R=1em @C=1em{
		& A \ar[ld]_-g \ar[rd]^-f & \\
		B \ar[rd]_-p && C \ar[ld]^-q \\
		& D &
	} \qquad
	\xymatrix@R=1em @C=-.5em{
		& \mc R(A) \ar[ld]_-{g_!} \ar@{<-}[rd]^-{f^*} & \\
		\mc R(B) \ar@{<-}[rd]_-{p^*} && \mc R(C) \ar[ld]^-{q_!} \\
		& \mc R(D) &
	} \]
	if either $p$ or $q$ is an $h$-fibration then the functors are coherently deformable. The Beck-Chevalley isomorphism of \autoref{prop:beck_chevalley_spaces} therefore gives an equivalence of derived functors
	\[ \L g_! \circ \R f^* \simeq \R p^* \circ \L q_!, \]
	cf. \cite[9.4.6]{ms}.
\end{ex}

\begin{ex}
	Consider the category whose objects are diagrams of the form
	\[ \xymatrix{ X_0 \ar[r] & X_1 \ar[r] & X_2 \ar[r] & \ldots } \]
	where each map is a closed inclusion, each space $X_i$ has a $\Z/2$-action, and the maps are $\Z/2$-equivariant. There are two operations we can perform, ``colimit'' and ``$\Z/2$-fixed points,'' that commute with each other. We capture this in the commuting square:
	\[ \xymatrix @C=4em{
		\cat{Top}^{\N \times \Z/2} \ar[r]^-{\underset{\N}\colim} \ar[d]_-{\underset{\Z/2}\lim} &
		\cat{Top}^{\Z/2} \ar[d]^-{\underset{\Z/2}\lim} \\
		\cat{Top}^{\N} \ar[r]_-{\underset{\N}\colim} &
		\cat{Top}
	} \]
	The colimits can be left-deformed and the limits can be right-deformed, but the square of functors is not coherently deformable. To prove this we take $X_n = \mathbb{RP}^n$ to be the $n$-skeleton of $\mathbb{RP}^\infty$ with the trivial $\Z/2$-action. The two composites of derived functors give
	\[ \hocolim \left((\mathbb{RP}^n)^{h\Z/2}\right), \qquad \left( \hocolim \mathbb{RP}^n \right)^{h\Z/2} \]
	which are not weakly equivalent,\footnote{The framework of \cite{shulman_comparing} gives us a map between them, but the map is not an equivalence.} because by the Sullivan conjecture \cite{miller_sullivan},
	\[ \hocolim \Map(\mathbb{RP}^\infty,\mathbb{RP}^n) \simeq \hocolim * \simeq * \not\simeq \Map(\mathbb{RP}^\infty,\mathbb{RP}^\infty). \]
\end{ex}

The ``coherently deformable'' condition can be strengthened. If each functor $F_i$ is left-deformable and each subcategory $\cat A_{i-1}$ contains the image of a single left deformation retract $Q_{i-1}$, we say the list $F_n \circ \ldots \circ F_1$ is \textbf{coherently left deformable}\index{coherently left/right deformable}. This is essentially the definition found in \cite[42]{dhks}. A \textbf{coherently right deformable} list is defined similarly. This stronger condition implies the following almost immediately.
\begin{lem}
	If $F_n \circ \ldots \circ F_1$ is coherently left deformable, then the composite functor $F_n \circ \ldots \circ F_1$ is also left-deformable and the canonical map
	\[ \L F_n \circ \ldots \circ \L F_1 \to \L(F_n \circ \ldots \circ F_1) \]
	is an equivalence.
\end{lem}

\begin{ex}
	Any string of pushforward functors is coherently left-deformable. Any string of pullbacks is coherently right-deformable.
	In the first diagram of \autoref{ex:derived_external_smash_and_base_change}, relating $\barsmash$ and $f_!$, both branches are coherently left deformable. Therefore $\L f_!X \barsmash^{\L} \L g_!Y$ is actually the left-derived functor of the composite $f_!X \barsmash g_!Y$. Of course, the second diagram cannot satisfy this stronger condition because it composes left and right deformable functors together.
\end{ex}

\begin{prop}\label{prop:derived_adjunction}(cf. \cite[43.2]{dhks}, \cite[6.4.13]{riehl_basic}) 
	Given an adjunction $(F \adj G)$ between homotopical categories, if $F$ is left deformable and $G$ is right deformable, then the pairs of lists $\id_{\cat C}, \R G \circ F$ and $\L F \circ G, \id_{\cat D}$ are coherently deformable, and the derivations
	\[ \xymatrix @R=0.5em{
		\id_{\cat C} \ar[r]^-\eta & G \circ F \ar[r]^-r & \R G \circ F & \ar@{~>}[r] & & \id_{\cat C} \ar[r]^-{\D (r \circ \eta)} & \R G \circ \L F \\
		\L F \circ G \ar[r]^-q & F \circ G \ar[r]^-\epsilon & \id_{\cat D} & \ar@{~>}[r] & & \L F \circ \R G \ar[r]^-{\D (\epsilon \circ q)} & \id_{\cat D}
	} \]
	give an adjunction $(\L F \adj  \R G)$ of functors on the homotopy category.\index{deformable!adjunction}
\end{prop}
\begin{proof}
	We check that the two diagrams below commute. The maps are taken either in the homotopy category of functors, or as functors of homotopy categories. The composite along each left-hand edge and top row satisfies the defining property of $\D (r \circ \eta)$ and along each right-hand column and bottom edge satisfies the defining property of $\D (\epsilon \circ q)$. This proves the triangle identities.
	\[ \xymatrix{
		FQQ \ar[d]^-{1q = q1}_-\sim \ar[r]^-\eta & FQGFQ \ar[r]^-r \ar[d]^-q & FQGRFQ \ar[d]^-q \\
		FQ \ar@{=}[rd] \ar[r]^-\eta & FGFQ \ar[r]^-r \ar[d]^-\epsilon & FGRFQ \ar[d]^-\epsilon \\
		& FQ \ar[r]^-r_-\sim & RFQ
	} \qquad
	\xymatrix{
		QGR \ar[d]^-q_-\sim \ar[r]^-\eta & GFQGR \ar[d]^-q \ar[r]^-r & GRFQGR \ar[d]^-q \\
		GR \ar@{=}[rd] \ar[r]^-\eta & GFGR \ar[r]^-r \ar[d]^-\epsilon & GRFGR \ar[d]^-\epsilon \\
		& GR \ar[r]^-{1r = r1}_-\sim & GRR
	} \]
\end{proof}

\begin{ex}
	We therefore get an adjunction of derived functors $(\L f_! \adj  \R f^*)$ for any map $f\colon A \to B$ of unbased spaces. We also get an adjunction $(\L f^* \simeq \R f^* \adj \R f_*)$ when $f$ is a fiber bundle where both base and fiber are cell complexes.
\end{ex}

\beforesubsection
\subsection{Deriving symmetric monoidal categories and bifibrations}\label{sec:deriving_SMBFs}\aftersubsection

If $(\mc C,\otimes,I)$ is both a symmetric monoidal category and a homotopical category, we might wish to lift the symmetric monoidal structure to $\ho\mc C$. We say that the symmetric monoidal structure is \textbf{left-deformable}\index{left-deformable!symmetric monoidal category} if there is a full subcategory $\cat A$ (whose objects are ``cofibrant'') containing a left deformation retract $Q$ of $\mc C$, such that
\begin{enumerate}
	\item[(SM1)] $\otimes$ preserves cofibrant objects and weak equivalences between them,
	\item[(SM2)] and $I$ is cofibrant.\footnote{This could be weakened by only asking for $QI \otimes QX \to I \otimes QX$ to be an equivalence, because in that case we simply enlarge $\cat A$ to include $I$. This returns us to the original condition (2) without breaking condition (1).}
\end{enumerate}

\begin{ex}
	The symmetric monoidal category $(\mc R,\barsmash,S^0)$ is left-deformable, using for instance the $h$-cofibrant spaces and the whiskering functor $W$.
\end{ex}

\begin{prop}\label{prop:deform_sym_mon_cat}
	If $(\mc C,\otimes,I)$ is left-deformable then there is a canonical symmetric monoidal structure on $\ho\mc C$ whose product is $\otimes^{\L}$ and whose unit is $I$.
\end{prop}

\begin{proof}
	This is well-known but we can give a proof by repeated application of \autoref{prop:passing_natural_trans_to_derived_functors}. The isomorphisms $\alpha,\lambda,\rho,\gamma$ ascend to the homotopy category because in each case the lists of functors are coherently left-deformable, using the subcategory $\cat A$ in every case. They are coherent because they were coherent in $\mc C$, and the recipe in \autoref{prop:passing_natural_trans_to_derived_functors} respects composition of natural transformations and identity natural transformations. Concretely, $\alpha$ on the homotopy category can be described by removing ``extraneous'' copies of $Q$, applying $\alpha$ from $\mc C$, then re-inserting the ``extraneous'' $Q$s,
	\[ QX \otimes Q(QY \otimes QZ) \simeq QX \otimes (QY \otimes QZ) \simeq (QX \otimes QY) \otimes QZ \simeq Q(QX \otimes QY) \otimes QZ. \]
	The other isomorphisms admit a similar description.
\end{proof}

\begin{rmk}
	We can dualize the above discussion and get a corresponding result for right-deformable symmetric monoidal structures. One example is the category $\mc R(B)$ with a cartesian monoidal structure $X \otimes Y := X \times_B Y$.
\end{rmk}

Similarly, if $\mc C$ is a bifibration over $\bS$ and each fiber $\mc C^A$ is a homotopical category, we might wish to make $\ho\mc C$ into a bifibration over $\bS$ as well.\footnote{Technically, to define $\ho\mc C$ we first make $\mc C$ into a homotopical category by closing the weak equivalences under 2-out-of-6. But $\ho\mc C$ is still described as the initial category in which all the weak equivalences in every $\mc C^A$ have been inverted.} We say that $\mc C$ is \textbf{bi-deformable}\index{deformable!bifibration} if the following three conditions hold.
\begin{enumerate}
\item[(BF3)] The pushforwards are coherently left-deformable: there exists a simultaneous choice of left deformation retract $Q_A$ on each fiber category $\mc C^A$ under which the pushforwards $f_!$ preserve cofibrant objects and weak equivalences between them.
\item[(BF4)] The pullbacks are coherently right-deformable: there exists a simultaneous choice of right deformation retract $R_A$ on each fiber category $\mc C^A$ under which the pullbacks $f^*$ preserve fibrant objects and weak equivalences between them.
\item[(BF5)] We choose a class of \textbf{homotopy Beck-Chevalley}\index{homotopy!Beck-Chevalley square} squares. Each one must be a Beck-Chevalley square in $\mc C$, and in addition, the two lists of functors in the Beck-Chevalley isomorphism must be coherently deformable.
\end{enumerate}

\begin{thm}\label{prop:deform_bifibration}
	If $\mc C$ is bi-deformable, then the category $\ho\mc C$ obtained by inverting the weak equivalences in $\mc C$, together with the homotopy Beck-Chevalley squares, is also a bifibration. Furthermore its fiber categories are canonically isomorphic to the homotopy categories of the fibers,
	\[ (\ho\mc C)^A \cong \ho(\mc C^A). \]
\end{thm}

\begin{proof}
	The first step is to show that $\ho\sC$ is a fibration with the desired fiber categories. For this, it suffices to define a second fibration $\bH \to \bS$ and an isomorphism of categories $\ho\sC\to \bH$ over $\bS$,
	such that the induced dotted map in the diagram
	\[\xymatrix @R=1.7em{\sC^A \ar[r]\ar[d]& (\ho\sC)^A \ar[r]&\bH^A\\\ho(\sC^A) \ar@{-->}[urr]}\]
	is an isomorphism for all $A$ in $\bS$.

	We build the fibration $\bH$ by a Grothendieck construction. For each $A \in \bS$ we take the homotopy category $\ho(\mc C^A)$. For each map $f\colon A \to B$ we take the derived pullback functor $f^*R_B\colon \ho(\mc C^B) \to \ho(\mc C^A)$. We extend this to an \textbf{indexed category}\index{indexed category} by picking composition isomorphisms and unit isomorphisms between these derived functors (see e.g. \cite[B1.2]{johnstone2002sketches} or \cite[7.1]{mp2}). We get them by right-deriving the same isomorphisms for the strict pullback functors in $\sC$, in other words using the maps
	\[ f^*R_Bg^*R_C \overset\sim\leftarrow f^*g^*R_C \cong (g \circ f)^*R_C \]
	\[ \id_A^* R_A \overset\sim\leftarrow \id_A^* \cong \id_{\sC^A}. \]
	These isomorphisms have to satisfy certain coherence conditions, but they are the same conditions the strict pullbacks already satisfy, so the coherence passes automatically to the derived pullbacks $\R f^*$ by the universal property of right-derived functors. By the general theory of Grothendieck constructions, $\bH \to \bS$ is a fibration.
	
	Concretely, the objects of $\bH$ are the objects of $\sC$, and a morphism from $X$ to $Y$ over $f\colon A \to B$ is a zig-zag in $\sC^A$ from $X$ to $f^*R_B Y$. The composition of morphisms in $\bH$ is given by
	\[ \xymatrix @R=1em{
		X \ar@{<->}[r] & f^*R_BY \ar@{<->}[r] & f^*R_Bg^*R_CZ \ar@{<-}[d]^-\sim \\
		&& f^*g^*R_CZ  \ar@{<-}[d]^-\cong \\
		&& (g \circ f)^*R_CZ.
	} \]
	The identity morphism at $X$ is $X \cong \id^* X \overset\sim\to \id^* R_A X$.
	
	Define $\sC \to \bH$ by taking each map $X \to Y$ over $f$, represented by a vertical map $X \to f^*Y$, to the composite $X \to f^*Y \to f^*R_BY$. Chasing diagrams, we check
	\begin{itemize}
		\item this respects composition and units, hence is a functor,
		\item it sends weak equivalences to isomorphisms, hence defines a functor $\ho\sC \to \bH$,
		\item the resulting map $\ho(\sC^A) \to (\ho\sC)^A \to \bH^A = \ho(\sC^A)$ is the identity,
		\item $\bH$ is generated by morphisms of the form $X \to X' \overset\sim\to \id^*R_A X'$, $X \overset\sim\leftarrow X' \overset\sim\to \id^*R_A X'$, and $f^*R_B Y = f^*R_B Y$, and 
		\item these are each in the image of $\ho\sC$, hence $\ho\sC \to \bH$ is surjective.
	\end{itemize}
	
	Then we map $\bH \to \ho\sC$ by sending a zig-zag $X \leftrightarrow f^*R_B Y$ to the zig-zag of maps of $\sC$, $X \leftrightarrow f^*R_B Y \to R_B Y \leftarrow Y$. We check
	\begin{itemize}
		\item this respects the equivalence relation that defines the fibers of $\bH$, hence is well-defined,
		\item this defines a functor $\bH \to \ho\sC$, and 
		\item the composite $\sC \to \bH \to \ho\sC$ is equal to the canonical inclusion.
	\end{itemize}
	Hence $\ho \sC \to \bH$ is also injective, and is therefore an isomorphism of categories.
	
	We deduce that there is a cartesian arrow in $\ho\mc C$ of the form $f^*RX \to RX \overset\sim\leftarrow X$ for each $X$ and $f$, identifying $f^*R$ with the pullback in $\ho\mc C$. Dualizing everything, $\ho\mc C$ has co-cartesian arrows of the form $X \overset\sim\leftarrow QX \to f_!QX$ that identify $f_!Q$ with the pushforward in $\ho\mc C$.
		
	Therefore each derived pushforward $f_!Q$ forms an adjunction with the derived pullback $f^*R$. The unit is the unique map $X \to f^*Rf_!QX$ lying over the map from each of these two terms to $f_!QX$ (one of which is co-cartesian and the other of which is cartesian):
	\[ \xymatrix @R=1.7em{
		X & \ar[l]_-\sim QX \ar@{-->}[d] \ar[rd] & \\
		& f^*f_!QX \ar@{-->}[d] \ar[r] & f_!QX \ar[d]^-\sim \\
		& f^*Rf_!QX \ar[r] & Rf_!QX
	} \]
	The dotted lines are filled in by the universal property of cartesian arrows in $\mc C$. Therefore the unit is the zig-zag $X \overset\sim\leftarrow QX \to f^*f_!QX \to f^*Rf_!QX$, and further the $Q$s can be dropped if $X$ is cofibrant. Similarly the counit is $f_!Qf^*RY \to f_!f^*RY \to RY \overset\sim\leftarrow Y$, and the $R$s can be dropped if $Y$ is fibrant.
	
	For each homotopy Beck-Chevalley square, the Beck-Chevalley map in $\mc C$ is an isomorphism by assumption. Because the functors are coherently deformable, as in \autoref{ex:derived_beck_chevalley} this gives an isomorphism of their derived functors on the homotopy category. It remains to show this isomorphism agrees with the Beck-Chevalley map that is constructed out of the adjunctions on the homotopy category. We check this by restricting to inputs in the common category on which $q_!$ and $f^*$ are derived, and inspecting the routes in the following diagram from $g_!Qf^*R$ to $p^*Rq_!Q$.
	\[ \resizebox{\textwidth}{!}{\xymatrix @C=-1em{
		&& p^*q_! \ar[rrr]^-r &&& p^*Rq_! &&& p^*Rq_!Q \ar[lll]_-q^-\sim \\
		&&& g_!g^*p^*q_! \ar@{-}[dd]^-\cong \ar[ul]_-\epsilon \ar[rrr]^-r &&& g_!g^*p^*Rq_! \ar@{-}[dd]^-\cong \ar[ul]_-\epsilon &&& g_!g^*p^*Rq_!Q \ar[ul]_-\epsilon \ar[lll]_-q^-\sim \\
		&&&&&&& g_!Qg^*p^*Rq_! \ar@{-}[dd]^-\cong \ar[ul]_-q &&& g_!Qg^*p^*Rq_!Q \ar@{-}[dd]^-\cong \ar[ul]_-q \ar[rd]^-r_-\sim \ar[lll]_-q^-\sim \\
		g_!f^* \ar[rrr]^-\eta &&& g_!f^*q^*q_! \ar[rrr]^-r &&& g_!f^*q^*Rq_! &&& && g_!Qg^*Rp^*Rq_!Q \\
		& g_!Qf^* \ar[rrr]^-\eta \ar[ul]_-q \ar[rd]^-r_-\sim &&& g_!Qf^*q^*q_! \ar[ul]_-q \ar[rd]^-r \ar[rrr]^-r &&& g_!Qf^*q^*Rq_! \ar[ul]_-q \ar[rd]^-r_-\sim &&& g_!Qf^*q^*Rq_!Q  \ar[rd]^-r_-\sim \ar[lll]_-q^-\sim \\
		&& g_!Qf^*R \ar[rrr]^-\eta &&& g_!Qf^*Rq^*q_! \ar[rrr]^-r &&& g_!Qf^*Rq^*Rq_! &&& g_!Qf^*Rq^*Rq_!Q \ar[lll]_-q^-\sim \\
	}} \]
\end{proof}

\begin{rmk}\label{expand_beck_chevalley}
	By standard properties of Beck-Chevalley isomorphisms, the Beck-Chevalley condition is preserved any time we replace a square in $\bS$ by a ``weakly equivalent'' square, so long as for each ``weak equivalence'' the functors $\L f_!$ and $\R f^*$ induce equivalences on the fiber categories $\ho\mc C^A$. This allows us to further expand the class of squares for which $\ho\mc C$ satisfies the Beck-Chevalley condition.
\end{rmk}

The following results fall out of the proof of the previous theorem.\index{homotopy!(co-)cartesian arrow}
\begin{prop}\label{prop:homotopy_cartesian_arrows}(cf. \cite{harpaz_prasma_grothendieck})
	An arrow is cartesian in $\ho\mc C$ iff it is isomorphic to a cartesian arrow in $\mc C$ whose target is fibrant. A cartesian arrow $f^*Y \to Y$ in $\mc C$ is cartesian in $\ho\mc C$ iff the map $f^*Y \to f^*RY$ is an isomorphism in $\ho(\mc C^A)$. The dual statements apply to the co-cartesian arrows.
\end{prop}

When the base category is $\bullet \to \bullet$, ignoring the Beck-Chevalley conditions, a bifibration $\mc C$ over $\bS$ is precisely an adjunction of two categories.

\begin{prop}
	In this case, when the left adjoint $F$ is left-deformable and the right adjoint $G$ is right-deformable, the induced bifibration $\ho\mc C$ gives the same adjunction $(\L F \adj \R G)$ that we constructed in \autoref{prop:derived_adjunction}.\index{deformable!adjunction}
\end{prop}

This gives a second proof of \autoref{prop:derived_adjunction} and shows that our ad-hoc construction of the adjunction $(\L F \adj \R G)$ was far more canonical than it appeared. It arises by inverting the weak equivalences in the bifibration $\mc C$ that encoded the adjunction $(F \adj G)$. This sort of canonical approach also appears in \cite{hinich2018so}.

If $\mc C$ is a symmetric monoidal bifibration over $\bS$, and each fiber category $\mc C^A$ is a homotopical category, we may wish to combine the above two constructions and make the homotopy category $\ho\mc C$ into a symmetric monoidal bifibration. To integrate the above two constructions, we make $\mc C$ itself into a homotopical category by taking the equivalences in each fiber and closing under 2-out-of-6. Then we ask for a left deformation $Q$ for $\boxtimes$ that is fiberwise over $\bS$, meaning that each map $QX \simar X$ lies over an identity map, and that $\boxtimes$ preserves the equivalences in some full subcategory of each fiber $\mc A^A \subseteq \mc C^A$ containing the image of $Q$.

Then we ask for the five conditions (SM1), (SM2), (BF3), (BF4), (BF5). This makes $\ho\mc C$ a bifibration with a symmetric monoidal structure. The last piece of the puzzle is to prove that the derived tensor product $\boxtimes^{\L}$ preserves cartesian and cocartesian arrows in $\ho\mc C$. We need the following final two conditions.
\begin{enumerate}
	\item[(SMBF6)] Suppose that $\boxtimes$ and \emph{all} of the pushforwards $f_!$ are coherently deformable, meaning there is a subcategory $\mc A$ containing a finite composite of fiberwise left and right deformation retracts of $\mc C$ over $\bS$, preserved by $\boxtimes$ and each $f_!$, on which both functors are equivalent to their derived functors,
	\item[(SMBF7)] and the same condition for $\boxtimes$ and all of the pullbacks $f^*$ (the category $\mc A$ can be different).
\end{enumerate}
These are not necessarily stronger than (BF4) and (BF5), because for instance there does not have to be a single right deformation retract that deforms the pullbacks and $\boxtimes$. We are allowed to use a composite of left and right deformation retracts.\index{deformable!SMBF}
\begin{prop}\label{prop:deform_SMBF}
	Under these assumptions, $\boxtimes^{\L}$ preserves cartesian and cocartesian arrows in $\ho\mc C$.
\end{prop}
\begin{proof}
	We give two proofs. By \autoref{prop:homotopy_cartesian_arrows}, a co-cartesian arrow in $\ho\mc C$ is isomorphic to a co-cartesian arrow of the form $Q_A X \to f_!Q_A X$, where $Q_A$ is the left deformation used for $f_!$. Taking $Y \in \mc A^A$ isomorphic to $X$ in $\ho\mc C^A$, we get a zig-zag of weak equivalences $Y \overset\sim\leftarrow Q_A Y \ldots Q_A X$ in $\mc C^A$, all of which are preserved by $f_!$. Therefore without loss of generality each of our co-cartesian arrows has source in $\mc A$. Since $\boxtimes$ preserves equivalences on $\mc A$, we have $\boxtimes^{\L} \simeq \boxtimes$ on this subcategory, hence $\boxtimes^{\L}$ of our two co-cartesian arrows is equivalent to $\boxtimes$ of them, which is co-cartesian in $\mc C$. The tensor product of the sources is still in $\mc A$, hence $(f\times g)_!Q_{A \times B} \simeq (f\times g)_!$ on this source, so this co-cartesian arrow is also co-cartesian in $\ho\mc C$. The proof for cartesian arrows is dual.
	
	The second proof is more explicit, but establishes the helpful corollary that the canonical isomorphisms
\begin{align*}
\L f_!X \boxtimes^{\L} \L g_!Y &\simeq \L(f \times g)_!(X \boxtimes^{\L} Y), \\
\R f^*X \boxtimes^{\L} \R g^*Y &\simeq \R(f \times g)^*(X \boxtimes^{\L} Y)
\end{align*}
	arising from the universal property of co/cartesian arrows in $\ho\mc C$ agree with the isomorphisms produced by applying \autoref{prop:passing_natural_trans_to_derived_functors} to the corresponding isomorphisms in $\mc C$. For the co-cartesian arrows, we assume $X$ and $Y$ are in $\mc A$ and form the following diagram in which the top row is the canonical co-cartesian arrow in $\ho\mc C$ for $X \boxtimes^{\L} Y$, and the left-hand column is $\boxtimes^{\L}$ of the canonical co-cartesian arrows for $X$ and $Y$. The bottom-right route is the isomorphism arising from \autoref{prop:passing_natural_trans_to_derived_functors}. All maps marked $\sim$ or $\cong$ lie over identity maps in $\bS$, and all others lie over $f \times g$. The commutativity of the diagram establishes that the isomorphism from \autoref{prop:passing_natural_trans_to_derived_functors} extends to an isomorphism of co-cartesian arrows, hence agrees with the canonical isomorphism arising from $\ho\mc C$.
	\[ \xymatrix @R=1.7em{
		QX \boxtimes QY \ar[rd]^-\sim & Q_{A \times B}(Q X \boxtimes Q Y) \ar[l]_-\sim \ar[d]^-\sim \ar[r] & (f \times g)_! Q_{A \times B}(Q X \boxtimes Q Y) \ar[d]^-\sim \\
		QQ_A X \boxtimes QQ_B Y \ar[u]_-\sim \ar[d] \ar[r]^-\sim & X \boxtimes Y \ar[d] \ar[r] & (f \times g)_! (X \boxtimes Y) \ar@{<->}[ld]^-\cong \\
		Qf_!Q_A X \boxtimes Qg_!Q_B Y \ar[r]^-\sim & f_!X \boxtimes g_!Y
	} \]
	For the pullbacks the argument is exactly the same, but the commuting diagram is as follows, with the two cartesian arrows we wish to compare running along the bottom and right-hand sides.
	\[ \xymatrix @R=1.7em{
		f^*X \boxtimes g^*Y \ar@{<->}[d]^-\cong \ar[rd] & Qf^*X \boxtimes Qg^*Y \ar[l]_-\sim \ar[r]^-\sim \ar[rddd] & Qf^*R_A X \boxtimes Qg^*R_B Y \ar[dd] \\
		(f \times g)^* (X \boxtimes Y) \ar[d]^-\sim \ar[r] & X \boxtimes Y \ar[d]^-\sim & \\
		(f \times g)^* R_{A \times B}(X \boxtimes Y) \ar[r] & R_{A \times B}(X \boxtimes Y) & QR_A X \boxtimes QR_B Y \\
		(f \times g)^* R_{A \times B}(QX \boxtimes QY) \ar[u]_-\sim \ar[r] & R_{A \times B}(QX \boxtimes QY) \ar[u]_-\sim & QX \boxtimes QY \ar[u]_-\sim \ar[l]_-\sim \ar[luu]_-\sim
	} \]
\end{proof}

\begin{cor}
	The homotopy category of all retractive spaces $\ho\mc R$ is a symmetric monoidal bifibration over $\cat{Top}$ with Beck-Chevalley on the homotopy pullback squares.
\end{cor}

\begin{proof}
	We verify the seven conditions (SM1), (SM2), (BF3), (BF4), (BF5), (SMBF6), and (SMBF7) using all of the earlier results concerning $\barsmash$, $f_!$, $f^*$, $f$-cofibrations, $h$-fibrations, and weak equivalences. This gives an SMBF with Beck-Chevalley for strict pullbacks along fibrations. By \autoref{expand_beck_chevalley}, and the fact that every homotopy pullback square is weakly equivalent to a strict pullback along a fibration, we therefore get Beck-Chevalley for all homotopy pullback squares.
\end{proof}

\beforesubsection
\subsection{Review of model categories}\label{sec:model_cats}\aftersubsection

A model category is a homotopical category with extra structure, which we now recall from \cite{hovey_model_cats}.

Let $\cat C$ be a category with all small colimits and limits. A \textbf{model structure}\index{model category} on $\cat C$ is a choice of three subcategories $W$, $C$, and $F$ whose maps are called ``weak equivalences,'' ``cofibrations,'' and ``fibrations,'' respectively, such that
\begin{itemize}
	\item The weak equivalences $W$ are closed under 2-out-of-3 (if $f$ and $g$ are composable, and two of the maps $f$, $g$, and $g \circ f$ are in the given class, so is the third).
	\item Each of the three classes $W$, $C$, and $F$ is closed under retracts.
	\item Given a commuting square in $\cat C$
	\[ \xymatrix @R=1.7em{
		A \ar[d]_-i \ar[r] & X \ar[d]^-p \\
		B \ar[r] \ar@{-->}[ur] & Y
	} \]
	the dotted lift exists if $i$ is an acyclic cofibration ($i \in W \cap C$) and $p$ is a fibration, or if $i$ is a cofibration and $p$ is an acyclic fibration ($p \in W \cap F$).
	\item There is a functorial factorization of each map $f$ into an acyclic cofibration followed by a fibration, and there is another functorial factorization into a cofibration followed by an acyclic fibration.	
\end{itemize}

In this case $\cat C$ is called a \textbf{model category}. In any model category $\cat C$:
\begin{itemize}
	\item Each isomorphism in $\cat C$ is a cofibration, a fibration, and a weak equivalence.
	\item The weak equivalences satisfy 2-out-of-6, so $\cat C$ is also a homotopical category.
	\item An object $X$ is \textbf{cofibrant} if $\emptyset \to X$ is a cofibration, where $\emptyset$ is the initial object. There exists a left retraction $Q$ onto the subcategory of cofibrant objects, and in addition $QX \to X$ is an acyclic fibration, not just a weak equivalence.\index{cofibrant replacement}
	\item An object $X$ is \textbf{fibrant} if $X \to *$ is a fibration, where $*$ is the terminal object. There exists a right retraction $R$ onto the subcategory of fibrant objects, and in addition $X \to RX$ is an acyclic cofibration, not just a weak equivalence. 
	\index{fibrant replacement}
	\item A map $i$ is an acyclic cofibration iff it has the left-lifting property with respect to the fibrations, i.e. the lift in the above square exists any time $p$ is a fibration. Similar statements characterize the cofibrations, the fibrations, and the acyclic fibrations.
	\item The cofibrations are closed under pushouts, transfinite compositions (therefore coproducts), and retracts. The same applies to acyclic cofibrations. The dual statements apply to the fibrations and to the acyclic fibrations.
	\item (Ken Brown's Lemma) If $F\colon \cat C \to \cat D$ is a functor of model categories, taking acyclic cofibrations between cofibrant objects to weak equivalences, then it takes all weak equivalences of cofibrant objects to weak equivalences. The same applies to fibrant objects and acyclic fibrations.
	\item (Fundamental Theorem) If $X$ is cofibrant and $Y$ is fibrant, then $\cat C(X,Y) \to \ho\cat C(X,Y)$ is surjective, and we can describe the kernel as well. For cofibrant $X$, let $``X \times I"$ be any factorization of $X \amalg X \to X$ into a cofibration followed by an acyclic fibration:
	\[ \xymatrix @R=1.7em{
		X \amalg X \ar@{^{(}->}[r] \ar[rd]_-{(\id_X,\id_X)} & ``X \times I" \ar@{->>}[d]^-\sim \\
		& X
	} \]
	Then two maps $f,g \colon X \rightrightarrows Y$ are equal in the homotopy category iff they are homotopic, meaning there is some extension of $(f,g)\colon X \amalg X \to Y$ to $X \times I \to Y$.
\end{itemize}

\begin{ex}
	The \textbf{Quillen model structure} on $\mc R(B)$ has $W$ the weak homotopy equivalences and $F$ the Serre fibrations ($q$-fibrations). The class of cofibrations $C$ must be determined by the other two; we will describe it more once we recall cofibrantly-generated model categories. This is sometimes called the \textbf{$q$-model structure}.
\end{ex}

\begin{ex}
	The \textbf{Hurewicz model structure} on $\mc R(B)$ has $W$ the maps that are homotopy equivalences on the total space, $C$ the $h$-cofibrations, and $F$ the $h$-fibrations. This explains the result of \autoref{prop:strom_lift} (although really that result is used to prove this is a model structure, see \cite{strom1972homotopy}). We won't use this model structure so much, because it has the wrong equivalences. This is sometimes called the \textbf{$h$-model structure}.
\end{ex}

\begin{ex}
	There is an \textbf{integral model structure} on the category $\mc R$ of all retractive spaces, obtained by combining together the Quillen model structures on the fiber categories $\mc R(B)$, see \cite{harpaz_prasma_grothendieck,cagne2017bifibrations,hebestreit_sagave_schlichtkrull}. There are actually two of these, one where we take the weak equivalences in the base category to be the usual equivalences, and another where we take the weak equivalences in the base to be only the isomorphisms. Inverting this second class of weak equivalences gives us the bifibration $\ho \mc R$ we constructed in the previous section.
\end{ex}

In this paper we put more emphasis on homotopical categories, but model structures are still useful. They make it easier to count maps in the homotopy category (using the fundamental theorem) and to left and right deform functors (using cofibrant and fibrant replacement).\footnote{A third advantage is that if one is interested in lifting things to $\infty$-categories, model categories are more useful than just categories with weak equivalences.}

Recall a functor $F\colon \cat C \to \cat D$ is \textbf{left Quillen}\index{Quillen adjunction} if it is a left adjoint and preserves both the cofibrations and the acyclic cofibrations. Equivalently, the right adjoint $G\colon \cat D \to \cat C$ is \textbf{right Quillen}, meaning it preserves the fibrations and the acyclic fibrations. By Ken Brown's lemma, if $(F \adj G)$ is a Quillen pair then $F$ is left-deformable and $G$ is right-deformable.

The pair $(F \adj G)$ is a \textbf{Quillen equivalence}\index{Quillen equivalence} if the adjunction $(\L F \adj \R G)$ defined in \autoref{prop:derived_adjunction} gives an equivalence of homotopy categories $\ho\cat C \simeq \ho\cat D$.\footnote{The recipe we used to define this adjunction, in fact, strictly agrees with the usual recipe in the setting of a model category.} Equivalently, for all cofibrant $X$ in $\cat C$ and fibrant $Y$ in $\cat D$, a map $FX \to Y$ is a weak equivalence in $\cat D$ iff its adjoint $X \to GY$ is a weak equivalence in $\cat C$.
%

\begin{ex}\label{ex:derived_adjunction_spaces}
	The pullback $f^*\colon \mc R(B) \to \mc R(A)$ is right Quillen with respect to the Quillen model structures on $\mc R(B)$ and $\mc R(A)$. Therefore the pushforward $f_!$ is left Quillen. 
	Moreover, the pair $(f_! \adj f^*)$ is a Quillen equivalence whenever $f\colon A \to B$ is a weak equivalence. To see this, just observe that if $X \in \mc R(A)$ is cofibrant and $Y \in \mc R(B)$ is fibrant, the following square in $\mc R$ has both horizontal maps weak equivalences.
	\[ \xymatrix @R=1.7em{
		X \ar[d] \ar[r]^-\sim & f_!X \ar[d] \\
		f^*Y \ar[r]_-\sim & Y
	} \]
	Therefore the left-hand map is an equivalence iff the right-hand map is.
\end{ex}

\begin{ex}\label{ex:other_adjunction_spaces}
	When $f$ is a fiber bundle with fiber a cell complex, the pullback $f^*\colon \mc R(B) \to \mc R(A)$ is also left Quillen. Therefore the sheafy pushforward $f_*$ is right Quillen. This gives a second proof of most of \autoref{prop:sheafy_pushforward_preserves}. Moreover, if $f$ is a weak equivalence, then $(f^* \adj f_*)$ is a Quillen equivalence, because $\L f^* \simeq \R f^*$ gives an equivalence of homotopy categories by the previous example. 
\end{ex}

\begin{ex}\label{ex:cg_quillen_equiv_to_cgwh}
	There are actually two categories that we call $\mc R(B)$, one using (CG) and a smaller one using (CGWH). Both are endowed with the Quillen model structure as above, and the forgetful functor (CGWH) $\to$ (CG) is a right Quillen equivalence. This implies that the weak-Hausdorffification functor $h(-)$ preserves equivalences between cofibrant objects. We could also see this directly: the cofibrant objects in (CG) are already weak Hausdorff, therefore $h(-)$ is isomorphic to the identity on that subcategory.
	
\end{ex}

A model category $\cat C$ is \textbf{left proper} if for every square as below with $i$ a cofibration and $f$ a weak equivalence, the map $g$ is also a weak equivalence.
\[ \xymatrix @R=1.7em{
	A \ar[r]^-i \ar[d]_-f & B \ar[d]^-g \\
	C \ar[r]_-j & B \cup_A C
} \]
In a left proper model category $\cat C$, the pushout is homotopical on those spans for which $i$ (or $f$) is a cofibration. This last condition is usually called the \textbf{gluing lemma}\index{gluing lemma}, compare \autoref{prop:technical_cofibrations}, and is equivalent to being left proper. It implies the pushout is left-deformable; its left-derived functor is the \textbf{homotopy pushout}\index{homotopy!pushout}.\footnote{The pushout is actually always left-deformable, but if $\cat C$ is not left proper, we have to additionally replace $A$, $B$, and $C$ by cofibrant spaces.}

Dualizing, and swapping out the cofibrations for fibrations, we get the definition of \textbf{right proper}, which is equivalent to the \textbf{dual gluing lemma}. The right-derived functor of pullback is called the \textbf{homotopy pullback}. We say $\cat C$ is \textbf{proper}\index{model category!proper} if it is both left and right proper.

\begin{ex}
	The Quillen model structure on $\mc R(B)$ is proper by \autoref{prop:technical_cofibrations}.
\end{ex}

\beforesubsection
\subsection{Cofibrantly generated model categories}\label{sec:cof_gen_model_cats}\aftersubsection

Let $I$ be any set of maps in $\cat C$. A \textbf{generalized $I$-cell complex} is a transfinite composition of pushouts of maps in $I$, see \cite{hovey_model_cats}. An \textbf{$I$-cell complex} is a sequential composition of pushouts of coproducts of maps in $I$. (Each of these is also a generalized $I$-cell complex.) An \textbf{$I$-injective map} is any map with the right lifting property with respect to any map in $I$ (equivalently, with respect to any retract of a generalized $I$-cell complex). An \textbf{$I$-cofibration} is any map with the left lifting property with respect to $I$-injective maps. In other words $i\colon A \to B$ is an $I$-cofibration if for every square
	\[ \xymatrix @R=1.7em @C=3em{
		A \ar[d]_-i \ar[r] & X \ar[d]^-p \\
		B \ar[r] \ar@{-->}[ur] & Y
	} \]
with $p\colon X \to Y$ an $I$-injective map, the dotted lift exists. A retract of a generalized $I$-cell complex is always an $I$-cofibration, but in general there could be more.

A model structure on $\cat C$ is \textbf{cofibrantly generated}\index{model category!cofibrantly generated} once we specify a subset $I$ of all cofibrations, and $J$ of all acyclic cofibrations, such that
\begin{itemize}
	\item The fibrations $F$ are the $J$-injective maps: $F = J$-inj.
	\item The acyclic fibrations $W \cap F$ are the $I$-injective maps: $W \cap F = I$-inj.
\end{itemize}
In addition, the domains of the maps in $I$ have to be ``small relative to the generalized $I$-cell complexes,'' and similarly for $J$. We call $I$ the \textbf{generating cofibrations} and $J$ the \textbf{generating acyclic cofibrations}. The smallness condition is laborious to spell out at the highest level of generality; for us the following will be sufficient.

\begin{df}\label{df:simple_smallness}
	$I$ has the \textbf{simplified smallness condition} if for each domain $K$ of a map of $I$, and each $I$-cell complex
	\[ X_0 \ra X_1 \ra \ldots X_{\infty} = \underset{n}\colim\, X_n, \]
	every map $K \to X_\infty$ factors through some $X_n$.\footnote{This simplification is also used in the notion of a ``compactly generated model category'' from \cite{mmss}. But we won't be able to use that notion directly, because the model structures for parametrized spectra that we are interested in only satisfy the ``cofibration hypothesis'' with respect to the category of sequences of unbased spaces (not retractive spaces), and this forgetful functor does not preserve colimits and limits.}
\end{df}

If $\cat C$ is a cofibrantly-generated model category with the simplified smallness condition, then the cofibrations are precisely the retracts of the $I$-cell complexes.\footnote{If we didn't have the simplified smallness condition, we would have to take the generalized $I$-cell complexes instead.} Similarly, the class of acyclic cofibrations is equal to the class of retracts of the $J$-cell complexes. Both of these claims are consequences of Quillen's small-object argument \cite{quillen, hovey_model_cats}.

\begin{ex}
	The Quillen model structure on $\mc R(B)$ is cofibrantly generated. Set
	\[ \begin{array}{ccrll}
	I &=& \{ \ S^{n-1}_{+B} \to D^n_{+B} & : n \geq 0, & D^n \to B \ \} \\
	J &=& \{ \ D^n_{+B} \to (D^n \times I)_{+B} &: n \geq 0, & (D^n \times I) \to B \ \}.
	\end{array} \]
	The simplified smallness condition holds for $I$ and $J$ because $S^{n-1}$ and $D^n$ are compact topological spaces, and cell complexes built out of such have the property that the filtration maps $X_{i-1} \to X_i$ are closed inclusions with weak Hausdorff quotient. The Serre fibrations are $J$-inj by definition, and the acyclic Serre fibrations are $I$-inj by an elementary argument. By the above discussion, the cofibrations in $\mc R(B)$ are precisely the maps that are retracts of relative cell complexes. We call such maps \textbf{$q$-cofibrations}.
\end{ex}

A common strategy for proving that $F\colon \cat C \to \cat D$ is left-deformable is to show that it preserves retracts of $J$-cell complexes, and then invoke Ken Brown's Lemma. Often this reduces further to checking that $F$ sends each generating acyclic cofibration to an acyclic cofibration. For instance, we can prove that $f_!$ is left-deformable this way.

%
%
%

The external smash product $\barsmash$ is a little more complicated because it is a functor in two inputs, and it is only a left adjoint in each variable separately. Suppose that $\mc C$, $\mc D$, and $\mc E$ are model categories, and $\otimes\colon \mc C \times \mc D \to \mc E$ a functor. We say that $\otimes$ is a \textbf{left Quillen bifunctor}\index{left Quillen bifunctor} if
	\begin{itemize}
		\item it is a left adjoint in each slot,
		\item for each cofibration $f$ in $\mc C$ and cofibration $g$ in $\mc D$, $f \square g$ is a cofibration, and
		\item if in addition one of $f$ or $g$ is a weak equivalence then $f \square g$ is a weak equivalence.
	\end{itemize} 
Then $X \otimes -$ is left Quillen whenever $X$ is cofibrant in $\cat C$, and similarly $- \otimes Y$ is left Quillen if $Y$ is cofibrant in $\cat D$. More informally, $X \otimes Y$ preserves equivalences when $X$ and $Y$ are both cofibrant. A standard application of the lemmas of \autoref{pushout_product_lemmas} is the following.\index{model category!monoidal}
\begin{prop}\label{prop:quillen_bifunctor_defn}
	If $\mc C$ and $\mc D$ are cofibrantly generated, then for $\otimes$ to be a Quillen bifunctor it suffices that
	\begin{itemize}
		\item it is a left adjoint in each slot,
		\item if $f,g \in I$ then $f \square g \in C$, and
		\item if $f \in I, g \in J$ or $f \in J, g \in I$ then $f \square g \in C \cap W$.
	\end{itemize} 
\end{prop}

\begin{ex}\label{ex:smashing_spaces_is_left_Quillen}
	The above conditions are straightforward to check for the external smash product $\barsmash\colon \mc R(A) \times \mc R(B) \to \mc R(A \times B)$, and therefore it is a left Quillen bifunctor. It therefore preserves equivalences when both inputs are $q$-cofibrant. Note that this is weaker than the conclusion of \autoref{prop:h_cofibrations_pushout_product}, but the dual conclusions about fibrations are not weaker than those found in \autoref{h_fibrations_pullback_hom}. 
	In particular, this proves that $\barmap_A(X,Y)$ preserves equivalences when $X$ is $q$-cofibrant and $Y$ is $q$-fibrant.
\end{ex}

The following standard result is also immediate from the above discussion and \autoref{prop:deform_sym_mon_cat}. It applies in particular to the integral model structure on $\mc R$.
\begin{cor}
	Suppose $\mc C$ is both a model category and a symmetric monoidal category, the tensor product $\otimes$ is a left Quillen bifunctor, and the unit $I$ is cofibrant. Then $\ho\mc C$ has a canonical left-derived symmetric monoidal structure.
\end{cor}

The final virtue of cofibrantly-generated model categories is that there is a standard method for proving that one exists. It suffices to check the following list of conditions, cf. \cite{hovey_model_cats}.
\begin{prop}\label{prop:construct_cofibrantly_generated_model_category}
	Given a category $\cat C$ with all small colimits and limits, a subcategory $W$ and sets of maps $I$ and $J$, for this data to make $\cat C$ into a cofibrantly generated model category, it is sufficient that:
	\begin{enumerate}
		\item $W$ is closed under 2-out-of-3 and retracts.
		\item $I$ satisfies the simplified smallness condition.
		\item $J$ satisfies the simplified smallness condition.
		\item $J$-cell complexes are in $W$ $\cap$ $I$-cof.
		\item $I$-inj $\subseteq$ $W$ $\cap$ $J$-inj.
		\item Either $W$ $\cap$ $I$-cof $\subseteq$ $J$-cof or $W$ $\cap$ $J$-inj $\subseteq$ $I$-inj.
	\end{enumerate}
\end{prop}

\newpage
\section{Parametrized spectra}

Next we will define parametrized spectra and level equivalences between them. We prove several technical preliminaries here before defining stable equivalences in the next section.

\beforesubsection
\subsection{Basic definitions}\aftersubsection

For us, a parametrized spectrum is simply a diagram spectrum in the sense of \cite{mmss}, except that the levels of the diagram are in $\mc R(B)$, and we use the external smash product when defining the bonding maps.

In detail, we define $S^n$ to be the one-point compactification of $\R^n$, regarded as a retractive space over $*$. For $X \in \mc R(B)$ we regard $\Sigma_B X := S^1 \barsmash X$ as a space over $B$, not $* \times B$.\footnote{The different ways of doing this are not strictly equal, but canonically isomorphic. Given any two models, a bonding map defined using one model carries over to a bonding map defined using the other model. The notion of a bonding map is therefore independent of which model we pick. Of course one encounters exactly the same issue when defining non-parametrized spectra, because there is technically not a canonical choice of model for the functor $\Sigma(-)$.}

\begin{df}\hfill
	
	\vspace{-1em}
	\begin{itemize}
		\item A \textbf{sequential spectrum} or \textbf{prespectrum} over $B$ is a sequence of retractive spaces
		\[ X_n \in \mc R(B), \qquad n \geq 0, \]
		together with \textbf{bonding maps}
		\[ \sigma\colon \Sigma_B X_n \to X_{1+n}. \]\index{sequential spectra $\Psp(B)$}
		\item A \textbf{symmetric spectrum} over $B$ is a sequential spectrum together with an action of the symmetric group
		\[ \Sigma_n \curvearrowright X_n \]
		through maps of retractive spaces, such that each composite
		\[ \sigma^p\colon S^p \barsmash X_q \to \ldots \to X_{p + q} \]
		is $\Sigma_p \times \Sigma_q$-equivariant.
		\item A \textbf{orthogonal spectrum} over $B$ is a sequential spectrum with a continuous action of the orthogonal group $O(n)$ on $X_n$ through maps of retractive spaces, such that the above composite is $O(p) \times O(q)$-equivariant.\index{orthogonal spectra!category $\Osp(B)$}
	\end{itemize}
\end{df}

Because external smash commutes with pullback, for any sequential spectrum $X$ over $B$ and $b \in B$, the fibers $X_b$ form a  sequential spectrum in the usual sense, which we call the \textbf{fiber spectrum}\index{fiber spectrum} at $b$. The same applies to symmetric and orthogonal spectra.

\begin{ex}[Suspension spectra] \hfill
	\vspace{-1em}
	
	\begin{itemize}
		\item The \textbf{fiberwise sphere spectrum} $\Sph_B$ is the sequential spectrum whose $n$th space is
		\[ S^n \times B \cong S^n \barsmash (B_{+B}). \]
		Each fiber spectrum is isomorphic to the sphere spectrum $\Sph$.
		\item More generally the \textbf{fiberwise suspension spectrum}\index{fiberwise!suspension spectrum} $\Sigma^\infty_B X$ of a retractive space $X$ over $B$ has $n$th space
		\[ S^n \barsmash X \]
		and bonding maps $S^1 \barsmash (S^n \barsmash X) \cong S^{1+n} \barsmash X$. Each of its fibers is the suspension spectrum $\Sigma^\infty X_b$ of the fiber $X_b$.
		\item If $X$ is only a space over $B$, we can add a basepoint section and then take its fiberwise suspension spectrum. We sometimes denote this as $\Sigma^\infty_{+B} X$ instead of $\Sigma^\infty_B X_{+B}$. Each of its fibers is the suspension spectrum $\Sigma^\infty_+ X_b$. If $X \to B$ happens to be a fibration or a bundle, we can think of $\Sigma^\infty_{+B} X$ as the associated bundle of suspension spectra.
		\item All three of these examples are in fact symmetric/orthogonal spectra using the usual action of the groups $\Sigma_n \leq O(n)$ on $\R^n$ and its compactification $S^n$. 
	\end{itemize}
\end{ex}

\begin{ex}[Bundles of spectra]
	Suppose that $E$ is a sequential spectrum, symmetric spectrum, or orthogonal spectrum, that $G$ is a topological group, and that that we specify a continuous action $G \curvearrowright E$. This means that $G$ acts continuously on each $E_n$ in a way that respects the bonding maps. Then to each principal $G$-bundle $P \to B$, we can associate a \textbf{bundle of spectra}\index{bundle of spectra} with fiber $E$. At level $n$ this is just the usual mixing construction $P \times_G E_n$, and there are canonical homeomorphisms $\Sigma_B(P \times_G E_n) \cong P \times_G (\Sigma_B E_n)$ that allow us to define the bonding maps. More concretely, we can define a bundle of spectra by giving an open cover of $B$ and assigning the overlaps to automorphisms of $E$, in a way that satisfies the cocycle condition.
	
	One special case is when the $G$-action on $E$ is trivial. The associated bundle is just $E_n \times B$ at level $n$.
	
	As another special case, if $A$ is an abelian group and $G = \textup{Aut}(A)$, the usual bar-construction model for the the Eilenberg-Maclane spectrum $HA$ admits an action by $G$, whose action on the homotopy groups $\pi_0(HA) \cong A$ is the canonical one. From this we deduce that any bundle of abelian groups $\mc A \to B$ has an associated \textbf{parametrized Eilenberg-Maclane spectrum} $H\mc A \to B$, whose fiber is $HA$.
\end{ex}

If $X$ is a parametrized spectrum, the bonding maps have adjoints $X_n \to \Omega_B X_{1+n} := \barmap_*(S^1,X_{1+n})$. Furthermore the loop space functor will be derived if $X_{1+n} \to B$ is a fibration. We say $X$ is a \textbf{(weak) $\Omega$-spectrum}\index{$\Omega$-spectrum} if the composite $X_n \to \Omega_B X_{1+n} \to \R\Omega_B X_{1+n}$ is a weak equivalence.

\begin{ex}[$\Omega$-spectra] \hfill
	\vspace{-1em}
	
	\begin{itemize}
		\item If $E$ is a weak $\Omega$-spectrum then so is any bundle of spectra with fiber $E$. The bundle of Eilenberg-Maclane spectra $H\mc A$ is one example of this.
		\item If $X \to B$ is a fibration and $X$ is weak Hausdorff,\footnote{This implies the maps of the colimit system are closed inclusions, so the colimit is a homotopy colimit.} we can form a weak $\Omega$-sequential spectrum over $B$ whose $n$th space is
		\[ Q_B(\Sigma^n_{+B} X) = \underset{k}\colim \Omega^k_B \Sigma^{k+n}_B X_{+B}. \]
		Using \autoref{prop:h_cofibrations_pushout_product} and \autoref{h_fibrations_pullback_hom}, this is a fibration at every spectrum level and the fiber is the classical fibrant replacement of the spectrum $\Sigma^\infty_+ X_b$. We could also use the $O(n)$ action on $S^n$ to form different bonding maps, giving a weak $\Omega$-orthogonal spectrum with the same levels, but whose adjoint bonding maps are only weak equivalences instead of homeomorphisms.
	\end{itemize}
\end{ex}

The above definitions can be re-phrased using diagrams of retractive spaces indexed by $\mathscr N, \Sigma_S, \mathscr J$ from \cite{mmss}. We will state this carefully for orthogonal spectra. Recall that the topological category $\mathscr J$\index{orthogonal spectra!indexing category $\mathscr J$} has objects the finite-dimensional subspaces $V$ of some fixed inner product space isomorphic to $\R^\infty$. For each pair $V$ and $W$, let $O(V,W)$ be the space of all linear isometric embeddings of $V$ into $W$. This has a vector bundle whose fiber consists of the orthogonal complement of $V$ in $W$, and we let $\mathscr J(V,W)$ be the Thom space of this vector bundle. More concretely, we have the formula
\[ \mathscr J(\R^n,\R^{m+n}) \cong O(m+n)_+ \sma_{O(m)} S^m. \]
The composition in $\mathscr J$ composes linear embeddings and adds points in their orthogonal complements.

This category contains inside it a smaller category $\mathscr N$, whose objects are the standard subspaces $\R^n \subseteq \R^\infty$ and whose morphisms are
\[ \mathscr N(\R^n,\R^{m+n}) = S^m. \]

\begin{df}
	A parametrized $\mathscr J$-space is an enriched $\mathscr J$-diagram in $\mc R(B)$. Concretely, this means for each object $V$ an object $X(V) \in \mc R(B)$, and for each pair $V,W$ a map in $\mc R(B)$
	\[ \mathscr J(V,W) \barsmash X(V) \to X(W) \]
	such that all of the following ``associativity'' and ``unit'' diagrams commute:
	\[ \xymatrix @R=1.7em{
		\mathscr J(V,W) \barsmash \mathscr J(U,V) \barsmash X(U) \ar[r] \ar[d] & \mathscr J(V,W) \barsmash X(V) \ar[d] \\
		\mathscr J(U,W) \barsmash X(U) \ar[r] & X(W)
	}
	\qquad
	\xymatrix @R=1.7em{
		X(V) \ar[r] \ar@{=}[dr] & \mathscr J(V,V) \barsmash X(V) \ar[d] \\
		& X(V).
	}
	\]
\end{df}

\begin{lem}
	Restricting to $V = \R^n$ gives an equivalence of categories from parametrized $\mathscr J$-spaces to parametrized orthogonal spectra. 
\end{lem}

\begin{lem}
	In the same way, parametrized $\mc N$-spaces are equivalent to parametrized sequential spectra.
\end{lem}

\begin{proof} Identical to the proof in the non-parametrized case \cite{mmss}. \end{proof}

Let $\Osp(B)$ denote the category of parametrized $\mathscr J$-spaces. In light of the previous lemma, it is harmless to call this the \textbf{category of orthogonal spectra over $B$}\index{orthogonal spectra!category $\Osp(B)$}. Similarly let $\Psp(B)$ denote the \textbf{category of sequential spectra (prespectra) over $B$}.\index{sequential spectra $\Psp(B)$}

\begin{rmk}
	We use the term \textbf{spectra} any time our discussion applies equally well to sequential spectra and to orthogonal spectra. We will not treat symmetric parametrized spectra in detail in this document, because they require an additional discussion of semistability as in \cite{schwede_symmetric_spectra} to make the theory work, see \cite{braunack2018rational,hebestreit_sagave_schlichtkrull}.
\end{rmk}

A colimit of a diagram of spectra is computed by taking the colimit in $\mc R(B)$ on each spectrum level separately, and then using the canonical commutation of $\Sigma_B(-)$ with colimits to define the bonding maps. The same applies to limits and $\Omega_B(-)$.\index{colimit}\index{limit}

Since spectra are diagrams, we can construct free spectra as free diagrams. Given $V$ and a retractive space $A$, define the \textbf{free orthogonal spectrum}\index{free spectrum $F_V A$} $F_V A$ to be the parametrized $\mathscr J$-space given by the formula
\[ (F_V A)(W) := \mathscr J(V,W) \barsmash A. \]
This is the left adjoint of the functor that sends a $\mathscr J$-space $X$ to $X(V)$. When $V = \R^n$ we call this orthogonal spectrum $F_n A$.
The analogous construction in sequential spectra gives the \textbf{free sequential spectrum} $F_n A$, which at spectrum level 0 to $(n-1)$ is the zero object of $\mc R(B)$, then the retractive space $A \in \mc R(B)$ and its fiberwise suspensions:
\[ F_n A = \{ \xymatrix{ B & B & \ldots & B & A & \Sigma_B A & \Sigma_B^2 A & \ldots } \} \]
In both cases, $F_0 = \Sigma^\infty$ is the suspension spectrum functor
\[ (\Sigma^\infty A)(W) = \Sigma^W_B A = S^W \barsmash A. \]

\begin{ex}\label{ex:thom_spectra}
	We may represent a virtual bundle $\xi$ over $B$ by a vector bundle $V$ and integer $n$, and define the \textbf{fiberwise Thom spectrum} of $\xi$ to be
	\[ \Th_B(\xi) := F_n \Th_B(V). \]\index{fiberwise!Thom spectrum $\Th_B(\xi)$}
	This depends on the choice of $V$ and $n$, an also on whether we take $F_n$ in orthogonal spectra or sequential spectra. \autoref{ex:equivalent_thom_spectra} and \autoref{prop:pre_equiv_to_orth} below imply that different choices give ``stably equivalent'' results.
\end{ex}

\beforesubsection
\subsection{Level equivalences, level fibrations, free cofibrations}\aftersubsection

Every result in this subsection applies to both sequential spectra and orthogonal spectra.

A \textbf{level equivalence}\index{level equivalence} is a map of spectra $X \to Y$ inducing an equivalence on each level $X(V) \to Y(V)$. A \textbf{level $h$-fibration}\index{$h$-fibration!level} is a map which is a $h$-fibration at every level, and similarly for $q$-fibrations\index{$q$-fibration!level}. A \textbf{level $f$-cofibration}\index{$f$-cofibration!level}\index{$h$-cofibration!level}\index{$i$-cofibration!level} is an $f$-cofibration at every level, and similarly for other kinds of cofibrations. Since $X(V) \cong X(\R^n)$ when $n = \dim V$, it is enough to check these conditions for $V = \R^n$.

We will need a slightly stronger kind of cofibration. Consider all maps of spectra over $B$ of the form $F_n(K \to X)$, where $n\geq 0$ and $K \to X$ is an $f$-cofibration of $f$-cofibrant retractive spaces over $B$. The class of \textbf{free $f$-cofibrations}\index{$f$-cofibration!free} is the smallest class of maps of spectra containing the above, and closed under pushout, transfinite composition, and retracts. In particular, any countable composition of pushouts of coproducts of maps of the form $F_n(K \to X)$ is a free cofibration.

The \textbf{free $q$-cofibrations}, \textbf{free $h$-cofibrations}, and \textbf{free $i$-cofibrations} are defined similarly\index{$q$-cofibration!free}\index{$h$-cofibration!free}\index{$i$-cofibration!free}. For $q$-cofibrations, this class is generated by free cells $F_k(S^{n-1}_{+B}) \to F_k(D^n_{+B})$. In the case of sequential spectra, a map $X \to Y$ is a free cofibration iff each of the maps
\[ \Sigma_B Y_n \cup_{\Sigma_B X_n} X_{1+n} \to Y_{1+n} \]
is a cofibration. 

When a spectrum is cofibrant in this sense, we say it is \textbf{freely cofibrant}.\footnote{The more obvious term ``free cofibrant'' turns out to be confusing because spectra that are cofibrant in this sense do not have to be free, only built out of free spectra in a certain way.}

\begin{lem}\label{prop:free_implies_level}
	Every free cofibration is a level cofibration.
\end{lem}

\begin{proof}
	It suffices to prove that $F_n(K \to X)$ is a level cofibration. At level $k$ it is the identity of $\mathscr J(\R^n,\R^k)$ or $\mathscr N(\R^n,\R^k)$ smashed with $K \to X$, which is a cofibration by \autoref{prop:h_cofibrations_pushout_product} (for $f$, $h$, or $i$-cofibrations) or \autoref{ex:smashing_spaces_is_left_Quillen} (for $q$-cofibrations).
\end{proof}

\begin{lem}
	The free spectrum functor $F_n$ sends cofibrant spaces to freely cofibrant spectra. It sends $h$-cofibrant, $h$-fibrant spaces to level $h$-fibrant spectra.
\end{lem}

\begin{proof}
	Also follows directly from \autoref{prop:h_cofibrations_pushout_product}.
\end{proof}

Now we turn our attention to cofibrant and fibrant replacement. It is possible, but tedious, to generalize the whiskering construction $WX$ from \autoref{prop:whiskering_properties}. It is much easier to prove that spectra have a model structure with the level equivalences.

\begin{prop}[Level model structure]\label{prop:level_model_structure}
	The level equivalences, free $q$-cofibrations, and level $q$-fibrations define a proper model structure on spectra (sequential spectra or orthogonal spectra). It is cofibrantly generated by
	\[ \begin{array}{ccrll}
	I &=& \{ \ F_k\left[ S^{n-1}_{+B} \to D^n_{+B} \right] & : n,k \geq 0, & D^n \to B \ \} \\
	J &=& \{ \ F_k\left[ D^n_{+B} \to (D^n \times I)_{+B} \right] &: n,k \geq 0, & (D^n \times I) \to B \ \}.
	\end{array} \]
\end{prop}

\begin{proof}
	This kind of result is a lifting of a model structure along an adjunction, and is very well-known. We have hopefully made it straightforward for the interested reader to verify the axioms laid out in \autoref{prop:construct_cofibrantly_generated_model_category}. By \autoref{prop:free_implies_level} for $h$-cofibrations, the cofibrations in this model structure are always level $h$-cofibrations, therefore it is left proper. Right properness is also immediate.
\end{proof}

\begin{rmk}
	So far in this section, everything works equally well in (CGWH) and (CG), including \autoref{prop:level_model_structure}. We therefore have two distinct model categories of sequential spectra and two distinct model categories of orthogonal spectra. In each case, the forgetful functor (CGWH) $\to$ (CG) is a right Quillen equivalence.
\end{rmk}

The above model structure gives us a cofibrant replacement functor $Q$\index{cofibrant replacement}, and a fibrant replacement functor $R^{lv}$\index{fibrant replacement!$R^{lv}$ for spectra}. In principle these are only defined on each fiber category $\Osp(B)$ or $\Psp(B)$ separately, but an examination of the small-object argument reveals that they actually define functors on all of $\Osp$ or $\Psp$, that preserve each fiber category.

On the other hand, we also want to generalize the monoidal fibrant replacement functor $P$ to spectra. If $X$ is a spectrum, we define $PX$ by applying $P$ to every level $X_n$\index{fibrant replacement!$P$ for spectra} and using the canonical homeomorphisms to define the bonding maps
\[ \mathscr J(V,W) \barsmash PX(V) \cong P(\mathscr J(V,W) \barsmash X(V)) \to PX(W) \]
or more simply in sequential spectra
\[ S^n \barsmash PX_m \cong P(S^n \barsmash X_m) \to PX_{n+m}. \]
Again, this defines a functor on each fiber category, or on the entire category of all spectra over all base spaces.

\begin{prop}\label{prop:spectrum_px}
	$X \to PX$ is a level equivalence. If $X$ is level $h$-cofibrant, then $PX$ is level $h$-fibrant. $P$ preserves free/level $f$-cofibrations and free/level $h$-cofibrations. If the spectra are all level $h$-cofibrant, it turns free/level $h$-cofibrations into free/level $f$-cofibrations.
\end{prop}

\begin{proof}
	\autoref{prop:px_properties} implies everything except for the statements about free cofibrations. For that part, we observe that $P$ preserves retracts because it is a functor, and preserves pushouts and transfinite compositions when they are taken along free cofibrations, because free cofibrations are level cofibrations and therefore level closed inclusions, and the pullback and pushforward that together form $P$ preserve those colimits by \autoref{lem:f_shriek_preserves} and \autoref{prop:f_star_colimits}. This reduces our work to checking that $P$ sends $F_n(K \to X)$ to a free $f$-cofibration if $K \to X$ is an $h$-cofibration between $h$-cofibrant retractive spaces. But this happens because
	\[ P(F_n(K \to X)) \cong F_n(PK \to PX) \]
	and $PK \to PX$ is an $f$-cofibration of $f$-cofibrant spaces by \autoref{prop:px_properties} again.
\end{proof}

In summary, any spectrum $X$ can be replaced by a freely $f$-cofibrant level $h$-fibrant spectrum $PQX$, and there is a zig-zag of level equivalences $X \overset\sim\leftarrow QX \overset\sim\to PQX$. This three-term zig-zag serves as a replacement for the nine-term zig-zag found in \cite[13.5.2]{ms}. 

\beforesubsection
\subsection{Base change}\aftersubsection

Every result in this subsection applies to both sequential spectra and orthogonal spectra, but to save space we will only state them for orthogonal spectra $\Osp(B)$.

Given $f\colon A \to B$, we define the functor $f^*\colon \Osp(B) \to \Osp(A)$ by pulling back every level of a spectrum $X$ over $B$ back to $A$, and defining the bonding maps using the canonical commutation of pullback and smash product from \autoref{lem:external_smash_and_base_change}:
\[ \mathscr J(V,W) \barsmash f^*X(V) \cong f^*(\mathscr J(V,W) \barsmash X(V)) \to f^*X(W) \]\index{pullback $f^*$}
The pushforward is defined in exactly the same way.\index{pushforward $f_{^^21}$} The adjunction $(f_! \adj f^*)$ on each level commutes with the action of $\mathscr J(V,W)$ and therefore defines an adjunction $(f_! \adj f^*)$ as functors between $\Osp(A)$ and $\Osp(B)$.

These functors could alternatively be defined by first defining the bifibration of all parametrized spectra $\Osp \to \cat{Top}$, whose fibers are the categories of spectra over $B$, $\Osp(B)$. A map in the category $\Osp$ is given at every spectrum level by maps in $\mc R$ that commute with the bonding maps. Then the above pullback and pushforward functors arise from the cartesian and cocartesian arrows in $\Osp$, as in \autoref{prop:spaces_SMBF}.

\autoref{lem:f_star_preserves} and \autoref{lem:f_shriek_preserves} give sufficient conditions under which $f^*$ and $f_!$ preserve level equivalences, level fibrations, and level cofibrations. Since $f_!(F_n K) \cong F_n f_!K$ and $f^*(F_n K) \cong F_n f^*K$ by \autoref{lem:external_smash_and_base_change} again, we also get that $f_!$ preserves all free cofibrations, and $f^*$ preserves the free $f$-cofibrations.\footnote{See the proof of \autoref{prop:spectrum_px} for how to show that $f^*$ preserves the compositions and pushouts that make up a general free cofibration.}  Collecting this together:
\begin{lem}\label{lem:spectrum_f_preserves}
	$f^*\colon \Osp(B) \to \Osp(A)$ preserves
	\begin{itemize}
		\item level $h$-fibrations and $q$-fibrations,
		\item level equivalences between level $q$-fibrant (or $h$-fibrant) spectra,
		\item free or level $f$-cofibrations, and free or level closed inclusions.
	\end{itemize}
	$f_!\colon \Osp(A) \to \Osp(B)$ preserves
	\begin{itemize}
		\item free or level $f$-cofibrations, $q$-cofibrations, $h$-cofibrations, or closed inclusions, and
		\item level equivalences between level $h$-cofibrant spectra.
	\end{itemize}
	Moreover if $f$ itself is a Hurewicz fibration ($h$-fibration), then $f^*$ preserves free or level $h$-cofibrations, and all level equivalences, while $f_!$ preserves spectra that are both level $h$-fibrant and level $h$-cofibrant.
\end{lem}

In particular, $(f_! \adj f^*)$ is a Quillen pair for the level model structure, which is a Quillen equivalence when $f$ is a weak equivalence.

\beforesubsection
\subsection{Smashing with a space}\label{sec:smash_with_space}\aftersubsection

Every result in this subsection applies to both sequential spectra and orthogonal spectra.

It is straightforward to define functors that smash a spectrum with a space,
\[ \barsmash\colon \mc R(A) \times \Psp(B) \to \Psp(A \times B), \qquad \barsmash\colon \mc R(A) \times \Osp(B) \to \Osp(A \times B). \]
To smash with a space $K$, just apply $K \barsmash -$ at each spectrum level, and use the symmetric monoidal structure of $\barsmash$ to define the bonding maps.\index{external smash product $X \barsmash Y$!of a spectrum and a space} This gives two functors as above, which are unique up to unique isomorphism. If we smash a spectrum with multiple spaces, the resulting functors are associative and unital, up to unique isomorphism (\autoref{thm:spectra_rigidity}).

\begin{ex}
	Every free spectrum $F_n K$ over $B$ is canonically isomorphic to the smash product of the space $K \in \mc R(B)$ with the free spectrum $F_n S^0 \in \Psp(*)$ or $\Osp(*)$.
\end{ex}

\begin{ex}
	The reduced suspension functor $\Sigma_B$ is defined on spectra over $B$ by smashing with the space $S^1 \in \mc R(*)$.\index{fiberwise!reduced suspension $\Sigma_B$}
\end{ex}

Smashing with a space has a right adjoint in each variable. For formal reasons, these right adjoints will also be uniquely defined up to unique isomorphism, and the choice of adjunction will also be unique.

If we fix the space $K \in \mc R(A)$, the right adjoint of $K \barsmash -$ is the \textbf{external mapping spectrum}\index{external!function spectrum $\barF_B(Y,Z)$} functor
\[ \barF_A(K,-)\colon \Osp(A\times B) \to \Osp(B) \quad \textup{or} \quad \barF_A(K,-)\colon \Psp(A\times B) \to \Psp(B) \]
that applies $\barmap_A(K,-)$ at each spectrum level. (We use the notation $\barF$ because the output is a spectrum, and reserve $\barmap$ for when the output is a space.) The adjunctions from \autoref{sec:external_smash} pass in a unique way to an adjunction 
\[ K \barsmash X \to Y \quad \textup{ over }A \times B \quad \longleftrightarrow \quad X \to \barF_A(K,Y) \quad \textup{ over }B. \]

\begin{ex}
	The based loops functor $\Omega_B$ is defined on spectra over $B$ by taking external maps out of the space $S^1 \in \mc R(*)$.\index{fiberwise!based loops $\Omega_B$}
\end{ex}

If we fix the orthogonal spectrum $X \in \Osp(B)$, the right adjoint of $- \barsmash X$ sends the spectrum $Y$ over $A \times B$ to the \textbf{external mapping space}\index{external!mapping space $\barmap_B(Y,Z)$} $\barmap_B(X,Y)$. This is defined from the space-level mapping space as the equalizer in $\mc R(A)$
\[ \barmap_B(X,Y) \to \prod_n \barmap_B(X_n,Y_n) \rightrightarrows \prod_{m,n} \barmap_B(X_m \barsmash \mathscr J(\R^m,\R^n),Y_n) \]
where the two parallel maps either
\begin{itemize}
	\item pre-compose the map $X_n \to Y_n$ with $X_m \barsmash \mathscr J(\R^m,\R^n) \to X_n$,
	\item or post-compose $X_m \to Y_m$ with $Y_m \barsmash \mathscr J(\R^m,\R^n) \to Y_n$.
\end{itemize}
In other words, for each $a \in A$ a collection of maps $X_n \to (Y_a)_n$ in $\mc R(B)$ lands in the equalizer subspace iff they give a map of orthogonal spectra $X \to Y_a$ over $B$.

In sequential spectra, we similarly get the equalizer
\[ \barmap_B(X,Y) \to \prod_n \barmap_B(X_n,Y_n) \rightrightarrows \prod_{m,n} \barmap_B(X_m \barsmash S^{n-m},Y_n) \]
and in both cases the above definition helps us define an adjunction
\[ K \barsmash X \to Y \quad \textup{ over }A \times B \quad \longleftrightarrow \quad K \to \barmap_B(X,Y) \quad \textup{ over }A. \]

\begin{rmk}
	Setting $A = *$, this gives an enrichment of the categories $\Osp(B)$ and $\Psp(B)$ in based spaces. In other words it endows each set of maps of spectra over $B$, $\Osp(B)(X,Y)$ or $\Psp(B)(X,Y)$, with a topology.
\end{rmk}

\beforesubsection
\subsection{Smashing with a spectrum}\aftersubsection

The results in this section apply to orthogonal spectra, but not to sequential spectra.

Let $\mathscr J \sma \mathscr J$ denote the category whose objects are pairs $(V,V')$ of objects of $\mathscr J$ and whose morphisms are smash products of the morphism spaces
\[ (\mathscr J \sma \mathscr J)((V,V'),(W,W')) = \mathscr J(V,W) \sma \mathscr J(V',W'). \]
We call a $\mathscr J \sma \mathscr J$-diagram a \textbf{bispectrum}. The direct sum of representations, plus some choice of embedding $\mc U \oplus \mc U \to \mc U$, gives a direct sum functor 
\[ \oplus\colon \mathscr J \sma \mathscr J \to \mathscr J. \]
Then we can turn any bispectrum into a spectrum by taking the left Kan extension along the direct sum functor $\oplus$.

In the non-parametrized case, we define the smash product of two orthogonal spectra $X$ and $Y$ by smashing their levels $X(V) \sma Y(V')$ for all $V$ and $V'$. This gives a bispectrum in an obvious way. So the left Kan extension along $\oplus$ gives a spectrum, which we call the smash product spectrum $X \sma Y$.\footnote{In fact, once the definition is written out, it becomes clear that the choice of embedding $\mc U \oplus \mc U \to \mc U$ does not affect the answer up to canonical isomorphism.}

In the parametrized case, we do exactly the same thing, using the external smash product of retractive spaces. Given $X \in \Osp(A)$ and $Y \in \Osp(B)$, the external smash products $X(V) \barsmash Y(V')$ form a bispectrum whose levels are in the category $\mc R(A \times B)$. Its left Kan extension along $\oplus$ is therefore a spectrum $X \barsmash Y \in \Osp(A \times B)$. We call this the \textbf{external smash product spectrum} of $X$ and $Y$.\index{external smash product $X \barsmash Y$!of parametrized spectra} This is the same definition as \cite[11.4.10]{ms}.\footnote{Unfortunately, the term ``external smash product'' is also used in \cite{mmss} to differentiate the bispectrum $\{X(V) \sma Y(V')\}$ from the spectrum $X \sma Y$. These are two different uses of the word ``external.'' Our product is ``internal'' over $\mathscr J$ but ``external'' over the base spaces. In other words it produces a spectrum (not a bispectrum) but the spectrum is over $A \times B$ (not just $A$). See also \cite[11.1.7]{ms}.}

The following lemma has the same proof as the non-parametrized case.
\begin{lem}\label{barsmash_free}
	There is an isomorphism natural in retractive spaces $X$ and $Y$
	\[ F_V X \barsmash F_W Y \cong F_{V \oplus W} (X \barsmash Y). \]
	If $X$ is a space and $Z$ is a spectrum, then smashing with the suspension spectrum of $X$ is naturally isomorphic to smashing with the space $X$,
	\[ (\Sigma^\infty X) \barsmash Z \cong X \barsmash Z. \]
\end{lem}
Both of these statements specialize to the statement that there is an isomorphism
\[ (\Sigma^\infty X) \barsmash (\Sigma^\infty Y) \cong \Sigma^\infty (X \barsmash Y). \]
We will soon see that all of these isomorphisms are not just canonical, but unique.

The external smash product of spectra also has a right adjoint in each variable. The proof of this follows the same formal recipe found in \cite[\S 21]{mmss}. To describe the resulting right adjoint let $Y \in \Osp(B)$ and $Z \in \Osp(A \times B)$. Recall that the \textbf{shift functor} $\sh^n$ simply re-indexes the levels of $Z$,
\[ (\sh^n Z)_m = Z_{m+n}. \]
Pulling $Z$ back along the direct sum map $\oplus$ gives a bispectrum whose $n$th level is the orthogonal spectrum $\sh^n Z$. Since this is an orthogonal spectrum of orthogonal spectra, we can apply $\barmap_B(Y,-)$ and get just an orthogonal spectrum, which we call $\barF_B(Y,Z)$:\index{external!function spectrum $\barF_B(Y,Z)$}
\[ \barF_B(Y,Z)_n = \barmap_B(Y,\sh^n Z). \]
The formal argument cited above gives an adjunction
\[ X \barsmash Y \to Z \quad \textup{ over }A \times B \quad \longleftrightarrow \quad X \to \barF_B(Y,Z) \quad \textup{ over }A. \]
As an important special case, the right adjoint of $- \barsmash F_n S^0$ is the shift functor $\sh^n$.

We finish this subsection by describing how external smashes of spectra interact with level equivalences, level fibrations, and free cofibrations.
\begin{thm}\label{prop:spectra_pushout_product}\hfill
	\vspace{-1em}
	
	\begin{itemize}
		\item Let $f: K \to X$ and $g: L \to Y$ be free $h$-cofibrations of spectra over $A$ and $B$, respectively. Then $f \square g$, constructed using $\barsmash$, is a free $h$-cofibration.
		\item The same is true for free $f$-cofibrations, free $q$-cofibrations, and free closed inclusions.
		\item If $X$ and $Y$ are freely $h$-cofibrant and level $h$-fibrant then $X \barsmash Y$ is level $h$-fibrant.
		\item If $X$ is freely $h$-cofibrant and $g: Y \to Y'$ is a level equivalence of freely $h$-cofibrant spectra then $\id_X \barsmash g$ is a level equivalence.
	\end{itemize}
\end{thm}

\begin{proof}
	At the moment we can only prove the first two claims. Start by assuming that $f$ and $g$ are free spectra on $h$-cofibrations. Then by \autoref{barsmash_free}, their pushout-product is a free spectrum on a pushout-product of $h$-cofibrations, which is an $h$-cofibration by \autoref{prop:h_cofibrations_pushout_product}, so $f \square g$ is a free $h$-cofibration.
	
	Now examine the class of pairs of maps $f,g$ for which $f \square g$ is a free $h$-cofibration. It contains the free spectra on the $h$-cofibrations. By \autoref{pushout_product_pushouts}, it is closed under pushouts. By \autoref{pushout_product_composition}, it is closed under transfinite composition. It is clearly closed under retracts, just because $\square$ is a bifunctor. Therefore it contains all pairs of free $h$-cofibrations.
\end{proof}

The final two claims of \autoref{prop:spectra_pushout_product} are a trivial corollary of \autoref{prop:reedy_pushout_product} below, which we will prove after introducing another kind of cofibration that is ``semifree'' instead of free. The payoff for this is the ability to commute $\barsmash$ and $f^*$ on the homotopy category. Note that a slightly weaker form of these two claims already appears in the literature in \cite[4.27, 4.31]{malkiewich2017coassembly}. Also, a different proof of the fourth claim appears in \cite[8.6]{mp1}, and a different, more ad-hoc argument that $\barsmash$ and $f^*$ commute in the homotopy category can be found in \cite[8.8]{mp1}.

\beforesubsection
\subsection{Rigidity and interaction with base-change}\aftersubsection

The results in this section apply to orthogonal spectra, but not to sequential spectra.

The universal property of pushforward and pullback give canonical maps
\[ (f \times g)_!(X \barsmash Y) \to f_!X \barsmash g_!Y \]
\[ f^*X' \barsmash g^*Y' \to (f \times g)^*(X' \barsmash Y'). \]
\begin{prop}\label{prop:spectra_external_smash_and_base_change}(cf. \cite[11.4.1]{ms}, \cite{shulman_framed_monoidal})\hfill
	\vspace{-1em}
	
	\begin{itemize}
		\item (CG) These maps are isomorphisms.
		\item (CGWH) The first map is an isomorphism. The second is an isomorphism when $X'$ and $Y'$ are freely $i$-cofibrant.\footnote{It seems to actually be enough if only one of them is semifreely $i$-cofibrant in the sense of the next section, and the other is arbitrary, but the proof is longer.}
	\end{itemize}
\end{prop}

\begin{proof}
	For the first map, this reduces to the fact that $f_!$ commutes with left Kan extensions and the space-level smash products. For the second map, it is an isomorphism when $X' = F_n K$ and $Y' = F_m L$ because $f^*$ commutes with free spectra. In (CG), $f^*$ preserves all colimits so the conclusion follows. In (CGWH), $f^*$ preserves pushouts along level closed inclusions. So if we take a pushout square along a level closed inclusion
	\[ \xymatrix @R=1.7em{
		X_1 \ar[r] \ar[d] & X_2 \ar[d] \\
		X_3 \ar[r] & X_4
	} \]
	and the claim is true for the first three vertices and $Y'$, then it is true for $X_4$ and $Y'$. (The desired map is a pushout of three isomorphisms and is therefore an isomorphism.) By a similar argument, if $X'$ is a colimit of a sequence of closed inclusions $X_i \to X_{i+1}$ and the claim is true for each $X_i$ and $Y'$ then it is true for $X'$ and $Y'$. Therefore the class of spaces for which it is true must include all freely $i$-cofibrant $X'$ and $Y'$.
\end{proof}

The rigidity theorem \autoref{prop:spaces_rigidity} generalizes to spectra, but it's a good idea to make the statement a little stronger to accommodate the fact that $f^*$ and $\barsmash$ don't always commute in (CGWH).

\begin{thm}[Rigidity]\label{thm:spectra_rigidity}
	Suppose $n \geq 0$ and we have maps of spaces
	\[ B \overset{f}\leftarrow A \overset{g}\to C_1 \times \ldots \times C_n \]
	such that $f$ is injective on each fiber of $g$.
	
	Then any functor $\Osp(C_1) \times \ldots \times \Osp(C_n) \to \Osp(B)$ isomorphic to
	\[ \Phi\colon (X_1,\ldots,X_n) \leadsto f_!g^*(X_1 \barsmash \ldots \barsmash X_n) \]
	is in fact uniquely isomorphic to $\Phi$.
	
	More generally, on any full subcategory of $\Osp(C_1) \times \ldots \times \Osp(C_n)$ containing all $n$-tuples of the form $(F_{V_1} *_{+C_1},\ldots, F_{V_n} *_{+C_n})$, any functor isomorphic to $\Phi$ is uniquely isomorphic to $\Phi$.\index{rigidity!for spectra}
\end{thm}

\begin{proof}
	Let $\eta$ be an automorphism of $\Phi$. Using the adjunction between free spectra and evaluation, the proof of \autoref{prop:spaces_rigidity} shows that $\eta$ is determined by its value on every tuple of free spectra of the form $(F_{V_1} *_{+C_1},\ldots, F_{V_n} *_{+C_n})$. It is therefore enough to show that $\eta$ is the identity on these spectra. By \autoref{prop:spaces_rigidity} again, $\eta$ must be trivial on the suspension spectra of any $n$-tuple of spaces of the form $(*_{+C_1},\ldots,*_{+C_n})$ where the free points map to $(c_1,\ldots,c_n) \in C_1 \times \ldots \times C_n$. We argue the same for the free spectra $(F_{V_1} *_{+C_1},\ldots, F_{V_n} *_{+C_n})$ using the argument from \cite[3.17]{malkiewich_cyclotomic_dx}: the action of $\eta$ over any point of $B$ is a self-map of a spectrum of the form $F_{V_1 + \ldots + V_n} S^0$ which agrees with the identity map of $F_0S^0$ along any choice of point in $S^{V_1 + \ldots + V_n}$. Since the orthogonal group acts faithfully on the sphere, this is enough to conclude that $\eta$ acts as the identity on this tuple of free spectra. This finishes the proof.
\end{proof}

\begin{rmk}
	The same rigidity theorem holds if some of the factors in the smash product are spaces. It also holds for sequential spectra so long as at most one of the factors in the smash product is a spectrum, and the rest are spaces.
\end{rmk}

The same corollaries follow as before, except that in (CGWH) we restrict the domain to freely $i$-cofibrant spectra any time we need to switch $f^*$ with $\barsmash$ to get an isomorphism. For instance:
\begin{itemize}
	\item The interchange transformations of \autoref{prop:spectra_external_smash_and_base_change} are unique on the subcategory of freely $i$-cofibrant spectra. This implies that, as natural transformations, they are unique on the entire category.
	\item The category $\Osp$ of all parametrized spectra over all parametrized spaces is a symmetric monoidal category.
	\item In (CG), $\Osp$ is a symmetric monoidal bifibration. In both (CG) and (CGWH), the full subcategory $\Osp'$ of freely $i$-cofibrant spectra is a symmetric monoidal bifibration.
	\item In (CG), $\Osp(B)$ is a symmetric monoidal category under the internal smash product $\Delta_B^*(- \barsmash -)$. In (CG) and (CGWH), the subcategory $\Osp(B)'$ is a symmetric monoidal category.
\end{itemize}

\begin{rmk}\label{rmk:general_rigidity}
	The proof of \autoref{thm:spectra_rigidity} actually establishes a stronger statement, but it is more complicated to spell out. Roughly, if we modify $\Phi$ by allowing smash products after the pullbacks, then any natural transformation $\Phi \Rightarrow \Theta$ that is an isomorphism on free spectra, is the unique natural transformation with that property. See \cite[\S 3.3]{malkiewich_cyclotomic_dx} for similar statements to this one.
\end{rmk}

We will not use the following corollary in the current work, but it would be strange to call $P$ a monoidal fibrant replacement functor and not mention it.
\begin{cor}\hfill
	\vspace{-1em}
	
	\begin{itemize}
		\item (CG) The functor $P\colon \Osp \to \Osp$ has a strong symmetric monoidal structure.
		\item (CGWH) The functor $P\colon \Osp \to \Osp$ has a lax symmetric monoidal structure that is an isomorphism on freely $i$-cofibrant spectra.
	\end{itemize}
\end{cor}

\begin{proof}
	The isomorphism $\Sph \cong P\Sph$ is unique. There is a canonical map $PX \barsmash PY \to P(X \barsmash Y)$ over the identity of $A \times B$, by the above and the fact that $P = (p_1)_!(p_0)^*$, which is an isomorphism (CG) always (CGWH) on freely $i$-cofibrant spectra. As in \autoref{prop:P_strong_monoidal}, by \autoref{rmk:general_rigidity} we can keep track of which natural transformation is the right one by restricting attention to $X = \Sigma^\infty_{+A} A$ and $Y = \Sigma^\infty_{+B} B$, where it gives a map $A^I \times B^I \to (A \times B)^I$. Using this principle it is straightforward to check the associator, unitor, and symmetry coherences.
\end{proof}

\beforesubsection
\subsection{Skeleta and semifree cofibrations}\label{sec:reedy}\aftersubsection

The results in this section apply to orthogonal spectra. The ones that do not involve smash products also hold for sequential spectra, but in that case they say nothing new, see \autoref{rmk:reedy_prespectra}.

This subsection is meant to be skipped unless you want the details of the remainder of the proof of \autoref{prop:spectra_pushout_product}. The remaining statements can't be proven directly using free cofibrations; instead we have to pass to a more general notion that enjoys a close relationship with the skeleta of an orthogonal spectrum.

Let $O$ be a compact Lie group. We will be interested in the case where $O = O(n)$ for $n \geq 0$. A retractive $O$-equivariant space is a retractive space $Y$ over $B$ with a fiber-preserving $O$-action. Furthermore, when we have an $O$-equivariant map of retractive spaces $X \to Y$, we say it is an $f$-cofibration or $h$-cofibration if this is true after forgetting the $O$-action.

\begin{prop}\label{smashing_with_free_complex}
	Under the conventions above:
	\begin{itemize}
		\item If $f: K \to L$ is a free $O$-cell complex of based spaces and $g: X \to Y$ is an $O$-equivariant map of retractive $B$-spaces that is an $h$-cofibration, then $f \square g$ constructed by $(- \barsmash -)_{O}$ is an $h$-cofibration.
		\item The same is true for $f$-cofibrations or closed inclusions.
		\item If $L$ is any finite free based $O$-cell complex, and $X$ is any $O$-equivariant retractive space that is both $h$-cofibrant and $h$-fibrant, then $(L \barsmash X)_O$ is $h$-fibrant.
		\item If $L$ is any finite free based $O$-cell complex, and $g: X \to Y$ is any $O$-equivariant map of retractive $h$-cofibrant spaces over $B$ that is nonequivariantly a weak equivalence, then $(\id_L \barsmash g)_{O}$ is a weak equivalence.
	\end{itemize}
\end{prop}

We will only use the variant of the first statement where $O = O(n)$, $K = *$, and $L = \mathscr J(\R^n,\R^m)$, but this general formulation with pushout-products $\square$ is easier to prove. Note also that the finiteness assumption on $L$ in the last two statements does not seem to be necessary, but shortens the proof.

\begin{proof}
	The key ingredient for the first two claims is that there is an isomorphism
	\[ ((O \times D)_+ \barsmash Y)_{O} \cong (D_+ \barsmash Y) \]
	of retractive spaces over $B$ without $O$-actions (so we forget the action on the right), which is natural in the unbased space $D$ and the retractive space $Y \in \mc R(B)$. This follows from the pushout square \eqref{external_half_smash}.
	Therefore these claims hold in the special case where  $f$ is a single free based $O$-cell,
	\[ (O \times S^{d-1})_+ \to (O \times D^d)_+, \]
	by reduction to \autoref{prop:h_cofibrations_pushout_product}. But the case of general $f$ immediately follows because $- \square g$ preserves pushouts and compositions in the sense explained in \autoref{pushout_product_lemmas}.
	
	For the fourth statement we induct on the skeleta of $L$. If $L$ is obtained from $K$ by pushout along a single cell
	\[ \xymatrix @R=1.7em{
		(O \times S^{d-1})_+ \ar[r] \ar[d] & (O \times D^d)_+ \ar[d] \\
		K \ar[r] & L
	} \]
	Then $g$ gives a map of two pushout squares of the form
	\[ \xymatrix @R=1.7em{
		S^{d-1}_+ \barsmash (-) \ar[r] \ar[d] & D^d_+ \barsmash (-) \ar[d] \\
		(K \barsmash (-))_O \ar[r] & (L \barsmash (-))_O.
	} \]
	The horizontal maps are cofibrations by the first part of this proposition, so each of these squares is a homotopy pushout square. The map $g$ induces equivalences on the top-left and top-right terms by \autoref{prop:h_cofibrations_pushout_product}, and on the bottom-left term by inductive hypothesis, and therefore induces an equivalence on the bottom-right term. By induction $(\id_{L'} \barsmash f)_O$ is an equivalence for all skeleta $L'$ of $L$, and therefore for $L$ itself.
	
	The third statement is proven by the same induction, and we only get one pushout square of the above form. The first three terms are $h$-fibrations and the top horizontal map is an $h$- (therefore $f$-)cofibration, so by \autoref{prop:clapp} the last term is an $h$-fibration.
\end{proof}

\begin{lem}
	$\mathscr J(\R^n,\R^{m+n})$ is a based free $O(n)$-cell complex.
\end{lem}
\begin{proof}
	This is a consequence of Illman's triangulation theorem \cite{illman}.
\end{proof}

The \textbf{semifree spectrum}\index{semifree spectrum} $\mc G_n K$ on an $O(n)$-equivariant retractive space $K$ at level $n$ is defined to be the $\mathscr J$-space
\[ \mathscr J(\R^n,-) \sma_{O(n)} K. \]
This gives the left adjoint to the operation that takes a spectrum $X$ to $X_n$ as an $O(n)$-space.

\begin{prop}\label{smash_two_semi_free}
	Let $K$ and $L$ be ex-spaces over $A$ and $B$, respectively. There is a natural isomorphism
	\[ \mc G_m K \barsmash \mc G_n L \cong \mc G_{m+n} (O(m+n)_+ \barsmash_{O(m) \times O(n)} K \barsmash L). \]
\end{prop}

\begin{proof}
	As in the non-fiberwise case (cf \cite[I.5.14]{schwede_symmetric_spectra}, \cite[2.3.4]{brun_dundas_stolz}), we look at the $\mathscr J \sma \mathscr J$-space defining the smash product on the left-hand side. At bilevel $(m+m',n+n')$ it is
	\[ (\mathscr J(\R^m,\R^{m+m'}) \sma \mathscr J(\R^n,\R^{n+n'})) \barsmash_{O(m) \times O(n)} (K \barsmash L). \]
	Since the left-hand side of the statement in the proposition is defined as a left Kan extension of the diagram just above, it is a left adjoint functor from $O(m) \times O(n)$-spaces over $A \times B$ to spectra over $A \times B$, applied to $K \barsmash L$. To make this agree with the right-hand side, we simply verify that both right adjoints send the spectrum $X$ over $A \times B$ to $X_{m+n}$ considered as an $O(m) \times O(n)$-space.
\end{proof}

Given an orthogonal spectrum $X$ over $B$, its $n$-skeleton $\sk^n X$ is defined by
\[ \sk^n X := (i_n)_!(i_n)^* X, \]
where $i_n$ is the inclusion of the full subcategory of $\mathscr J$ on the objects $\R^0$ through $\R^n$ back into $\mathscr J$, $(i_n)^*$ is the restriction to this subcategory, and $(i_n)_!$ is the enriched left Kan extension. Writing the definition out gives the presentation
\[ \bigvee_{i,j \leq n} \mathscr J(\R^j,-) \barsmash \mathscr J(\R^i,\R^j) \barsmash X_i \rightrightarrows
\bigvee_{i \leq n} \mathscr J(\R^i,-) \barsmash X_i \to \sk^n X. \]
The $n$-skeleton also has a ``semifree'' presentation
\[ \bigvee_{i,j \leq n} \mathscr J(\R^j,-) \barsmash_{O(j)} \mathscr J(\R^i,\R^j) \barsmash_{O(i)} X_i \rightrightarrows
\bigvee_{i \leq n} \mathscr J(\R^i,-) \barsmash_{O(i)} X_i \to \sk^n X. \]
By a standard categorical argument, $\sk^{n-1} X$ admits a map to $X$ which is an isomorphism levels 0 through $n-1$. At level $n$ this map gives an $O(n)$-equivariant map $L_n X \to X_n$ over $B$ called the $n$th \textbf{latching map}\index{latching map} of $X$. (We define the 0th latching map to be $B \to X_0$.)

\begin{ex}
	Let $X = \mc G_n K$. Then $\sk^m X$ is the zero object for $m < n$ and $\mc G_n K$ for $m \geq n$. By comparing universal properties, it has a single nontrivial latching map, which is at level $n$ and is just the inclusion of the zero object $B \to K$.
\end{ex}

More generally, any map of spectra $X \to Y$ over $B$ has a sequence of relative skeleta, defined as pushouts of the skeleta of $Y$ with $X$ along the skeleta of $X$:
\[ \xymatrix @R=1.7em{
	\sk^n X \ar[r] \ar[d] & \sk^n Y \ar[d] \\
	X \ar[r] & \sk^n (X \to Y) } \]
The map from $\sk^{n-1}(X \to Y)$ to $Y$ on level $n$ can then be written
\[ L_n Y \cup_{L_n X} X_n \to Y_n, \]
and we call this the $n$th relative latching map $X \to Y$. Notice that when $X$ is the zero object in spectra (all levels equal to $B$) these relative latching maps are the latching maps of $Y$.

We recall how the skeleta are built inductively from the latching maps.
\begin{prop}\label{prop:skeleta_pushout}
	For each spectrum $X$ over $B$ there is a natural pushout square
	\[ \xymatrix @R=1.7em{
		\mc G_n L_n X \ar[r] \ar[d] & \mc G_n X_n \ar[d] \\
		\sk^{n-1} X \ar[r] & \sk^n X }. \]
\end{prop}
\begin{proof}
	We check that $\sk^n X$ has the universal property of the pushout. To define a map from $\sk^n X$ into $Y$ is to restrict to the objects $0$ through $n$ and define a map $X \to Y$. This map on $X_0$ through $X_{n-1}$ is specified by $\sk^{n-1} X \to Y$, and on $X_n$ is specified by $\mc G_n X_n \to Y$. It is therefore clear that the map out of $\sk^n X$ into $Y$ would be determined uniquely by the maps on $\sk^{n-1} X$ and $\mc G_n X$, and it remains to check that such a map always exists if the other two agree along $\mc G_n L_n X$. We clearly get the spaces $X_0$ through $X_n$ and all the maps between them, except for agreement along morphisms in $\mathscr J(a,n)$ for every $a < n$. But this follows from the claim that two maps $L_n X \rightrightarrows Y_n$ are the same, one coming about from level $n$ of $\sk^{n-1} X \to Y$ and the other coming about by restricting $X_n \to Y_n$ to the subspace $L_n X$. Of course this agreement holds when we have a map from the pushout diagram into $Y$, so we also have the existence part of the universal property for $\sk^n X$.
\end{proof}

\begin{prop}\label{prop:map_skeleta_pushout}
	For each map of spectra $X \to Y$ over $B$ there is a natural pushout square
	\[ \xymatrix @R=1.7em @C=3em{
		\mc G_n (L_n Y \cup_{L_n X} X_n) \ar[r] \ar[d] & \mc G_n Y_n \ar[d] \\
		\sk^{n-1} (X \to Y) \ar[r] & \sk^n (X \to Y) }. \]
\end{prop}
\begin{proof}
	Consider the colimit of the diagram
	\[ \xymatrix @R=1.7em{
		X & \sk^{n-1} X \ar[l] \ar[r] & \sk^{n-1} Y \\
		\mc G_n X_n \ar@{=}[d] \ar[u] & \mc G_n L_n X \ar[r]  \ar[d] \ar[l] \ar[u] & \mc G_n L_n Y \ar[d] \ar[u] \\
		\mc G_n X_n & \mc G_n X_n \ar@{=}[l] \ar[r]  & \mc G_n Y_n
	}\]
	It may be calculated as the pushout of the pushouts of the columns, or as the pushout of the pushouts of the rows. The first recipe, by  \autoref{prop:skeleta_pushout}, gives the pushout square that defines $\sk^n(X \to Y)$. The second recipe gives the pushout square in the statement of this proposition.
\end{proof}

\begin{df}
	A map of spectra $X \to Y$ over $B$ is a \textbf{semifree} or \textbf{Reedy $h$-cofibration}\index{$h$-cofibration!semifree} if each relative latching map $L_n Y \cup_{L_n X} X_n \to Y_n$ is an $h$-cofibration. The notions of semifree $f$-cofibration and semifree closed inclusion are defined similarly.\index{$f$-cofibration!semifree}\index{$i$-cofibration!semifree}\footnote{Loosely, a flat map of orthogonal spectra is just a semifree $q$-cofibration. More precisely, to be flat, the latching map has to be not just a relative cell complex but a relative $O(n)$-cell complex at spectrum level $n$.}
\end{df}

We emphasize that in this and the previous definition, the latching map is $O(n)$-equivariant, but that action is ignored when determining if the map is a cofibration. This means our proofs below will be more general than the ones done in \cite[\S 4]{malkiewich2017coassembly}.\footnote{However when you develop this theory equivariantly, it becomes important to stick with homotopy extension properties that respect the $G \times O(n)$-action.}

\begin{rmk}\label{rmk:reedy_prespectra}
	The notion of a semifree cofibration makes sense for sequential spectra, but it turns out to be the same thing as a free cofibration. The presence of the $O(n)$-action is what makes the two notions different in orthogonal spectra.
\end{rmk}

To simplify our language a bit, in the rest of this section we will use the word ``cofibration'' when making a statement that is equally true for $f$-cofibrations, $h$-cofibrations, or closed inclusions.

\begin{prop}\label{reedy_preserved}
	Semifree cofibrations are closed under pushouts, transfinite compositions (therefore also coproducts), and retracts.
\end{prop}

\begin{proof}
	The proofs are very similar to those in \autoref{pushout_product_composition}, but formally a little different.
	
	If $X \to X \cup_K M$ is a pushout of the semifree cofibration $K \to M$, then its latching map
	\[ L_n (X \cup_K M) \cup_{L_n X} X_n \to X_n \cup_{K_n} M_n \]
	rearranges as
	\[ L_n M \cup_{L_n K} X_n \to M_n \cup_{K_n} X_n \]
	This is a pushout of the original latching map $L_n M \cup_{L_n K} K_n \to M_n$ by the diagram in which every square is a pushout
	\[ \xymatrix @R=1.7em @!C{
		L_n K \ar[r] \ar[d] & L_n M \ar[d] \\
		K_n \ar[r] \ar[d] & L_n M \cup_{L_n K} K_n \ar[r] \ar[d] & M_n \ar[d] \\
		X_n \ar[r] & L_n M \cup_{L_n K} X_n \ar[r] & M_n \cup_{K_n} X_n
	} \]
	and is therefore a cofibration. So the pushout $X \to X \cup_K M$ is a semifree cofibration.
	
	For compositions, suppose $X \to Y \to Z$ are both semifree cofibrations. In the following diagram, every square is a pushout:
	\[ \xymatrix @R=1.7em @!C=6em{
		L_n X \ar[r] \ar[d] & X_n \ar[d] \ar[rd] & \\
		L_n Y \ar[r] \ar[d] & L_n Y \cup_{L_n X} X_n \ar[r]^-{c} \ar[d] & Y_n \ar[d] \ar[rd] \\
		L_n Z \ar[r] & L_n Z \cup_{L_n X} X_n \ar[r]^-{c} & L_n Z \cup_{L_n Y} Y_n \ar[r]^-{c} & Z_n
	}\]
	Therefore the maps labeled $c$ are cofibrations; two of them by assumption, and the last because it is a pushout of a cofibration. It follows that the latching map $L_n Z \cup_{L_n X} X_n \to Z_n$ is also a cofibration.
	
	Next we'll do countable sequential colimits; the case of a transfinite composition is essentially the same. If $X^{(\infty)}$ is the sequential colimit of spectra $X^{(m)}$ and each $X^{(m-1)} \to X^{(m)}$ is a semifree cofibration then the latching map for $X^{(0)} \to X^{(\infty)}$ factors as a sequential colimit of maps of the form
	\[ L_n X^{(\infty)} \cup_{L_n X^{(m-1)}} X^{(m-1)}_n \to L_n X^{(\infty)} \cup_{L_n X^{(m)}} X^{(m)}_n. \]
	Each of these fits into a larger diagram in which the large squares are pushout squares:
	\[ \xymatrix  @R=1.7em @C=1em{
		L_n X^{(m-1)} \ar[r] \ar[d] & \ldots \ar[r] & X^{(m-1)}_n \ar[d] \ar[rrd] && \\
		L_n X^{(m)} \ar[r] \ar[d] & \ldots \ar[r] & L_n X^{(m)} \cup_{L_n X^{(m-1)}} X^{(m-1)}_n \ar[rr]^-{c} \ar[d] && X^{(m)}_n \ar[d] \ar[rd] \\
		\vdots \ar[d] && \vdots \ar[d] && \vdots \ar[d] & \ddots \ar[rd] \\
		L_n X^{(\infty)} \ar[r] & \ldots \ar[r] & L_n X^{(\infty)} \cup_{L_n X^{(m-1)}} X^{(m-1)}_n \ar[rr]^-{c} && L_n X^{(\infty)} \cup_{L_n X^{(m)}} X^{(m)}_n \ar[r] & \ldots \ar[r] & X^{(\infty)}_n
	}\]
	So, as before, each of these component maps is a cofibration. Therefore their sequential colimit $L_n X^{(\infty)} \cup_{L_n X^{(0)}} X^{(0)}_n \to X^{(\infty)}_n$ is also a cofibration.
	
	Finally, retracts are preserved by the latching map construction, and a retract of a cofibration is a cofibration, so a retract of a semifree cofibration is a semifree cofibration.
\end{proof}

\begin{cor}
	The class of semifree cofibrations is the smallest class containing the semifree spectra on cofibrations of spaces, and closed under pushouts, transfinite compositions, and retracts. In particular, every free cofibration is a semifree cofibration.
\end{cor}

Next we show that semifree cofibrations are level cofibrations, generalizing \autoref{prop:free_implies_level}.
\begin{prop}\label{reedy_implies_skeleta_level}
	If the map of spectra $X \to Y$ is a semifree cofibration then each map of skeleta $\sk^n X \to \sk^n Y$ is a level cofibration.
\end{prop}

\begin{proof}
	We induct on $n$. The $0$-skeleta $\sk^0 X \to \sk^0 Y$ are just suspension spectra of the map of spaces $X_0 \to Y_0$, which is a cofibration by assumption (it is the 0th relative latching map). Since $S^m \barsmash -$ preserves cofibrations by \autoref{prop:h_cofibrations_pushout_product}, this gives a cofibration on each spectrum level $m$.
	
	For the inductive step, we focus on spectrum level $m$ for some $m \geq 0$. When $m < n$ we get the square in which the horizontal maps are homeomorphisms,
	\[ \xymatrix @R=1.7em{
		(\sk^{n-1} X)_m \ar[r] \ar[d] & (\sk^n X)_m \ar[d] \\
		(\sk^{n-1} Y)_m \ar[r] & (\sk^n Y)_m
	}. \]
	Since the left-vertical is a cofibration by inductive hypothesis, so is the right vertical. When $m \geq n$, two applications of \autoref{prop:skeleta_pushout} gives us two pushout squares
	\[ \resizebox{\textwidth}{!}{$
		\xymatrix{
		\mathscr J(\R^n,\R^m) \barsmash_{O(n)} L_n X \ar[d] \ar[r] & \mathscr J(\R^n,\R^m) \barsmash_{O(n)} X_n \ar[d] \\
		(\sk^{n-1} X)_m \ar[r] & (\sk^n X)_m }
	\quad \ra \quad
	\xymatrix{
		\mathscr J(\R^n,\R^m) \barsmash_{O(n)} L_n Y \ar[d] \ar[r] &\mathscr J(\R^n,\R^m) \barsmash_{O(n)}  Y_n \ar[d] \\
		(\sk^{n-1} Y)_m \ar[r] & (\sk^n Y)_m }
	$}
	\]
	Applying \autoref{cofibration_of_pushouts} with the lower-left corners as $B$ and the upper-right corners as $C$, for the map of pushouts to be a cofibration, it suffices that the following two maps are cofibrations.
	\[ (\sk^{n-1} X)_m \ra (\sk^{n-1} Y)_m \]
	\[ \mathscr J(\R^n,\R^m) \barsmash_{O(n)} (L_n Y \cup_{L_n X} X_n \to Y_n) \]
	The first is by inductive hypothesis and the second is by \autoref{smashing_with_free_complex}. This completes the induction.
\end{proof}

\begin{cor}\label{reedy_implies_level}
	Every semifree cofibration is a level cofibration.
\end{cor}

The remaining results concern level equivalences and fibrations.

\begin{prop}
	A level equivalence of semifreely $h$-cofibrant spectra $f: X \simar X'$ gives a level equivalence of skeleta $\sk^n X \to \sk^n X'$ for all $n \geq 0$.
\end{prop}

\begin{proof}
	Induction on $n$. Again the $0$-skeleton $\sk^0 X$ is the suspension spectrum $\Sigma^\infty X_0$. Since both $X_0$ and $X'_0$ are $h$-cofibrant the map between their suspension spectra is an equivalence on every level by \autoref{prop:h_cofibrations_pushout_product}.
	
	For the inductive step we again use the pushout square of \autoref{prop:skeleta_pushout}, and prove the claim for each spectrum level $m \geq n$, assuming the claim is true for $n-1$ and for every value of $m$.
	
	In particular, we can assume that the map of spaces $L_n X \to L_n X'$ is a weak equivalence, since this is the statement of the result for $n-1$ at spectrum level $n$. We also know that the spaces $L_n X$, $L_n X'$, $X_n$, and $X_n'$ are $h$-cofibrant by \autoref{reedy_implies_skeleta_level}. Combining these observations with \autoref{smashing_with_free_complex}, we conclude that $f$ induces an equivalence on the top-left and top-right terms of the pushout square
	\[ \xymatrix @R=1.7em{
		\mathscr J(\R^n,\R^m) \barsmash_{O(n)} L_n X \ar[d] \ar[r] & \mathscr J(\R^n,\R^m) \barsmash_{O(n)} X_n \ar[d] \\
		(\sk^{n-1} X)_m \ar[r] & (\sk^n X)_m } \]
	Again by \autoref{smashing_with_free_complex} the top horizontal is an $h$-cofibration, and the inductive hypothesis tells us that $f$ induces an equivalence on the bottom-left term. We conclude that $f$ gives an equivalence on $(\sk^n X)_m$, completing the induction.
\end{proof}

\begin{cor}\label{cor:latching_objects_derived}
	A level equivalence of semifreely $h$-cofibrant spectra $X \simar X'$ gives an equivalence of the latching objects $L_n X \simar L_n X'$. Therefore the relative latching map $L_n X' \cup_{L_n X} X_n \to X_n'$ is a weak equivalence for every $n \geq 0$.
\end{cor}

\begin{prop}
	If $X$ is semifreely $h$-cofibrant and level $h$-fibrant then the each skeleton $\sk^n X$ is level $h$-fibrant.
\end{prop}

\begin{proof}
	The same induction on $n$ as above. The levels $S^n \barsmash X_0$ of the 0-skeleton are $h$-fibrant by \autoref{prop:h_cofibrations_pushout_product}. For the inductive step we get to assume that $L_n X$ is $h$-fibrant, and so both $L_n X$ and $X_n$ are both $h$-cofibrant and $h$-fibrant. Combining these observations with \autoref{smashing_with_free_complex}, we conclude that the first three vertices in the square
	\[ \xymatrix @R=1.7em{
		\mathscr J(\R^n,\R^m) \barsmash_{O(n)} L_n X \ar[d] \ar[r] & \mathscr J(\R^n,\R^m) \barsmash_{O(n)} X_n \ar[d] \\
		(\sk^{n-1} X)_m \ar[r] & (\sk^n X)_m } \]
	are $h$-fibrant, and the top horizontal is an $h$-cofibration. By \autoref{prop:clapp} again the last vertex is also $h$-fibrant.
\end{proof}

\begin{cor}\label{cor:latching_objects_fibrant}
	If $X$ is semifreely $h$-cofibrant and level $h$-fibrant then each latching space $L_n X$ is $h$-fibrant.
\end{cor}

The climax of this section is the following extension of \autoref{prop:spectra_pushout_product}.
\begin{thm}\label{prop:reedy_pushout_product}
	Let $f: X \to Y$ and $g: X' \to Y'$ be maps of orthogonal spectra over $A$ and $B$, respectively. Form $f \square g$ using the external smash product of spectra.
	\begin{enumerate}
		\item If $f$ and $g$ are semifree $h$-cofibrations, so is their pushout-product $f \square g$. The same applies to semifree $f$-cofibrations and semifree closed inclusions.
		\item If $Y$ and $Y'$ are both semifreely $h$-cofibrant and level $h$-fibrant then so is $Y \barsmash Y'$.
		\item If $f$ is a semifree $h$-cofibration, $g$ is a level equivalence, and all four spectra are semifreely $h$-cofibrant, then $f \square g$ is a level equivalence.
	\end{enumerate}
\end{thm}

\begin{proof}
	Using \autoref{prop:map_skeleta_pushout} and the formal results \autoref{pushout_product_pushouts} and \autoref{pushout_product_composition}, (1) reduces to the case where the maps $f$ and $g$ are of the form
	\[ f: \mc G_m(K \to L), \qquad g: \mc G_n(K' \to L'), \]
	where $\phi: K \to L$ and $\gamma: K' \to L'$ are cofibrations of ex-spaces over $A$ and $B$, respectively. By \autoref{smash_two_semi_free} and the fact that the semifree spectrum functor preserves colimits we get
	\[ f \square g \cong \mc G_{m+n}(O(m+n)_+ \barsmash_{O(m) \times O(n)} (\phi \square \gamma)) \]
	Therefore it suffices to prove that $O(m+n)_+ \sma_{O(m) \times O(n)} (\phi \square \gamma)$ is a $h$-cofibration. This follows from one application of \autoref{smashing_with_free_complex} and one of \autoref{prop:h_cofibrations_pushout_product}.
	
	For (2), it suffices to prove this when $Y$ is equal to its $m$-skeleton and $Y'$ is equal to its $m'$-skeleton, for any value of $m$ and $m'$. When both are equal to 0 it is clearly true. Incrementing $m$ or $m'$ requires us to replace one of the two terms by a pushout along a map of the form $\mc G_{m+1}(K \to L)$ where $K$ and $L$ are both $h$-cofibrant and $h$-fibrant (by \autoref{cor:latching_objects_fibrant}), and $K \to L$ is an $h$-cofibration. By \autoref{smashing_with_free_complex} we know that each level of $\mc G_{m+1} K$ and $\mc G_{m+1} L$ is $h$-fibrant, therefore at each spectrum level we get a square of the form found in \autoref{prop:clapp}, therefore the final term is also $h$-fibrant.
	
	For (3) assume first that $f$ and $g$ are in the special form described in (1) above, with $\gamma\colon K' \to L'$ a weak equivalence. Then $f \square g$ is a semifree spectrum on a weak equivalence, and is therefore a level equivalence by \autoref{smashing_with_free_complex}. Next, if $g$ is a semifree cofibration and level equivalence, it factors as a sequence of pushouts of maps of the previous form, by \autoref{prop:map_skeleta_pushout} and \autoref{cor:latching_objects_derived}. Therefore $f \square g$ is a composition of pushouts of maps that are level $h$-cofibrations and weak equivalences, so $f \square g$ is also a level $h$-cofibration and weak equivalence. By the same argument, we may also allow $f$ to be a general semifree $h$-cofibration.
	
	Finally we adapt the proof of Ken Brown's lemma. For fixed $f$, the class of maps $g$ between semifreely cofibrant spectra such that $f \square g$ is a level equivalence includes the acyclic cofibrations (i.e. semifree cofibrations that are level equivalences) and is closed under 2-out-of-3. Given an arbitrary level equivalence $X \to Y$ of semifree-cofibrant spectra, we factor the obvious map $X \vee Y \to Y$ into a cofibration $X \vee Y \to Z$ followed by a level equivalence $Z \to Y$, using \autoref{prop:level_model_structure}. Then the inclusions of $X$ and $Y$ into $Z$ are semifree cofibrations and level equivalences, so they are in the class. The composite $Y \to Z \to Y$ is the identity, which is in the class, and so by 2 out of 3 the projection $Z \to Y$ is in the class. Then $X \to Z \to Y$ is a composite of two maps in the class and is therefore in the class as well. This finishes the proof.

\end{proof}

\begin{rmk}
	(3) can be strengthened: If $f$ is a semifree $h$-cofibration, $g$ is a level equivalence, $X$, $Y$, and $X'$ are semifreely $h$-cofibrant, and $Y'$ is level $h$-cofibrant, then $f \square g$ is a level equivalence. However the proof of this strengthening is more complicated.
\end{rmk}


\newpage
\section{Stable equivalences}

This section completes the foundations of parametrized spectra. We define the stable equivalences and extend them to a model structure. Then we describe how to derive the operations we have encountered with respect to the stable equivalences. Finally, using the techniques from \autoref{sec:composing_comparing} we lift the compatibilities between these operations to compatibilities between their derived functors.

\beforesubsection
\subsection{Stable homotopy groups}\aftersubsection

If $X$ is a parametrized sequential spectrum or orthogonal spectrum over $B$, its \textbf{(stable) homotopy groups}\index{homotopy!groups} are the stable homotopy groups of the fiber spectra $X_b$, for $b \in B$ and $n \in \Z$:
\[ \pi_{n,b}(X) := \pi_n(X_b) = \underset{k}\colim \pi_{k+n}((X_b)_n). \]
These homotopy groups do not preserve level equivalences. For instance the suspension spectrum of $\{0\}_{+I} \to I_{+I}$ over $I$ is a level equivalence but not an isomorphism on $\pi_{0,1/2}$.

However, the homotopy groups are right-deformable. Recall that the level fibrant replacement functor $R^{lv}$ from \autoref{prop:level_model_structure} is a retract into the subcategory of level $q$-fibrant spectra. Between two spectra in this subcategory, any level equivalence $X \to Y$ induces a level equivalence of fiber spectra $X_b \to Y_b$ for each $b \in B$, which is clearly an isomorphism on $\pi_*$. Therefore the homotopy groups send level equivalences in this subcategory to isomorphisms.

We therefore get \textbf{right-derived homotopy groups} $\R \pi_{n,b}(X)$, cf. \cite[12.3.4]{ms}. These can be described in more than one way up to natural isomorphism:
\[ \R\pi_{n,b}(X) \cong \pi_{n,b}(R^{lv}X) \cong \pi_{n,b}(PQX). \]
A map $X \to Y$ is an \textbf{$\R\pi_*$-isomorphism} or \textbf{stable equivalence}\index{stable equivalence} if it induces isomorphisms on the derived homotopy groups for all $b \in B$ and $n \in \Z$. So informally, a stable equivalence is a map that induces isomorphisms on $\pi_*$ after we make the levels of our spectra into fibrations. Since the stable equivalences are precisely those maps that the functors $\{\R\pi_{n,b}(-)\}$ send to isomorphisms, they make the category of orthogonal spectra $\Osp(B)$ or of sequential spectra $\Psp(B)$ into a homotopical category.

\begin{ex}\label{ex:equivalent_thom_spectra} \hfill
	\vspace{-1em}
	
	\begin{itemize}
		\item If $E \to E'$ is a stable equivalence of spectra then it induces a map of trivial bundles $r^*E \to r^*E'$, or $E \times B \to E' \times B$, which is a stable equivalence.
		\item If $X \to B$ is a weak Hausdorff fibration then the map from the fiberwise suspension spectrum of $X$ to its classical fibrant replacement
		\[ \Sigma^\infty_{+B} X \to Q_B(\Sigma^\infty_{+B} X) = \underset{k}\colim \Omega^k_B \Sigma^{k}_B \Sigma^\infty_{+B} X \]
		is a stable equivalence. This is because both spectra are fibrations at every level, and on each fiber this statement is a classical fact in stable homotopy theory.
		\item If $X \in \mc R(B)$ there is a standard map
		\[ F_{n+1} \Sigma_B X \to F_n X. \]
		In sequential spectra, we define this by observing that the truncation of $F_n X$ that replaces level $n$ by the zero object returns $F_{n+1} \Sigma_B X$. In orthogonal spectra, this map is adjoint to the map
		\[ X \barsmash S^1 \to X \barsmash \mathscr J(\R^n,\R^{n+1}) \]
		that identifies $S^1$ with the fiber over the standard embedding $\R^n \to \R^{n+1}$.
		
		As long as $X$ is an $h$-cofibrant space, this standard map is a stable equivalence. The proof is simple: apply $P$ to make the levels fibrations. Since $P$ commutes with $\Sigma_B$ and $F_n$, this now follows from the same statement for non-parametrized spectra.
		\item By the previous example and \autoref{ex:thom_smash_product}, we have for any vector bundle $W$ over $B$ a stable equivalence
		\[ F_{m+n} \Th_B(\R^m \times W) \overset\sim\to F_n \Th_B(W). \]
		Therefore different models for the Thom spectrum of a virtual bundle $\xi$ (see \autoref{ex:thom_spectra}) are stably equivalent.
	\end{itemize}
\end{ex}

When we want to prove that a functor $F$ preserves stable equivalences, it is easiest to break up the stable equivalence $X \to Y$ into a zig-zag of simpler pieces, and prove that each piece is separately preserved by $F$. The two most useful decompositions are
\begin{equation}\label{eq:stable_equivalence_decomposition}
\xymatrix @R=1.7em{
	X \ar[r]^-\sim & R^{lv}X \ar[d] \\
	Y \ar[r]^-\sim & R^{lv}Y
}
\qquad
\xymatrix @R=1.7em{
	X & \ar[l]_-\sim QX \ar[r]^-\sim & PQX \ar[d] \\
	Y & \ar[l]_-\sim QY \ar[r]^-\sim & PQY.
}
\end{equation}
The maps marked $\sim$ are level equivalences, and the vertical maps are $\pi_*$-isomorphisms on each fiber spectrum. The following is more or less immediate.

\begin{prop}\label{stable_equivalences}
	The class of stable equivalences is generated under 2-out-of-3 by the level equivalences and the maps of level $h$-fibrant spectra inducing isomorphisms on the homotopy groups of each fiber. A map of level-quasifibrant spectra is a stable equivalence iff it is a $\pi_*$-isomorphism on each fiber.
\end{prop}

\begin{ex}\label{ex:smash_with_cell_complex_is_derived} \hfill
	\vspace{-1em}
	
	\begin{itemize}
		\item If $K \in \mc R(A)$ is an $h$-cofibrant space then $K \barsmash -$ preserves level equivalences of level $h$-cofibrant spectra. It also preserves $\pi_*$-isomorphisms of spectra that are level $h$-cofibrant and $h$-fibrant, because on each fiber this is smashing a stable equivalence of ordinary spectra with a space, and all the spaces involved are well-based, hence the result is still a stable equivalence \cite[7.4(i)]{mmss}. Using the second decomposition in \eqref{eq:stable_equivalence_decomposition}, we conclude that $K \barsmash -$ preserves stable equivalences between level $h$-cofibrant spectra.
		\item Similarly, if $K$ is a finite cell complex then $\barF_*(K,-)$ preserves $\pi_*$-isomorphisms of level $q$-fibrant spectra, by an induction on the cells of $K$ and the long exact sequence for fiber sequences of spectra \cite[7.4.(vi)]{mmss}. Hence it preserves stable equivalences of level $q$-fibrant spectra.
		\item The maps $X \to \Omega_B\Sigma_B X$ and $\Sigma_B\Omega_B X \to X$ are always $\pi_*$-isomorphisms on each fiber, using \cite[7.4.(i), 7.4.(i')]{mmss}, but they are not always stable equivalences, see for instance the counterexample in \autoref{space_cofiber_fiber}. However if we left-derive $\Sigma_B$ and right-derive $\Omega_B$ then they are stable equivalences.\footnote{There are two notions of equivalence floating around, the level equivalences and the stable equivalences. A priori, deriving a functor with respect to each one could produce different results, see \autoref{warn:changing_equivs}. But in this case the two derived functors we get are equivalent.} To prove this we observe that if $X$ is level $f$-cofibrant, level $h$-fibrant, and weak Hausdorff, then by repeated application of \autoref{prop:h_cofibrations_pushout_product} and \autoref{h_fibrations_pullback_hom}, the source and target of both maps are level $h$-fibrant. This, and the fact that the map is a $\pi_*$-isomorphism, implies that it is a stable equivalence. Moreover under these assumptions on $X$, the functors $\Sigma_B$ and $\Omega_B$ are equivalent to their derived functors, so the same is true for the map of derived functors for any input $X$.
	\end{itemize}
\end{ex}

\begin{prop}(cf. \cite[12.4.3]{ms}, \cite[7.4]{mmss})\label{prop:coproduct_colimit_stable_equivalences}
	The derived homotopy groups commute with arbitrary coproducts of level $h$-cofibrant spectra, and sequential colimits along level $h$-cofibrations.\footnote{In (CGWH) these assumptions can be weakened. The maps of the colimit system only have to be level closed inclusions. In the coproduct, the $h$-cofibrant assumption can be dropped when $B = *$ but probably not in general.} Hence both of these operations preserve stable equivalences.
\end{prop}

\begin{proof}
	If $X$ and $Y$ are parametrized spectra over $B$, their coproduct $X \cup_B Y$ is just the union along $B$ on each spectrum level. When $X^i$ is an arbitrary collection of spectra over $B$, their coproduct is similarly $\bigcup_B X^i$. If the $X^i$ are all level $h$-cofibrant then this preserves level equivalences, so to compute the homotopy groups of the coproduct, we can make all the $X^i$ level $h$-cofibrant and $h$-fibrant. In this case the coproduct is also level $h$-fibrant because the proof of \autoref{prop:clapp} works just as well for these arbitrary pushouts along a common subspace as it does for pushouts of two spaces. Alternatively, we could use $P$ to make the fibrant replacement, and $P$ commutes with the coproduct by \autoref{prop:f_star_colimits}. At any rate, this means the derived homotopy groups of the coproduct are the actual homotopy groups of its fibers, which are the coproduct of the homotopy groups of the fibers of $X^i$ by \cite[7.4]{mmss}.
	
	For the sequential colimit of spectra $X^i$ over $B$, we repeatedly factor the maps into $q$-cofibrations followed by level equivalences, producing an equivalent colimit system of spectra where the maps of the system are $q$-cofibrations and the first term is $q$-cofibrant. This new system has colimit level equivalent to the original system by \autoref{prop:technical_cofibrations}. Then we apply $P$, giving another colimit system with equivalent colimit, where each map is an $f$-cofibration and each term is level $h$-fibrant. Since $P$ commutes with sequential colimits along level closed inclusions, the colimit is also level $h$-fibrant. So its derived homotopy groups are the ordinary homotopy groups, which on each fiber over $B$ are the colimit of the homotopy groups of the other spectra in the system by \cite[7.4]{mmss}.
\end{proof}

\begin{ex} \hfill
	\vspace{-1em}
	
	\begin{itemize}
		\item Recall that the product $X \times_B Y$ is the fiber product over $B$ on each spectrum level. Using either decomposition in \eqref{eq:stable_equivalence_decomposition}, we see the product preserves stable equivalences on level $q$-fibrant spectra. By induction, the same is true for finite products. This turns out to be false however for infinite products -- an infinite product has to be right-derived using the stable model structure below before it will preserve stable equivalences.
		\item The canonical map $X \cup_B Y \to X \times_B Y$ is always a $\pi_*$-isomorphism on each fiber separately \cite[7.4(ii)]{mmss}. On inputs that are both level $f$-cofibrant and level $h$-fibrant, it is therefore a stable equivalence. Hence the canonical map of derived functors $X \cup_B^{\L} Y \to X \times_B^{\R} Y$ is a stable equivalence.\footnote{We could have waited until later and then recovered this statement from the fact that homotopy cofiber and fiber sequences coincide, see \autoref{prop:pushout_equals_pullback}.}
	\end{itemize}
\end{ex}

As the above examples indicate, we usually prove things about stable equivalences by assuming the inputs are level fibrant, or applying $P$, and then reducing to the non-parametrized case. This justifies the work in earlier sections proving that various operations preserve fibrations or commute with $P$.

\beforesubsection
\subsection{Cofibers, fibers, pushouts and pullbacks}\aftersubsection

Every result in this subsection applies to both sequential spectra and orthogonal spectra.

Let $f\colon X \to Y$ be a map of sequential spectra or orthogonal spectra. Define the mapping cone or \textbf{uncorrected homotopy cofiber}\index{uncorrected homotopy cofiber}\index{mapping cone} of $f$ by the formula
\[ C_Bf = (X \barsmash I) \cup_{X \barsmash S^0} Y. \]
As in \autoref{space_cofiber_fiber}, on each fiber this gives the usual mapping cone $Cf_b = (X_b \sma I) \cup_{X_b} Y_b$, and we have an isomorphism of spectra $PC_B f \cong C_B Pf$.

The \textbf{homotopy cofiber}\index{homotopy!cofiber} the left-derived functor of the mapping cone, as a functor from maps of spectra to spectra, using the level equivalences. It is given by the formula $C_B(Qf)$, using \autoref{prop:free_implies_level} and the discussion in \autoref{space_cofiber_fiber}. As in \autoref{ex:strict_and_uncorrected}, the map to the strict cofiber
\[ C_B f \to Y \cup_X B \]
induces an equivalence of left-derived functors, so we could also think of the homotopy cofiber as left-derived from the cofiber.

\begin{lem}\label{lem:LES}
	There is a natural long exact sequence
	\[ \xymatrix{
		\ldots \ar[r] & \R\pi_{n,b}(X) \ar[r]^-{\R\pi_{n,b}(f)} & \R\pi_{n,b}(Y) \ar[r] & \R\pi_{n,b}(\L C_B f) \ar[r] & \R\pi_{n-1,b}(X) \ar[r] & \ldots
	}. \]
	Therefore $\L C_B f \simeq C_BQf$ preserves stable equivalences.
\end{lem}

\begin{proof}	
	We take the usual the non-parametrized long exact sequence for the homotopy groups of the fibers of $PQX \to PQY \to C_B PQf$, then identify the homotopy groups of $C_B PQf$ with those of $PQC_B Qf$ along the functorial string of maps
	\[ \xymatrix{
		C_B PQf \ar@{<->}[r]^-\cong & PC_B Qf & \ar[l]_-\sim PQC_B Qf
	} \]
	inducing isomorphisms on the (underived) homotopy groups of every fiber spectrum. As usual, it is not necessary to worry whether all possible reasonable choices of isomorphism give the same map. We only need to know that some functorial long exact sequence exists. 
\end{proof}

\begin{lem}
	$C_B f$ preserves stable equivalences when $f$ is a level $h$-cofibration or both $X$ and $Y$ are level $h$-cofibrant. $\L C_B f$ is also a left-derived functor of $C_B f$ using the \emph{stable} equivalences.
\end{lem}

\begin{proof}
	The discussion in \autoref{space_cofiber_fiber} tells us that for such maps $C_BQf \to C_B f$ is a level equivalence, so $C_B f$ preserves stable equivalences in these cases by \autoref{lem:LES}. Therefore the same left deformation retract $Q$ derives $C_B f$ with respect to the level and the stable equivalences.
\end{proof}

\begin{cor}[Left-properness]\label{cor:spectra_left_proper}
	In a strict pushout square of spectra over $B$
	\[ \xymatrix @R=1.7em{
		X \ar[r]^-i \ar[d]^-f & Y \ar[d] \\
		Z \ar[r] & Y \cup_X Z
	} \]
	if $f\colon X \to Z$ is a stable equivalence and either $i$ or $f$ is a level $h$-cofibration then $Y \to Y \cup_X Z$ is also a stable equivalence.
\end{cor}
\begin{cor}[Gluing lemma]\label{cor:spectra_gluing}
	Any diagram of spectra over $B$
	\[ \xymatrix @R=1.7em{
		Y \ar[d]^-\sim & X \ar[l]_-i \ar[r]^-f \ar[d]^-\sim & Z \ar[d]^-\sim \\
		Y' & X' \ar[l]^{i'} \ar[r]_-{f'} & Z'
	} \]
	in which the vertical maps are stable equivalences and both $i$ and $i'$ are level $h$-cofibrations, induces a stable equivalence of pushout spectra
	\[ Y \cup_X Z \to Y' \cup_{X'} Z' \]
\end{cor}\index{gluing lemma}

\begin{proof}
	For left properness, if $i$ is the cofibration, we factor $X \to Z$ into a level $h$-cofibration and a level equivalence, and from this reduce to the case where $f$ is the cofibration. Then take mapping cones in the vertical direction. For $X \to Z$ this cone is weakly contractible by \autoref{lem:LES}, but the cones are homeomorphic because the above square is a pushout. Therefore the cone of $Y \to Y \cup_X Z$ is also weakly contractible. By \autoref{lem:LES} again, $Y \to Y \cup_X Z$ is a stable equivalence.
	
	The gluing lemma is known to follow from left-properness by a long diagram-chase. Alternatively, we compare the mapping cones of $i\colon X \to Y$ and $\bar i\colon Z \to Y \cup_X Z$. In the square of mapping cones
	\[ \xymatrix @R=1.7em{
		\L C_B i \ar[d] \ar[r] & \L C_B \bar i \ar[d] \\
		\L C_B i' \ar[r] & \L C_B \bar i',
	} \]
	the horizontal maps are equivalences because they are isomorphisms before $\L$ and the maps $i, \bar i, i', \bar i'$ are all level $h$-cofibrations, hence $\L C_B \simeq C_B$. The left-vertical map is a stable equivalence by one application of \autoref{lem:LES} to $X \to Y$, therefore the right-vertical is a stable equivalence as well. By one more application of \autoref{lem:LES} to $\bar i$ and $\bar i'$, the map of pushouts is a stable equivalence.
\end{proof}

\begin{rmk}\label{fixed_may_sigurdsson_2}
	In \cite{ms}, the corresponding results assume that $i$ is a level $f$-cofibration, which is stricter than a level $h$-cofibration. Furthermore, the arguments employed in \cite[Ch. 5-6]{ms} and the counterexample \cite[6.1.5]{ms} suggest that no further improvement is possible. We would like to expand on \autoref{fixed_may_sigurdsson} by explaining why this approach avoids the difficulty and proves a stronger theorem. One difference is that we are using a symmetric monoidal fibrant replacement functor $P$, allowing us to commute fibrant replacement with the mapping cone.
	
	However, the most significant change is that we are thinking in terms of $\barsmash$ instead of $\sma_B$. If we think in terms of $\sma_B$, we are led to define the uncorrected homotopy cofiber as
	\[ C_B f = (X \sma_B I_B) \cup_{X \sma_B S^0_B} Y \]
	instead of $(X \barsmash I) \cup_{X \barsmash S^0} Y$. Though these two models are homeomorphic, the one with internal smash products suggests that $C_B f$ only preserves equivalences if $X$ and $Y$ are level $f$-cofibrant, even though it also preserves equivalences when they are level $h$-cofibrant. This small difference in assumptions leads to a large difference once we construct the model structure. If the cofibrant objects need to be level $f$-cofibrant, then the usual cells $S^{n-1}_{+B} \to D^n_{+B}$ won't work, we need to further restrict to cells where the map $S^{n-1} \to D^n$ is an $f$-cofibration, in other words the ``$qf$-cells.''\footnote{The language of ``well-grounded model categories'' is used in \cite{ms} to explain what assumptions are needed to get the gluing lemma to work. In that language, the $q$-model structure is not well-grounded \cite[6.1.3]{ms}. However, it actually is well-grounded, using a \emph{different ground structure}, that we obtain from the one in \cite[5.3.6]{ms} by replacing the $f$-cofibrations with $h$-cofibrations. If we examine the second half of the counterexample \cite[6.1.5]{ms} carefully, we see that it does not obstruct the existence of this ground structure, because $X \to Y'$ is not an $h$-cofibration to begin with.}
\end{rmk}

Next, define the \textbf{uncorrected homotopy fiber}\index{uncorrected homotopy fiber} of $f\colon X \to Y$ by
\[ F_B f = X \times_{\barF_*(S^0,Y)} \barF_*(I,Y). \]
On each fiber this gives the usual homotopy fiber $F_B f_b = X_b \times_{Y_b} F(I,Y_b)$. Note however that it may not preserve level equivalences because the fibers $X_b$ and $Y_b$ could themselves change.

The \textbf{homotopy fiber}\index{homotopy!fiber} is right-derived from this, again using the level equivalences. It is equivalent to $FR^{lv}f$, or $F_B f$ whenever $f$ is a level $q$-fibration or a map between level $q$-fibrant spectra. On the subcategory of spectra whose levels are $h$-cofibrant, it is also equivalent to $F_B Pf$. This is all by the discussion in \autoref{space_cofiber_fiber}.
\begin{lem}\label{lem:LES2}
	There is a natural long exact sequence
	\[ \xymatrix{
		\ldots \ar[r] & \R\pi_{n,b}(X) \ar[r]^-{\R\pi_{n,b}(f)} & \R\pi_{n,b}(Y) \ar[r] & \R\pi_{n-1,b}(F_B R^{lv}f) \ar[r] & \R\pi_{n-1,b}(X) \ar[r] & \ldots
	}. \]
	Therefore $\R F f \simeq F_B R^{lv}f$ preserves stable equivalences.
\end{lem}

\begin{proof}
	The proof is like the one for cofibers, but easier. We take the usual the non-parametrized long exact sequence for the fibers of $F_B R^{lv}f \to R^{lv}X \to R^{lv}Y$. We check that $F_B R^{lv}X$ is level $q$-fibrant and therefore $\R\pi_{n,b}(F_B R^{lv}f) \cong \pi_{n,b}(F_B R^{lv}f)$.
\end{proof}

\begin{lem}
	$F_B f$ preserves stable equivalences when $f$ is a level $q$-fibration or both $X$ and $Y$ are level $q$-fibrant. $\R F_B f$ is also the right-derived functor of $F_B f$ using the \emph{stable} equivalences.
\end{lem}
\begin{cor}[Right-properness]\label{cor:spectra_right_proper}
	In a strict pullback square of spectra over $B$
	\[ \xymatrix @R=1.7em{
		Y \times_W Z \ar[r] \ar[d] & Y \ar[d]^-p \\
		Z \ar[r]_-f & W
	} \]
	if $f\colon Z \to W$ is a stable equivalence and either $p$ or $f$ is a level $q$-fibration then $Y \times_W Z \to Y$ is also a stable equivalence.
\end{cor}
\begin{cor}[Dual gluing lemma]
	Any diagram of spectra over $B$
	\[ \xymatrix @R=1.7em{
		Y \ar[d]^-\sim & W \ar@{<-}[l]_-p \ar@{<-}[r]^-f \ar[d]^-\sim & Z \ar[d]^-\sim \\
		Y' & W' \ar@{<-}[l]^{p'} \ar@{<-}[r]_-{f'} & Z'
	} \]
	in which the vertical maps are stable equivalences and both $p$ and $p'$ are level $q$-fibrations, induces a stable equivalence of pullback spectra
	\[ Y \times_W Z \to Y' \times_{W'} Z' \]
\end{cor}

\begin{proof}
	Dual to the above.
\end{proof}

\begin{ex}
	There is a natural equivalence $(\R F)f \simeq (\R\Omega_B)(\L C)f$ for any map $f\colon X \to Y$ of parametrized spectra. To see this, restrict attention to $X$ and $Y$ both free $f$-cofibrant and level $h$-fibrant. Then the output of each of these functors is still level $h$-fibrant, so it suffices to construct under these assumptions a natural map $F_B f \to \Omega_B C_B f$ that is an equivalence on each fiber. For this we can use the usual recipe from \cite{mmss}, which takes a point in $X$ and path in $Y$ to the concatenation of that path in $Y \subseteq C_B f$ with the canonical path from its endpoint in $X$ to the cone point of $C_B f$.
\end{ex}

The above two gluing lemmas imply that pushouts of spectra over $B$ can be left-derived, and that pullbacks can be right-derived, and that furthermore it does not matter whether we use level or stable equivalences to derive them. This allows us to make the following definition.

Given a strictly commuting square of spectra over $B$,
\[ \xymatrix @R=1.7em{
	X \ar[d] \ar[r] & Y \ar[d] \\
	Z \ar[r] & W,
} \]
we say that it is a \textbf{homotopy pushout square}\index{homotopy!pushout} if the map from the left-derived pushout to $W$ is a stable equivalence, and a \textbf{homotopy pullback square}\index{homotopy!pullback} if the map from $X$ to the right-derived pullback is a stable equivalence.
\begin{prop}\label{prop:pushout_equals_pullback}
	A square of spectra over $B$ is homotopy pushout iff it is homotopy pullback.
\end{prop}

\begin{proof}
	By construction, these conditions are unchanged if we replace the square up to stable equivalence, so we can first apply $PQ$ to make all four spectra free $f$-cofibrant and level $h$-fibrant. Then using \autoref{prop:h_cofibrations_pushout_product}, since $X$ is level $h$-cofibrant the left-derived functor of the pushout is modeled by the double mapping cylinder
	\[ Y \cup_X (X \barsmash I_+) \cup_X Z, \]
	and since $W$ is level $h$-fibrant the right-derived functor of the pullback is modeled by the double mapping co-cylinder
	\[ Y \times_W \barF_*(I_+,W) \times_W Z. \]
	In this setting, both of these constructions produce a level $h$-fibrant spectrum, so each of the two maps
	\begin{align*}
		Y \cup_X (X \barsmash I_+) \cup_X Z &\to W \\
		X &\to Y \times_W \barF_*(I_+,W) \times_W Z
	\end{align*}
	is a stable equivalence iff it is a $\pi_*$-isomorphism on each fiber. This reduces the proposition to the non-parametrized case, handled in \cite{mmss}. An alternate proof proceeds by using the equivalence $F_B f \simeq \Omega_B C_B f$ twice to show that the ``total homotopy fiber'' of the square is contractible iff the ``total homotopy cofiber'' is contractible.
\end{proof}

The above proof tells us that a square of level fibrant spectra is homotopy pushout/pullback iff each of the squares of fiber spectra
\[ \xymatrix @R=1.7em{
	X_b \ar[d] \ar[r] & Y_b \ar[d] \\
	Z_b \ar[r] & W_b,
} \]
is homotopy pushout/pullback, so we can also copy over the Mayer-Vietoris sequence from the non-parametrized case.

\begin{cor}
	There is a natural long exact sequence for homotopy pushout/pullback squares
	\[ \xymatrix{
		\ldots \ar[r] & \R\pi_{n,b}(X) \ar[r] & \R\pi_{n,b}(Y) \oplus \R\pi_{n,b}(Z) \ar[r] & \R\pi_{n,b}(W) \ar[r] & \R\pi_{n-1,b}(X) \ar[r] & \ldots
	}. \]
\end{cor}

\beforesubsection
\subsection{The stable model structure}\aftersubsection

The results in this section work for both sequential spectra and orthogonal spectra.

To finish the process of proving that $f_!$ and $\barsmash$ preserve stable equivalences, we need a model structure whose weak equivalences are the stable equivalences.

\begin{thm}[Stable model structure]\label{thm:stable_model_structure}
	There is a proper model structure on $\Osp(B)$ and on $\Psp(B)$ whose weak equivalences are the stable equivalences. Each of these model structures is cofibrantly generated by the sets of maps
	\[ \begin{array}{ccrll}
		I &=& \{ \ F_k\left[ S^{n-1}_{+B} \to D^n_{+B} \right] & : n,k \geq 0, & D^n \to B \ \} \\
		J &=& \{ \ F_k\left[ D^n_{+B} \to (D^n \times I)_{+B} \right] &: n,k \geq 0, & (D^n \times I) \to B \ \} \\
		& \cup & \{ \ k_{i,j} \ \square \ \left[ S^{n-1}_{+B} \to D^n_{+B} \right] & : i,j,n \geq 0, & D^n \to B \ \}.
	\end{array} \]
\end{thm}\index{stable model structure}

\begin{rmk}
	At the time of \cite{ms}, it was not known whether this model structure existed. It was referred to as the ``$q$-model structure.'' The ``$qf$-model structure'' built in \cite{ms} is similar except that one restricts further to those cells for which the maps $S^{n-1} \to D^n$ are also $f$-cofibrations. As a result, its cofibrations are all free $f$-cofibrations, not just free $h$-cofibrations.
\end{rmk}

In the statement of \autoref{thm:stable_model_structure}, the pushout-product is carried out using the operation of smashing a spectrum over $*$ with a space over $B$, see \autoref{sec:smash_with_space}. The map of non-parametrized spectra $k_{i,j}$ includes the front end of the mapping cylinder $\Cyl_{i,j}$ for a map
\[ \lambda_{i,j}\colon F_{i+j} S^j \ra F_i S^0. \]
As in \autoref{ex:equivalent_thom_spectra}, in sequential spectra $\lambda_{i,j}$ is just a truncation, while in orthogonal spectra it arises from the map of spaces $S^j \to \mathscr J(\R^i,\R^{i+j})$ that identifies $S^j$ with the fiber over the standard embedding $\R^i \to \R^{i+j}$. So $\Cyl_{i,j}$ is defined as the pushout
\begin{equation}\label{eq:expanded_cylinder}
\xymatrix @R=1.7em{
	F_{i+j} S^j \ar@{<->}[r]^-\cong \ar[d]_-{\lambda_{i,j}} & \{1\}_+ \sma F_{i+j} S^j \ar[d]_-{\lambda_{i,j}} \ar[r] \ar@{}[rd]|(.85)*\txt{\huge $\ulcorner$} & I_+  \sma F_{i+j} S^j \ar[d] & \{0\}_+ \sma F_{i+j} S^j \ar[l] \ar[ld]^-{k_{i,j}} \\
	F_i S^0 \ar@{<->}[r]^-\cong & \{1\}_+ \sma F_i S^0 \ar[r] & \Cyl_{i,j}
}
\end{equation}
where all the horizontal maps are all free $q$-cofibrations by \autoref{ex:smashing_spaces_is_left_Quillen}. This implies all the horizontal maps in the diagram below are free $q$-cofibrations, hence so is $k_{i,j}$.
\[ \xymatrix @R=1.7em{
	& \{0,1\}_+ \sma F_{i+j} S^j \ar[d] \ar[r] \ar@{}[rd]|(.85)*\txt{\huge $\ulcorner$} & I_+ \sma F_{i+j} S^j \ar[d] \\
	\{0\}_+ \sma F_{i+j} S^j \ar[r] \ar@/_2em/[rr]_-{k_{i,j}} & \{0\}_+ \sma F_{i+j} S^j \vee \{1\}_+ \sma F_i S^0  \ar[r] & \Cyl_{i,j}
} \]

\begin{proof}
	We verify the six conditions listed in \autoref{prop:construct_cofibrantly_generated_model_category}.
	\begin{enumerate}
		\item \textbf{$W$ is closed under 2-out-of-3 and retracts.} The stable equivalences are defined by a collection of functors $\R \pi_{n,b}(-)$, so this is automatic.
		
		\item \textbf{$I$ satisfies the simplified smallness condition.} A map from $F_k S^{n-1}_{+B}$ into a spectrum $X$ is determined by a map of retractive spaces $S^{n-1}_{+B} \to X_k$, which is determined by a map $S^{n-1} \to X_k$ of unbased spaces over $B$. If $X$ is an $I$-cell complex with skeleta $X^{(m)}$, then each map $X^{(m)} \to X^{(m+1)}$ is a free $q$-cofibration, therefore a level $h$-cofibration by \autoref{prop:free_implies_level}, therefore a level closed inclusion. Since the sphere is compact and the quotients $X^{(m)}_k/X^{(m-1)}_k$ are weak Hausdorff, $S^{n-1}$ factors through some skeleton $X^{(m)}_k$. This gives the desired factorization of $F_k S^{n-1}_{+B}$ through some $X^{(m)}$. In fact, the factorization is unique as soon as $m$ is large enough that it exists.
		
		\item \textbf{$J$ satisfies the simplified smallness condition.} Each map in $J$ is a $q$-cofibration by the discussion just before this proof. Therefore so they are also level closed inclusions. By the point above, any map from $F_k D^n_{+B}$ into a $J$-cell complex therefore factors through some finite stage. If instead we take a map from the pushout
		\[ \xymatrix @R=1.7em{
			F_{i+j} S^j \barsmash F_k S^{n-1}_{+B} \ar[r] \ar[d] & Cyl_{i,j} \barsmash F_k S^{n-1}_{+B} \\
			F_{i+j} S^j \barsmash F_k D^n_{+B}
		} \]
		into a $J$-cell complex, the map on $Cyl_{i,j} \barsmash F_k S^{n-1}_{+B}$ is determined by its restriction to each of the terms in \eqref{eq:expanded_cylinder}, smashed with $F_k S^{n-1}_{+B}$. Using \autoref{barsmash_free}, we rewrite each such term as a free spectrum on a compact retractive space. Therefore each term separately factors through some level $X^{(m)}$. Taking the maximum of the resulting values of $m$, these factorizations are unique, and therefore agree along the maps in \eqref{eq:expanded_cylinder}, so they give a factorization of all of $Cyl_{i,j} \barsmash F_k S^{n-1}_{+B}$ through some finite level of the colimit system. Repeating the same argument with the rest of the pushout just above, we conclude that the pushout also factors through some finite level $X^{(m)}$.
		
		\item \textbf{$J$-cell complexes are in $W$ $\cap$ $I$-cof.}\footnote{It's worth pointing out that the earlier proof of the $q$-model structure (see e.g. \cite{hu_duality}) was nearly complete, but had a gap in this step. As explained in \cite{ms}, this step relies in a critical way on the left-properness statement in \autoref{cor:spectra_left_proper}.} Again, we already know they are in $I$-cof. First we argue that each map in $J$ is a stable equivalence. For the first set of maps in $J$ this is clear because they are level equivalences. For the second set, we draw the following diagram in which the square is a pushout.
		\[ \xymatrix @R=1.7em{
			F_{i+j} S^j \barsmash S^{n-1}_{+B} \ar[r] \ar[d]  \ar@{}[rd]|(.85)*\txt{\huge $\ulcorner$} & Cyl_{i,j} \barsmash S^{n-1}_{+B} \ar[d] \ar[rd] \\
			F_{i+j} S^j \barsmash D^n_{+B} \ar[r] & Y \ar[r] & Cyl_{i,j} \barsmash D^n_{+B}
		} \]
		Since the left vertical map is a $q$-cofibration, to prove that the map from $Y$ into the final term is a stable equivalence, by \autoref{cor:spectra_left_proper} it suffices to argue that both horizontal maps of the form
		\[ F_{i+j} S^j \barsmash Z_{+B} \to Cyl_{i,j} \barsmash Z_{+B} \]
		are stable equivalences. In each case, the source and target are level $h$-cofibrant so we may apply $P$ to get
		\[ \xymatrix @R=1.7em{
			F_{i+j} [ S^j \barsmash (Z \times_B B^I)_{+B} ] \ar[r]^-\sim & F_{i+j} [ S^j \sma I_+ \barsmash (Z \times_B B^I)_{+B} ] \ar[r]  & Cyl_{i,j} \barsmash (Z \times_B B^I)_{+B} \\
			& F_{i+j} [ S^j \sma \{1\}_+ \barsmash (Z \times_B B^I)_{+B} ] \ar[r]^-\sim \ar[u] & F_{i} [ S^0 \barsmash (Z \times_B B^I)_{+B} ] \ar[u] \\
		} \]
		The first marked equivalence is a level equivalence. The second is a stable equivalence, because on each fiber it is the $\pi_*$-isomorphism $\lambda_{i,j}$ smashed with the well-based space $[(Z \times_B B^I)_b]_+$, see \autoref{ex:smash_with_cell_complex_is_derived}. Since the vertical maps are level $h$-cofibrations, all horizontal maps are therefore stable equivalences by \autoref{cor:spectra_left_proper}, finishing the argument.
		
		Now we know that each individual map in $J$ is a stable equivalence, and also a level $h$-cofibration (using the discussion before this proof and \autoref{prop:spectra_pushout_product}). Any coproduct of such maps is an equivalence by \autoref{prop:coproduct_colimit_stable_equivalences}, and any pushout is by \autoref{cor:spectra_left_proper}. Then a sequential colimit of such is an equivalence by \autoref{prop:coproduct_colimit_stable_equivalences} again.
		
		\item \textbf{$I$-inj $\subseteq$ $W$ $\cap$ $J$-inj.} The work here is to characterize the classes $I$-inj and $J$-inj more explicitly. Let $p\colon X \to Y$ be a map of spectra over $B$. The adjunction between free spectra and the forgetful functor to spaces implies that $p$ is $I$-injective iff each level $p_n\colon X_n \to Y_n$ is an acyclic $q$-fibration. Moreover, $p$ is $J$-injective iff
		\begin{itemize}
			\item each $p_n$ is just a $q$-fibration (the first set of maps in $J$),
			\item and the map of retractive spaces $\Hom_\square(k_{i,j},p)$ over $B$ is an acyclic $q$-fibration (by \autoref{lem:pullback_hom_adjunction}).
		\end{itemize}
		Since $k_{i,j}$ is always a $q$-cofibration, when $p$ is a level $q$-fibration the map $\Hom_\square(k_{i,j},p)$ is also a $q$-fibration. So we can rearrange the above necessary and sufficient condition for $p$ to be $J$-injective:
		\begin{itemize}
			\item each $p_n$ is a $q$-fibration,
			\item and $\Hom_\square(k_{i,j},p)$ is a weak equivalence.
		\end{itemize}
		It will be convenient to rearrange this one more time. Writing out the definition of $\Hom_\square(k_{i,j},p)$ as
		\begin{equation}\label{eq:horrible_map}
		\barmap_*(\Cyl_{i,j},X) \to \barmap_*(F_{i+j} S^j,X) \times_{\barmap_*(F_{i+j} S^j,Y)} \barmap_*(\Cyl_{i,j},Y),
		\end{equation}
		we see using the level model structure that when $p$ is a level $q$-fibration, the map
		\[ \barmap_*(F_{i+j} S^j,X) \to  \barmap_*(F_{i+j} S^j,Y) \]
		is a $q$-fibration of spaces, hence the pullback is a homotopy pullback. Therefore we can replace the terms of the form $\barmap_*(\Cyl_{i,j},Y)$ with the equivalent terms $\barmap_*(F_i S^0,Y)$\footnote{One might think we need to assume $Y$ is $q$-fibrant for this, but that is not necessary. The inclusion $F_i S^0 \to \Cyl_{i,j}$ is a homotopy equivalence of spectra, in the sense that it has an inverse up to homotopy. The functor $\barmap_*(-,Y)$ preserves this homotopy equivalence with no assumptions on $Y$.}, giving the weakly equivalent map
		\[ \barmap_*(F_{i} S^0,X) \to \barmap_*(F_{i+j} S^j ,X) \times_{\barmap_*(F_{i+j} S^j,Y)} \barmap_*(F_{i} S^0,Y). \]
		Using the universal property of free spectra, this map is homeomorphic to
		\[ X_i \to \Omega_B^j X_{i+j} \times_{\Omega_B^j Y_{i+j}} Y_i. \]
		We also verify that $X_i \to \Omega_B^j X_{i+j}$ is the map we expect by tracing through the composite $F_{i+j} S^j \to F_i S^0 \to X$ at spectrum level $i+j$. In summary, $p$ is $J$-injective iff
		\begin{itemize}
			\item each $p_n$ is a $q$-fibration,
			\item and the following square is a homotopy pullback, cf \cite[12.5.6]{ms}.
		\end{itemize}
		\begin{equation}\label{eq:fibration_of_spectra_means_this_is_a_pullback}
		\xymatrix @R=1.7em{
			X_i \ar[r] \ar[d]^-{p_i} & \Omega^j_B X_{i+j} \ar[d]^-{\Omega_B^j p_{i+j}} \\
			Y_i \ar[r] & \Omega^j_B Y_{i+j}
		}
		\end{equation}
		The functors $\Omega_B$ in the square above are not derived, but we can replace them with their derived functors without changing the condition (assuming the $p_n$ are fibrations). To prove this, we model $\R\Omega_B$ by picking any level equivalence from $p\colon X \to Y$ to another level $q$-fibration $p'\colon X' \to Y'$ in which $Y'$ is level $q$-fibrant. This gives the square below on the left, which is a homotopy pullback. Since its vertical maps are fibrations, their fibers are fibrant over $B$, hence when we take $\Omega^j_B$ they are still equivalent. Therefore the square on the right induces an equivalence on the strict fibers of the vertical maps. Since we already know the vertical maps are fibrations, this implies it is a homotopy pullback square.
		\[ \xymatrix @R=1.7em{
			X_{i+j} \ar@{->>}[d]^-{p_{i+j}} \ar[r]^-\sim & X_{i+j}' \ar@{->>}[d]^-{p_{i+j}'} && \Omega^j_B X_{i+j} \ar@{->>}[d]^-{\Omega_B^j p_{i+j}} \ar[r] & \Omega^j_B X_{i+j}' \ar@{->>}[d]^-{\Omega_B^j p_{i+j}'} \\
			Y_{i+j} \ar[r]^-\sim & Y_{i+j}' && \Omega^j_B Y_{i+j} \ar[r] & \Omega^j_B Y_{i+j}'.
		} \]
		Therefore square \eqref{eq:fibration_of_spectra_means_this_is_a_pullback} is a homotopy pullback with strict $\Omega^j_B$ iff it is a homotopy pullback with derived $\Omega^j_B$.
		
		Now we may finish this step of the proof. Suppose $p$ is $I$-injective. Then it is an acyclic $q$-fibration on every level, hence a level equivalence, hence a stable equivalence. It is certainly at least a $q$-fibration on every level. Furthermore the version of \eqref{eq:fibration_of_spectra_means_this_is_a_pullback} with derived $\Omega^j_B$ has both verticals weak equivalences, hence it is a homotopy pullback square. So $p \in W \cap J$-inj.
		
		\item \textbf{$W$ $\cap$ $J$-inj $\subseteq$ $I$-inj.} If $p \in W \cap J$-inj, then by the above criterion, each $p_i$ is a $q$-fibration and the squares \eqref{eq:fibration_of_spectra_means_this_is_a_pullback} are homotopy pullback squares. We just need to prove that each $p_i$ is also a weak equivalence, using the fact that $p$ is a stable equivalence.
		
		Examine the strict fiber spectrum $F$ of the map $p$, which at each spectrum level is the pullback $B \times_{Y_n} X_n$. Because $p$ is a level $q$-fibration, this is equivalent to the homotopy fiber spectrum. By \autoref{lem:LES2}, $F$ is contractible in the sense that its derived homotopy groups vanish. Since $p$ is a level $q$-fibration, $F$ is level $q$-fibrant, so the homotopy groups of its fiber spectra $F_b$ are also zero. The following diagram of pullbacks is helpful for following the rest of the proof.
		\[ \xymatrix @R=1.7em{
			(F_b)_i \ar@{->>}[d] \ar[r] & F_i \ar@{->>}[d] \ar[r] & X_i \ar@{->>}[d]^-{p_i} \\
			\{b\} \ar[r] & B \ar[r] & Y_i
		} \]
		
		In the diagram below, the bottom square is a homotopy pullback with vertical maps fibrations, so the top vertical map is an equivalence. Hence $F_i$ is a weak $\Omega$-spectrum over $B$ with fibrant levels.
		\[ \xymatrix @R=1.7em{
			F_i \ar[d] \ar[r]^-\sim & \Omega^j_B F_{i+j} \ar[d] \\
			X_i \ar[r] \ar[d]^-{p_i} & \Omega^j_B X_{i+j} \ar[d]^-{\Omega_B^j p_{i+j}} \\
			Y_i \ar[r] & \Omega^j_B Y_{i+j}
		} \]
		In conclusion, each fiber spectrum $F_b$ is a weak $\Omega$-spectrum with vanishing homotopy groups. It is standard that this implies the levels $(F_b)_i$ are weakly contractible; for instance the colimit systems below are isomorphic, hence the top one consists of isomorphisms, and since it has zero colimit, the terms are all zero.
		\[ \xymatrix @R=1.7em{
			\pi_n((F_b)_i) \ar@{=}[d] \ar[r] & \pi_{n+1}((F_b)_{i+1}) \ar@{<->}[d]^-\cong \ar[r] & \pi_{n+2}((F_b)_{i+2}) \ar@{<->}[d]^-\cong \ar[r] & \ldots & \pi_{n-i}(F_b) \cong 0 \\
			\pi_n((F_b)_i) \ar[r]^-\cong & \pi_{n}(\Omega (F_b)_{i+1}) \ar[r]^-\cong & \pi_{n}(\Omega^2 (F_b)_{i+2}) \ar[r]^-\cong & \ldots
		} \]
		Since $p_i\colon X_i \to Y_i$ is a $q$-fibration and its fiber $(F_b)_i$ over $i(b) \in Y_i$ is contractible, it is an equivalence on every component of $Y_i$ containing $i(B)$. Therefore $\R\Omega_B p_i$ is an equivalence on every component. By condition \eqref{eq:fibration_of_spectra_means_this_is_a_pullback}, this implies $p_{i-1}$ is a weak equivalence, for every value of $i$.
	\end{enumerate}
	Left and right properness follow immediately from \autoref{cor:spectra_left_proper} and \autoref{cor:spectra_right_proper}.
\end{proof}

The following was established inside the previous proof.
\begin{cor}\label{cor:stable_fibrations}
	A map $p\colon X \to Y$ of spectra over $B$ is a fibration in the stable model structure iff it is a Serre fibration ($q$-fibration) on each level and each square \eqref{eq:fibration_of_spectra_means_this_is_a_pullback} is a homotopy pullback square.
\end{cor}

We say $X$ is \textbf{stably fibrant}\index{stably fibrant} when it is fibrant in the stable model structure. By the above, this is equivalent to the condition that $X$ is level $q$-fibrant and the adjoint bonding maps $X_n \to \Omega_B X_{n+1}$ are equivalences. The stable model structure gives us a stable fibrant replacement functor $R^{st}$.\index{fibrant replacement!$R^{st}$ for spectra} On spectra that are at least freely $i$-cofibrant, we could get a different stable fibrant replacement functor by generalizing the classical fibrant replacement operation from \autoref{ex:equivalent_thom_spectra}.

\begin{lem}
	The suspension spectrum functor
	\[ \Sigma^\infty\colon \mc R(B) \to \Psp(B) \textup{ or } \Osp(B) \]
	is left Quillen. More generally the same is true for the free spectrum functor $F_n$.
\end{lem}

\begin{proof}
	We already know it's a left adjoint, so we just check it preserves generating cofibrations and acyclic cofibrations.
\end{proof}

\begin{prop}\label{prop:pre_equiv_to_orth}
	The forgetful functor $\Osp(B) \to \Psp(B)$ is a right Quillen equivalence.
\end{prop}

\begin{proof}
	The forgetful functor $\mathbb U$ is a right adjoint by the abstract theory of diagrams: restricting a diagram of spaces from $\mathscr J$ to $\mathscr N$ has as its left adjoint the left Kan extension from $\mathscr N$ to $\mathscr J$. More concretely, the left adjoint $\mathbb P$ sends free sequential spectra to free orthogonal spectra, and preserves colimits, from which we can figure out what it does to an arbitrary sequential spectrum. See \cite{mmss} for more details.
	
	Since \autoref{cor:stable_fibrations} applies equally well to $\Psp(B)$ and $\Osp(B)$, the forgetful functor $\mathbb U$ preserves fibrations. It also preserves all weak equivalences using \autoref{stable_equivalences}. It is therefore a Quillen right adjoint.
	
	It remains to show it gives an equivalence on the homotopy category. This will follow if the unit map $X \to \mathbb U\mathbb PX$ is stable equivalence when $X$ is an $I$-cell complex. Since $\mathbb P$ is left Quillen it preserves preserves stable equivalences of $q$-cofibrant spectra. Using this fact and inducting up the cell complex structure, we reduce to the case where $X$ is of the form $F_k Z_{+B}$. Then we reduce this to the non-parametrized version of the same statement \cite[10.3]{mmss} by applying $P$, commuting it with the free spectrum construction, and then restricting to a fiber over $B$.
\end{proof}

As a first application of the stable model structure, we verify that homotopy cofiber sequences of spectra over $B$ give long exact sequences of maps into a third spectrum $Z$.
\begin{prop}
	If $f\colon X \to Y$ is a map of spectra over $B$ and $Z$ is a spectrum over $B$ then there is a long exact sequence
	\[ \xymatrix{
		\ldots \ar[r] & [\Sigma_B X,Z] \ar[r] & [\L C_B f,Z] \ar[r] & [Y,Z] \ar[r] & [X,Z] \ar[r] & \ldots
	}. \]
	where $[-,-]$ denotes maps in the homotopy category (stable equivalences inverted). It is natural in both $f$ and $Z$.
\end{prop}
Similar statements apply to double mapping cylinders, homotopy coequalizers, mapping telescopes, together with homotopy fiber sequences and the duals of these constructions.

\begin{proof}
	Applying cofibrant replacement to $X \to Y$ and stable fibrant replacement to $Z$, the above maps in the homotopy category can be described as point-set maps up to homotopy, and $\L C_B f$ can be described as just $C_B f$. Then the usual argument applies.
\end{proof}

\beforesubsection
\subsection{Derived base change and external smash product functors}\aftersubsection

Every result in this section applies to orthogonal spectra. The ones without $\barsmash$, or that only smash a spectrum with a space, apply also to sequential spectra.

We now have all the ingredients we need to construct pullback, pushforward, and smash product functors that are derived with respect to the stable equivalences.

\begin{lem}\label{lem:pullback_stable_Quillen}
	$(f_! \adj f^*)$ is a Quillen pair for the stable model structure.
	If $f\colon A \to B$ is a weak equivalence then it is a Quillen equivalence.
\end{lem}

\begin{proof}
	From \autoref{cor:stable_fibrations} we see that $f^*$ preserves fibrations and acyclic fibrations, hence it is right Quillen. When $f$ is a weak equivalence, we need to check that the derived unit and counit maps are stable equivalences. By \autoref{cor:level_derived_equals_stable_derived} and \autoref{ex:derived_adjunction_spaces}, there is a model for the derived functors in which these maps are level equivalences, so we are done.
\end{proof}

\begin{lem}\label{lem:other_stable_adjunction}
	If $f$ is a fiber bundle with base and fiber both cell complexes, $f^*$ is a Quillen left adjoint. If in addition $f$ is a weak equivalence, $f^*$ is a left Quillen equivalence.
\end{lem}

\begin{proof}
	We construct the right adjoint $f_*$ in spectra by applying the space-level $f_*$ on each level, and using the canonical commutation with $\Omega_B(-)$. Since the fiber of $f$ is a cell complex, $f^*$ sends generating cofibrations to cofibrations, and similarly for generating acyclic cofibrations, therefore $f^*$ is left Quillen. As in \autoref{ex:other_adjunction_spaces}, $f^*$ is homotopical in this case and so its left-derived and right-derived functors are equivalent, so the Quillen equivalence part follows from \autoref{lem:pullback_stable_Quillen}.
\end{proof}

\begin{lem}\label{lem:stable_smash_is_left_quillen}
	Each of the external smash product functors
	\begin{align*}
	\mc R(A) \times \Psp(B) &\to \Psp(A \times B), \\
	\qquad \mc R(A) \times \Osp(B) &\to \Osp(A \times B), \\
	\qquad \Osp(A) \times \Osp(B) &\to \Osp(A \times B)
	\end{align*}
	is a left Quillen bifunctor with respect to the stable model structure on spectra (and the Quillen model structure on spaces).
\end{lem}

\begin{proof}
	For the last of these smash product operations, we use \autoref{barsmash_free} and the fact that $\barsmash$ commutes with pushouts to define isomorphisms between each map in $I \square I$ and some map of $I$, and each map in $I \square J$ and some map of $J$. The proof for the other two smash product functors is easier.
\end{proof}

\begin{rmk}
	Taking $A = *$, this implies that $\Psp(B)$ and $\Osp(B)$ are simplicial model categories.
\end{rmk}

As a formal consequence of the above (and Ken Brown's lemma),
	\begin{itemize}
		\item $f^*$ preserves all stable equivalences between level $q$-fibrant spectra.
		\item $f_!$ preserves all stable equivalences between freely $q$-cofibrant spectra.
		\item $\barsmash$ preserves all stable equivalences between freely $q$-cofibrant spectra (and if smashing with a space, $q$-cofibrant spaces).
	\end{itemize}
In each of these cases, we can weaken the assumptions further. This will be important for commuting the derived left and right adjoints past each other.

\begin{prop}\label{prop:stably_derived}\hfill
	\vspace{-1em}
	
	\begin{itemize}
		\item $f^*$ preserves all stable equivalences between level-quasifibrant spectra.
		\item $f_!$ preserves all stable equivalences between level $h$-cofibrant spectra.
		\item $\barsmash$ preserves all stable equivalences between freely $h$-cofibrant spectra (and if smashing with a space, $h$-cofibrant spaces).
	\end{itemize}
\end{prop}
	
\begin{proof}
	The first case is immediate from \autoref{stable_equivalences}. For the second and third cases, by \autoref{lem:spectrum_f_preserves} and \autoref{prop:spectra_pushout_product} respectively we already know the functor preserves level equivalences under this assumption. Therefore we may use the level model structure to replace the inputs by level-equivalent ones that are freely $q$-cofibrant, and we are done.
\end{proof}

\begin{cor}\label{cor:level_derived_equals_stable_derived}
	The left-derived functor $\L f_!$ with respect to the level equivalences, is also the left-derived functor with respect to the stable equivalences. The same applies to $\R f^*$ and $\barsmash^{\L}$.
\end{cor}
\begin{proof}
	Each one arises as a deformation onto a subcategory on which the original functor $f_!$, $f^*$, or $\barsmash$ preserves stable equivalences.
\end{proof}

\begin{rmk}
	The sheafy pushforward $f_*$ is the one exception to the above rule: deriving it with respect to the level equivalences does not produce the same result as deriving it with respect to the stable equivalences.\footnote{Although it appears they agree if the fiber of $f$ is a \emph{finite} cell complex.} In general the input of $f_*$ needs to be stably fibrant.
\end{rmk}

\begin{ex}(Homotopy colimits and limits)
	We can use \autoref{prop:stably_derived} to imitate the definition of homotopy colimits and limits in either $\Psp(B)$ or $\Osp(B)$, as a model for the left-derived colimit or right-derived limit. Without going into too much detail, we build these as in \autoref{space_cofiber_fiber}, requiring for colimits that the spectra in the diagram are level $h$-cofibrant, and for limits that the spectra in the diagram are stably fibrant (or just level fibrant if the homotopy limit is finite). Both of these commute with derived pullback, and the homotopy colimit commutes with derived pushforward.
\end{ex}
%

Using \autoref{prop:stably_derived}, we can now pass from the point-set Beck-Chevalley isomorphism \autoref{prop:beck_chevalley_spaces} and commutation isomorphisms \autoref{prop:spectra_external_smash_and_base_change} to isomorphisms of derived functors on the stable homotopy category.

\begin{thm}\label{thm:spectra_SMBF}
	The homotopy category of parametrized orthogonal spectra $\ho\Osp$ forms a symmetric monoidal bifibration over $\cat{Top}$, with Beck-Chevalley squares the homotopy pullback squares of spaces.
\end{thm}\index{SMBF!of parametrized spectra $\ho\Osp$}

\begin{proof}
	We check the seven conditions from \autoref{prop:deform_SMBF}.
	\begin{enumerate}
		\item $\barsmash$ preserves cofibrant objects and weak equivalences between them by \autoref{prop:spectra_pushout_product} and \autoref{prop:stably_derived}.
		\item $F_0 S^0$ is an $I$-cell complex, therefore cofibrant.
		\item The pushforwards $f_!$ are coherently left-deformable, using for instance the level $h$-cofibrant spectra, by \autoref{prop:stably_derived}.
		\item The pullbacks $f^*$ are coherently right-deformable using level $q$-fibrant spectra.
		\item For a pullback square along an $h$-fibration, the Beck-Chevalley maps are coherently deformable using \autoref{lem:spectrum_f_preserves} and \autoref{prop:stably_derived}, and the freely $f$-cofibrant, level $h$-fibrant spectra as inputs.
		\item $\barsmash$ and all of the pushforwards $f_!$ are coherently deformable, using for instance the freely $f$-cofibrant spectra, by \autoref{prop:stably_derived}.
		\item $\barsmash$ and all of the pullbacks $f^*$ are coherently deformable using the freely $f$-cofibrant, level $h$-fibrant spectra. This uses \autoref{lem:spectrum_f_preserves}, \autoref{prop:stably_derived}, and especially \autoref{prop:spectra_pushout_product}(3). For (CGWH), note that although technically $\Osp$ is not an SMBF because $\barsmash$ does not always preserve cartesian arrows, it preserves them when the target is in this subcategory, so the proof in \autoref{prop:deform_SMBF} still establishes that $\ho \Osp$ is an SMBF.
	\end{enumerate}
	As in \autoref{expand_beck_chevalley}, since $\L f_!$ is an equivalence when $f$ is a weak equivalence, and every homotopy pullback square is equivalent to a strict pullback along an $h$-fibration, we therefore get Beck-Chevalley for every homotopy pullback square of spaces.
\end{proof}

The above theorem tells us that the three operations $\barsmash^\L$, $\L f_!$, and $\R f^*$ on the homotopy category can be interchanged by the three canonical commutation isomorphisms
\begin{align*}
\L f_!X \barsmash^{\L} \L g_!Y &\simeq \L(f \times g)_!(X \barsmash^{\L} Y), \\
\R f^*X \barsmash^{\L} \R g^*Y &\simeq \R(f \times g)^*(X \barsmash^{\L} Y), \\
\L g_! \circ \R f^* &\simeq \R p^* \circ \L q_!.
\end{align*}
Each of these isomorphisms arises from the SMBF structure on $\ho\Osp$. Alternatively, by the proofs of \autoref{prop:deform_bifibration} and \autoref{prop:deform_SMBF}, the same isomorphisms arise by deforming the corresponding isomorphism on the point-set level using \autoref{prop:passing_natural_trans_to_derived_functors}. We can therefore also think of these as canonical weak equivalences of homotopy functors.

The same argument shows that the category of sequential spectra $\ho \Psp$ is a bifibration with the same Beck-Chevalley squares. We also get isomorphisms as above that commute pullback and pushforward with the operation that smashes a spectrum with a space.

\begin{ex}
	A somewhat surprising corollary is that suspension spectrum commutes with pullback,
	\[ \Sigma^\infty_A f^*X \simeq f^*\Sigma^\infty_B X. \]
	Tracing through the above argument, we actually have this isomorphism on the point-set level (and it is unique). Then since these two lists of functors are coherently deformable this passes to a canonical equivalence of derived functors,
	\[ (\L \Sigma^\infty_A)(\R f^*)X \simeq (\R f^*)(\L\Sigma^\infty_B) X. \]
	Of course $\Sigma^\infty$ commutes with pushforward as well, both on the nose and in the homotopy category, but that is less surprising. These statements can be formalized into the larger statement that $\L \Sigma^\infty$ defines a map of symmetric monoidal bifibrations $\ho\mc R \to \ho\Osp$.
\end{ex}

If we allow ourselves to change the category $\ho \Osp$ up to equivalence, we could simplify the construction of its SMBF structure by restricting attention to the freely $f$-cofibrant spectra $\Osp^{c} \subseteq \Osp$. The induced map $\ho \Osp^{c} \to \ho \Osp$ is a fiberwise equivalence of categories over $\cat{Top}$, and by the proof of \autoref{thm:spectra_SMBF}, $\ho\Osp^{c}$ is a symmetric monoidal bifibration as well.

All three of the derived functors are easier to describe in $\ho\Osp^{c}$. The left-derived pushforward $\L f_!$ and smash product $\barsmash^{\L}$ are modeled by the actual pushforward $f_!$ and smash product $\barsmash$, respectively. And although $f^*$ still needs to be derived, here we can use $P$ to derive it, $\R f^* = f^* P$. This is especially convenient for proving things because $P$ commutes with many constructions, and has an explicit geometric description, as opposed to the functors $R^{lv}$ and $R^{st}$ that are produced by the small-object argument.\footnote{We could also pass further to freely $f$-cofibrant, level $h$-fibrant spectra $\ho \Osp^{cf}$. This receives an equivalent SMBF structure where all three of the functors $f_!$, $f^*$, and $\barsmash$ are equivalent to their derived functors. However, since the $h$-fibrant spaces are not preserved by $f_!$, the pushforward functor cannot be modeled directly by $f_!$. We model it by $P f_!$ instead. Therefore this alternate route doesn't completely eliminate the need for fibrant replacement, unless we have no need for pushforwards, or only have to push forward along an $h$-fibration.}
%

\beforesubsection
\subsection{Derived smash products over $B$}\aftersubsection

The homotopy category of orthogonal spectra over $B$, $\ho\Osp(B)$, has a derived internal smash product
\[ \sma_B^{\M} := (\R\Delta_B^*)(\barsmash^{\L}). \]
We emphasize that $\sma_B$ is not left- or right-deformable, merely a composite of coherently deformable functors. So, its derived functor $\sma_B^{\M}$ is only unique if we stipulate that it must arise as a composite of the derived functors for $\barsmash$ and $\Delta_B^*$.

This extends to a symmetric monoidal structure on the homotopy category, with unit $\Sph_B := \Sph \times B \simeq \R r_B^* \Sph$. The most satisfactory description of this in the literature is \cite[12.8]{shulman_framed_monoidal}, though that point of view may be overly abstract for some applications. So consider the following three more concrete approaches.
\begin{itemize}
	\item[(1)] The functors obtained by iterating $\sma_B = \Delta_B^* \circ \barsmash$ and $r_B^*$ are coherently deformable, using \autoref{prop:stably_derived}. Therefore the recipe of \autoref{prop:passing_natural_trans_to_derived_functors} passes the symmetric monoidal structure of $\Osp(B)$ up to the homotopy category.
	\item[(2)] Restrict to the equivalent subcategory $\ho\Osp^{c}(B)$ of freely $f$-cofibrant spectra. By \autoref{prop:stably_derived}, $\sma_B$ preserves this subcategory and is right-deformable using $P$. Therefore the recipe of \autoref{prop:passing_natural_trans_to_derived_functors} passes the symmetric monoidal structure of $\Osp(B)^c$ up to the homotopy category.
	\item[(3)] Restrict to the equivalent subcategory $\ho\Osp^{cf}(B)$ of freely $f$-cofibrant, level $h$-fibrant spectra. By \autoref{prop:stably_derived}, $\sma_B$ is homotopical here and sends this subcategory into itself. Therefore the symmetric monoidal structure on $\Osp^{cf}(B)$ that arose from the rigidity theorem, passes immediately to the homotopy category.
\end{itemize}

By contrast, the more general approach of \cite[12.8]{shulman_framed_monoidal} proves that a symmetric monoidal structure on a Grothendieck fibration can always be pulled back along any map of cartesian monoidal categories. Applying this to the homotopy category $\ho \Osp$ gives a symmetric monoidal structure on the fiber homotopy category $\ho\Osp(B)$. We call this recipe (4).

Let's unpack this and then prove it. Suppose that $\Phi\colon \sC \to \bS$ is any symmetric monoidal bifibration with product $\boxtimes$.\footnote{This works for symmetric monoidal fibrations, or just monoidal fibrations, as well.} Let $F\colon \bT \to \bS$ be any functor from another category $\bT$ that has all finite products. The functor $F$ does not have to preserve products!

Form the pullback category $F^*\sC$. Its objects are pairs $(X,a)$ with $X \in \ob\sC$, $a \in \ob\bT$, and $F(a) = \Phi(X)$. Similarly, the morphisms are morphisms in $\sC$ and $\bT$ lying over the same morphism in $\bS$.

Given any two pairs $(X,a)$ and $(X',b)$ in the pullback category $F^*\sC$, we can take the product $X \boxtimes X'$ in $\sC$, and then choose a cartesian arrow in $\sC$ of the form
\[ \xymatrix @R=1.7em{ X \otimes X' \ar@{~>}[d] \ar[r] & X \boxtimes X' \ar@{~>}[d] \\
	F(a \times b) \ar[r] & F(a) \times F(b).
} \]
This defines a new object $X \otimes X' \in \sC$. Similarly, we can define an object $I_\otimes$ using a cartesian arrow
\[ \xymatrix @R=1.7em{ I_\otimes \ar@{~>}[d] \ar[r] & I_\boxtimes \ar@{~>}[d] \\
	F({*}_\bT) \ar[r] & {*}_\bS
} \]
where $*$ denotes a terminal object, or empty product. These constructions define a tensor product and unit on the pullback category $F^*\sC$: the product of $(X,a)$ and $(X',b)$ is the pair $(X \otimes X',a \times b)$, and the unit is $(I_\otimes,*_\bT)$.\footnote{While this notation is convenient, we have to remember that $X \otimes X'$ is not a function of just $X$ and $X'$ -- it depends on $a$ and $b$ as well.} When $\sC = \ho\Osp$ is the homotopy category of parametrized spectra, and $F\colon * \to \cat{Top}$ is the functor picking out the object $B$, the tensor product so constructed on $F^*\ho\Osp \cong \ho\Osp(B)$ is the derived internal smash product $\sma_B^{\M}$.

We complete this to a symmetric monoidal structure as follows. Given three pairs $(X,a)$, $(X',b)$, $(X'',c)$, we tensor the first two to get a cartesian arrow as above, then tensor with the third pair, giving two cartesian arrows
\[ \xymatrix @R=1.7em{
	(X \otimes X') \otimes X'' \ar@{~>}[d] \ar[r] & (X \otimes X') \boxtimes X'' \ar@{~>}[d] \ar[r] & (X \boxtimes X') \boxtimes X'' \ar@{~>}[d] \\
	F((a \times b) \times c) \ar[r] & F(a \times b) \times F(c) \ar[r] & (F(a) \times F(b)) \times F(c).
} \]
Repeating the other direction gives
\[ \xymatrix @R=1.7em{
	X \otimes (X' \otimes X'') \ar@{~>}[d] \ar[r] & X \boxtimes (X' \otimes X'') \ar@{~>}[d] \ar[r] & X \boxtimes (X' \boxtimes X'') \ar@{~>}[d] \\
	F(a \times (b \times c)) \ar[r] & F(a) \times F(b \times c) \ar[r] & F(a) \times (F(b) \times F(c)).
} \]
By the universal property of cartesian arrows, there is a unique map
\[ (X \otimes X') \otimes X'' \to X \otimes (X' \otimes X'') \]
lying over the associator $(X \boxtimes X') \boxtimes X'' \to X \boxtimes (X' \boxtimes X'')$ in $\sC$ and the associator $(a \times b) \times c \to a \times (b \times c)$ in $\bT$. By the universal property again, this map has an inverse and is therefore an isomorphism. This gives an associator isomorphism $\alpha$ for the product $\otimes$ in the pullback category $F^*\sC$. We define the remaining isomorphisms $\lambda, \rho, \gamma$ by an analogous trick. To prove these are coherent, we first observe that a map such as
\[ ((X \otimes X') \otimes X'') \otimes X''' \overset{\alpha \otimes \id}\to (X \otimes (X' \otimes X'')) \otimes X''' \]
that tensors one of the isomorphisms $\alpha, \lambda, \rho$ or $\gamma$ with an identity map, again satisfies a universal property similar to the one we used to define $\alpha$. Then the coherences follow from this universal property and the coherence for the product $\boxtimes$ in $\sC$ and $\times$ in $\bT$.

\begin{lem}\label{lem:two_smashes_over_b_isomorphic}
	The four recipes discussed thus far give canonically isomorphic symmetric monoidal structures on $\ho\Osp(B)$.
\end{lem}

\begin{proof}
	An isomorphism of symmetric monoidal structures is determined by what it does on an equivalent subcategory, so we will restrict to $\ho\Osp^{cf}(B)$, the freely $f$-cofibrant level $h$-fibrant spectra. The first, second, and fourth recipes use $\sma_B^\M$ and $\R r_B^* \Sph$ here, while the third recipe uses $\sma_B$ and $r_B^* \Sph$. These admit canonical isomorphisms because of the uniqueness of right-derived and left-derived functors. The work is to check that the associator, unitor, and symmetry isomorphism in the three recipes commute along these isomorphisms. For the first, second, and third recipes, this is true just by examination of \autoref{prop:passing_natural_trans_to_derived_functors}.
	
	To show that the fourth recipe agrees with the third, we first remove the $\L$ decorations in the fourth recipe. The canonical isomorphism $\barsmash^\L \cong \barsmash$ gives an isomorphism of symmetric monoidal structures on $\ho\Osp^{cf}$, because that's how the structure using $\barsmash^\L$ was defined. Therefore we can apply the fourth recipe to either one, and get canonically isomorphic symmetric monoidal structures on the fiber category $\ho\Osp^{cf}(B)$. In effect, this means that if we remove the $\L$ decorations in the fourth recipe, the resulting associator (and unitor and symmetry isomorphism) is also obtained by applying the fourth recipe to $\ho\Osp^{cf}$ and $\barsmash$.
	
	Next we remove the $\R$ decorations. Because every spectrum is level fibrant, when we carry out the fourth recipe, we can select the cartesian arrow for $f^*$ instead of $\R f^*$ in every step. This gives a new associator (and unitor and symmetry isomorphism) that agrees with the earlier one along the canonical isomorphisms $\R f^* \cong f^*$.
	
	So we have shown that along the canonical isomorphism $\sma_B^\M \cong \sma_B$, the associator in the fourth recipe agrees with an associator we get by applying the fourth recipe to $\ho\Osp^{cf}$, using $\barsmash$ and strict pullbacks. We get the same associator by applying the fourth recipe to the point-set category $\Osp^{cf}$, then passing to the homotopy category. This point-set associator is natural in the three inputs, and extends to a natural isomorphism on all of $\Osp$ (between functors that don't pass to the homotopy category). Hence by the rigidity theorem \autoref{thm:spectra_rigidity}, it must coincide with the point-set associator isomorphism we used in the third recipe. This finishes the proof that the associators from the third and fourth recipes agree along $\sma_B^\M \cong \sma_B$. The argument for the unitor and symmetry isomorphism proceeds the same way.
\end{proof}

The third recipe is the most concrete, because we can just work in the point-set category and draw conclusions in the homotopy category. The only downside is that we have to insist our spectra are freely $f$-cofibrant and level $h$-fibrant. Fortunately, this often happens in applications anyway. For instance, Thom spectra $F_n \Th_B(V)$ satisfy these assumptions, so their geometric behavior passes to the same behavior on the homotopy category. For instance if $\xi = V - n$ and $\omega = W - m$ are virtual bundles over $B$, we get using \autoref{ex:internal_smash_thom_spectra} the point-set isomorphism
\[ \Th_B(\xi) \sma_B \Th_B(\omega) := F_n \Th_B(V) \sma_B F_m \Th_B(W) \cong F_{n+m} \Th_B(V \times_B W) = \Th_B(\xi \oplus \omega). \]
Because everything is in $\Osp^{cf}(B)$, this implies the same isomorphism on the homotopy category.

In \cite{shulman_framed_monoidal}, this structure is extended to that of an indexed symmetric monoidal category. We will do just one part of this, the fact that the pullbacks $\R f^*$ are strong symmetric monoidal.

For $f\colon A \to B$, the strict pullback functor $f^*\colon \Osp(B)\to\Osp(A)$ has a unique symmetric monoidal structure, by \autoref{thm:spectra_rigidity}. Each of the above four recipes suggests a way of passing this to a symmetric monoidal structure on the derived pullback $\R f^*$: coherent deformability, restricting to the freely $f$-cofibrant or the freely $f$-cofibrant and level $h$-fibrant spectra, or using the universal property of cartesian arrows in the fibration $\ho\Osp$.
\begin{prop}\label{prop:unique_sm_structure_on_pullback}
	These give identical symmetric monoidal structures on $\R f^*\colon \ho\Osp(B) \to \ho\Osp(A)$.
\end{prop}

\begin{proof}
	We just have to show they give the same maps in the homotopy category
	\[ \R f^*(X) \sma_B^{\M} \R f^*(Y) \sim \R f^*(X \sma_B^{\M} Y), \qquad \R f^*(r_B^* \Sph) \sim r_A^* \Sph. \]
	Technically, the third recipe produces such maps without the $\M$ and $\R$ decorations, so we use the canonical isomorphisms to compare those maps to the maps furnished by the other three recipes. As before, it suffices to restrict to freely $f$-cofibrant, level $h$-fibrant spectra, and then the first and second recipes give the same map as the third, essentially by their definition.
	
	The proof that the third and fourth recipes produce the same result also proceeds as before. In the fourth recipe, we first remove the $\L$ from $\barsmash$, then remove the $\R$ from the pullbacks. The resulting isomorphisms arise by the universal property of cartesian arrows, but they also arise from the point-set level and are natural, hence they agree with the point-set isomorphism that we used as input for the third recipe. This verifies that the third and fourth recipes agree.
\end{proof}

In a somewhat different direction, we can extend the symmetric monoidal structure described in \autoref{lem:two_smashes_over_b_isomorphic} to the larger category $\Osp_{(B)}$ of all spectra over all spaces over $B$. This might seem like a silly level of generality, but we will actually need it for the fiberwise refinement of the Reidemeister trace (\autoref{ex:fiberwise_reidemeister_trace}).

We define $\Osp_{(B)}$ as the pullback of categories
\[ \xymatrix @R=1.7em{
	\Osp_{(B)} \ar[d] \ar[r] & \Osp \ar[d] \\
	\cat{Fib}_B \ar[r] & \cat{Top}
} \]
where $\cat{Fib}_B$ is the category of $h$-fibrations $p\colon E \to B$, mapping to $\cat{Top}$ by retaining the total space and forgetting the fibration. So an object of $\Osp_{(B)}$ is a pair $(p\colon E \to B, X \in \Osp(E))$. The morphisms are computed by forgetting $p$ and taking morphisms in $\Osp$.

Following the abstract (fourth) recipe we gave earlier in this section, we can give the point-set category $\Osp_{(B)}$ a symmetric monoidal structure. For each pair of spectra $X \in \Osp(D)$ and $Y \in \Osp(E)$, we define the \textbf{external smash product rel $B$}\index{external!smash rel $B$ $X \barsma{B} Y$} as
\[ X \barsma{B} Y := \Delta_{D,E}^*(X \barsmash Y), \qquad \Delta_{D,E}\colon D \times_B E \to D \times E. \]
This is a direct generalization of the internal smash product $\sma_B$.

Continuing to follow the fourth recipe, this extends to a symmetric monoidal structure in which the unit is $\Sph \times B \in \Osp(B)$. In fact, the functors obtained by iterating $\barsma{B}$ are rigid by \autoref{thm:spectra_rigidity}, so there is only one choice for the associator, unitor, and symmetry isomorphisms.

Next we pass to the homotopy category $\ho\Osp_{(B)}$ by inverting the stable equivalences. The result is also the pullback
\[ \xymatrix @R=1.7em{
	\ho \Osp_{(B)} \ar[d] \ar[r] & \ho \Osp \ar[d] \\
	\cat{Fib}_B \ar[r] & \cat{Top}
} \]
by comparing fiber categories and using \autoref{prop:deform_bifibration}.

Then we give the homotopy category $\ho\Osp_{(B)}$ a symmetric monoidal structure by following any of the four recipes covered in \autoref{lem:two_smashes_over_b_isomorphic}: coherent deformability of $\barsma{B}$ and its iterates on the whole category, or on the subcategory of freely $f$-cofibrant spectra, passing to freely $f$-cofibrant level $h$-fibrant spectra where $\barsma{B}$ is homotopical, or formally pulling back the symmetric monoidal structure with product $\barsmash^{\L}$ on $\ho\Osp$. The proof of \autoref{lem:two_smashes_over_b_isomorphic} applies verbatim and tells us that these symmetric monoidal structures are canonically isomorphic.

We conclude with a proof that this makes $\ho\Osp_{(B)}$ into a symmetric monoidal bifibration over $\cat{Fib}_B$.

\begin{prop}
	Suppose $F\colon \bT \to \bS$ is a functor of cartesian monoidal categories, and that $\bS$ and $\bT$ are endowed with a class of Beck-Chevalley squares, preserved by $F$. Then for any SMBF $\mc A$ over $\bS$, the symmetric monoidal structure discussed in this section makes $F^*\mc A$ into an SMBF, provided that for any two maps $A \to A'$, $B \to B'$ in $\bT$ the square
	\[ \xymatrix @R=1.7em{
		F(A \times B) \ar[d] \ar[r] & F(A) \times F(B) \ar[d] \\
		F(A' \times B') \ar[r] & F(A') \times F(B')
	} \]
	is Beck-Chevalley in $\bS$.
\end{prop}

\begin{cor}\label{prop:SMBF_over_b}
	The above makes $\ho\Osp_{(B)}$ into a symmetric monoidal bifibration.
\end{cor}

\begin{proof}	  
	We encounter no difficulty in checking that $F^*\mc A$ is a bifibration with (co)cartesian arrows those arrows that are (co)cartesian in $\mc A$. So we just have to check that the tensor product preserves (co)cartesian arrows.
	
	For the cocartesian arrows, it suffices to check that in the square of objects in $\mc A$ depicted on the left, lying over the square in $\bS$ on the right, the left vertical arrow is cocartesian.
	\[ \xymatrix{
		X \otimes Y \ar@{-->}[d] \ar[r] & X \boxtimes Y \ar[d] \\
		\phi_!X \otimes \gamma_!Y \ar[r] & \phi_!X \boxtimes \gamma_!Y
	}
	\qquad
	\xymatrix{
		F(A \times B) \ar[d]_-g \ar[r]^-f & F(A) \times F(B) \ar[d]^-k \\
		F(A' \times B') \ar[r]_-h & F(A') \times F(B').
	} \]
	To do this, replace the left-hand square by the isomorphic square
	\[ \xymatrix @R=1.7em{
		f^*(X \boxtimes Y) \ar@{-->}[d] \ar[r] & X \boxtimes Y \ar[d] \\
		h^*k_!(X \boxtimes Y) \ar[r] & k_!(X \boxtimes Y).
	} \]
	and show that the induced map $g_!f^*(X \boxtimes Y) \to h^*k_!(X \boxtimes Y)$ is an isomorphism. As one might expect, this turns out to be precisely the Beck-Chevalley map. It is defined by the universal property of cartesian arrows, which re-arranges into the statement that it is the unique map making this square commute (here $Z = X \boxtimes Y$).
	\[ \xymatrix @R=1.7em{
		f^*Z \ar[d] \ar[r] & Z \ar[r] & k_!Z \\
		g_!f^*Z \ar@{-->}[rr] && h^*k_!Z \ar[u]
	} \]
	We prove that the Beck-Chevalley map has this defining property by subdividing this square as follows.
	\[ \xymatrix @R=1.7em{
		f^*Z \ar[r]^-{cart} \ar[rd]^-{cocart} \ar[d]^-{cocart} & Z \ar[r]^-{cocart} & k_!Z & \\
		g_!f^*Z \ar[dd]^-{coev} \ar[rd]^-{cocart} & f_!f^*Z \ar[u]^-{ev} \ar[rd]^-{cocart} & & h^*k_!Z \ar[ul]^-{cart} \\
		& h_!g_!f^*Z \ar[r]^-\cong & k_!f_!f^*Z \ar[uu]^-{ev} & \\
		h^*h_!g_!f^*Z \ar[ur]^-{cart} \ar[rrr]^-\cong &&& h^*k_!f_!f^*Z \ar[ul]^-{cart} \ar[uu]^-{ev}
	} \]
	This finishes the proof that the tensor preserves cocartesian arrows. The proof for cartesian arrows is much shorter.
\end{proof}

\begin{rmk}
	The original symmetric monoidal category $\ho\Osp(B)$ sits inside $\ho\Osp_{(B)}$ as the fiber category over the object $\id_B\colon B \to B$.
\end{rmk}

\begin{rmk}
	If we enlarge the base category of $\Osp_{(B)}$ to be all spaces with maps to $B$, not just fibrations, then by the same argument we would get a symmetric monoidal structure on $\Osp_{(B)}$ and the homotopy category $\ho\Osp_{(B)}$. However the proof of \autoref{prop:SMBF_over_b} fails, and in fact the smash product does not preserve all co-cartesian arrows in the homotopy category.
\end{rmk}
%

\beforesubsection
\subsection{Parametrized homology and cohomology}\aftersubsection

If $B$ is a space, $\mc E$ is a parametrized sequential spectrum or orthogonal spectrum over $B$, and $X \to B$ is an unbased space over $B$, we define the twisted homology and cohomology spectra of $X$ with coefficients in $\mc E$ as follows.
\[ H_\bullet(X;\mc E) := (\L r_!)(\R p^*) \mc E \]
\[ H^\bullet(X;\mc E) := (\R r_*)(\R p^*) \mc E. \]
Here $r\colon X \to *$ and $p\colon X \to B$ are the projections. If $\mc E$ is freely $f$-cofibrant and level $h$-fibrant then the base-change operations are already derived, so we get the simpler formula
\[ H_\bullet(X;\mc E) \simeq r_!p^* \mc E \]
\[ H^\bullet(X;\mc E) \simeq r_*p^* \mc E. \]
In other words, we pull back $\mc E$ to $X$, then either quotient out the basepoint (homology) or take sections (cohomology).

Note that these are ordinary spectra, with no parametrization. The twisted homology and cohomology groups are just the homotopy groups of these spectra,
\[ H_n(X;\mc E) := \pi_n(H_\bullet(X;\mc E)) \]
\[ H^n(X;\mc E) := \pi_{-n}(H^\bullet(X;\mc E)). \]\index{homology}

\begin{lem}
	The homology and cohomology spectra of $X$ can also be described by the formulas
\[ H_\bullet(X;\mc E) \simeq (\L r_!)(\R\Delta_B^*) (X_{+B} \barsmash^\L \mc E),
\qquad H^\bullet(X;\mc E) \simeq \R \barF_B(X_{+B},\mc E) \]
where $r\colon B \to {*}$ and $\Delta_B\colon B \to B \times B$ are 0-fold and 2-fold diagonal maps.
\end{lem}
\begin{proof}
	This follows from \eqref{half_smash_internal} and \eqref{half_map_sections} and a careful check that the inputs are appropriately cofibrant and fibrant.
\end{proof}

\begin{ex} \hfill
	\vspace{-1em}
	
	\begin{itemize}
		\item When $\mc E = r^*E = B \times E$ is a trivial bundle of spectra, we get the classical definitions of homology and cohomology spectra:
		\[ H_\bullet(X;E) \simeq X_+ \sma^{\L} E, \qquad H^\bullet(X;E) \simeq \R F(X_+,E). \]
		If $X$ is a cell complex then the first is just $X_+ \sma E$, and if $X$ is finite or $E$ is an $\Omega$-spectrum then the second is just $F(X_+,E)$.
		\item Another common case is $X = B$, where we get
		\[ H_\bullet(B;\mc E) \simeq (\L r_!)\mc E, \qquad H^\bullet(B;\mc E) \simeq (\R r_*)(\mc E). \]
		Again, homology pushes $\mc E$ forward to a point, while cohomology takes sections.
		\item Let $\mc A$ be a bundle of abelian groups and $\mc E = H\mc A$ the associated bundle of Eilenberg-Maclane spectra. Then the above definition of $H_*(X;H\mc A)$ is isomorphic to classical twisted ordinary homology $H_*(X;\mc A)$, and similarly for cohomology. To see this it's enough to observe that it satisfies a twisted version of the Eilenberg-Steenrod axioms, in which the dimension axiom is stated as an isomorphism $h_*(\{b\}) \cong \mc A_b$ for all $b \in B$, and for any path $\gamma$ from $b$ to $b'$ the composite $h_*(\{b\}) \cong h_*(\gamma) \cong h_*(\{b'\})$ agrees with the monodromy isomorphism $\mc A_b \cong \mc A_{b'}$ induced by $\gamma$. Then a slightly modified version of the usual proof of uniqueness of homology shows that $H_*(X;H\mc A) \cong H_*(X;\mc A)$. 
	\end{itemize}
\end{ex}

It is possible to characterize these theories by a variant of the Eilenberg-Steenrod axioms for parametrized spaces over $B$. See \cite[20.1]{ms} for more details. In particular, the twisted $K$-theory of \cite{atiyah_segal} is represented by a parametrized spectrum, as discussed in detail in \cite{hebestreit_sagave_schlichtkrull}.

\newpage
\section{Duality and traces}\label{sec:duality}

We can produce fixed-point invariants by taking ``traces'' in the homotopy category of parametrized spectra. There are two conceptually different ways to do this, leading respectively to the Lefschetz number $L(f)$ and the Reidemeister trace $R(f)$. The first is the usual notion of trace in a symmetric monoidal category from \cite{dold_puppe}, while the second is a more general notion of trace in a shadowed bicategory from \cite{ponto_asterisque,kaledin_traces}.

The following table gives an overview. Each entry is a type of product of parametrized spectra, the duality theory associated to it, and a resulting fixed-point invariant.
\[
\resizebox{\textwidth}{!}{$
	\begin{array}{|c|c|c|}\hline
	\textup{smash product} & \textup{external smash product} & \textup{composition product} \\
	\sma & \barsmash & \odot \\
	\textup{Spanier-Whitehead duality} & \textup{(no interesting} & \textup{Costenoble-Waner duality} \\
	\textup{Lefschetz number } L(f) & \textup{duality theory)}  & \textup{Reidemeister trace } R(f)
	\\\hline
	\textup{smash product rel }B & \textup{external smash product rel }B & \textup{composition product rel }B \\
	\sma_B & \barsma{B} & \odot_B \\
	\textup{fiberwise Spanier-Whitehead duality} & \textup{(no interesting} & \textup{fiberwise Costenoble-Waner duality} \\
	\textup{fiberwise Lefschetz number } L_B(f) & \textup{duality theory)} & \textup{fiberwise Reidemeister trace } R_B(f) \\\hline
	\end{array}
$}
\]
Each of these products is defined from $\barsmash$ as indicated below.
\[ \xymatrix @C=8em @R=2em{
	\sma &
	\barsmash \ar[l]_-*\txt{special case} \ar[d]^-*\txt{used to define} \ar[r]^-*\txt{used to define} &
	\odot \\
	\sma_B &
	\barsma{B} \ar[l]^-*\txt{special case} \ar[r]_-*\txt{used to define} &
	\odot_B \\
} \]
The above table is not exhaustive. For instance the refined Reidemeister trace of \cite{ponto_shulman_mult} is also defined using the composition product $\odot$.

\beforesubsection
\subsection{Duality and traces in a symmetric monoidal category}\aftersubsection

Suppose that $(\mc C,\otimes,I,\alpha,\ell,r,\gamma)$ is a symmetric monoidal category. Recall that $(X,Y)$ is a \textbf{dual pair} in $\mc C$ if there are maps
\[
c: I \ra X \otimes Y \qquad e: Y \otimes X \ra I
\]
such that the composites
\[ \xymatrix @R=0.5em @C=3em{
	X \cong I \otimes X \ar[r]^-{c \otimes \id_X} & X \otimes Y \otimes X \ar[r]^-{\id_X \otimes e} & X \otimes I \cong X \\
	Y \cong Y \otimes I \ar[r]^-{\id_Y \otimes c} & Y \otimes X \otimes Y \ar[r]^-{e \otimes \id_Y} & I \otimes Y \cong Y
} \]
are the identity maps of $X$ and $Y$, respectively. We say that the duality $(X,Y,c,e)$ is \textbf{invertible}\index{invertible!in a symmetric monoidal category} if $c$ and $e$ are isomorphisms.

We usually just say that $X$ is \textbf{dualizable}\index{dualizable!in a symmetric monoidal category} if such data exists, because it turns out that when it exists, it is unique up to isomorphism. Similarly $X$ is invertible if this duality is an invertible one.
\begin{lem}\hfill
	\vspace{-1em}
	
	\begin{itemize}
		\item $X$ is dualizable iff $- \otimes X$ has a right adjoint of the form $- \otimes Y$.
		\item $X$ is invertible iff	$X \otimes Y \cong I$ for some $Y$.
		\item Duals are unique: for any two duals $Y$, $Y'$ of $X$ there is a unique isomorphism $Y \cong Y'$ along which the coevaluation and evaluation maps agree.
		\item Strong symmetric monoidal functors preserve dual pairs.
	\end{itemize}
\end{lem}

\begin{ex}[Examples of dualizable objects]\label{ex:dualizable_objects}\hfill
	\vspace{-1em}
	
	\begin{itemize}
		\item If $k$ is a field, a $k$-vector space $V$ is dualizable iff it is finite-dimensional, and invertible iff it is 1-dimensional. If $k$ is instead a commutative ring, a module $M$ is dualizable iff it is finitely generated projective.
		\item In $k$-chain complexes up to quasi-isomorphism, a chain complex is dualizable iff it is quasi-isomorphic to one with finitely many nonzero levels, each of which is finitely generated projective.
		\item In the stable homotopy category $\ho\Osp(*)$, a spectrum $X$ is dualizable iff it is equivalent to $F_n K$ where $K$ is a finite complex, iff the total homology $H_*(X;\Z)$ is finitely generated. This is \textbf{Spanier-Whitehead duality}\index{Spanier-Whitehead duality}. A spectrum is invertible iff it is equivalent to a sphere $F_n S^m$. If $X$ is a finitely dominated cell complex (a retract up to homotopy of a finite complex) then $\Sigma^\infty_+ X$ is dualizable.
		\item In the homotopy category of spectra over $B$, $\ho\Osp(B)$, a spectrum $X$ is dualizable iff its derived fiber spectra $(\R b^*)X$ are dualizable. This is \textbf{fiberwise Spanier-Whitehead duality}\index{Spanier-Whitehead duality!fiberwise}. The spectrum is furthermore invertible iff its fiber spectra are spheres. (See \autoref{fiberwise_dualizability}.)
	\end{itemize}
\end{ex}

\begin{rmk}
	In $\cat{Set}$ or $\cat{Top}$ with the product $\times$, the only dualizable objects are the singletons $\{*\}$. As a result, if we look at the external smash product $\barsmash$ on the homotopy category of all parametrized spectra $\ho\Osp$, the only dualizable objects are non-parametrized spectra in $\ho \Osp(*)$.
\end{rmk}

The above symmetric monoidal categories are all closed. A \textbf{closed}\index{symmetric monoidal!category, closed} symmetric monoidal category is one in which $- \otimes X$ has a right adjoint $F(X,-)$ for each object $X$. Therefore $X$ is dualizable iff $F(X,-)$ is isomorphic to a functor of the form $- \otimes Y$. This implies that any dual of $X$ must be isomorphic to $F(X,I)$, the \textbf{functional dual} of $X$.

If $X$ is dualizable, the \textbf{trace}\index{trace!in a symmetric monoidal category} of a map $f\colon X \to X$ is defined as the composite
\[ \xymatrix @R=0.5em @C=3em{
	I \ar[r]^-c & X \otimes Y \ar[r]^-{f \otimes \id} & X \otimes Y \ar@{<->}[r]^-{\cong} & Y \otimes X \ar[r]^-{e} & I.
} \]
More generally, the \textbf{twisted trace} of a map $f\colon Q \otimes X \to X \otimes P$ is the composite
\[ \xymatrix @R=0.5em @C=3em{
	Q \ar@{<->}[r]^-\cong & Q \otimes I \ar[r]^-{\id \otimes c} & Q \otimes X \otimes Y \ar[r]^-{f \otimes \id} & X \otimes P \otimes Y \ar@{<->}[r]^-\cong & Y \otimes X \otimes P \ar[r]^-{e \otimes \id} & P.
} \]
\begin{rmk}
	By the coherence theorem for symmetric monoidal categories \cite{maclane}, we get the same trace no matter how we parenthesize the terms above, and no matter how we define the isomorphism $\cong$, so long as we pick a definition that is formally obtained as composites of the structure isomorphisms $\alpha, \ell, r, \gamma$ of $\mc C$.
\end{rmk}
If $F\colon \mc C \to \mc D$ is a strong symmetric monoidal functor, then it preserves tensor products and dual pairs, so it also preserves traces. In other words $\tr(F(f)) \cong F(\tr(f))$, along the isomorphism $I \cong F(I)$ coming from the symmetric monoidal structure on $F$. A similar statement holds for twisted traces. 

\begin{ex}[Examples of traces]\label{ex:traces}\hfill
	\vspace{-1em}
	
	\begin{itemize}
		\item If $f\colon V \to V$ is a linear transformation of finite-dimensional vector spaces, or more generally finite free modules, the trace $\tr(f)$ in the above sense is a map $k \to k$ that multiplies by the trace of $f$ in the usual sense from linear algebra.
		\item If $f\colon C_* \to C_*$ is a map of free finitely-generated chain complexes, its trace is the alternating sum of the traces
		\[ \tr(f) = \sum_i (-1)^i \tr(f_i). \]
		The signs arise from the symmetry isomorphism $X \otimes Y \cong Y \otimes X$, which introduces a $(-1)^i$ in degree $i$. Without this kind of sign convention, the symmetry isomorphism would not be a map of chain complexes, and rational homology $H_*(-;\Q)$ would not be a symmetric monoidal functor.
		\item If $f\colon X \to X$ is a map of finite or finitely dominated cell complexes, the trace of $\Sigma^\infty_+ f$ in the stable homotopy category is a map $\Sph \to \Sph$, in other words an integer $L(f) \in \Z$. We obtain the same number by taking the trace of $f$ on $C_*(X;\Z)$ in the homotopy category of chain complexes, since $C_*(-;\Z)$ is strong symmetric monoidal by a variant of the Eilenberg-Zilber theorem.
		
		Applying the strong symmetric monoidal functor $H_*(-;\Q)$, we conclude that $L(f)$ is equal to the trace of the map $f_*$ on rational homology $H_*(X;\Q)$, or the alternating sum
		\[ L(f) = \sum_i (-1)^i \tr(f_*\colon H_i(X;\Q) \to H_i(X;\Q)). \]
		This is the classical \textbf{Lefschetz number}\index{Lefschetz number $L(f)$} of $f$, see for instance \cite{dold1974fixed,dold1976fixed}.
		\item If $f\colon E \to E$ is a map of fibrations over $B$ with finite or finitely dominated fibers, the trace of $\Sigma^\infty_{+B} f$ in $\ho\Osp(B)$ is a map of trivial bundles
		\[ L_B(\Sph)\colon B \times \Sph \to B \times \Sph, \]
		called the \textbf{fiberwise Lefschetz number}, cf \cite[II.10]{crabb_james}. By \autoref{prop:unique_sm_structure_on_pullback}, pulling back to a single point of $B$ is strong symmetric monoidal, so $L_B(f)$ is on each fiber equivalent to the Lefschetz number $L(f_b)$ of the fiber map $f_b\colon E_b \to E_b$. By the adjunction $(r_! \adj r^*)$ we can also rearrange $L_B(f)$ to a map of spectra $\Sigma^\infty_+ B \to \Sph$, or equivalently a map of spaces $B \to QS^0$. \index{Lefschetz number $L(f)$!fiberwise}\index{fiberwise!Lefschetz number $L(f)$} 
		\item If $E$ is a fibration over $B$ with finite or finitely dominated fibers, the twisted trace of the diagonal map $\Sigma^\infty_{+B} E \to \Sigma^\infty_{+B} E \times_B E$ is a map of spectra $\Sigma^\infty_{+B} B \to \Sigma^\infty_{+B} E$. Applying $r_!$ gives a map of spectra
		\[ \tau\colon \Sigma^\infty_+ B \to \Sigma^\infty_+ E. \]
		This is the \textbf{Becker-Gottlieb transfer} \cite{becker1975transfer,becker1976transfer}.
	\end{itemize}
\end{ex}

We conclude with three elementary properties of the trace. We state them for the ordinary trace, but they also have versions that are twisted by $Q$ and $P$.
\begin{prop}\label{prop:trace_properties}\hfill
	\vspace{-1em}
	
	\begin{itemize}
		\item (Additivity) \cite{may_additivity} If $\mc C$ is an additive category with biproduct $\oplus$, and the product $\otimes$ from the symmetric monoidal structure distributes over $\oplus$, given two self-maps $f_X\colon X \to X$, $f_Y\colon Y \to Y$ of dualizable objects we have
		\[ \tr(f_X \oplus f_Y) = \tr(f_X) + \tr(f_Y). \]
		More generally, if $\mc C$ has a compatible triangulated structure, given a map of cofiber sequences
		\[ \xymatrix @R=1.7em @C=3em{
			X \ar[d]^-{f_X} \ar[r] & Z \ar[d]^-{f_Z} \ar[r] & Y \ar[d]^-{f_Y} \\
			X \ar[r] & Z \ar[r] & Y
		} \]
		we have $\tr(f_Z) = \tr(f_X) + \tr(f_Y)$.
		\item (Multiplicativity) Given two self-maps $f_X\colon X \to X$, $f_Y\colon Y \to Y$ of dualizable objects we have
		\[ \tr(f_X \otimes f_Y) \cong \tr(f_X) \otimes \tr(f_Y). \]
		\item (Composition) \cite{mp1} Given a cycle of maps
		\[ \xymatrix{
			X_0 \ar[r]^-{f_1}& X_1 \ar[r]^-{f_2}& \ldots  \ar[r]^-{f_{n-1}}& X_{n-1} \ar[r]^-{f_1}& X_0
		} \]
		if we regard their tensor product as a self-map
		\[ \psi(f_1,\ldots,f_n)\colon (X_0 \otimes X_1 \otimes \ldots \otimes X_{n-1}) \to (X_0 \otimes X_1 \otimes \ldots \otimes X_{n-1}), \]
		its trace is equal to the trace of the composite
		\[ \tr(\psi(f_1,\ldots,f_n)) = \tr(f_n \circ \ldots \circ f_1). \]
		As a corollary, the trace has cyclic invariance,
		\[ \tr(f_n \circ \ldots \circ f_2 \circ f_1) = \tr(f_1 \circ f_n \circ \ldots \circ f_2). \]
	\end{itemize}
\end{prop}

\beforesubsection
\subsection{Spanier-Whitehead duality and applications}\aftersubsection

Let $X$ be a finite cell complex, or retract of such (a compact ENR).\index{compact ENR} Then $\Sigma^\infty_+ X$ is dualizable in the stable homotopy category. This can be deduced indirectly by an inductive argument on the number of cells, without too much work. However, if we want explicit coevaluation and evaluation maps, we need to do more work.

If $X$ is a finite cell complex, pick an embedding $i\colon X \to \R^n$ that has a mapping cylinder neighborhood. By this we mean a closed neighborhood $N$ containing $X$, with boundary $\partial N$ a CW complex, such that $N$ is homeomorphic to the mapping cylinder of a continuous map $\partial N \to X$. Let $p\colon N \to X$ denote the projection map.

For each $\epsilon > 0$ we define the $\epsilon$-tube, $\epsilon$-ball, and $\epsilon$-sphere as follows.
\[ \begin{array}{rcl}
	N_\epsilon &=& \{ u \in \R^n : d(u,i(X)) \leq \epsilon \} \\
	B_\epsilon &=& \{ u \in \R^n : d(u,0) \leq \epsilon \} \\
	S^n_\epsilon &=& \R^n / \{ u : d(u,0) \geq \epsilon \} \cong B_\epsilon / \partial B_\epsilon
\end{array} \]
We pick $\epsilon > 0$ so that $N_\epsilon \subseteq N$; this is possible by a straightforward compactness argument. Finally, we let $S^n$ refer to the one-point compactification of $\R^n$. We identify $S^n$ it with $S^n_\epsilon$ by the homotopy equivalence $S^n \overset\sim\to S^n_\epsilon$ that is the identity inside $B_\epsilon$ and sends the rest to the basepoint. We also identify the orthogonal spectrum $F_n S^n$ with $\Sph = F_0 S^0$ along the stable equivalence $F_n S^n \overset\sim\to \Sph$ discussed in \autoref{ex:equivalent_thom_spectra}.

\begin{thm}[Spanier-Whitehead duality, neighborhood version]\label{thm:sw_duality_1}
	If $X$ is a finite cell complex, the dual of $\Sigma^\infty_+ X$ in the stable homotopy category is $F_n (N/\partial N)$, with coevaluation and evaluation maps
	\begin{align*}
		F_n S^n &\to F_n (N/\partial N) \sma X_+ \\
		X_+ \sma F_n (N/\partial N) &\to F_n S^n
	\end{align*}
	obtained by applying $F_n$ to the ``collapse'' and ``scanning'' maps
	\[ \begin{array}{ccccccc}
	S^n & \overset{c}\to & N/\partial N \sma X_+ && X_+ \sma N/\partial N & \overset{e}\to & S^n_\epsilon \\[.5em]
	v &\mapsto & \left\{ \begin{array}{ccl} v \sma p(v) && v \in N \\ {*} && v \not\in N \end{array} \right. && x \sma u &\mapsto & u-i(x)
	\end{array} \]
\end{thm}\index{Spanier-Whitehead duality}

\begin{proof}
	Below we generalize this to \autoref{thm:sw_duality_2}, then we prove that generalization in \autoref{sec:sw_proof}.
\end{proof}

\begin{rmk}
	If $X = M$ is a smooth manifold and $i$ is a smooth embedding, we can take $N$ to be a tubular neighborhood of $X$. Then $N/\partial N$ is the Thom space of the normal bundle $\nu$, hence its desuspension is the Thom spectrum $Th(-TM)$ of the negative tangent bundle, $-TM = \nu - \R^n$. This result is commonly called \textbf{Atiyah duality} after \cite{atiyah_thom}, though this special case was done first by Milnor and Spanier \cite{milnor_spanier}.
\end{rmk}

As a corollary, if $X = M$ is an orientable smooth manifold with normal bundle $\nu$, then the above theorem identifies $F_n Th(\nu) \sma -$ with $F(M_+,-)$, therefore
\[ \begin{array}{c}
	F(M_+,H\Z) \simeq F_n Th(\nu) \sma H\Z \\[.5em]
	\Rightarrow H^{i}(M;\Z) \cong H_{n-i}(Th(\nu),*;\Z) \cong H_{d-i}(M;\Z)
\end{array} \]
where $d = \dim M$ and the last isomorphism is the Thom isomorphism. This is a particularly clean way to prove Poincar\'e duality. To generalize to extraordinary homology theories and non-orientable manifolds, one only needs to understand the twisted version of the Thom isomorphism in these more general cases.
%

Before proving \autoref{thm:sw_duality_1} we first give a more general version with mapping cones. For each inclusion $A \subseteq X$ we let $C^u(X,A)$ denote the unreduced mapping cone, and if $A$ is based we let $C(X,A)$ denote the reduced mapping cone:
\begin{align*}
	C(X,A) &:= A \sma I \cup_{A \sma \{1\}_+} X \\
	& \cong C^u(X,A) / C^u(*,*) \\
	C^u(X,A) &:= {*} \cup_{A \times \{0\}} (A \times I) \cup_{A \times \{1\}} X \\
	& \cong C(X_+,A_+).
\end{align*}\index{mapping cone}
Every map of pairs gives a map of mapping cones, and there is a canonical homeomorphism $C^u(X,A) \sma Y_+ \cong C^u(X \times Y,A \times Y)$.

Let $X$ be a compact ENR, $i\colon X \to \R^n$ a topological embedding, and $p\colon N \to X$ be a map that retracts the closed neighborhood $N$ back to $X$. It may not be a mapping cylinder neighborhood, so we may not have much control over the homotopy type of $N/\partial N$.\footnote{It might be helpful to picture the case where $i(X) \subset \R^3$ is Alexander's horned sphere.} We can always at least make $\partial N \to N$ into an $h$-cofibration, if we wish, by taking a fine triangulation of $\R^n$ and taking $N$ to be the union of the closed simplices meeting $X$.

In this more general case, the dual is $F_n C^u(N,N-X)$. The coevaluation map becomes any route through the following diagram from top-left to bottom-right.
\[ \xymatrix @R=2em{
	S^n \ar[d] \ar[r] & N/\partial N \ar[r]^-p & N/\partial N \sma X_+ \\
	C(S^n,S^n - \textup{int} N) \ar[d] \ar[ur]_-{(\sim)} & \ar[l] C^u(N,\partial N) \ar[u]_-{(\sim)} \ar[d] \ar[r]^-p & C^u(N,\partial N) \sma X_+ \ar[u]_-{(\sim)} \ar[d] \\
	C(S^n,S^n-X) & \ar[l]_-\sim C^u(N,N-X) \ar[r]^-p & C^u(N,N-X) \sma X_+ \\
} \]
If $\partial N \to N$ is an $h$-cofibration then the maps marked $(\sim)$ are equivalences, so we can go through the middle of the diagram. Otherwise we have to take the far-left route. The equivalence on the bottom is by excision, \autoref{lem:mapping_cones_equivalences}.

The evaluation map becomes the bottom edge of the following diagram.
\[ \xymatrix @R=2em{
	X_+ \sma N/\partial N \ar[r] & B_\epsilon / \partial B_\epsilon \\
	X_+ \sma C^u(N,\partial N) \ar[u]_-{(\sim)} \ar[d] \ar[r] & C^u(\R^n,\R^n - \textup{int} B_\epsilon) \ar[u]_-\sim \ar[d]^-\sim & \ar@{<->}[l]_-\sim \ar@{<->}[lu]_-\sim\ar@{<->}[ld]_-\sim S^n \\
	X_+ \sma C^u(N,N-X) \ar[r] & C^u(\R^n, \R^n - \{0\})	
} \]
The equivalences on the right are any degree-one maps between different models for the sphere $S^n$, while the horizontal maps in the first column arise from the map $X \times N \to \R^n$ sending $x,u$ to $u - i(x)$.

\begin{thm}[Spanier-Whitehead duality, mapping cone version]\label{thm:sw_duality_2}\cite{dold_puppe}
	If $X$ is a compact ENR, the dual of $\Sigma^\infty_+ X$ in the stable homotopy category is $F_n C^u(N,N-X)$, with coevaluation and evaluation maps
	\begin{align*}
	F_n S^n &\to F_n C^u(N,N-X) \sma X_+ \\
	X_+ \sma F_n C^u(N,N-X) &\to F_n S^n
	\end{align*}
	obtained by applying $F_n$ to the collapse and scanning maps defined above.
\end{thm}\index{Spanier-Whitehead duality}

The proof is in \autoref{sec:sw_proof}. When $N$ is a mapping cylinder neighborhood, $C^u(N,\partial N) \to C^u(N,N-X)$ is an equivalence, so this theorem implies the previous one.

\begin{cor}\label{cor:lefschetz_formula}
	When $X$ is a compact ENR, the Lefschetz number $L(f)$ of a map $f\colon X \to X$ is the degree of the map of spheres
	\[ \xymatrix @R=0em{
		S^n \ar[r] & S^n_\epsilon \\
		v \ar@{|->}[r] & v - f(p(v)).
	} \]
\end{cor}\index{Lefschetz number $L(f)$}

\begin{proof}
	When $X$ is a finite complex, this follows from \autoref{thm:sw_duality_1} by writing out the formulas for the pieces of the trace of $\Sigma^\infty_+ f\colon \Sigma^\infty_+ X \to \Sigma^\infty_+ X$ and composing them together:
	\[ \xymatrix @R=0.5em{
		S^n \ar[r] & N/\partial N \sma X_+ \ar[r]^-f & N/\partial N \sma X_+ \ar[r] & S^n_\epsilon \\
		v \ar@{|->}[r] & v \sma p(v) \ar@{|->}[r] & v \sma f(p(v)) \ar@{|->}[r] & v - f(p(v))
	} \]
	When $X$ is a compact ENR, we refine $N$ so that $\partial N \to N$ is an $h$-cofibration. Then the commuting diagrams before the statement of \autoref{thm:sw_duality_2} imply that the trace defined using mapping cones is homotopic to the above composite as well.
\end{proof}

Assume the fixed points of $f$ are isolated (this can be arranged by a homotopy) and then choose $\epsilon$ so that outside a small neighborhood of each fixed point, $f(x)$ is always at least $\epsilon$ away from $p^{-1}(x)$. Then the trace of $\Sigma^\infty_+ f$ vanishes away from the fixed points. Near each fixed point $x \in \Fix(f)$, we get a map of spheres homotopic to
\[ \begin{array}{rcl}
	B_\epsilon(x)/\partial & \to & B_\epsilon(0)/\partial \\
	v &\mapsto & v - f(p(v)).
\end{array} \]
The degree of this map is the standard definition of ``index'' of a fixed point when $X$ is a finite complex or compact ENR.\footnote{When $X$ is a smoothly embedded submanifold of dimension $d$, this reduces to a more classical definition, namely the degree of the map $v \mapsto v - f(v)$ defined on a punctured small $d$-dimensional coordinate ball in $X$.} Therefore the degree of $\tr(\Sigma^\infty_+ f)$ is the sum of these degrees, proving:
\begin{thm}[Lefschetz-Hopf]
	\[ \sum_{x \in \Fix(f)} \ind(x) = \sum_i (-1)^i \tr(f_*\colon H_i(X;\Q) \to H_i(X;\Q)). \]
\end{thm}
This approach to the Lefschetz-Hopf theorem is originally due to Dold and Puppe \cite{dold_puppe}. It is more conceptual than earlier proofs, and easier to generalize to the case of bundles, because it doesn't require us to triangulate $X$ and use simplicial homology. Instead, it goes straight from topology to stable homotopy theory.

\begin{ex}
	The flip map of $S^1$ has two fixed points, each of index $+1$, so $L(f) = 2$. On homology, the flip map is the identity on $H_0$ and negation on $H_1$, so the alternating sum of traces is also equal to 2.
\end{ex}

This entire discussion can now be done fiberwise over $B$, cf. \cite[II.6-10]{crabb_james}. Suppose that $E \to B$ is a fiber bundle in which both the fiber $X$ and the base $B$ are finite cell complexes, and $B$ is connected. Pick a fiberwise embedding $i\colon E \to B \times \R^n$ over $B$. To generalize the first theorem above, we assume that $i$ has a fiberwise mapping cylinder neighborhood $N \to E$. This is true for instance if $E$ is a trivial bundle, or a smooth fiber bundle with closed manifold fiber.

Let $N_\epsilon$ be the $\epsilon$-tube about $i(E)$ in $B \times \R^n$, distance measured only in the $\R^n$ direction, and chosen so that $N_\epsilon \subseteq N$. Let $N/_B\partial N$ refer to the cofiber in retractive spaces over $B$, i.e. the pushout $N \cup_{\partial N} B$.

\begin{thm}[Fiberwise Spanier-Whitehead duality, neighborhood version]\label{thm:fiberwise_sw_duality_1}
	Under these assumptions, the dual of $\Sigma^\infty_{+B} E$ in spectra over $B$ is $F_n N/_B\partial N$, with coevaluation and evaluation maps
	\begin{align*}
	F_n S^n_B &\to F_n N/_B\partial N \sma_B E_{+B} \\
	E_{+B} \sma_B F_n N/_B\partial N &\to F_n S^n_B
	\end{align*}
	obtained by applying $F_n$ to the maps of retractive spaces
	\[ \begin{array}{ccccccc}
	S^n \times B & \to & N/_B\partial N \sma_B E_{+B} && E_{+B} \sma_B N/_B\partial N & \to & S^n_\epsilon \times B \\[.5em]
	v &\mapsto & \left\{ \begin{array}{ccl} v \sma p(v) && v \in N \\ {*} && v\not\in N \end{array} \right. && x \sma u &\mapsto & u-i(x)
	\end{array} \]
\end{thm}\index{Spanier-Whitehead duality!fiberwise}

Again, we prove this by generalizing it. For an inclusion $A \subseteq X$ of spaces over $B$, let $C_B(X,A)$ and $C_B^u(X,A)$ denote the fiberwise versions of the reduced and unreduced mapping cone:
\begin{align*}
C_B(X,A) &:= A \barsmash I \cup_{A \barsmash \{1\}_+} X \\
& \cong C^u_B(X,A) / C^u_B(B,B) \\
C^u_B(X,A) &:= {B} \cup_{A \times \{0\}} (A \times I) \cup_{A \times \{1\}} X \\
& \cong C_B(X_{+B},A_{+B}).
\end{align*}\index{mapping cone}

Let $E \to B$ be a fiber bundle in which the base $B$ and fiber $X$ are compact ENRs. Pick an embedding $i\colon E \to B \times \R^n$ over $B$ and let $N$ be any neighborhood that retracts fiberwise onto $E$ (possible by \cite[1.8]{dold1974fixed}). It does not have to be a mapping cylinder neighborhood.

\begin{thm}[Fiberwise Spanier-Whitehead duality, mapping cone version]\label{thm:fiberwise_sw_duality_2}
	Under these assumptions, the dual of $\Sigma^\infty_{+B} E$ in spectra over $B$ is $F_n C^u_B(N,N-E)$, with maps
	\begin{align*}
	F_n S^n_B &\to F_n C^u_B(N,N-E) \sma_B E_{+B} \\
	E_{+B} \sma_B F_n C^u_B(N,N-E) &\to F_n S^n_B
	\end{align*}
	given by the following fiberwise analogs of the maps from \autoref{thm:sw_duality_2}.
\end{thm}\index{Spanier-Whitehead duality!fiberwise}
\[ \xymatrix{
	S^n_B \ar[d] \ar[r] & N/_B\partial N \ar[r]^-p & N/_B\partial N \sma_B E_{+B} \\
	C_B(S^n_B,S^n_B - \textup{int} N) \ar[d] \ar[ur]_-{(\sim)} & \ar[l] C^u_B(N,\partial N) \ar[u]^-{(\sim)} \ar[d] \ar[r]^-p & C^u_B(N,\partial N) \sma_B E_{+B} \ar[u]_-{[\sim]} \ar[d] \\
	C_B(S^n_B,S^n_B-E) & \ar[l]_-\sim C^u_B(N,N-E) \ar[r]^-p & C^u_B(N,N-E) \sma_B E_{+B} \\
} \]

\[ \xymatrix{
	E_{+B} \sma_B N/_B\partial N \ar[r] & B_\epsilon / \partial B_\epsilon \times B \\
	E_{+B} \sma_B C_B^u(N,\partial N) \ar[u]^-{[\sim]} \ar[d] \ar[r] & C^u(\R^n,\R^n - \textup{int} B_\epsilon) \times B \ar[u]_-\sim \ar[d]^-\sim & \ar@{<->}[l]_-\sim \ar@{<->}[lu]_-\sim\ar@{<->}[ld]_-\sim S^n_B \\
	E_{+B} \sma_B C_B^u(N,N-E) \ar[r] & C^u(\R^n, \R^n - \{0\})	\times B
} \]
Note that all the terms in the bottom two rows are $f$-cofibrant and that $E_{+B}$ is $h$-fibrant. By \autoref{cor:internal_smash_properties} this makes the smash products model the smash product in the homotopy category. The maps marked $(\sim)$ are equivalences when $\partial N \to N$ is an $h$-cofibration, while $[\sim]$ is an equivalence if $\partial N \to N$ is an $f$-cofibration. So if $N$ is a fiberwise mapping cylinder neighborhood then these are the same coevaluation and evaluation maps from before.

\begin{cor}\label{cor:fiberwise_lefschetz_formula}
	When $E$ is a fiber bundle with base $B$ and fiber $X$ both compact ENRs, the fiberwise Lefschetz number $L_B(f)$ of a fiberwise map $f\colon E \to E$ arises from the map of trivial sphere bundles over $B$
\[ \xymatrix @R=0em{
	S^n \times B \ar[r] & S^n_\epsilon \times B \\
	(v,b) \ar@{|->}[r] & (v - f_b(p_b(v)),b),
} \]
by de-suspension (equivalently, by taking free spectra $F_n$ of both sides).
\end{cor}\index{Lefschetz number $L(f)$!fiberwise}\index{fiberwise!Lefschetz number $L(f)$}
So $L_B(f)$ is just $L(f)$ carried out over each point of $B$, giving a map of bundles because the formula for $L(f)$ varies continuously in $f$.

\begin{proof}
	 If $N$ is a fiberwise mapping cylinder neighborhood, this is direct from \autoref{thm:fiberwise_sw_duality_1}. Otherwise, we first assume that $B$ is a finite simplicial complex, and find an $N$ so that $\partial N \to N$ is an $h$-cofibration. We do this by triangulating $B \times \R^n$, subdividing, and taking the union of closed simplices meeting $i(E)$. Then the maps marked $(\sim)$ are equivalences, and the same argument as before gives the result.
	 
	 If $B$ is a compact ENR, it is a retract of a finite simplicial complex $K$ along an inclusion $j\colon B \to K$. Since $j^*\colon \ho\Osp(K) \to \ho\Osp(B)$ is strong symmetric monoidal, it preserves the trace. The above formula gives a map of level-fibrant spectra over $K$, so its strict pullback is equivalent to the derived pullback, and is therefore a formula for the trace over $B$.
\end{proof}

We can even interpret the above formula as a weighted sum of fixed points of $f$ if we count the degrees in the right way. If $B$ is a smooth $d$-dimensional manifold and the fixed points are arranged to be smoothly embedded $d$-dimensional submanifolds of $E$, then this trace as an element in $[\Sigma^\infty_+ B, \Sph]$ counts the fixed points with framing of their normal bundle in $B \times \R^n$ given by $v \mapsto f(p(v))$ as a framed manifold in $B \times \R^n$ up to framed cobordism.\footnote{Technically $v \mapsto f(p(v))$ is not a bundle map, but as long as it is smooth and full-rank at the zero section it can be deformed to one.} This is by a variant of the Pontryagin-Thom isomorphism, see for instance \cite[2.5]{cw_transversality}.

When $B$ is a sphere, it turns out we can embed $B \times \R^n$ into $\R^m \times \R^n$ and count the framing there. If in addition $B$ is the circle, the fiberwise Lefschetz number $L_{S^1}(f)$ lands in the group
\[ [\Sigma^\infty_+ S^1, \Sph] \cong \Z \oplus \Z/2. \]
The first generator is the Lefschetz number of the fiber $L(f_b)$, while the second generator counts the total number of twists in the framing of the fixed points, mod 2.

\begin{ex}\label{ex:lefschetz_fiberwise}
	Consider the trivial bundle $S^3 \times S^1 \to S^1$ and the fiberwise self-map that over $\theta \in [0,2\pi]/(0 = 2\pi)$ is the one-point compactification of the map $\R^3 \to \R^3$ that scales by $1+\epsilon$ and rotates by $\theta$ about the $z$-axis. The fixed points of this map are always $0$ and $\infty$, with indices $-1$ and $+1$, respectively. In the bundle $S^3 \times S^1$ these fixed points form two disjoint circles, and on each one the framing $v \mapsto v-f(v)$ contributes one twist and zero twists, respectively, for a total of one twist. Therefore
	\[ L_{S^1}(f) = (0,1) \in \Z \oplus \Z/2. \]
	This proves that even though the action of $f$ on each copy of $S^3$ can have its fixed points removed, we cannot remove the fixed points by a fiberwise homotopy on all of $S^3 \times S^1$.
\end{ex}

\beforesubsection
\subsection{Proof of Spanier-Whitehead duality}\label{sec:sw_proof}\aftersubsection

We now recall the argument from \cite{dold_puppe}, \cite[III.4]{lms}, and \cite{ms} that proves \autoref{thm:sw_duality_2}, and generalize it to \autoref{thm:fiberwise_sw_duality_2}. Note that the fiberwise theorem \autoref{thm:fiberwise_sw_duality_2} is not handled directly in \cite{ms}, since their motivation is to prove dualizability, which reduces to checking what happens on a single fiber, whereas our motivation is to get an explicit formula for the trace.

Recall that for a pair $(X,A)$ of (not retractive) spaces over $B$, the unbased mapping cone $C^u_B(X,A)$ is the retractive space defined as the colimit of the diagram
\[ \xymatrix{
	B & \ar[l] A \times 0 \ar[r] & A \times I & \ar[l] A \times 1 \ar[r] & X \times 1. } \]\index{mapping cone}
For two inclusions $A \to X$ and $A' \to X'$, let $(X,A) \square (X',A')$ denote the pair
\[ (X,A) \square (X',A') = (X \times X',X \times A' \cup_{A \times A'} A \times X') \]
where the second is regarded as a subspace of the first for the purpose of giving it a topology. We will only consider this product when we have open inclusions, so this is the same as the pushout topology. If $(X,A)$ is a pair over $B$ and $(X',A')$ is a pair over $B'$, we define a map
\[ \mu\colon C^u_B(X,A) \barsmash C^u_{B'}(X',A') \to C^u_{B \times B'}((X,A) \square (X',A')) \]
as follows. The left-hand side rearranges into the colimit the diagram below.
\[ \xymatrix @R=1em{
	& \ar[ddddl] A \times X' \times 0 \times 1 \ar[r] & A \times X' \times I \times 1 & \ar[l] A \times X' \times 1 \times 1 \ar[r] & X \times X' \times 1 \times 1 \\
	& & A \times A' \times I \times 1 \ar[u] \ar[d] & & X \times A' \times 1 \times 1 \ar[u] \ar[d] \\
	& \ar[ddl] A \times A' \times 0 \times I \ar[r] & A \times A' \times I \times I & \ar[l] A \times A' \times 1 \times I \ar[r] & X \times A' \times 1 \times I \\
	& & A \times A' \times I \times 0 \ar[u] \ar[dll] & & X \times A' \times 1 \times 0 \ar[u] \ar[dllll] \\
	B \times B'
} \]
Then we use the formula
\[ X \times X' \times I \times I \to X \times X' \times I, \qquad (x,y,s,t) \mapsto (x,y,\min(s,t)) \]
to give maps from all these pieces into the mapping cone for $(X,A) \square (X',A')$ over $B \times B'$.

This product map $\mu$ can be defined for any number of pairs, and is associative, unital, and commutative in an appropriate sense, since this follows directly from its formula. When $A' = \emptyset$, all terms below the top row of the above diagram disappear, and $\mu$ becomes the canonical homeomorphism
\[ C^u_B(X,A) \barsmash X'_{+B'} \cong C^u_{B \times B'}(X \times X',A \times X'). \]

In the following lemma, ``normal'' means that the total space (forgetting the map to $B$) is normal in the sense of point-set topology, i.e. closed sets can be separated by open sets. In particular, if $B$ is a cell complex and $X$ and $X'$ are subspaces of $B \times \R^n$, then $X \times X'$ is a metric space and therefore normal.

\begin{lem}\label{product_of_cones}\hfill
	\vspace{-1em}
	
	\begin{itemize}
		\item (Excision) If $U \subset A \subset X$ are inclusions of fiberwise spaces over $B$, $\bar U \subset \textup{int} A$, and $X$ is normal, then the map of unbased cones
		\[ C^u_B(X-U,A-U) \to C^u_B(X,A) \]
		is a fiberwise homotopy equivalence.
		\item (Product) If $f\colon A \subset X$ and $f'\colon A' \subset X'$ are open inclusions of fiberwise spaces and $X \times X'$ is normal, then the product map $\mu$ is a fiberwise homotopy equivalence.
	\end{itemize}
\end{lem}

\begin{proof}
	Identical to \cite[III.4.3 and III.4.4]{lms}, only we interpret the formulas as taking place in spaces over $B$ and $B \times B'$, respectively.
\end{proof}

\begin{cor}\label{lem:mapping_cones_equivalences}
	Suppose $X$ is any compact subspace of $\R^n$, and $N$ is any neighborhood of $X$, not necessarily open. Then the following maps are homotopy equivalences.
	\begin{itemize}
		\item $C^u(N,N-X) \to C^u(\R^n,\R^n-X)$
		\item $C^u(N,N-X) \to C(S^n,S^n-X)$
	\end{itemize}
	Similarly the fiberwise versions of these maps
	\begin{itemize}
		\item $C^u_B(N,N-E) \to C^u_B(\R^n,\R^n-E)$
		\item $C^u_B(N,N-E) \to C_B(S^n,S^n-E)$
	\end{itemize}
	are homotopy equivalences if $B$ is a cell complex, $E$ is a subspace of $B \times \R^n$ with $E \to B$ proper, and $N$ any neighborhood in $B \times \R^n$.
\end{cor}

\begin{proof}
	The first part follows by excision on $\R^n - N \subseteq \R^n - X \subseteq \R^n$. For the second part, using the first part, it suffices to assume that $N$ is a large disc centered at the origin. Then we get a deformation retract by gluing on a deformation retract of $C(D^n,D^n)$ onto $C^u(S^{n-1},S^{n-1})$, where $D^n$ is given basepoint $*$.
	
	The fiberwise version of the first part has the same argument. For the second part,	we merely have to let the radius of the disc $D^n$ vary continuously over the base $B$ so that it always contains $X$. Over every cell of $B$, the preimage in $X$ is compact, making it straightforward to do this.
\end{proof}

Since the product map $\mu$ is a fiberwise homotopy equivalence, any pullback or pushforward of $\mu$ is also a weak equivalence. Notice that since the mapping cone is built from pushouts along closed inclusions, it commutes up to canonical homeomorphism with pullbacks $g^*$ and pushforwards $f_!$. Therefore for any span ending in $B \times B'$
\[ \xymatrix{
	D & \ar[l]_-f C \ar[r]^-g & B \times B'
} \]
composing these commutations with the equivalence $f_!g^*\mu$ gives an equivalence
\[ f_!g^*(C^u_B(X,A) \barsmash C^u_{B'}(X',A')) \simar f_!g^*C^u_{B \times B'}((X,A) \square (X',A')) \cong C^u_{D}f_!g^*((X,A) \square (X',A')). \]
So any ``operation'' on parametrized spaces formed by an external smash product, a pullback, and then a pushforward, must commute with mapping cones up to a canonical fiberwise homotopy equivalence.

The ordinary smash product $\sma$ and the internal smash product $\sma_B$ are two examples of such operations. We therefore get a homotopy equivalence
\[ \mu_\sma \colon C^u(X,A) \sma C^u(X',A') \to C^u((X,A) \square (X',A')) \]
for spaces over $*$ and a fiberwise homotopy equivalence
\[ \mu_{\sma_B}\colon C^u_B(X,A) \sma_B C^u_B(X',A') \to C^u_B((X,A) \square_B (X',A')) \]
when both pairs $(X,A)$ and $(X',A')$ are over the same base space $B$ and
\[ (X,A) \square_B (X',A') =  (X \times_B X',X \times_B A' \cup_{A \times_B A'} A \times_B X'). \]
These are still associative, commutative, and unital in an appropriate sense.

Next we reduce \autoref{thm:sw_duality_2} to a space-level statement. Let $n \geq 0$. Recall that two based spaces $X,Y$ are $n$-dual if there are maps $c\colon S^n \to Y \sma X$ and $e\colon X \sma Y \to S^n$ such that the first composite below is homotopic to the transposition map, and the second is homotopic to the transposition followed by the map $S^n \to S^n$ that is the one-point compactification of $v \mapsto -v$.\index{$n$-duality!in a symmetric monoidal category} In particular, it is homotopic to the identity if $n$ is even and a reflection if $n$ is odd.
\[ \xymatrix @R=.5em{
	X \sma S^n \ar[r]^-{1 \sma c} & X \sma Y \sma X \ar[r]^-{e \sma 1} & S^n \sma X \\
	S^n \sma Y \ar[r]^-{c \sma 1} & Y \sma X \sma Y \ar[r]^-{1 \sma e} & Y \sma S^n
} \]
When $X,Y \in \mc R(B)$ we formulate the same definition using $\sma_B$ and the trivial sphere over $B$, $S^n_B = S^n \times B$. We restrict to $f$-cofibrant spaces so that the smash product can be right-derived using $P$, and we ask for all the smash products that appear to be equivalent to their right-derived smash products. It is not necessary that $c$ and $e$ be maps; they can be zig-zags that give maps in the homotopy category.
\begin{lem}\label{lem:n-dual}
	If $X$ and $Y$ are $n$-dual and $h$-cofibrant, then applying $F_n$ to $c$ and $e$ gives a duality between $F_0 X$ and $F_n Y$ in the homotopy category of orthogonal spectra. The same holds if $X$ and $Y$ are $f$-cofibrant retractive spaces over $B$.
\end{lem}

\begin{proof}
	In the non-fiberwise case, we interpret the maps
	\[ F_n c\colon F_n S^n \to F_n Y \sma X, \qquad \qquad F_n e\colon X \sma F_n Y \to F_n S^n \]
	as maps in the homotopy category
	\[ \Sph \to F_n Y \sma^\L X, \qquad \qquad X \sma^\L F_n Y \to \Sph \]
	and the same is true of all the smash products in the proof below. Hence we can prove the triangle identities in the homotopy category, using the strict smash product $\sma$ in the place of $\sma^\L$ and proving the maps agree up to homotopy.
	
	The triangle identities in spectra become
	\[ \xymatrix @R=.5em{
		X \sma F_n S^n \ar[r]^-{1 \sma F_n c} & X \sma F_n Y \sma X \ar[r]^-{F_n e \sma 1} & F_n S^n \sma X \\
		F_n S^n \sma F_n Y \ar[r]^-{F_n c \sma 1} & F_n Y \sma X \sma F_n Y \ar[r]^-{1 \sma F_n e} & F_n Y \sma F_n S^n.
	} \]
	The symmetry isomorphism for $X \sma F_n S^n$ is just $F_n$ of the transposition map, so the first triangle identity for spaces directly gives the first triangle identity for spectra. The symmetry isomorphism for $F_n S^n \sma F_n Y$ is determined by what it does at spectrum level $2n$, where it is a map of quotients of $O(2n) \times S^n \times Y$ sending $(\rho,v,y)$ to $(\rho\rho_{n,n},y,v)$. Here $\rho_{n,n}$ is the automorphism of $\R^{2n}$ shuffling the first $n$ coordinates past the remaining coordinates. It is in the same path component of $O(2n)$ as the map $(v,w) \mapsto (-v,w)$, and this gives a homotopy from the symmetry isomorphism to the map that sends $(\rho,v,y)$ to $(\rho,y,-v)$. This latter map is $F_{2n}$ of the map of spaces provided by the $n$-duality, so we get the second triangle identity in spectra as well.
	
	For general $B$, the proof starts by observing that $F_n$ commutes with $P$. Hence the derived smash product of free spectra is isomorphic to a free spectrum on the derived smash product. The proof then proceeds as before.
\end{proof}

\begin{proof} of \autoref{thm:sw_duality_2}.
	Define two closed subspaces of $X \times N$ by
	\[ \Delta_X = \{ (x,i(x)) : x \in X \},
	\qquad
	\Gamma_\epsilon = \{ (x,u) : d(i(x),u) \leq \epsilon \}. \]
	Intuitively, $\Delta_X$ is an embedded copy of $X$ and $\Gamma_\epsilon$ is its tubular neighborhood.
	
	Let $D$ be a large disc centered at the origin containing $i(X)$. By \autoref{lem:mapping_cones_equivalences}, our proposed coevaluation and evaluation maps can be rearranged slightly so that they only use unbased mapping cones:
	\[ \xymatrix @C=1em @R=.8em{
		C^u(\R^n,\R^n-D) \ar[r] & C^u(\R^n,\R^n-X) & \ar[l]_-\sim C^u(N,N-X) \ar[r]^-p & C^u(N,N-X) \sma X_+ \\
		X_+ \sma C^u(N,N-X) \ar[rr]^-{(x,u) \mapsto u-x} && C^u(\R^n, \R^n - \{0\}).
	} \]
	At either end, we still identify $C^u(\R^n,\R^n-D)$ and $C^u(\R^n, \R^n - \{0\})$ with $S^n$ along a degree one map.
	
	By \autoref{lem:n-dual} it suffices to show these give an $n$-duality between $X_+$ and $C^u(N,N-X)$, which are both $h$-cofibrant based spaces. For the first triangle identity, the composite of coevaluation and evaluation maps becomes the top and right edges of the following diagram.
	\[ \xymatrix @R=1.5em{
		X_+ \sma C^u(\R^n,\R^n-D) \ar[dd] \ar[r] & X_+ \sma C^u(\R^n,\R^n-X) \ar@/_1em/[ldd] & \ar[l]_-\sim X_+ \sma C^u(N,N-X) \ar[ld] \ar[dd] \\
		& C^u(X \times N, (X \times N) - \Delta_X) \ar[ld] \ar@/^5em/[dd] & \\
		X_+ \sma C^u(\R^n,\R^n-0) \ar@{<->}[rd]^-\cong & C^u(\Gamma_\epsilon,\Gamma_\epsilon - \Delta_X) \ar[l] \ar[d] \ar[u]_-\sim & X_+ \sma C^u(N,N-X) \sma X_+ \ar[ld] \\
		& C^u(\R^n,\R^n-0) \sma X_+ &
	} \]
	These maps all arise as maps of pairs, and so to check the regions commute up to homotopy it suffices to give the maps and homotopies of pairs as indicated below.
	\[ \xymatrix @R=1.5em @C=6em{
		X \times (\R^n,\R^n-D) \ar[dd] \ar[r] & X \times (\R^n,\R^n-X) \ar@/_1em/[ldd]_-{(x,v) \mapsto (x,v-x)} & \ar[l]_-\sim X \times (N,N-X) \ar[ld] \ar[dd]^-{(x,v) \mapsto (x,v,p(v))} \\
		& (X \times N, (X \times N) - \Delta_X) \ar[ld]_-{(x,v) \mapsto (x,v-x)} \ar@/^5em/[dd]^-(.3){(x,v) \mapsto (v-x,p(v))} & \\
		X \times (\R^n,\R^n-0) \ar@{->}[rd]^-\cong & (\Gamma_\epsilon,\Gamma_\epsilon - \Delta_X) \ar[l]_-{(x,v) \mapsto (x,v-x)} \ar[d]_-(.3){(x,v) \mapsto (v-x,p(v))} \ar[u]_-\sim & X \times (N,N-X) \times X \ar[ld]^-{(x,v,y) \mapsto (v-x,y)} \\
		& (\R^n,\R^n-0) \times X &
	} \]
	The unlabeled maps are all inclusions and collapses, in other words $(x,v) \mapsto (x,v)$. The top-left region commutes by the homotopy $(x,v-tx)$. This gives a map of pairs because $v - tx = 0 \Rightarrow v \in D$. The bottom-left region commutes by the homotopy $(p((1-t)v+tx),v-x)$. This gives a map of pairs because $v - x = 0 \Rightarrow (x,v) \in \Delta_X$. The remaining regions strictly commute. In summary, the first triangle identity follows from the homotopies $$(p(v),v-x) \sim (x,v-x) \sim (x,v).$$
	
	For the second triangle identity, the composite of coevaluation and evaluation maps becomes the top, right, and bottom edges of the following diagram. Along our identifications with $S^n$, the map we wish to compare it to is given on the left-hand edge, where the second map swaps the cones and also negates the coordinate in $\R^n$.
	\[ \xymatrix @R=1.5em{
		C^u(\R^n,\R^n-D) \sma C^u(N,N-X) \ar[d]_-\sim \ar[r] & C^u(\R^n,\R^n-X) \sma C^u(N,N-X) \\
		C^u(\R^n,\R^n-0) \sma C^u(N,N-X) \ar[d]_-\sim & C^u(N,N-X) \sma C^u(N,N-X) \ar[u]_-\sim \ar[d] \\
		C^u(N,N-X) \sma C^u(\R^n,\R^n-0) & C^u(N,N-X) \sma X_+ \sma C^u(N,N-X) \ar[l]
	} \]
	Applying $\mu_\sma$ to every term gives a termwise homotopy equivalence to a new diagram, where every term arises from a single pair and the maps arise as maps of pairs, indicated below.
	\[ \xymatrix @R=1.5em{
		(\R^n,\R^n-D) \square (N,N-X) \ar[d]_-\sim \ar[rr] && (\R^n,\R^n-X) \square (N,N-X) \ar[ld]_-{(v,w) \mapsto (w,w-v)} \\
		(\R^n,\R^n-0) \square (N,N-X) \ar[d]_-{(v,w) \mapsto (w,-v)} & (\R^n,\R^n-X) \square (\R^n,\R^n-0) & (N,N-X) \square (N,N-X) \ar[l]^-{(v,w) \mapsto (v,w-v)} \ar[u]_-\sim \ar[d]^-{(v,w) \mapsto (v,p(v),w)} \\
		(N,N-X) \square (\R^n,\R^n-0) \ar[ur]^-\sim && ((N,N-X) \times X) \square (N,N-X) \ar[ll]^-{(v,x,w) \mapsto (v,w-x)}
	} \]
	The unlabeled maps are all inclusions and collapses, in other words $(v,w) \mapsto (v,w)$. The top-left region commutes by the homotopy $(w,tw-v)$. This gives a map of pairs because if $w \in X$ and $tw - v = 0$ then $(v,w) \in D \times X$. The right-hand region commutes by the homotopy $((1-t)w+tv,w-v)$. This gives a map of pairs because if $w = v$ and $(1-t)w+tv = v \in X$ then $(v,w) \in X \times X$. The bottom region commutes by the homotopy $(v,w-(1-t)v-tp(v))$. This gives a map of pairs because if $v \in X$ and $w - (1-t)v - tp(v) = w-v = 0$ then $(v,w) \in X \times X$. In summary, the second triangle identity follows from the homotopies $$(v,w-p(v)) \sim (v,w-v) \sim (w,w-v) \sim (w,-v).$$
	
	Since all the mapping cones are $h$-cofibrant, these two point-set triangle identities imply the triangle identities on the homotopy category.
\end{proof}
\begin{proof} of \autoref{thm:fiberwise_sw_duality_2}.
	Define $\Delta_E \subseteq \Gamma_\epsilon \subseteq E \times_B N$ in the same way as above, and perform the same rearrangement of the coevaluation and evaluation maps, so that everything arises from maps of pairs. Then the point-set version of the first triangle identity follows by the same formulas. All the smash products are equivalent to the derived smash product because $E_{+B}$ is $h$-fibrant, by \autoref{cor:internal_smash_properties}. This gives the first triangle identity in the homotopy category.
	
	In the second triangle identity, we can still apply $\mu_{\sma_B}$ termwise and get an equivalent diagram that commutes up to homotopy. The backwards map is still an equivalence, by excision for $N \times_B N \subset \R^n_B \times_B N$, after we have applied $\mu_{\sma_B}$. This gives a point-set version of the second triangle identity; namely, the following diagram commutes in the homotopy category.
	\[ \xymatrix @R=1.5em{
		C^u_B(\R^n_B,\R^n_B-D_B) \sma_B C^u_B(N,N-E) \ar[d]_-\sim \ar[r] & C^u_B(\R^n_B,\R^n_B-E) \sma_B C^u_B(N,N-E) \\
		C^u_B(\R^n_B,\R^n_B-0_B) \sma_B C^u_B(N,N-E) \ar[d]_-\sim & C^u_B(N,N-E) \sma_B C^u_B(N,N-E) \ar[u]_-\sim \ar[d] \\
		C^u_B(N,N-E) \sma_B C^u_B(\R^n_B,\R^n_B-0_B) & C^u_B(N,N-E) \sma_B E_{+B} \sma_B C^u_B(N,N-E) \ar[l]
	} \]
	This diagram admits a map to the same diagram with derived smash products. The map is an equivalence on the left-hand column, but not in the right-hand column.\footnote{The derived smash product in the right-hand column would give us the cone on the same pair as before, but with expressions such as $(N-E) \times_B (N-E)$ replaced by $(N-E) \times_B B^I \times_B (N-E)$.} However, by a brief diagram chase, this is enough to deduce that the diagram with derived smash products also commutes in the homotopy category. This verifies the second triangle identity in the homotopy category.
\end{proof}

\beforesubsection
\subsection{Duality and traces in a shadowed bicategory}\label{sec:traces_in_bicategories}\aftersubsection

To get stronger fixed-point invariants, we generalize what it means to take a trace of a morphism $f$. Instead of working in a symmetric monoidal category, such as modules over a commutative ring $k$, we work in a shadowed bicategory, such as the bicategory of modules over \emph{non}-commutative rings.

To illustrate the difference, recall that for non-commutative algebras $A$, $B$, and $C$ over a commutative ring $k$ there is a tensor product operation $\otimes_B$ that takes each $(A,B)$-bimodule ${}_A M_B$ and $(B,C)$-bimodule ${}_B N_C$ to an $(A,C)$-bimodule $M \otimes_B N$. As the algebras vary, these tensor products are associative and unital up to coherent isomorphism. We can therefore tensor long strings of bimodules together,
\[ M_0 \otimes_{A_1} M_1 \otimes_{A_2} M_2 \ldots M_{n-1} \otimes_{A_n} M_n, \]
without having to worry about parentheses, because any two ways of forming this product are canonically isomorphic.

If the same ring acts on both sides of this string of bimodules, we can also close the string into a circle. Define the \textbf{shadow} of an $(A,A)$-bimodule $M$ to be the $k$-module
\[ \shad{M} := M/(am \sim ma). \]
Intuitively, this is $M$ tensored with itself. Then define the \textbf{circular product} of a list of bimodules $_{A_i} (M_i)_{A_{i+1}}$, indices taken mod $n$, by the formula
\[ \shad{M_0,\ldots,M_n} := \shad{ M_0 \otimes_{A_1} M_1 \otimes_{A_2} \ldots \otimes_{A_n} M_n }. \]
This is well-defined up to canonical isomorphism, even if we ``rotated'' the list of modules before applying this formula. So we have an essentially unique way to define the circular product of a circular list of bimodules.

This is formalized in the notion of a \textbf{bicategory with shadow}\index{shadowed bicategory}\index{bicategory with shadow}, or shadowed bicategory \cite{ponto_asterisque}, sometimes also called a trace theory after \cite{kaledin_traces}. Formally, a bicategory with shadow $\mathscr B$ is defined by the following data.
\begin{itemize}
	\item A collection of objects or ``0-cells'' $\ob \mathscr B$. We typically label these $A$, $B$, $C$, ...
	\item For each pair of 0-cells $A$, $B$, a category $\mathscr B(A,B)$ whose objects are called ``1-cells'' $X$, $Y$, ... and whose morphisms are called ``2-cells'' $f$, $g$, ...
	\item A category $\mathscr B_{\shad{}}$ called the ``shadow category.''
	\item For every triple of 0-cells $A$, $B$, $C$ a functor $\odot\colon \mathscr B(A,B) \times \mathscr B(B,C) \to \mathscr B(A,C).$
	\item For every 0-cell $A$ a functor $\shad{}\colon \mathscr B(A,A) \to \mathscr B_{\shad{}}$ and a 1-cell $U_A \in \ob\mathscr B(A,A)$.
	\item Natural isomorphisms $\alpha$, $\lambda$, $\rho$, and a ``rotator'' isomorphism $\theta$ making the following four diagrams of functors commute.
	\[
\begin{array}{cc}
	\xymatrix @R=1.5em{
		\mathscr B(A,B) \times \mathscr B(B,C) \times \mathscr B(C,D) \ar[d]_-{\odot \times 1} \ar[r]^-{1 \times \odot} & \mathscr B(A,B) \times \mathscr B(B,D) \ar[d]^-{\odot} \\
		\mathscr B(A,C) \times \mathscr B(C,D) \ar[r]_-{\odot} & \mathscr B(A,D)
	}
	& \raisebox{-2em}{$\alpha\colon (X \odot Y) \odot Z \cong X \odot (Y \odot Z)$} \\[5em]
	\xymatrix @R=1.5em{
		\mathscr B(A,B) \ar[d]_-{U_A \times 1} \ar@{=}[dr] \ar[r]^-{1 \times U_B} & \mathscr B(A,B) \times \mathscr B(B,B) \ar[d]^-{\odot} \\
		\mathscr B(A,A) \times \mathscr B(A,B) \ar[r]_-{\odot} & \mathscr B(A,B)
	}
	& \raisebox{-3em}{$\lambda\colon U_A \odot X \cong X$}
	\raisebox{-1em}{$\rho\colon X \odot U_B \cong X$} \\[5em]
	\xymatrix @R=1.5em{
		\mathscr B(A,B) \times \mathscr B(B,A) \ar[dr]_-{\shad{}} \ar[r]^-{\textup{swap}} & \mathscr B(B,A) \times \mathscr B(A,B) \ar[d]^-{\shad{}} \\
		& \mathscr B_{\shad{}}
	}
	& \raisebox{-2em}{$\theta\colon \shad{X \odot Y} \cong \shad{Y \odot X}$}
\end{array}
	\]
\end{itemize}
This data must also satisfy a coherence condition. If we start with any expression for an $n$-fold product of 1-cells, either arranged along a line or a circle, then if we re-arrange this expression using the isomorphisms $\alpha$, $\lambda$, $\rho$, and $\theta$, and eventually come back to the same expression, the composite isomorphism must be the identity. As in the case of symmetric monoidal categories, it suffices to check this condition for four particular strings of the above isomorphisms. See \cite{benabou} and \cite[4.4.1]{ponto_asterisque} for the four conditions and \cite{mp2} for the proof of this coherence theorem.

\begin{ex}\label{ex:ring_bicategories}\hfill
	\vspace{-1em}
	
	\begin{itemize}
		\item If $\mc C$ is a symmetric monoidal category, it also defines a shadowed bicategory with a single 0-cell $*$. The categories $\mathscr B(*,*)$ and $\mathscr B_{\shad{}}$ are both equal to $\mc C$, the product $\odot$ is the tensor product in $\mc C$, the shadow $\shad{}$ is the identity functor, and $U_*$ is the unit of $\mc C$. The associator, unitors, and symmetry isomorphism come from the symmetric monoidal structure of $\mc C$.
		\item For any commutative ring $k$, there is a shadowed bicategory $\mathscr Bimod_k$ whose 0-cells are $k$-algebras, 1-cells are bimodules, and 2-cells are bimodule maps. The horizontal product, shadow, and unit were defined earlier in this section. The coherence conditions are checked using the universal property of the tensor product.\index{bicategory!of bimodules $\mathscr Bimod_k$}
		
		\item There is also a shadowed bicategory $ Ch\mathscr Bimod_k$ whose 0-cells are $k$-algebras that are projective as $k$-modules, 1-cells are unbounded chain complexes of bimodules, and 2-cells are maps in the derived category, i.e. maps of chain complexes with quasi-isomorphisms inverted. The horizontal product and shadow are left-derived from the same operations in $\mathscr Bimod_k$, and the coherence is deduced using the general theory of left-derived functors.\footnote{We need to know that the $k$-algebras $A$ are projective over $k$ to deduce that these products can be coherently deformed. To study algebras that are not projective over $k$, we replace them by DGAs that are levelwise projective.} We could similarly restrict to nonnegatively-graded chain complexes $ Ch^{\geq 0}\mathscr Bimod_k$.
		
		In particular, if $_A M_B$ and $_B N_C$ are concentrated in degree 0 and projective over $k$, their horizontal product is the two-sided bar complex
		\[ B_i(M;B;N) = M \otimes_k B^{\otimes_k i} \otimes_k N. \]
		Similarly if $_A M_A$ is concentrated in degree 0 and projective over $k$, the shadow of $M$ is the Hochschild complex \cite{loday_cyclic}
		\[ B^{\cyc}_i(A;M) = M \otimes_k A^{\otimes_k i}. \]
		
		\item There is a shadowed bicategory $ \mathscr Bimod_{\Sph}$ whose 0-cells are orthogonal ring spectra that are cofibrant\footnote{It seems fairly certain that some kind of cofibrancy hypothesis is needed to make the composition products and shadows coherently deformable. Though it would be interesting to work out if it is enough to assume the rings are flat-cofibrant spectra.} as $\Sph$-modules, 1-cells are orthogonal bimodule spectra, and 2-cells are maps in the homotopy category of bimodule spectra. We define the tensor and shadow as in $\mathscr Bimod_k$, using the smash product in the place of $\otimes_k$, then left-derive the results so that they pass to the homotopy category.\index{bicategory!of bimodule spectra $\mathscr Bimod_{\Sph}$}
		
		The shadow is therefore given by \textbf{topological Hochschild homology}
		\[ \THH(A;M) := |n \mapsto M \sma A^{\sma n}| \]
		on the subcategory of bimodules $M$ that are cofibrant as orthogonal spectra, see for instance \cite{shipley_thh,angeltveit2014relative}. The rotator isomorphism $\theta$ is left-derived from the same isomorphism on the point-set level, but it can be described more explicitly as an isomorphism
		\[ B^{\cyc}(A;B(M;B;N)) \cong B^{\cyc}(B;B(N;A;M)) \]
		that arises by recognizing that both are realizations of bisimplicial spectra, and on the diagonal the two simplicial spectra are isomorphic, both are given by
		\[ A^{\sma k} \sma M \sma B^{\sma k} \sma N \]
		with the same face and degeneracy maps. This is sometimes called the \textbf{Dennis-Waldhausen-Morita argument}, see \cite[p.36]{waldhausen_2}, \cite{blumberg_mandell_localization}.
		
		\item The shadowed bicategory $ \mathscr Bimod_{\Sph}$ can be extended from orthogonal ring spectra to categories enriched in orthogonal spectra, i.e. ``ring spectra on many objects,'' see \cite{campbell_ponto,blumberg_mandell_thc,clmpz}. The shadow in this case is the categorical cyclic bar construction
		\[ \THH(A;M) := \left|n \mapsto \bigvee_{a_0,\ldots,a_n} M(a_n,a_0) \sma A(a_0,a_1) \sma \ldots \sma A(a_{n-1},a_n)\right|. \]

		\item There is a bicategory of group actions $\mathscr Act$, whose 0-cells are well-based topological groups $G$, whose 1-cells from $G$ to $H$ are spaces $X$ with a left action of $G$ and right action of $H$ that commute with each other, and 2-cells are maps of such $G$-$H$ spaces with the weak equivalences inverted. As in the case of bimodules, the product of $_G X_H$ and $_H Y_J$ is the left-derived functor of $X \times_H Y$, in other words the unbased bar construction $B(X,H,Y)$.\index{bicategory!of group actions $\mathscr Act$} We could alternatively take based $G$-$H$ spaces with product $X \sma_H Y$, which left-derives to the based bar construction $B(X,H_+,Y)$ on the subcategory of well-based $X$ and $Y$.
		\item There is a similar bicategory of group actions on spectra $\mathscr Act_\Sph$, whose 0-cells are $q$-cofibrant topological groups $G$, whose 1-cells from $G$ to $H$ are spectra $X$ with a left action of $G$ and right action of $H$ that commute with each other, and 2-cells are maps of such with the weak equivalences inverted. This is just a sub-bicategory of $\mathscr Bimod_\Sph$ on the ring spectra of the form $\Sigma^\infty_+ G$.\index{bicategory!of group actions on spectra $\mathscr Act_\Sph$}
\end{itemize}
\end{ex}

\begin{rmk}\label{rmk:left_deformable_bicategory}
	Most of the examples in the above list follow a general pattern, where we pass from a ``point-set'' bicategory $\mathscr B$ to a ``homotopy'' bicategory $\ho\mathscr B$ by inverting certain 2-cells and left-deriving all the bicategory operations. Using the theory from \autoref{sec:composing_comparing} (or simply citing \cite[22.11]{shulman2006homotopy}), this is guaranteed to work if there exists a left retraction $Q$ on each category $\mathscr B(A,B)$, landing in a designated subcategory of cofibrant 1-cells, such that
	\begin{itemize}
		\item $\odot$ preserves cofibrant objects,
		\item $\odot$ and $\shad{}$ preserve weak equivalences between cofibrant objects,
		\item $QU_A \odot QX \to U_A \odot QX$ is always a weak equivalence, and
		\item $QX \odot QU_B \to QX \odot U_B$ is always a weak equivalence.
	\end{itemize}
	We say that $\mathscr B$ is a \textbf{left-deformable bicategory}\index{left-deformable!bicategory} when this happens. It appears that in most examples, for $\mathscr B$ to be left-deformable we need to restrict to the 0-cells that are cofibrant in some sense.
\end{rmk}

The notion of duality and trace generalizes from symmetric monoidal categories to bicategories with shadow. In the case of bimodules over non-commutative algebras, an $(A,B)$ bimodule $M$ is \textbf{left-dualizable}, or \textbf{dualizable over $A$}, if there is a $(B,A)$ bimodule $M^*$ and maps
\[ c\colon B \to M^* \otimes_A M \]
\[ e\colon M \otimes_B M^* \to A \]
such that the following two composites are identity maps:
\[ M \cong M \otimes_B B \to M \otimes_B M^* \otimes_A M \to A \otimes_A M \cong M \]
\[ M^* \cong B \otimes_B M^* \to M^* \otimes_A M \otimes_B M^* \to M^* \otimes_A A \cong M^* \]
Equivalently, every functor of the form $M \otimes_B -$ has right adjoint isomorphic to $M^* \otimes_A -$. As in the symmetric monoidal case, this implies that $M^* \cong \Hom_A(M,A)$, the set of left $A$-linear maps $M \to A$, with $(B,A)$-action given by pre-composition with the right $B$-action and post-composition with the right $A$-action.

Similarly, we say that $M$ is \textbf{right-dualizable} or \textbf{dualizable over $B$} if there are maps
\[ c\colon A \to M \otimes_B M' \]
\[ e\colon M' \otimes_A M \to B \]
satisfying similar triangle identities.\footnote{The author generally finds the term ``dualizable over $A$'' to be less confusing than ``left-dualizable.'' For instance, when $M$ is dualizable over $A$, its dual $M^*$ is also dualizable over $A$, and its dual over $A$ is $M$ again. This is a little cleaner than saying that its left-dual $M^*$ has a right-dual which is $M$ again. It also makes us less attached conceptually to writing things in a particular order, which is always good.} The dual over $B$ is then the $(B,A)$ bimodule $\Hom_B(M,B)$.

In general, these are separate conditions and the two duals of $M$ do not have to be isomorphic. However they coincide when $M$ is \textbf{invertible}, meaning it has a dual on one side for which the maps $c$ and $e$ are isomorphisms. From this it follows that $M$ has a dual on the other side, and the two duals are isomorphic. Such an invertible $(A,B)$-bimodule $M$ induces the famous \textbf{Morita equivalence} equivalence between the category of $B$-$k$ bimodules and of $A$-$k$-bimodules.\index{Morita equivalence}

The above discussion generalizes to any bicategory $\mathscr B$.
\begin{df}\label{def:dual_in_bicategory}
	A 1-cell $X$ from $A$ to $B$ is dualizable over $A$ if there is another 1-cell $Y$ from $B$ to $A$ and maps
	\[ c\colon U_B \to Y \odot X \]
	\[ e\colon X \odot Y \to U_A \]
	such that the following two composites are identity maps.
	\[ X \cong X \odot U_B \to X \odot Y \odot X \to U_A \odot X \cong X \]
	\[ Y \cong U_B \odot Y \to Y \odot X \odot Y \to Y \odot U_A \cong Y \]
\end{df}\index{dualizable!in a bicategory}
Dualizability over $B$ is defined similarly. The mnemonic is that the 0-cell you are dualizable over, is the one that receives the evaluation map. Furthermore $X$ is invertible if it is dualizable on at least one side and the maps $c$ and $e$ are isomorphisms.\index{invertible!in a bicategory}

\begin{ex}[Examples of dualizable objects]\hfill
	\vspace{-1em}
	
	\begin{itemize}
		\item If the shadowed bicategory comes from a symmetric monoidal category $\mc C$, then a 1-cell $X$ is left- or right-dualizable in this sense iff it was dualizable in $\mc C$.
		\item In $\mathscr Bimod_k$, an $(A,B)$-bimodule $M$ is dualizable over $A$ iff it is finitely generated projective as an $A$-module (forgetting the $B$-action). It is invertible if in addition the map $B \to \Hom_A(M,M)$ is an isomorphism, and $M$ is a ``generator'' meaning that $A$ is a direct summand of $M^{\oplus k}$ for some $k$ \cite[Ch II]{bass_moritach2}.
		\item In $Ch\mathscr Bimod_k$, a chain complex of $(A,B)$-bimodules $M_\bullet$ is dualizable over $A$ iff it is quasi-isomorphic as a chain complex of $A$-modules to one that is nonzero in only finitely many degrees, each of which is finitely generated projective.
		\item In $\mathscr Bimod_{\Sph}$, an $(A,B)$-bimodule $M$ is dualizable over $A$ iff as an $A$-module it is a retract in the homotopy category of a finite $A$-cell complex, in other words iff it is in the thick subcategory of $A$-modules generated by $A$.\footnote{It appears that the same argument applies to bimodules of spectral categories, proving that an $(A,B)$ bimodule $M(-,-)$ is dualizable over $A$ iff for each object $b$, the $A$-module $M(-,b)$ is a retract in the homotopy category of a finite cell complex of representable $A$-modules $A(-,a)$.}
		\item In $\mathscr Act_\Sph$, the suspension spectrum of a $G$-$H$-space $X$ is dualizable over $G$ if $X$ is a homotopy retract of a finite $G$-cell complex, after forgetting the $H$-action. The case of $H = 1$ is sometimes called \textbf{Ranicki duality}, see for instance \cite[5.3]{ponto_asterisque}.
		\item The 2-category $\cat{Cat}$ of categories, functors, and natural transformations forms a bicategory without a shadow.\footnote{We are not saying no shadow exists, only that we are not considering a shadow here.} A functor is dualizable on one side iff it is a left adjoint, and dualizable on the other side iff it is a right adjoint.
	\end{itemize}
\end{ex}

For bimodules, we arrive at these characterizations by comparing to the adjunction between $M \otimes_A -$ and $\Hom_A(M,-)$. We see that $M$ is dualizable over $A$ iff the composition map
\[ \Hom_A(M,A) \otimes_A M \cong \Hom_A(M,A) \otimes_A \Hom_A(A,M) \to \Hom_A(M,M) \]
is an isomorphism of $(B,B)$ bimodules, cf \cite[16.4.12]{ms}, \cite[6.3]{lind2019transfer}. The fact that this corresponds to a ``finitely generated projective'' condition is a standard argument: we prove the above map map is an isomorphism when $M$ is finitely generated projective by an induction on the generators. Conversely, when that map is an isomorphism we show that the coevaluation map $k \to M^* \otimes_A M$ factors through $M^* \otimes_A M'$ some finitely generated free module (or finite $A$-cell complex) $M'$. Using the triangle identities we then deduce that $M$ is a retract of $M'$, see e.g. \cite[p.79]{ekmm}.
This style of argument actually works in any ``closed'' bicategory, meaning one in which the operations $X \odot -$ and $- \odot X$ always have right adjoints, see \cite[\S 16.4]{ms}.\index{bicategory!closed}

We never use the shadow $\shad{}$ when we talk about dualizability, but we need it to take traces. It acts as a substitute for the fact that we can't swap terms in a tensor product past each other. We ``rotate'' them instead.

Suppose $M$ is a 1-cell from $A$ to $B$, dualizable over $B$, with dual $M^*$. Then the \textbf{trace}\index{trace!in a bicategory} of a map $f\colon M \to M$ is defined as the composite
\[ \xymatrix @R=0.5em @C=3em{
	\shad{U_A} \ar[r]^-{\shad{c}} & \shad{M \odot M^*} \ar[r]^-{\shad{f \odot \id}} & \shad{M \odot M^*} \ar@{<->}[r]^-{\cong} & \shad{M^* \odot M} \ar[r]^-{\shad{e}} & \shad{U_B}.
} \]
More generally, if $Q$ is a 1-cell from $A$ to $A$, and $P$ goes from $B$ to $B$, the \textbf{twisted trace} of a map $f\colon Q \odot M \to M \odot P$ is the composite
\[ \xymatrix @R=0.5em @C=3em{
	\shad{Q} \ar@{<->}[r]^-\cong & \shad{Q \odot U_A} \ar[r]^-{\shad{\id \odot c}} & \shad{Q \odot X \odot Y} \ar[r]^-{\shad{f \odot \id}} & \shad{X \odot P \odot Y} \ar@{<->}[r]^-\cong & \shad{Y \odot X \odot P} \ar[r]^-{\shad{e \odot \id}} & \shad{P}.
} \]

\begin{ex}[Examples of bicategorical traces]\hfill
	\vspace{-1em}
	
	\begin{itemize}
		\item If the shadowed bicategory comes from a symmetric monoidal category $\mc C$, this is exactly the same trace we defined earlier.
		\item In $\mathscr Bimod_k$, suppose $A$ is a $k$-algebra and $M \cong A^{\oplus n}$ is a free module of finite rank, regarded as a $k$-$A$ bimodule that is dualizable over $A$. Then the trace of $f\colon M \to M$ is a map
		\[ k \to A/(aa' = a'a) \]
		where the quotient is taken in $k$-modules. It sends $1 \in k$ to the sum of the diagonal entries of the matrix for $f$.\footnote{This equivalence relation is exactly what one needs for the trace to be invariant under coordinate change -- even when $n = 1$, we do not have the equality $a = bab^{-1}$ in the non-commutative algebra $A$ until we apply the above equivalence relation. And if $A$ is commutative, the equivalence relation is trivial, so this recovers the usual trace.}
		\item The trace of $f\colon P \to P$ when $P$ is a projective $A$-module is equal to the trace of the self-map of $A^{\oplus n}$ that retracts to $P$ then applies $f$. Doing this for every finitely generated projective module gives a map
		\[ K_0(A) \to A/(aa'=a'a) = HH_0(A) \]
		which is exactly the Hattori-Stallings trace \cite{hattori,stallings}.
		\item Let $X$ be a connected finite cell complex with fundamental group $\pi$, and $\ti X$ is its universal cover as a space with a left $\pi$-action. Since $X$ is finite, $\ti X$ is built from finitely many free $\pi$-cells, which implies that the $\Z[\pi]$-$\Z$ bimodule chain complex $C_*(\ti X;\Z)$ is dualizable over $\Z[\pi]$. Its dual is the cochain complex $C^*(\ti X;\Z)$ with the homological (i.e. negative) grading. This dualizability is true in the homotopy category always, and true on the nose if we insist that $C_*(\ti X;\Z)$ means cellular chains.
		
		If $f\colon X \to X$ is a self-map that fixes the basepoint, it induces a homomorphism $f_*\colon \pi \to \pi$ and a map $\ti f\colon \ti X \to \ti X$ that fixes a chosen basepoint of $\ti X$. The map $\ti f$ is not quite $\pi$-equivariant, instead satisfying the equation $\ti f(ax) = f_*(a)\ti f(x).$ Therefore $\ti f$ gives a linear map
		\[ C_*(\ti X;\Z) \to \Z[\pi]^f \otimes_{\Z[\pi]} C_*(\ti X;\Z) \]
		where $\Z[\pi]^f$ is the $\Z[\pi]$-$\Z[\pi]$ bimodule with actions $a \cdot b \cdot c := f_*(a)bc.$ The twisted trace of the map induced by $\ti f$ is therefore a map of Hochschild chain complexes
		\[ \shad{\Z} \to \shad{\Z[\pi]^f} \]
		which is determined by its effect on homology,\footnote{Note that over the group ring $\Z[\pi]$, there is no result stating that an alternating sum of traces on a chain complex is equal to the alternating sum of traces on homology. This is why we take the trace at the chain level, then measure the result on homology.} a map
		\[ \Z \to HH_0(\Z[\pi];\Z[\pi]^f) \cong \Z[\pi]/(f_*(a)b = ba). \]
		The image of 1 is called the \textbf{Reidemeister trace} $R(f)$,\index{Reidemeister trace $R(f)$} see \cite{wecken,husseini,ponto_asterisque}. Concretely, the action of $\ti f$ on each $C_n(\ti X;\Z)$ gives a matrix with entries in $\Z[\pi]$. We can compute $R(f)$ as the sum of the diagonal entries of all these matrices, each one weighted by $(-1)^n$, and then we apply the equivalence relation $f_*(a)b = ba$.\footnote{A good exercise is to use this description to show that when $X$ is the circle and $f$ is the map that reflects across a line, the resulting element of $\Z \oplus \Z$ is $(1,1)$.}
		
		\item Morita equivalent algebras $A$ and $B$ have isomorphic Hochschild homology groups, $HH_n(A) \cong HH_n(B)$, see e.g. \cite{loday_cyclic}. The fastest way to prove this is to let $M$ be an invertible 1-cell from $A$ to $B$, and take the trace of the identity map of $M$. Since the coevaluation and evaluation maps are isomorphisms, this trace is an isomorphism. Taken in the bicategory of bimodules, we therefore get an isomorphism
		\[ HH_0(A) = A/(aa'=a'a) \cong B/(bb' = b'b) = HH_0(B). \]
		To get the higher Hochschild homology groups, we think of $M$ as a chain complex concentrated in degree 0. Then it gives an invertible 1-cell in the bicategory of bimodule chain complexes, by taking the coevaluation and evaluation maps from before and composing them with the quasi-isomorphisms
		\[ B(M^*,A,M) \simar M^* \otimes_A M, \qquad B(M,B,M^*) \simar M \otimes_B M^*. \]
		(They are quasi-isomorphisms because dualizability forces $M$ and $M^*$ to be projective over both $A$ and $B$.) The trace of the identity of $M$ therefore gives an equivalence in the derived category between the Hochschild chains,
		\[ HC_\bullet(A;A) \simeq HC_\bullet(B;B). \]
		Taking homology gives $HH_n(A) \cong HH_n(B)$ for all $n \geq 0$. This is a conceptual reinterpretation of the classical use of the Dennis-Waldhausen-Morita argument, see also \cite[p.36]{waldhausen_2}, \cite{mccarthy_cyclic_homology_of_exact_category,blumberg_mandell_localization,campbell_ponto,clmpz}.
		\item Carrying out the same argument for Morita equivalent ring spectra $A$ and $B$ gives an equivalence of spectra
		\[ \THH(A) \simeq \THH(B). \]
		\item The more classical argument for Morita invariance of Hochschild homology goes by arguing first that
		\begin{equation}\label{morita_equivalent_to_perfect_modules}
			HH_\bullet(A) \cong HH_\bullet(\mc P(A)),
		\end{equation}
		where $\mc P(A)$ denotes the category of finitely-generated projective $A$-modules, enriched in $k$-modules by $\Hom_A(-,-)$. Then Morita equivalent rings have equivalent categories of modules, hence equivalent Hochschild homology.
		
		However, the equivalence \eqref{morita_equivalent_to_perfect_modules} could also be proven by a Dennis-Waldhausen-Morita argument. It suffices to show that in the bicategory of bimodules-on-many-objects, there is an invertible 1-cell from $A$ to $\mc P(A)$. The proof of this turns out to be surprisingly short -- it amounts to checking that $A \to \mc P(A)$ is fully faithful and essentially surjective up to retracts and cofiber sequences. See \cite{blumberg_mandell_localization,campbell_ponto,clmpz}.
	\end{itemize}
\end{ex}

\beforesubsection
\subsection{Shadowed bicategories of parametrized spaces and spectra}\aftersubsection

In this section we construct the shadowed bicategory of parametrized spectra. To warm up, we first construct a bicategory $\mc R/\cat{Top}$ whose 0-cells are topological spaces. The 1-cells from $A$ to $B$ are retractive spaces $X$ over $A \times B$, and 2-cells are maps of retractive spaces. The product of $X \in \mc R(A \times B)$ and $Y \in \mc R(B \times C)$ is defined by the formula
		\[ X \odot Y := (r_B)_!(\Delta_B)^*(X \barsmash Y), \]
		\[ \xymatrix @R=1.7em @C=4em{
			X \odot Y \ar[d] & \ar[l] \Delta_B^*(X \barsmash Y) \ar[d] \ar[r] & X \barsmash Y \ar[d] \\
			A \times C & \ar[l]_-{r_B} A \times B \times C \ar[r]^-{\Delta_B} & A \times B \times B \times C.
		} \]
We sometimes call this $X \odot^B Y$ if the choice of $B$ is not clear. The shadow and unit are
\[ \shad{X} := (r_A)_!(\Delta_A)^*X, \qquad U_A := (\Delta_A)_!(r_A)^* S^0 \cong A_{+(A \times A)}. \]
Intuitively, the product is a space over $A \times C$ whose fiber over $(a,c)$ is the sum of $X_{a,b} \sma Y_{b,c}$ over all $b \in B$. Another intuition is that away from the basepoint sections, this is just the fiber product $X \times_B Y$.\index{bicategory!of retractive spaces $\mc R/\cat{Top}$}

We have the rigidity theorem \autoref{prop:spaces_rigidity} waiting on standby, so as soon as we show that isomorphisms $\alpha$, $\lambda$, $\rho$, and $\theta$ exist, we will know for free that they are unique and coherent. To show these isomorphisms exist, we decompose each of the operations $\odot$, $\shad{}$, and $U_A$ into three pieces. Then we compose several smaller isomorphisms from the SMBF structure on $\mc R$, one for each small square or triangular region in the diagrams below.\footnote{These diagrams were typeset by Kate Ponto.} In fact, this procedure works for any symmetric monoidal bifibration (SMBF), not just $\mc R \to \cat{Top}$.

\begin{thm}\label{thm:SMBF_to_bicategory}
	For any SMBF $(\mc C,\boxtimes,I)$ over $\bS$, there is a shadowed bicategory $\mc C/\bS$ whose 0-cells are the objects of $\bS$, 1- and 2-cells the objects and morphisms in the fiber category $\mc C^{A \times B}$, shadow category $\mc C^*$, three operations $\odot$, $\shad{}$, $U_A$ defined by
	\[ \begin{array}{rclcrcl}
	\mc C^{A \times B} \times \mc C^{B \times C} & \overset\odot\to & \mc C^{A \times C}
	&& X \odot Y &=& (r_B)_!(\Delta_B)^*(X \boxtimes Y) \\
	\mc C^{A \times A} & \overset{\shad{}}\to & \mc C^{*}
	&& \shad{X} &=& (r_A)_!(\Delta_A)^*X \\
	{*} & \overset{U_A}\to & \mc C^{A \times A}
	&& U_A &=& (\Delta_A)_!(r_A)^* I
	\end{array} \]
	and isomorphisms $\alpha$, $\lambda$, $\rho$, $\theta$ given by \autoref{fig:assoc_indexed}, \autoref{fig:unit_indexed}, and \autoref{fig:shadow_indexed}.
\end{thm}

\begin{figure}[h]
	{
		\begin{tikzcd}[column sep=30pt,row sep=15pt]
			{\begin{pmatrix} \mathscr{C}^{A\times B}
					\\
					\mathscr{C}^{B\times C}
					\\
					\mathscr{C}^{C\times D}
				\end{pmatrix}
			}
			\ar[d,"{\boxtimes \times 1}"]
			\ar[r,"{1 \times \boxtimes}"]
			&
			{\begin{pmatrix} \mathscr{C}^{A\times B}
					\\
					\mathscr{C}^{B\times C\times C\times D}
				\end{pmatrix}
			}\ar[r,"{1 \times \Delta_C^*}"]\ar[d,"\boxtimes"]
			&
			{\begin{pmatrix} \mathscr{C}^{A\times B}
					\\
					\mathscr{C}^{B\times C\times D}
				\end{pmatrix}
			}\ar[r,"{1 \times (r_C)_!}"]\ar[d,"\boxtimes"]
			&
			{\begin{pmatrix} \mathscr{C}^{A\times B}
					\\
					\mathscr{C}^{B\times D}
				\end{pmatrix}
			}\ar[d,"\boxtimes"]
			\\
			{\begin{pmatrix} \mathscr{C}^{A\times B\times B\times C}
					\\
					\mathscr{C}^{C\times D}
				\end{pmatrix}
			}\ar[r,"\boxtimes"]\ar[d,"{\Delta_B^* \times 1}"]
			&\mathscr{C}^{A\times B\times B\times C\times C\times D}\ar[r,"\Delta_C^*"]\ar[d,"\Delta_B^*"]
			&\mathscr{C}^{A\times B\times B\times C\times D}\ar[r,"(r_C)_!"]\ar[d,"\Delta_B^*"]
			&\mathscr{C}^{A\times B\times B\times D}\ar[d,"\Delta_B^*"]
			\\
			{\begin{pmatrix} \mathscr{C}^{A\times B\times C}
					\\
					\mathscr{C}^{C\times D}
				\end{pmatrix}
			}\ar[r,"\boxtimes"]\ar[d,"{(r_B)_! \times 1}"]
			&\mathscr{C}^{A\times B\times C\times C\times D}\ar[r,"\Delta_C^*"]\ar[d,"(r_B)_!"]
			&\mathscr{C}^{A\times B\times  C\times D}\ar[r,"(r_C)_!"]\ar[d,"(r_B)_!"]
			&\mathscr{C}^{A\times B\times D}\ar[d,"(r_B)_!"]
			\\
			{
				\begin{pmatrix} \mathscr{C}^{A\times C}
					\\
					\mathscr{C}^{C\times D}
				\end{pmatrix}
			}\ar[r,"\boxtimes"]
			&
			\mathscr{C}^{A\times C\times C\times D}\ar[r,"\Delta_C^*"]
			&
			\mathscr{C}^{A\times C\times D}\ar[r,"(r_C)_!"]
			&
			\mathscr{C}^{A\times D}
		\end{tikzcd}
	}
	\caption{Associator (products of categories are written vertically in parentheses)}\label{fig:assoc_indexed}
\end{figure}
\begin{figure}[h]
	\begin{tikzcd}[column sep=50pt,row sep=30pt]
		\sC^{A\times B}\times \ast \ar[d,"1 \times I"]\ar[dr,"\cong"] \\
		\sC^{A\times B}\times \sC^\ast \ar[d,"1 \times r_B^*"]\ar[r,"\boxtimes"]
		&\sC^{A\times B}\ar[d,"r_B^*"]\ar[dr,"\cong"]
		\\
		\sC^{A\times B}\times \sC^B\ar[d,"1 \times (\Delta_B)_!"]\ar[r,"\boxtimes"]
		&\sC^{A\times B\times B}\ar[d,"(\Delta_B)_!"]\ar[r,"\Delta_B^*"]
		&\sC^{A\times B}\ar[d,"(\Delta_B)_!"]\ar[dr,"\cong"]
		\\ 
		\sC^{A\times B}\times \sC^{B\times B}\ar[r,"\boxtimes"]
		&
		\sC^{A\times B\times B\times B}\ar[r,"\Delta_B^*"]
		&
		\sC^{A\times B\times B}\ar[r,"(r_B)_!"]
		&\sC^{A\times B}
	\end{tikzcd}
	\caption{Unitor}\label{fig:unit_indexed}
\end{figure}
\begin{figure}[h]
	\begin{tikzcd}[column sep=50pt,row sep=30pt]
		{\mathscr{C}^{A\times B} \times \mathscr{C}^{B\times A}}
		\ar[r,"{\boxtimes}"]
		\ar[d,"{\boxtimes}"]
		&\mathscr{C}^{B\times A\times A\times B}
		\ar[ld,"\simeq"]
		\ar[r,"\Delta_A^*"]
		\ar[d,"\simeq"]
		&\mathscr{C}^{B\times A\times B}
		\ar[r,"(r_A)_!"]
		\ar[d,"\simeq"]
		&\mathscr{C}^{B\times B}\ar[d,"="]
		\\
		\mathscr{C}^{A\times B \times B\times A}
		\ar[d,"\Delta_B^*"]
		\ar[r,"\simeq"]
		&\mathscr{C}^{A\times A \times B\times B}
		\ar[r,"\Delta_A^*"]
		\ar[d,"\Delta_B^*"]
		&\mathscr{C}^{A\times B\times B}
		\ar[r,"(r_A)_!"]
		\ar[d,"\Delta_B^*"]
		&\mathscr{C}^{B\times B}
		\ar[d,"\Delta_B^*"]
		\\
		\mathscr{C}^{A\times B\times A}
		\ar[d,"(r_B)_!"]
		\ar[r,"\simeq"]
		&\mathscr{C}^{A\times A\times B}
		\ar[r,"\Delta_A^*"]
		\ar[d,"(r_B)_!"]
		&\mathscr{C}^{A\times B}
		\ar[r,"(r_A)_!"]
		\ar[d,"(r_B)_!"]
		&\mathscr{C}^{B}
		\ar[d,"(r_B)_!"]
		\\
		\mathscr{C}^{A\times A}
		\ar[r,"="]
		&
		\mathscr{C}^{A\times A}
		\ar[r,"\Delta_A^*"]
		&
		\mathscr{C}^{A}\ar[r,"(r_A)_!"]
		&
		\mathscr{C}^{\ast}
	\end{tikzcd}
	\caption{Rotator}\label{fig:shadow_indexed}
\end{figure}

The content of this theorem is that the isomorphisms $\alpha$, $\lambda$, $\rho$, and $\theta$ are coherent. For $\mc R/\cat{Top}$, this is true by rigidity, but it turns out the isomorphisms are coherent even when we don't have such a rigidity theorem. The literature contains at least two different proofs of this, \cite{ponto_shulman_indexed} and \cite{mp2}.\footnote{The bicategory this produces also has an obvious symmetry that flips the direction of all the 1-cells, showing that the sidedness is artificial. We are therefore free to swap the direction of everything if we wish. In \cite{ms} this is encoded by saying the bicategory is ``symmetric,'' while in \cite{mp2} this is encoded in a graphical calculus that is insensitive to the orientation of the 1-cells.}

Applying \autoref{thm:SMBF_to_bicategory} to the point-set category $\Osp$ of parametrized orthogonal spectra over all base spaces, gives the \textbf{point-set bicategory of parametrized spectra} $\Osp/\cat{Top}$. The 0-cells are spaces, 1-cells from $A$ to $B$ are parametrized orthogonal spectra over $A \times B$, and 2-cells are maps in $\Osp(A \times B)$. The product, unit, and shadow are defined just as we did for $\mc R/\cat{Top}$. By the rigidity theorem, since we know that isomorphisms $\alpha$, $\lambda$, $\rho$, and $\theta$ exist, they are also unique.\footnote{If we use (CGWH) then not all of these maps are isomorphisms, so we have to restrict to the subcategory of freely $i$-cofibrant spectra, or perhaps the further subcategory of freely $f$-cofibrant spectra. This does not introduce any difficulty in the rest of this section, so we won't comment on it again.}

\begin{rmk}
	If $K \in \mc R(A \times B)$ is a retractive space and $X \in \Osp(B \times C)$ is a parametrized spectrum, we can also define a product $K \odot X \in \Osp(A \times C)$ by following the above recipe and using the smash product with a space operation $K \barsmash X$ for the external smash product. This is uniquely naturally isomorphic to $\Sigma^\infty K \odot X$. Under this description, the monoidal fibrant replacement functor $P$ on both spaces and spectra is given by the composition product $PX \cong B^I_{+(B \times B)} \odot X$.
\end{rmk}

Applying \autoref{thm:SMBF_to_bicategory} to the homotopy category $\ho\Osp$ gives the \textbf{homotopy bicategory of parametrized spectra} $\ho \Osp/\cat{Top}$, or $\Ex$ for short. In this bicategory, the 0-cells are again topological spaces and the 1-cells are again orthogonal spectra over $A \times B$, but the 2-cells are maps in the homotopy category $\ho \Osp(A \times B)$ with the stable equivalences inverted.\index{bicategory!of parametrized spectra $\Ex$}

The bicategory $\Ex$ is the example we are most interested in, so we will take a little time to describe its product, unit, and shadow in several equivalent ways. As when we defined $\sma_B$ and $\barsma{B}$, there are four equivalent ways of defining $\odot$, $\shad{}$, and the rest of the bicategory structure on the homotopy category.
\begin{enumerate}
	\item Define $\odot$ and $\shad{}$, and $U_A$ as composites of left- and right-deformable functors
	\[ X \odot^\M Y := \L(r_B)_!\R(\Delta_B)^*(X \barsmash^\L Y), \qquad \M\shad{X} := \L(r_A)_!\R(\Delta_A)^*X, \]
	\[ U_A = \Sigma^\infty_{+(A \times A)} A = (\Delta_A)_!(r_A)^*\Sph \simeq \L(\Delta_A)_!\R(r_A)^*\Sph. \]
	Define the isomorphisms $\alpha$, $\lambda$, $\rho$, $\theta$ from the point-set level isomorphisms in the bicategory $\Osp/\cat{Top}$, using \autoref{prop:passing_natural_trans_to_derived_functors}. The relevant composites of functors are coherently deformable, using the subcategory of freely $f$-cofibrant, level $h$-fibrant spectra in the source category, and all of its images.\footnote{Note that the coherent deformability condition is slightly complicated to check for the unit isomorphism $U_A \odot X \cong X$. $X$ is freely $f$-cofibrant and level $h$-fibrant, but $U_A \barsmash X$ is only freely $f$-cofibrant. However the map
	\[ (\Delta_A)^*(U_A \barsmash X) \to \R(\Delta_A)^*(U_A \barsmash X) := (\Delta_A)^*(PU_A \barsmash X) \]
	is isomorphic at each spectrum level to
	\[ A \times_A X_n \to A^I \times_A X_n \]
	and is therefore an equivalence.}
	\item Restrict to freely $f$-cofibrant spectra. Then the resulting bicategory is right-deformable, i.e. the compositions of $\odot$ and $\shad{}$ are coherently right-deformable (see \autoref{rmk:left_deformable_bicategory}), for instance by applying $P$ to the inputs.\index{right-deformable!bicategory} Therefore the isomorphisms between their composites on the point-set level pass to the homotopy category.
	\item Start with the point-set bicategory $\Osp/\cat{Top}$ and restrict each of the categories $\Osp(A \times B)$ to the subcategory $\Osp(A\times B)^{cfu}$ containing the freely $f$-cofibrant level $h$-fibrant spectra, along with the unit $U_A$ if $A = B$. Observe that the operation $\odot$ preserves these objects and weak equivalences between them. Therefore it passes directly to operations and coherent isomorphisms between them on the homotopy categories $\ho\Osp(A \times B)^{cfu}$, giving a bicategory.\footnote{While this recipe is the simplest, we actually don't think of it as giving a shadowed bicategory because $\shad{U_A}$ has the wrong homotopy type. If we really want to use shadows with this recipe, we can, at the cost of removing the strict units $U_A$.}
	\item As above, apply \autoref{thm:SMBF_to_bicategory} to the homotopy category $\ho\Osp$. In other words, pass $\barsmash$, $f^*$, and $f_!$ and the isomorphisms between them to the homotopy category before assembling them together into $\odot$, $U_A$, and $\shad{}$.
\end{enumerate}

\begin{thm}\label{thm:four_bicategories_of_spectra}
	These give canonically isomorphic (shadowed) bicategory structures on the homotopy categories $\ho\Osp(A \times B)$.
\end{thm}

This says in essence that we can pass from the SMBF to the bicategory before or after taking homotopy categories, and the results are canonically isomorphic.

\begin{rmk}
	Description (4) has been the most common definition of $\Ex$ in the literature so far, but (2) and (3) are often more useful on a technical level. For instance, using (3) it is easy to prove that $\ho\Osp/\cat{Top}$ has a ``shadowed $n$-Fuller structure'' as defined in \cite{mp1} -- such a structure exists on the point-set category, using the rigidity theorem to check all the coherence conditions. It then passes immediately to the homotopy category. It should also be straightforward to use (3) to prove that $\Ex$ it is a symmetric monoidal bicategory, even though the list of axioms for this structure is incredibly long \cite[4.4-4.8]{stay_compact_closed}. Unfortunately (3) has a drawback that it doesn't extend well to base-change objects, where (2) is better.
\end{rmk}

\begin{proof} of \autoref{thm:four_bicategories_of_spectra}.
	The canonical isomorphisms $\odot \cong \odot'$ and $\shad{} \cong \shad{}'$ arise because in every recipe $\odot$ and $\shad{}$ are defined as composites of deformations of the same functors. So the content of this theorem is that these canonical isomorphisms commute with the associator, unitor, and rotator isomorphisms. When comparing recipes (1), (2), and (3) this is true essentially by the definition of passing an isomorphism of functors to an isomorphism between its derived functors, \autoref{prop:passing_natural_trans_to_derived_functors}. So it remains to compare (2) to (4) on the subcategory of freely $f$-cofibrant spectra.
	
	Examine \autoref{fig:assoc_indexed}, \autoref{fig:unit_indexed}, and \autoref{fig:shadow_indexed}. We derive each of the functors along each little edge by right-deriving the pullbacks with $P$, and leaving the smash products and pushforwards alone. The isomorphism furnished by (2) says, restrict to fibrant inputs at the top-left, then remove the $P$s, and use the unique point-set isomorphism between the outside routes of the diagram on this input. This large isomorphism is the composite of the isomorphisms for the small square and triangular regions arising from the SMBF $\Osp^c$. So we could get the same isomorphism by leaving the $P$s in, removing and then re-inserting them every time we traverse one square or triangle using them. In other words, we compose the right-derived isomorphisms in each little region. By the proofs of \autoref{prop:deform_bifibration} and \autoref{prop:deform_SMBF}, these right-derived isomorphisms agree with the isomorphisms coming from the SMBF structure on $\ho\Osp^{c}$. This leaves us with the isomorphism from recipe (4).
	
	This argument runs beautifully for the associator and rotator but runs into a small hiccup for the unitor: at the bottom of the second column, in $\mathscr{C}^{A\times B\times B\times B}$, the image of our input is not fibrant, hence we are at risk of not being able to remove $P$. However, it comes from a fibrant object in the category $\mathscr{C}^{A \times B \times B}$ just above. Using the proof of \autoref{prop:deform_SMBF}, we get a commuting square in the homotopy category
	\[ \xymatrix @R=1.7em{
		(\Delta_B)_!(\Delta_B)^* \ar[d] \ar[r]^-\cong & (\Delta_B)^*(\Delta_B)_! \ar[d] \\
		(\Delta_B)_!\R(\Delta_B)^* \ar[r]^-\cong & \R(\Delta_B)^*(\Delta_B)_!
	} \]
	where the horizontal maps are Beck-Chevalley isomorphisms. On fibrant inputs, the left vertical is an isomorphism, hence the right vertical is as well. Therefore on $(\Delta_B)_!$ of a fibrant input, $(\Delta_B)^*$ is equivalent to its right-derived functor. (This is essentially a formalization of the observation we made when we defined recipe (1).) The proof can then proceed as in the other cases.
\end{proof}

\begin{ex}[Dualizable objects and traces in $\Ex$]\label{ex:traces_in_ex}\hfill
	\vspace{-1em}
	
	\begin{itemize}
		\item A parametrized spectrum $X$ over $B$, can be regarded as a 1-cell between $*$ and $B$ in the bicategory $\Ex$. We will see that $X$ is dualizable over $*$ when its derived fibers are dualizable, and dualizable over $B$ when it is a retract in the homotopy category of a finite cell complex spectrum over $B$, the cells being taken from the stable model structure \autoref{thm:stable_model_structure}. In \cite{ms} dualizability over $B$ is called \textbf{Costenoble-Waner duality}, after \cite{cw_book}.\index{Costenoble-Waner duality}
		\item If $X$ is level $h$-cofibrant, its shadow of $X$ over $B$ in the homotopy category is the strict shadow of $PX$, which at each level comes out to
		\[ (X_n \times_{B \times B} B^I) / ((B \times B) \times_{B \times B} B^I). \]
		In particular, if $X = U_B = \Sigma^\infty_{+(B \times B)} B$, the shadow is the free loop space $\Sigma^\infty_+ \Lambda B$.
		\item If $X$ is a compact ENR and $f\colon X \to X$, we define the \textbf{Reidemeister trace}\index{Reidemeister trace $R(f)$} as follows. Note that $f$ is in general not a fiberwise map over $X$, so it cannot be directly interpreted as a 2-cell in $\Ex$. Instead, we think $f$ as a map
		\[ \Sph \odot \Sigma^\infty_{+(* \times X)} X \to \Sigma^\infty_{+(* \times X)} X \odot \Sigma^\infty_{+(X \times X)} X^f, \]
		where $X^f$ is just $X$ regarded as a space over $X \times X$ by the map $(f,\id_X)$. The above map of spectra arises from the isomorphism of spaces over $X$
		\[ X \to X \times_X X^f, \qquad x \mapsto (f(x),x). \]
		We note that with this twisting, $f$ is an isomorphism of parametrized spectra.\footnote{This heuristically suggests that the Reidemeister trace is the most refined lift of the Lefschetz number that we could possibly construct.} The twisted trace of this isomorphism is then a map in the homotopy category
		\[ \xymatrix{ R(f)\colon \Sph \simeq \shad{\Sph} \ar[r] & \shad{\Sigma^\infty_{+(X \times X)} X^f} \simeq \Sigma^\infty_+ \Lambda^f X } \]
		where $\Lambda^f X$ is the twisted free loop space
		\[ X^f \times_{X \times X} X^I \cong \{ \gamma\colon I \to X \ | \ \gamma(0) = f(\gamma(1)) \} \cong \textup{hoeq}(f,\id_X). \]
		We will later give a concrete geometric formula for this trace, and indicate how to show that on homology it agrees with the algebraic definition of $R(f)$.
		\item The above definition can be generalized to the \textbf{refined Reidemeister trace} $R(f,f')$ for a fibration $E \overset{p}\to B$ with compact ENR fibers and a commuting square
		\[ \xymatrix @R=1.7em{
			E \ar[r]^-f \ar[d]_-p & E \ar[d]^-p \\
			B \ar[r]_-{f'} & B.
		} \]
		We take the trace of the isomorphism
		\[ \Sigma^\infty_{+(B \times B)} B^{f'} \odot \Sigma^\infty_{+(B \times E)} E^p \to \Sigma^\infty_{+(B \times E)} E^p \odot \Sigma^\infty_{+(E \times E)} E^f \]
		given by $(p(e),e) \mapsto (f(e),e)$. This gives a map
		\[ \xymatrix{ R(f,f')\colon \Sigma^\infty_+ \Lambda^{f'} B \simeq \shad{\Sigma^\infty_{+(B \times B)} B^{f'}} \ar[r] & \shad{\Sigma^\infty_{+(E \times E)} E^f} \cong \Sigma^\infty_+ \Lambda^f E. } \]
		Using the multiplicativity of the Reidemeister trace, pasting such commuting squares bottom-to-top corresponds to composing these traces, see \cite{ponto_shulman_mult}.
	\end{itemize}
\end{ex}

If we fix a base $B$, then we can run the entire discussion above on the category $\Osp_{(B)}$ constructed in \autoref{prop:SMBF_over_b}. We first get a point-set bicategory $\Osp_{(B)}/\cat{Fib}_B$. The 0-cells are fibrations $E \to B$, and the 1- and 2-cells from $D$ to $E$ are the objects and morphisms, respectively of the category $\Osp(D \times_B E)$. The product of $X \in \mc R(D \times_B E)$ and $Y \in \mc R(E \times_B F)$ is defined by the formula
\[ X \odot_B Y := (r_E)_!(\Delta_E)^*(X \barsma{B} Y), \]
\[ \xymatrix @R=1.7em @C=4em{
	X \odot_B Y \ar[d] & \ar[l] \Delta_E^*(X \barsma{B} Y) \ar[d] \ar[r] & X \barsma{B} Y \ar[d] \\
	D \times_B F & \ar[l]_-{r_E} D \times_B E \times_B F \ar[r]^-{\Delta_E} & D \times_B E \times_B E \times_B F.
} \]
We sometimes call this $X \odot_B^E Y$ if the choice of $E$ is not clear. Similarly, the shadow and unit are
\[ \shad{X} := (r_E)_!(\Delta_E)^*X, \qquad U_E := (\Delta_E)_!(r_E)^* \Sph_B \cong \Sigma^\infty_{+(E \times_B E)} E. \]
It is helpful to imagine that these are the same constructions as before, carried out over every point of $B$. In particular, the shadow lands in spectra over $B$, and the shadow of $U_E$ is the fiberwise free loop space $\Sigma^\infty_{+B} \Lambda_B E$, where $\Lambda_B E$ consists of those loops in $\Lambda E$ that are contained in a single fiber over $B$. The same figures as before define associator, unitor, and rotator isomorphisms, and these in turn are unique by the rigidity theorem.

We then build the homotopy bicategory $\ho\Osp_{(B)}/\cat{Fib}_B$, sometimes abbreviated as $\Ex_B$, whose 0- and 1-cells are the same but whose 2-cells are morphisms in the homotopy categories $\ho\Osp(D \times_B E)$. There are again four ways to do this, which are canonically isomorphic, by the same proof as above.

\begin{ex}[Dualizable objects and traces in $\Ex_B$]\label{ex:fiberwise_reidemeister_trace}\hfill
	\vspace{-1em}
	
	\begin{itemize}
		\item A parametrized spectrum $X$ over $E$ over $B$, can be regarded as a 1-cell between $B$ and $E$ in the bicategory $\Ex_B$. It turns out that $X$ is dualizable over $B$ iff its fiber over $E$ is finite, and dualizable over $E$ iff its fiber over $B$ is finite \cite[6.3]{lind2019transfer}. This may seem backwards, but a convenient mnemonic is: to see if it's dualizable over $A$, pull it back so that it's a spectrum over $A$, then check if it's a retract of a finite cell complex.
		\item If $E \to B$ is a fibration with compact ENR fibers then $\Sigma^\infty_{+(B \times_B E)} E$ is dualizable over $E$. Given a fiberwise self-map $f\colon E \to E$, the trace of the map
		\[ \Sigma^\infty_{+(B \times_B B)} B \odot \Sigma^\infty_{+(B \times_B E)} E \to \Sigma^\infty_{+(B \times_B E)} E \odot \Sigma^\infty_{+(E \times_B E)} E^f \]
		given by $(p(e),e) \mapsto (f(e),e)$, gives a map
		\[ \xymatrix{ R_B(f)\colon \Sph_B \simeq \shad{U_B} \ar[r] & \shad{\Sigma^\infty_{+(E \times E)} E^f} \cong \Sigma^\infty_{+B} \Lambda^f_B E. } \]
		This is the \textbf{fiberwise Reidemeister trace}.\index{Reidemeister trace $R(f)$!fiberwise}\index{fiberwise!Reidemeister trace $R(f)$} Intuitively, it is just $R(f)$ carried out over every point of $B$. With a small amount of masochism, one can also define $R_B(f,f')$ for a self-map $(f,f')$ of a fibration $E \to A$ over $B$.
	\end{itemize}
\end{ex}

We end with a technical lemma about $\odot$ and $\odot_B$ at the space level that we will need for the proof of Costenoble-Waner duality. We use it in conjunction with recipe (2) to tell when the composition product agrees with the right-derived composition product.
\begin{lem}\label{lem:derived_circle_product}
	If $A, C, D$ are $h$-fibrant spaces over $B$, then for $X \in \mc R(A \times_B C)$ and $Y \in \mc R(C \times_B D)$ both $f$-cofibrant, the relative composition product $X \odot_B Y$ is $f$-cofibrant and preserves equivalences so long as $X \to C$ or $Y \to C$ is at least a $q$-fibration.
\end{lem}

\begin{cor}
	For $X \in \mc R(A \times C)$ and $Y \in \mc R(C \times D)$ both $f$-cofibrant, the composition product $X \odot Y$ is $f$-cofibrant and preserves equivalences so long as $X \to C$ or $Y \to C$ is at least a $q$-fibration.
\end{cor}

\begin{cor}
	If $A,D$ are $h$-fibrant spaces over $B$, then for $X \in \mc R(A)$ and $Y \in \mc R(D)$ both $f$-cofibrant, the relative smash product $X \barsma{B} Y$ is $f$-cofibrant and preserves equivalences so long as $X \to B$ or $Y \to B$ is at least a $q$-fibration.
\end{cor}

\begin{proof}
	Applying $\Delta_C^*$ to the square defining the external smash product gives the pushout square
	\[ \xymatrix @R=1.7em{
		(A \times_B C \times_C Y) \cup_{A \times_B C \times_B D} (X \times_C C \times_B D) \ar[d] \ar[r] & X \times_C Y \ar[d] \\
		A \times_B C \times_B D \ar[r] & \Delta_C^*(X \barsmash Y).
	} \]
	The horizontal maps are $f$-cofibrations over $A \times_B C \times_B D$, so pushing forward to $A \times_B D$ will preserve equivalences, so we only have to focus on $\Delta_C^*(X \barsmash Y)$. Since the square is a homotopy pushout square, it suffices to check that the top two terms preserve equivalences. In the top-left, this is because $A \to B$ and $D \to B$ are fibrations. In the top-right, this uses the assumption that $X \to C$ or $Y \to C$ is a $q$-fibration.
\end{proof}

We can also apply this lemma to the shadow by making the following observation. If $X \in \mc R(A \times C)$ and $Y \in \mc R(C \times A)$, then we can regard them as 1-cells in $\Ex$ and form their circular product $\shad{X \odot Y}$. Or, we could regard $X$ as a 1-cell from $*$ to $A\times C$ and $Y$ as a 1-cell from $A \times C$ to $*$ and take their composition product $X \odot^{A \times C} Y$. These two constructions are canonically isomorphic, because both return the external smash product $X \barsmash Y$ pulled back along the diagonal maps of $A$ and $C$, then pushed forward to a point.

\begin{cor}\label{lem:derived_shadow}
	For $X \in \mc R(A \times C)$ and $Y \in \mc R(C \times A)$ both $f$-cofibrant, the circular product $\shad{X \odot Y}$ is $f$-cofibrant and preserves equivalences so long as $X \to A \times C$ or $Y \to A \times C$ is at least a $q$-fibration.
\end{cor}

The same reasoning applies rel $B$ to $X \in \mc R(A \times_B C)$ and $Y \in \mc R(C \times_B A)$, using the fact that in the following diagram the square is a homotopy pullback.
\[ \xymatrix @R=1.7em{
	A \times_B C \ar[d] \ar[r] & A \times_B C \times_B A \ar[d] \ar[r] & A \times_B C \times C \times_B A \\
	A \ar[d] \ar[r] & A \times_B A & \\
	B
} \]
\begin{cor}
	For $A$ and $C$ $h$-fibrant over $B$ and $X \in \mc R(A \times_B C)$ and $Y \in \mc R(C \times_B A)$ both $f$-cofibrant, the circular product $\shad{X \odot_B Y}_B$ is $f$-cofibrant and preserves equivalences so long as $X \to A \times_B C$ or $Y \to A \times_B C$ is at least a $q$-fibration.
\end{cor}	
		
\beforesubsection
\subsection{Properties of the trace; base-change objects}\aftersubsection

It gets easier to use the bicategorical trace when we know a few basic properties. The first few properties are direct generalizations of additivity, multiplicativity, and composition for traces in a symmetric monoidal category (\autoref{prop:trace_properties}). Again, for ease of exposition, we state these properties without the extra twisting by $Q$ and $P$.
\begin{prop}\label{prop:bicat_trace_properties}\hfill
	\vspace{-1em}
	
	\begin{itemize}
		\item (Additivity) If each category $\mathscr B(A,B)$ is additive under $\oplus$, $\odot$ distributes over $\oplus$, and the shadow category $\mathscr B_{\shad{}}$ and shadow $\shad{}$ are both additive, then given two 2-cells $f_X\colon X \to X$, $f_Y\colon Y \to Y$ of right-dualizable 1-cells from $A$ to $B$, we have
		\[ \tr(f_X \oplus f_Y) = \tr(f_X) + \tr(f_Y). \]
		More generally, if these categories have compatible triangulated structures, given a map of cofiber sequences of 1-cells from $A$ to $B$
		\[ \xymatrix @R=1.7em @C=3em{
			X \ar[d]^-{f_X} \ar[r] & Z \ar[d]^-{f_Z} \ar[r] & Y \ar[d]^-{f_Y} \\
			X \ar[r] & Z \ar[r] & Y
		} \]
		we have $\tr(f_Z) = \tr(f_X) + \tr(f_Y)$ \cite{ponto_shulman_linearity}.
		\item (Multiplicativity) Given right-dualizable 1-cells $X \in \mathscr B(A,B)$ and $Y \in \mathscr B(B,C)$, and self-maps $f_X,f_Y$, we have by \cite[4.5.5]{ponto_asterisque}
		\[ \tr(f_X \odot f_Y) \cong \tr(f_Y) \circ \tr(f_X). \]
		\item (Composition) If $\mathscr B$ has a notion of $n$-fold tensor product that commutes with $\odot$ as described in \cite[\S 5]{mp1}, given a cycle of maps of right-dualizable 1-cells from $A$ to $B$
		\[ \xymatrix{
			X_0 \ar[r]^-{f_1}& X_1 \ar[r]^-{f_2}& \ldots  \ar[r]^-{f_{n-1}}& X_{n-1} \ar[r]^-{f_1}& X_0
		} \]
		if we regard their tensor product as a self-map of a 1-cell from $A^{\times n}$ to $B^{\times n}$,
		\[ \psi(f_1,\ldots,f_n)\colon (X_0 \otimes X_1 \otimes \ldots \otimes X_{n-1}) \to (X_0 \otimes X_1 \otimes \ldots \otimes X_{n-1}), \]
		its trace is equal to the trace of the composite
		\[ \tr(\psi(f_1,\ldots,f_n)) = \tr(f_n \circ \ldots \circ f_1). \]
		As a corollary, the trace has cyclic invariance,
		\[ \tr(f_n \circ \ldots \circ f_2 \circ f_1) = \tr(f_1 \circ f_n \circ \ldots \circ f_2), \]
		though this corollary is true without any additional structure on $\mathscr B$ \cite[4.5.4]{ponto_asterisque}.
	\end{itemize}
\end{prop}
There is a string diagram calculus that gives an elegant proof of the multiplicativity property -- it is like the calculus for monoidal categories except that the strings are drawn on a cylinder instead of the plane. See \cite{ponto_shulman_general}.

The next property generalizes the fact that the trace in a symmetric monoidal category are preserved by a strong symmetric monoidal functor. A \textbf{strong shadow functor}\index{strong shadow functor} $F\colon \mathscr B \to \mathscr C$ is a homomorphism of shadowed bicategories. In detail, it is given by
\begin{itemize}
	\item A function of 0-cells $\ob \mathscr B \to \ob \mathscr C$, abusively called $F$.
	\item For each pair of 0-cells $A$, $B$, a functor $F\colon \mathscr B(A,B) \to \mathscr C(FA,FB)$.
	\item A functor $F\colon \mathscr B_{\shad{}} \to \mathscr C_{\shad{}}$.
	\item Natural isomorphisms making the following three squares commute.
	\[
	\begin{array}{cc}
	\xymatrix @R=1.5em{
		\mathscr B(A,B) \times \mathscr B(B,C) \ar[d]_-{F} \ar[r]^-{\odot} & \mathscr B(A,C) \ar[d]^-{F} \\
		\mathscr C(FA,FB) \times \mathscr C(FB,FC) \ar[r]_-{\odot} & \mathscr C(FA,FC)
	}
	& \raisebox{-2em}{$m_F\colon F(X) \odot F(Y) \cong F(X \odot Y)$} \\[5em]
	\xymatrix @R=1.5em{
		\mathscr B(A,A) \ar[d]_-{F} \ar[r]^-{\shad{}} & \mathscr B_{\shad{}} \ar[d]^-{F} \\
		\mathscr C(FA,FA) \ar[r]_-{\odot} & \mathscr C_{\shad{}}
	}
	& \raisebox{-2em}{$s_F\colon \shad{F(X)} \cong F \shad{X}$} \\[5em]
	\xymatrix @R=1.5em{
		{*} \ar@{=}[d] \ar[r]^-{\odot} & \mathscr B(A,A) \ar[d]^-{F} \\
		{*} \ar[r]_-{\odot} & \mathscr C(FA,FA)
	}
	& \raisebox{-2em}{$i_F\colon U_{FA} \cong F(U_A)$}
	\end{array}
	\]
\end{itemize}
This must satisfy four coherence conditions, stating in effect that the four isomorphisms $\alpha$, $\lambda$, $\rho$, and $\theta$ agree with the action of $F$. 
If we drop all the references to $\shad{}$ and $\theta$, we get the notion of a \textbf{pseudofunctor}\index{pseudofunctor}, or homomorphism of bicategories. A strong shadow functor is an \textbf{equivalence}\index{bicategory!equivalence of} if it induces equivalences on the categories of 1-cells and 2-cells, and every 0-cell is connected to a 0-cell in the image by an least one invertibly dualizable 1-cell.

\begin{ex}[Strong shadow functors]\hfill
	\vspace{-1em}
	
	\begin{itemize}
		\item Thinking of each symmetric monoidal category as a shadowed bicategory, each strong symmetric monoidal functor $F$ is also a strong shadow functor.
		\item The isomorphisms of shadowed bicategory structures described in \autoref{thm:four_bicategories_of_spectra} can be interpreted as strong shadow functors in which $F$ is the identity but $m_F$, $s_F$, and $i_F$ are not.
		\item Fiberwise suspension spectrum defines a strong shadow functor $\Sigma^\infty\colon \mc R/\cat{Top} \to \mc Osp/\cat{Top}$ from the point-set bicategory of spaces to the point-set bicategory of spectra. The isomorphisms and their coherence all come for free by rigidity. Passing to cofibrant-fibrant spaces, we get a strong shadow functor on the homotopy categories as well.
		\item Restricting from $B$ to a single point $b \in B$ defines a strong shadow functor $\Ex_B \to \Ex$.
	\end{itemize}
\end{ex}

\begin{thm}\label{bicategory_map_preserves_traces}\cite[8.3]{ponto_shulman_general}
	Suppose $F$ is a strong shadow functor and $(M,N)$ is a dual pair in $\sB$.
	\begin{enumerate}
		\item Then $(F(M),F(N))$ forms a dual pair in $\sC$, with coevaluation and evaluation maps canonically isomorphic to $F$ of the same maps for $(M,N)$.
		\item  For any $f\colon Q\odot M\rightarrow M\odot P$, $F(f)$ is canonically isomorphic to a map $$F(Q) \odot F(M) \rightarrow F(M) \odot F(P),$$ and the trace of this map is canonically isomorphic to $F(\tr(f))$. In other words, $F$ commutes with bicategorical traces.
	\end{enumerate}
\end{thm}

So for instance, the fiberwise Reidemeister trace constructed in $\Ex_B$ produces a map of spectra over $B$, that over each point of $b$ is isomorphic to the ordinary Reidemeister trace carried out on that fiber.

Finally we discuss base-change 1-cells. Before trying to formalize them, we will simply say what they are in all the examples we have encountered so far.\index{base-change objects}

\begin{ex}[Base-change 1-cells]\hfill
	\vspace{-1em}
	
	\begin{itemize}
		\item In the shadowed bicategory $\mathscr Bimod_k$ of $k$-algebras and bimodules, for any homomorphism $f\colon A \to B$ of $k$-algebras, the base-change 1-cell $$\bcr{A}{f}{B} = {}_A B_B$$ is the $(A,B)$ bimodule given by $B$ with the evident $A$ and $B$ actions. Notice that $- \odot \bcr{A}{}{B}$ extends scalars from $A$ to $B$, while $\bcr{A}{}{B} \odot -$ restricts scalars (i.e. restricts the $B$-action to an $A$-action without changing the underlying abelian group).
		\item In $Ch\mathscr Bimod_k$, $\mathscr Bimod_{\Sph}$, and $\mathscr Act$, any homomorphism of rings or groups $A \to B$ gives a base-change 1-cell $\bcr{A}{}{B} = {}_A B_B$ by essentially the same rule as above. In all of these settings, multiplying by this 1-cell induces derived extension and restriction of scalars.
		\item In $\mc R/\cat{Top}$, for any map of spaces $f\colon A \to B$, the base-change 1-cell $\bcr{A}{}{B}$ is the retractive space $A_{+(A \times B)}$ with $A$ mapping to the base by $(\id,f)$. Notice that $- \odot \bcr{A}{}{B}$ is isomorphic to the pushforward $f_!$ and $\bcr{A}{}{B} \odot -$ is isomorphic to the pullback $f^*$.
		\item Similarly in $\Ex = \ho\mc Osp/\cat{Top}$, for any map of spaces $f\colon A \to B$, the base-change 1-cell $\bcr{A}{}{B}$ is the parametrized spectrum $\Sigma^\infty_{+(A \times B)} A$. Again $- \odot \bcr{A}{}{B}$ is isomorphic to the (derived) pushforward $f_!$ and $\bcr{A}{}{B} \odot -$ is isomorphic to the (derived) pullback $f^*$.
	\end{itemize}
\end{ex}

The explanation for the notation $\bcr{A}{}{B}$ is that in the ``ring'' examples the object is $B$, but in the ``parametrized'' examples it is $A$. What they all have in common is that the bicategorical product with the base-change 1-cell on either side is isomorphic to some familiar operation and its right adjoint.

As $f$ varies, the base-change 1-cells form a pseudofunctor from some ``base category'' $\bS$ into the bicategory $\mathscr B$. This means we have canonical isomorphisms
\[ \bcr{A}{f}{B} \odot \bcr{B}{g}{C} \cong \bcr{A}{g \circ f}{C} \]
\[ \bcr{A}{\id}{A} \cong U_A \]
that satisfy the same coherences as for a monoidal functor. 
In the ``ring'' examples, these isomorphisms are fairly obvious before passing to the homotopy category. To argue that they pass to the homotopy category it's enough to see that the maps
\[ Q\bcr{A}{}{B} \odot Q\bcr{B}{}{C} \to \bcr{A}{}{B} \odot \bcr{B}{}{C} \]\index{left-deformable!base-change objects}
are all equivalences. The construction of these isomorphisms in the ``parametrized'' examples will be done in \autoref{prop:base_change_all_isomorphic}.

The other similarity between these cases is that the base-change objects are always dualizable over the ``target,'' so that the duals form another pseudofunctor pointing the other way, that we might call $\bcl{A}{}{B}$. We will not formalize base-change 1-cells any further than this, but see \cite{shulman_framed_monoidal}.

\begin{ex}[Duals of base-change 1-cells]\hfill
	\vspace{-1em}
	
	\begin{itemize}
		\item In bimodules, the dual of $\bcr{A}{}{B}$ over $B$ is $\bcl{A}{}{B} = {}_B B_A$ defined in the same way but with the $A$-action on the right. The coevaluation and evaluation maps are induced by $f$ and multiplication:
		\[ \xymatrix{ A \ar[r]^-c & B \otimes_B B = B & B \otimes_A B \ar[r] & B. } \]
		The triangle identities reduce to the fact that $f$ is a ring homomorphism. The same rule describes the duals of the base-change 1-cells in $Ch\mathscr Bimod_k$, $\mathscr Bimod_{\Sph}$, and $\mathscr Act$.
		\item In $\mc R/\cat{Top}$, the dual of $\bcr{A}{}{B}$ over $B$ is $\bcl{A}{}{B} = A_{+(B \times A)}$. The coevaluation and evaluation maps arise from the maps of spaces given by the diagonal of $A$ and by $f$:
		\[ \xymatrix{ A \ar[r]^-c & A \times_B A & A \times_A A \ar[r] & B. } \]
		The same applies to parametrized spectra $\Ex = \ho\mc Osp/\cat{Top}$.
	\end{itemize}
\end{ex}

In a shadowed bicategory with base-change objects, the trace of the identity of a base-change 1-cell $\bcr{A}{}{B}$ gives a map of shadows $\shad{U_A} \to \shad{U_B}$. This map is usually something very familiar. For instance, tracing the identity of $_A B_B$ gives the map $HH(A) \to HH(B)$ induced by the ring homomorphism $f$, and similarly for topological Hochschild homology. Tracing the identity of $\Sigma^\infty_{+(A \times B)} A$ in parametrized spectra gives the map $\Sigma^\infty_+ LA \to \Sigma^\infty_+ LB$ arising from the map of spaces $f\colon A \to B$.

In some cases, we can also trace the identity of $\bcr{A}{}{B}$ over $A$ instead of $B$. This gives a \textbf{transfer} or wrong-way map on Hochschild homology or the free loop space. It does not always exist -- we have to assume dualizability over $A$, which amounts to some finiteness condition on $f$. In the ring case, this dualizability is equivalent to $B$ being finitely built as an $A$-module, while in the space case it corresponds to $A \to B$ having homotopy fiber that is finite CW or finitely dominated.

The Reidemeister trace defined in \autoref{ex:traces_in_ex} is a twisted version of this transfer, as discussed at length in \cite{lind2019transfer,campbell_ponto}. It is the twisted trace of the isomorphism of base-change objects
\[ \xymatrix{ \bcr{X}{}{*} \ar[r]^-\cong & \bcr{X}{}{X} \odot \bcr{X}{}{*} } \]
arising from the trivial fact that the composite $X \overset{f}\to X \to *$ agrees with the unique map $X \to *$.

To compare Reidemeister traces in different bicategories, we need to know that the base-change objects are preserved. Given a strong shadow functor $F\colon \sB \to \sC$ if $\sB$ and $\sC$ have base-change objects then we ask for $F$ to come with a functor of base categories $F\colon \bS \to \bT$, and isomorphisms $\bcr{A}{f}{B} \cong \bcr{FA}{Ff}{FB}$ that agree with the compositions and units up to coherent isomorphism. See \cite[\S 14]{mp2} for the axioms.

One situation where this happens is when $F$ came from a map of symmetric monoidal bifibrations. Each SMBF $\sC \to \bS$ gives a bicategory $\calBi CS$, as we have seen in \autoref{thm:SMBF_to_bicategory}, but this bicategory also has a canonical set of base-change objects. The base-change object for a map $f\colon A \to B$ in the base category $\bS$ is given by
\[ \bcr{A}{f}{B} = (\id,f)_!(r_A)^* I \in \sC^{A \times B} \]
and the isomorphism $\bcr{A}{\id}{A} \cong U_A$ is because they are both defined by the same formula $(\Delta_A)_!(r_A)^* I$. The composition isomorphisms arise from \autoref{fig:base_change_composition_indexed}, which is formally similar to the definition of the unitor isomorphism from \autoref{fig:unit_indexed}.

\begin{figure}[h]
	\begin{tikzcd}[column sep=50pt,row sep=30pt]
		{\ast}
		\ar[d,"I \times I"]
		\ar[rd,"I"]
		\\
		{\mathscr{C}^{\ast} \times \mathscr{C}^{\ast}}
		\ar[r,"{\boxtimes}"]
		\ar[d,"r_A^* \times r_B^*"]
		& \mathscr{C}^{\ast}
		\ar[d,"r_{A \times B}^*"]
		\ar[rd,"r_A^*"]
		\\
		{\mathscr{C}^{A} \times \mathscr{C}^{B}}
		\ar[r,"{\boxtimes}"]
		\ar[d,"{(1,f)_! \times (1,g)_!}"]
		&\mathscr{C}^{A\times B}
		\ar[r,"{(1,f)^*}"]
		\ar[d,"{((1,f)\times(1,g))_!}"]
		&\mathscr{C}^{A}
		\ar[rd,"{(1,g \circ f)_!}"]
		\ar[d,"{(1,f,g \circ f)_!}"]
		\\
		{\mathscr{C}^{A \times B} \times \mathscr{C}^{B \times C}}
		\ar[r,"\boxtimes"]
		&\mathscr{C}^{A\times B \times B\times C}
		\ar[r,"(1\Delta_B1)^*"]
		&\mathscr{C}^{A\times B\times C}
		\ar[r,"(1r_B1)_!"]
		&\mathscr{C}^{A\times C}
	\end{tikzcd}
	\caption{Base change composition 
	}\label{fig:base_change_composition_indexed}
\end{figure}

\begin{prop}\cite[14.1]{mp2}\label{prop:little_bit_of_functoriality}
	Each map of symmetric monoidal bifibrations $H\colon (\sC,\bS) \to (\sD,\bT)$ induces a strong shadow functor $F\colon \calBi CS \to \calBi DT$, and a coherent isomorphism of base-change objects $F \circ [] \cong [] \circ H$. Therefore the trace of an isomorphism of base-change objects in $\sC$ is carried to a trace of an isomorphism of base-change objects in $\sD$.
\end{prop}

We will use this for \autoref{thm:reidemeister_traces_agree}, showing that the Reidemeister trace in parametrized spectra is equivalent to a similar Reidemeister trace defined in ring spectra.

For now, if we want to use this formal description to define the Reidemeister trace in parametrized spectra, we at least need to know that different ways of constructing this isomorphism give the same result. The recipes (1), (2)\index{right-deformable!base-change objects} and (4) from \autoref{thm:four_bicategories_of_spectra} extend in a natural way to base-change objects and their composition isomorphisms, by the same construction that we used for the unit isomorphism.\footnote{Recipe (3) fails because the base-change objects aren't fibrant, and if you include them, tensoring on one side has the effect of pushing forward fibrant spectra, which creates additional non-fibrant spectra.} The same proof then establishes

\begin{prop}\label{prop:base_change_all_isomorphic}
	These three pseudofunctors from spaces into $\Ex$ are canonically isomorphic.
\end{prop}

Specifically, along the canonical isomorphisms between the different models for $\odot$ and $U_A$, the composition and unit isomorphisms for these pseudofunctors agree with each other. Therefore the Reidemeister trace in parametrized spectra can be defined as the trace of the point-set map described in \autoref{ex:traces_in_ex}, or the formal isomorphism of composites of base-change objects just above.

\beforesubsection
\subsection{Costenoble-Waner duality and applications}\aftersubsection

Recall from \autoref{ex:traces_in_ex} that a spectrum over $X \times *$ is called Costenoble-Waner dualizable if it is dualizable over $X$ in the bicategorical sense. In \autoref{sec:applications}, we will recall a full characterization of Costenoble-Waner dualizable spectra, similar to the classification of Spanier-Whitehead dualizable spectra. In this section, we focus on compact ENRs and bundles of such, and give explicit formulas for the coevaluation and evaluation maps, from which we get a geometric description of the Reidemeister trace.

For any space $X$, the parametrized spectrum $\bcr{X}{}{*} = \Sigma^\infty_{+(X \times *)} X$ gives a 1-cell in $\Ex$ from $X$ to $*$. For this to be dualizable over $X$, by \autoref{def:dual_in_bicategory} we need a parametrized spectrum $Y$ over $X$, regarded as a 1-cell from $*$ to $X$, a coevaluation map of in the homotopy category spectra
\[ U_* = \Sph \to Y \odot \Sigma^\infty_{+(X \times *)} X, \]
and an evaluation map in the homotopy category of spectra over $X \times X$
\[ \Sigma^\infty_{+(X \times *)} X \odot Y \to \Sigma^\infty_{+(X \times X)} X = U_X, \]
satisfying two triangle identities. We identify $X$ with $X^I$ along inclusion of constant paths, so that we can have the evaluation map land in $X^I_{+(X \times X)}$ instead of $X_{+(X \times X)}$.

To describe $Y$ explicitly, we first assume that $X$ is a finite complex. Let $i\colon X \to \R^n$ be an embedding with mapping cylinder neighborhood $N$ and projection $p\colon N \to X$. We pick an $\epsilon > 0$ so that $N_{\epsilon} \subseteq N$, in other words the closed tube of radius $\epsilon$ about $X$ is contained in $N$. Let $N /_X \partial N$ denote the fiberwise quotient over $X$, as a 1-cell from $*$ to $X$. We have canonical isomorphisms
\[ N /_X \partial N \odot X_{+(X \times *)} \cong N / \partial N, \qquad
X_{+(X \times *)} \odot N /_X \partial N \cong (X \times N)/_{(X \times X)} (X \times \partial N). \]
The slickest way to prove that these are homeomorphisms is to remember that $X_{+(X \times *)}$ is a base-change object, so the first expression pushes forward $N /_X \partial N$ to a point, while the second pulls it back to $X \times X$.

Define two closed subspaces of $X \times N$ by
\[ \Delta_X = \{ (x,i(x)) : x \in X \},
\qquad
\Gamma_\epsilon = \{ (x,u) : d(i(x),u) \leq \epsilon \}. \]
Intuitively, $\Delta_X$ is an embedded copy of $X$ and $\Gamma_\epsilon$ is its tubular neighborhood. For every $(x,u) \in \Gamma_\epsilon$, since the line segment from $i(x)$ to $u$ is contained in $N_\epsilon$, we may apply $p$ to that line segment and get a path in $X$ that we call $\gamma_{x,u}$:
\[ \gamma_{x,u}(t) = p((1-t)i(x) + tu). \]
This defines a continuous function $\Gamma_\epsilon \to X^I$ of spaces over $X \times X$.

\begin{thm}[Costenoble-Waner duality, neighborhood version] \label{thm:cw_duality_1}\cite[D]{klein2001dualizing}\cite[18.5]{ms}\cite[2.4]{ww1}
	If $X$ is a finite cell complex, the dual of $\Sigma^\infty_{+(X \times *)} X$ over $X$ is the 1-cell from $*$ to $X$ given by $F_n (N/_X\partial N)$, with coevaluation and evaluation maps
	\begin{align*}
	F_n S^n &\to F_n (N/_X\partial N) \odot X_{+(X \times *)} \\
	X_{+(X \times *)} \odot F_n (N/_X\partial N) &\to F_n S^n \barsmash X^I_{+(X \times X)}
	\end{align*}
	obtained by applying $F_n$ to the ``collapse'' and ``scanning'' maps
	\[ \begin{array}{cclcc}
	S^n & \overset{c}\to & N/\partial N && (X \times N)/_{(X \times X)} (X \times \partial N) \overset{e}\to S^n_\epsilon \barsmash X^I_{+(X \times X)} \\[.5em]
	v &\mapsto & \left\{ \begin{array}{ccl} v && v \in N \\ {*} && v \not\in N \end{array} \right. && (x,u) \mapsto \left\{ \begin{array}{ccl} (u-i(x),\gamma_{x,u}) && (x,u) \in \Gamma_\epsilon \\ {*} && (x,u) \not\in \Gamma_\epsilon \end{array} \right.
	\end{array} \]
	of retractive spaces over $*$ and $X \times X$, respectively.
\end{thm}\index{Costenoble-Waner duality}

\begin{rmk}\label{twisted_poincare_duality}
	For a smooth closed manifold $M$ of dimension $d$, the quotient $N/_M \partial N$ is an $(n-d)$-dimensional sphere bundle over $M$, constructed from the normal bundle by taking a fiberwise one-point compactification. Therefore \autoref{thm:cw_duality_1} gives a Poincar\' e duality theorem with coefficients in any bundle of spectra $\mc E$ over $M$. To deduce it, consider the operation $r_M^*(-) \cong \bcr{M}{}{*} \odot -$ from spectra over $*$ to spectra over $M$. Its right adjoint in the homotopy category, when applied to a freely $f$-cofibrant level $h$-fibrant spectrum $\mc E$, is both $\Gamma_M(\mc E)$ and $F_n N/_M \partial N \odot \mc E$, so we get an isomorphism
	\[ H^q(X;\mc E) = \pi_{-q}(\Gamma_M(\mc E)) \cong \pi_{-q}(F_n N/_M \partial N \odot \mc E) \cong H_{d-q}(M;\Sigma^{d-n}_M(N/_M \partial N) \sma_M \mc E), \]
	cf \cite{ms,malkiewich2014duality}. The fibers of $\Sigma^{d-n}_M(N/_M \partial N)$ are sphere spectra $\Sph$, so smashing it with $\mc E$ gives a new parametrized spectrum $\ti{\mc E}$ with the same fibers as $\mc E$, but perhaps different monodromy. So we get a twisted Poincar\' e duality theorem with spectrum coefficients,
	\[ H^q(X;\mc E) \cong H_{d-q}(M;\ti{\mc E}). \]
	This is particularly useful when $M$ is ``$\mc E$-orientable,'' meaning the two coefficient systems $\ti{\mc E}$ and $\mc E$ are equivalent. When $\mc E$ is a bundle of Eilenberg-Maclane spectra, the above theorem becomes the variant of Poincar\' e duality for non-orientable manifolds, the two coefficient systems differing by a twist by the orientation sheaf of $M$.
\end{rmk}

To generalize \autoref{thm:cw_duality_1} to $X$ a compact ENR, we recall from \autoref{sec:sw_proof} that when $(X,A)$ is a pair over $B$ and $(X',A')$ is a pair over $B'$, the pair $(X,A) \square (X',A')$ is over $B \times B'$ and there is a fiberwise homotopy equivalence
\[ \mu\colon C^u_B(X,A) \barsmash C^u_{B'}(X',A') \to C^u_{B \times B'}((X,A) \square (X',A')). \]

Now let $X$ be a compact ENR, $i\colon X \to \R^n$ a topological embedding, and $p\colon N \to X$ be a map that retracts the closed neighborhood $N$ back to $X$. Again, we may or may not ask for $\partial N \to N$ to be an $h$-cofibration. We take the dual to be $F_n C^u_{* \times X}(N,N-X)$ as a 1-cell from $*$ to $X$. The coevaluation map becomes any route through the following diagram from top-left to bottom-right.
\[ \xymatrix{
	S^n \ar[d] \ar[r] & N/\partial N \ar[r]^-\cong & N/_X \partial N \odot X_{+(X \times *)} \\
	C(S^n,S^n - \textup{int} N) \ar[d] \ar[ur]_-{(\sim)} & \ar[l] C^u(N,\partial N) \ar[u]_-{(\sim)} \ar[d] \ar[r]^-\cong & C^u_{* \times X}(N,\partial N) \odot X_{+(X \times *)} \ar[u]_-{(\sim)} \ar[d] \\
	C(S^n,S^n-X) & \ar[l]_-\sim C^u(N,N-X) \ar[r]^-\cong & C^u_{* \times X}(N,N-X) \odot X_{+(X \times *)} \\
} \]
The evaluation map becomes the left edge of the following diagram.
\[ \resizebox{\textwidth}{!}{
	\xymatrix{
	X_{+(X \times *)} \odot C^u_{* \times X}(N,N-X) \ar[d]^-\cong_-\mu &
	X_{+(X \times *)} \odot C^u_{* \times X}(N,\partial N) \ar[r]_-{(\sim)} \ar[l] \ar[d]^-\cong_-\mu &
	X_{+(X \times *)} \odot N/_X\partial N \ar[d]^-\cong \\
	C^u_{X \times X}(X \times N,(X \times N) - (X \times X)) \ar[d] &
	C^u_{X \times X}(X \times N,X \times \partial N) \ar[r]_-{(\sim)} \ar[l] \ar[d] &
	(X \times N)/_{(X \times X)}(X \times \partial N) \ar[d] \\
	C^u_{X \times X}(X \times N,(X \times N)-\Delta_X) &
	C^u_{X \times X}(X \times N,(X \times N)- \Gamma_{\frac\epsilon{2}}) \ar[r] \ar[l] &
	(X \times N)/_{(X \times X)}((X \times N) - \textup{int} \Gamma_{\frac\epsilon{2}}) \\
	C^u_{X \times X}(\Gamma_\epsilon,\Gamma_\epsilon - \Delta_X) \ar[u]_-\sim \ar[d] &
	C^u_{X \times X}(\Gamma_\epsilon, \Gamma_\epsilon - \Gamma_{\frac\epsilon{2}}) \ar[u]_-\sim \ar[d] \ar[r] \ar[l] &
	\Gamma_{\frac\epsilon{2}}/_{(X \times X)}\partial \Gamma_{\frac\epsilon{2}} \ar[u]_-\cong \ar[d] \\
	C^u(\R^n,\R^n-0) \barsmash X^I_{+(X \times X)} &
	C^u(\R^n,\R^n - B_{\frac\epsilon{2}}) \barsmash X^I_{+(X \times X)} \ar[r]_-{(\sim)} \ar[l] &
	B_{\frac\epsilon{2}} / \partial B_{\frac\epsilon{2}} \barsmash X^I_{+(X \times X)} \\
	& \ar@{<->}[u]_-\sim \ar@{<->}[ur]_-\sim\ar@{<->}[lu]_-\sim S^n \barsmash X^I_{+(X \times X)} &
}
} \]
All unlabeled maps are inclusions or collapses, except for the vertical maps in the last row, which arise from the formula $(x,u) \mapsto (u - i(x),\gamma_{x,u})$. The homeomorphism in the final column arises by observing that $X \times N$ is the pushout of $\Gamma_{\frac\epsilon{2}}$ and the closure of its complement along $\partial \Gamma_{\frac\epsilon{2}}$, and so we get this homeomorphism by comparing universal properties.

\begin{thm}[Costenoble-Waner duality, mapping cone version]\label{thm:cw_duality_2}
	If $X$ is a compact ENR, the dual of $\Sigma^\infty_{+(* \times X)} X$ over $X$ is $F_n C^u_{* \times X}(N,N-X)$, with coevaluation and evaluation maps
	\begin{align*}
	F_n S^n &\to F_n C^u_{* \times X}(N,N-X) \odot \Sigma^\infty_{+(X \times *)} X \\
	\Sigma^\infty_{+(X \times *)} X \odot F_n C^u_{* \times X}(N,N-X) &\to F_n S^n \barsmash X^I_{+(X \times X)}
	\end{align*}
	obtained by applying $F_n$ to the collapse and scanning maps defined above.
\end{thm}\index{Costenoble-Waner duality}

\begin{rmk}
	We can re-derive Spanier-Whitehead duality from Costenoble-Waner duality. If we compose the dual pair $(X_{+(X \times *)},F_n C_{* \times X}(N,N-X))$ with the dual pair $(X_{+(* \times X)},X_{+(X \times *)})$, the resulting dual pair $(X_+,F_n C(N,N-X))$ is just a pair of spectra over $* \times *$, with the same coevaluation and evaluation maps as before.
\end{rmk}

\begin{rmk}
	Using \autoref{thm:cw_duality_2}, the twisted Poincar\' e duality theorem from \autoref{twisted_poincare_duality} generalizes from a smooth manifold to any Poincar\' e duality space $X$, because the condition of being a Poincar\'e duality space is equivalent to $C_X(N,N-X) \to X$ being stably equivalent to a spherical fibration, see \cite{browder_spivak}, \cite[Thm A]{klein2001dualizing}. The same argument as in \autoref{twisted_poincare_duality} gives an isomorphism
	\[ H^q(X;\mc E) \cong H_{d-q}(X;\ti{\mc E}) \]
	where $d$ is the formal dimension of the Poincar\' e duality space $X$. In particular, we get this isomorphism when $X$ is any closed topological manifold.
\end{rmk}

\begin{cor}\label{cor:reidemeister_formula}
	When $X$ is a compact ENR, the Reidemeister trace $R(f)$ of a map $f\colon X \to X$ is the desuspension of the map of spaces
	\[ \xymatrix @R=0em{
		S^n \ar[r] & S^n_\epsilon \sma (\Lambda^f X)_+ \\
		v \ar@{|->}[r] & (v - f(p(v))) \sma \gamma_{f(p(v)),v}.
	} \]
\end{cor}\index{Reidemeister trace $R(f)$}

\begin{proof}
	We use the observation before \autoref{lem:derived_shadow} to express our shadows as composition products. When $X$ has a mapping cylinder neighborhood, we compose the maps of the Reidemeister trace to get
	\[ \resizebox{\textwidth}{!}{
		\xymatrix @R=0.5em{
		S^n \ar[r] &
		N/_X\partial N \odot^X X_{+(X \times *)} \ar[r]^-\cong &
		N/_X\partial N \odot^X \bcr{X}{f}{X} \odot^X X_{+(X \times *)} \ar[dd]^-\cong \\
		\\
		S^n_\epsilon \sma (\Lambda^f X)_+ &
		\ar[l]_-\cong S^n_\epsilon \barsmash X^I_{+(X \times X)} \odot^{X \times X} \bcr{X}{f}{X} &
		\ar[l] (X \times N)/_{(X \times X)}(X \times \partial N) \odot^{X \times X} \bcr{X}{f}{X} \\
		v \ar@{|->}[r] & (v,p(v)) \ar@{|->}[r] & (v,p(v),f(p(v)) \ar@{|->}[dd] \\
		\\
		(v - f(p(v)),\gamma_{f(p(v)),v}) & (v - f(p(v)),\gamma_{f(p(v)),v},p(v)) \ar@{|->}[l] & (f(p(v)),v,p(v)) \ar@{|->}[l]
		}
	} \]
	Applying $P$ to the base-change 1-cell makes all the products derived, but does not change them up to equivalence, so this composite is the trace as calculated in the homotopy category. If $X$ is a general compact ENR, we argue that this composite is still the trace of $f$ using the commuting diagrams before \autoref{thm:cw_duality_2}. The diagram is quite large, but most of its terms are the terms in the diagram for the evaluation map, tensored over $X \times X$ with the base-change 1-cell $\bcr{X}{f}{X}$. To make the circle products and shadows derived it is sufficient to put a $P$ on the base-change 1-cell. This makes all of the backwards maps into weak equivalences, so that we get a commuting diagram in the homotopy category.
\end{proof}

Examining the above formula, when the fixed points of $f$ are isolated we see that it vanishes away from the fixed points, and near each one it gives a loop in $\Lambda^f X$ that is very nearly constant, together with a map of spheres whose degree is the index of that fixed point. This leads to a formula:
\begin{cor}\label{reidemeister_trace_is_weighted_sum_of_fixed_points}
	When $f\colon X \to X$ has isolated fixed points $x$, regarded as constant paths $[x] \in \pi_0(\Lambda^f X)$, then
	\[ R(f) = \left( \sum_{x \in \Fix(f)} \ind(x) \cdot [x] \right) \in \pi_0(\Sigma^\infty_+ \Lambda^f X). \]
\end{cor}

\begin{ex}
	The flip map of $S^1$ has two fixed points, each of index $+1$, and $\Lambda^f S^1$ has two components. Each fixed point lands in a different component, so that the Reidemeister trace $R(f) = (1,1)$. This proves that any map homotopic to $f$ must have at least two fixed points.
\end{ex}

We can run the discussion fiberwise again. Let $E \to B$ be any Hurewicz fibration. To form the fiberwise version of $X^I$, define $\ti E$ by the pullback
\[ \xymatrix @R=1.5em{
	\ti E \ar[d] \ar[r] & E^I \ar[d] \\
	B \ar[r] & B^I,
} \]
in other words the closed subspace of $E^I$ of paths whose projection to $B$ is constant. To check that the evaluation at the endpoints map $\ti E \to E \times_B E$ is a fibration, we compare universal properties and see that if we add a disjoint copy of $E \times_B E$ to the source, the map is homeomorphic to
\[ \barmap_*(I_+,E_{+B}) \to \barmap_*((\partial I)_+,E_{+B}) \]
and is therefore a fibration by \autoref{h_fibrations_pullback_hom}. The inclusion of constant paths $E \to \ti E$ is a weak equivalence over $E \times_B E$.

Now suppose that $E \to B$ is a fiber bundle fiber $X$ and base $B$ both compact ENRs. Pick a fiberwise embedding $i\colon E \to B \times \R^n$ over $B$, a neighborhood $N$ that fiberwise projects back to $E$ by $p$ (possible by \cite[1.8]{dold1974fixed}), and choose $\epsilon > 0$ so that the fiberwise $\epsilon$-tube $N_\epsilon$ is contained in $N$.
Define two closed subspaces of $E \times_B N$ by
\[ \Delta_E = \{ (e,i(e)) : e \in E \},
\qquad
\Gamma_\epsilon = \{ (e,u) : d(i(e),u) \leq \epsilon \}. \]
Intuitively, $\Delta_E$ is an embedded copy of $E$ and $\Gamma_\epsilon$ is its tubular neighborhood. For every $(e,u) \in \Gamma_\epsilon$, since the line segment in $\R^n$ from $i(e)$ to $u$ is contained in $N_\epsilon$, we may apply $p$ to that line segment and get a path in $E$ that we call $\gamma_{x,u}$:
\[ \gamma_{x,u}(t) = p((1-t)i(x) + tu). \]
Note that this lies over a constant path in $B$. So this defines a continuous function $\Gamma_\epsilon \to \ti E$ of spaces over $E \times_B E$.

\begin{thm}[Fiberwise Costenoble-Waner duality, neighborhood version]\label{thm:fiberwise_cw_duality_1}
	If in addition $N$ is a fiberwise mapping cylinder neighborhood, the dual of $\Sigma^\infty_{+(E \times_B B)} E$ in $\Ex_B$ is $F_n N/_E\partial N$ as a 1-cell from $B$ to $E$, with coevaluation and evaluation maps
	\begin{align*}
	F_n S^n_B &\to F_n N/_E\partial N \odot^E_B E_{+(E \times_B B)} \\
	E_{+(E \times_B B)} \odot^B_B F_n N/_E\partial N &\to F_n S^n \barsmash \ti E_{+(E \times_B E)}
	\end{align*}
	obtained by applying $F_n$ to the maps of retractive spaces
	\[ \begin{array}{cclcc}
	S^n_B & \overset{c}\to & N/_B \partial N && (E \times_B N)/_{(E \times_B E)} (E \times_B \partial N) \overset{e}\to S^n_\epsilon \barsmash \ti E_{+(E \times_B E)} \\[.5em]
	v &\mapsto & \left\{ \begin{array}{ccl} v && v \in N \\ {*} && v \not\in N \end{array} \right. && (e,u) \mapsto \left\{ \begin{array}{ccl} (u-i(e),\gamma_{e,u}) && (e,u) \in \Gamma_\epsilon \\ {*} && (e,u) \not\in \Gamma_\epsilon \end{array} \right.
	\end{array} \]
	over $B$ and $E \times_B E$, respectively.
\end{thm}\index{Costenoble-Waner duality!fiberwise}

\begin{thm}[Fiberwise Costenoble-Waner duality, mapping cone version]\label{thm:fiberwise_cw_duality_2}
	In general, the dual of $\Sigma^\infty_{+E \times_B B} E$ in $\Ex_B$ is $F_n C_{B \times_B E}(N,N-E)$, with maps
	\begin{align*}
	F_n S^n_B &\to F_n C_{B \times_B E}(N,N-E) \odot^E_B E_{+(E \times_B B)} \\
	E_{+(E \times_B B)} \odot^B_B F_n C_{B \times_B E}(N,N-E) &\to F_n S^n \barsmash \ti E_{+(E \times_B E)}
	\end{align*}
	given by the analogs of the maps from \autoref{thm:cw_duality_2}, with all products taken over $B$ and all instances of $X$ replaced by $E$.
\end{thm}

\begin{cor}\label{cor:fiberwise_reidemeister_formula}
	When $E$ is a fiber bundle over a compact ENR $B$ with fiber $X$ a compact ENR, the fiberwise Reidemeister trace $R_B(f)$ of a fiberwise map $f\colon E \to E$ is the desuspension over $B$ of the map of spaces over $B$
	\[ \xymatrix @R=0em{
		S^n_B \ar[r] & S^n_\epsilon \barsmash (\Lambda^f_B E)_{+B} \\
		(v,b) \ar@{|->}[r] & (v - f_b(p_b(v))) \sma \gamma_{f_b(p_b(v)),v}.
	} \]
\end{cor}\index{Reidemeister trace $R(f)$!fiberwise}\index{fiberwise!Reidemeister trace $R(f)$}
So $R_B(f)$ is just $R(f)$ carried out over each point of $B$, giving a map of bundles because the formula for $R(f)$ varies continuously in $f$.

\begin{proof}
	We follow the same steps as in \autoref{cor:fiberwise_lefschetz_formula} to reduce to the case where $B$ is a finite simplicial complex and then to make $\partial N \to N$ into an $h$-cofibration of spaces over $E$. Then we compare the trace using mapping cones to the trace using neighborhoods as in \autoref{cor:reidemeister_formula}. Finally, when we write out the trace using neighborhoods, we get the expression
	\[ \xymatrix @R=0em{
		S^n_B \ar[r] & S^n_\epsilon \barsmash \ti E_{+(E \times_B E)} \odot^{E \times_B E}_B \bcr{E}{f}{E}_B \\
		(v,b) \ar@{|->}[r] & (v - f_b(p_b(v)),\gamma_{f_b(p_b(v)),v},p(v)).
	} \]
	This circle product is canonically isomorphic to the fiberwise twisted free loop space $(\Lambda^f_B E)_{+B}$, taking the above formula to the path $\gamma_{f(p(v)),v}$.
\end{proof}

As before, if $B$ is a smooth $d$-dimensional manifold and the fixed points are arranged to be smoothly embedded $d$-dimensional submanifolds of $E$, we can interpret this as a weighted sum of the framed $d$-dimensional fixed-point manifolds of $f$, labeled by reference maps to $\Lambda^f_B E$. If the fiber $X$ is highly connected, this gives the same information as the Lefschetz number. By \cite[10.5]{klein_williams}, it is a complete obstruction to removing fixed points from a smooth fiber bundle in the range where the dimension of $X$ is at least 3 more than the dimension of $B$.

\begin{ex}\label{ex:reidemeister_fiberwise}
The trivial bundle $S^3 \times S^1 \to S^1$ has 2-connected fiber $S^3$, so the fibers of $\Lambda^f_B E$ are simply-connected. Therefore $R_{S^1}(f)$ contains the same information as $L_{S^1}(f)$, regardless of the homotopy class of $f$. If we modify \autoref{ex:lefschetz_fiberwise} by having $f$ go through two whole rotations as we go around the base $S^1$, then
\[ R_{S^1}(f) = L_{S^1}(f) = (0,0) \in \Z \oplus \Z/2. \]
In fact, $f$ can have its fixed points removed by a fiberwise homotopy. Observe that the corresponding path in $\pi_1(O(3)) \cong \Z/2$ is homotopic to a constant path, and therefore $f$ is fiberwise homotopic to the identity. The identity can then be slightly deformed to have no fixed points because $S^3$ is an odd-dimensional sphere.
\end{ex}

\begin{rmk}
	There are more variants of Costenoble-Waner duality that can be deduced from these. For instance when $E \to B$ is a fiber bundle with compact ENR base and fiber, the 1-cell $E_{+(E \times B)}$ in $\Ex$ from $E$ to $B$ is dualizable over $E$. As a consequence of \cite[7.1]{lind2019transfer}, we can get its coevaluation and evaluation maps by pushing forward the maps from \autoref{thm:fiberwise_cw_duality_2} along the diagonal $B \to B \times B$ and the inclusion $E \times_B E \to E \times E$, respectively.
\end{rmk}

\beforesubsection
\subsection{Proof of Costenoble-Waner duality}\aftersubsection

We now recall the argument \cite{ms} that proves \autoref{thm:cw_duality_2}, including some additional detail that seems to be necessary outside the case of a smooth manifold or finite complex with particularly nice retract $p$. We also make heavy use of our rigidity results to be assured that the various identifications between different models of the same space can be made in a coherent way. We then generalize the argument to prove \autoref{thm:fiberwise_cw_duality_2}. Again, \autoref{thm:fiberwise_cw_duality_2} is not handled directly in \cite{ms} since their motivation is not to get an explicit formula for the trace.

We begin with an elementary but important observation. For $X \in \mc R(A \times C)$, $Y \in \mc R(C \times D)$, and $Z \in \mc R(*)$, the three expressions
\[ Z \barsmash (X \odot Y), \quad (Z \barsmash X) \odot Y, \quad X \odot (Z \barsmash Y) \]
are canonically isomorphic. Indeed, each of them can be rearranged to a pushforward of a pullback of $X \barsmash Y \barsmash Z$, and the pullback is along an injective map, so the rigidity theorem \autoref{prop:spaces_rigidity} applies. The same applies if instead $X$, $Y$, and $Z$ are spectra, and/or if everything is done relative to $B$, so that $X \in \mc R(A \times_B C)$, $Y \in \mc R(C \times_B D)$, and $Z \in \mc R(B)$, and we are comparing
\[ Z \barsma{B} (X \odot_B Y), \quad (Z \barsma{B} X) \odot_B Y, \quad X \odot_B (Z \barsma{B} Y). \]
We are therefore free to switch between these expressions without further comment.

In \autoref{product_of_cones} and the discussion that followed, we proved a general statement that specializes to the following.\index{mapping cone} Given an open inclusion $K \to X$ of spaces over $A \times C$ and $L \to Y$ of spaces $C \times D$, if we define
\[ (X,K) \square_C (Y,L) = (X \times_C Y,X \times_C L \cup_{K \times_C L} K \times_C Y) \]
as a pair of spaces over $A \times D$, then there is an equivalence of mapping cones
\[ \mu_\odot\colon C^u_{A \times C}(X,K) \odot C^u_{C \times D}(Y,L) \to C^u_{A \times D}((X,K) \square_C (Y,L)) \]
formed from $\mu$ by pulling back and pushing forward along
\[ \xymatrix{
	A \times D & \ar[l] A \times C \times D \ar[r] & A \times C \times C \times D.
} \]
The equivalence $\mu_\odot$ is associative and unital in an appropriate sense. If $(Z,M)$ is over $*$, so that we also have an instance of $\mu$
\[ \xymatrix{
	C^u(Z,M) \barsmash C^u_{A \times C}(X,K) \ar[r]^-{\mu} & C^u_{A \times C}((Z,M) \square (X,K)),
} \]
then the following square commutes.
\[ \xymatrix @R=1.7em{
	C^u(Z,M) \barsmash C^u_{A \times C}(X,K) \odot C^u_{C \times D}(Y,L) \ar[d]^-{1 \odot \mu} \ar[r]^-{\mu \odot 1} &
	C^u_{A \times C}((Z,M) \square (X,K)) \odot C^u_{C \times D}(Y,L) \ar[d]^-{\mu_\odot} \\
	C^u_{A \times C}(X,K) \odot C^u_{C \times D}((Z,M) \square (Y,L)) \ar[r]^-{\mu_\odot} &
	C^u_{A \times D}((Z,M) \square (X,K) \square_C (Y,L))
} \]
We don't need to formalize the coherence between all possible combinations of these maps -- it is enough to observe that all terms are subquotients of $(X \times Y \times Z \times I^n) \amalg (A \times D)$ and all maps lie underneath the identity of $X \times Y \times Z$ and functions $I^n \to I^m$ in which each coordinate is a minimum of some subset of the coordinates in the source.

Similarly, if $K \to X$ is over $A \times_B C$ and $L \to Y$ is over $C \times_B D$ then $(X,K) \square_C (Y,L)$ is a pair over $A \times_B D$, and we get an equivalence of mapping cones
\[ \mu_{\odot_B}\colon C^u_{A \times_B C}(X,K) \odot_B C^u_{C \times_B D}(Y,L) \to C^u_{A \times_B D}((X,K) \square_C (Y,L)). \]
We get a similar coherence between $\mu_{\odot_B}$ and the product map for the smash product rel $B$,
\[ \xymatrix{
	C^u(Z,M) \barsma{B} C^u_{A \times C}(X,K) \ar[r]^-\sim & C^u_{A \times C}((Z,M) \square_B (X,K)).
} \]

Next we reduce \autoref{thm:cw_duality_2} to a space-level statement. Let $n \geq 0$. Suppose that $X \in \mc R(A \times C)$ and $Y \in \mc R(C \times A)$ are $f$-cofibrant, so that their bicategorical product can be derived by taking $P$. In what follows, all maps are in the homotopy category, and we right-derive all circle products $\odot$ when necessary.

We say $X$ and $Y$ are $n$-dual over $A$ if there are maps in the homotopy category of retractive spaces $c\colon S^n \barsmash U_C \to Y \odot X$ and $e\colon X \odot Y \to S^n \barsmash U_A$, such that the first square below commutes in $\ho\mc R(A \times A)$ and the second commutes in $\ho\mc R(C \times C)$, where $\sigma$ denotes the one-point compactification of $v \mapsto -v$, cf \cite{ponto_asterisque}.\index{$n$-duality!in a bicategory}
\[ \xymatrix @R=1.7em{
	X \odot (S^n \barsmash U_C) \ar@{<->}[d]_-\cong \ar[r]^-{1 \odot c} & X \odot Y \odot X \ar[d]^-{e \odot 1} &&
	(S^n \barsmash U_C) \odot Y \ar@{<->}[d]_-\cong \ar[r]^-{c \odot 1} & Y \odot X \odot Y \ar[d]^-{1 \odot e} \\
	S^n \barsmash X \ar@{<->}[r]^-\cong & (S^n \barsmash U_A) \odot X &&
	S^n \barsmash Y \ar@{<->}[r]^-\sigma & Y \odot (S^n \barsmash U_A)
} \]
When $X \in \mc R(A \times_B C)$ and $Y \in \mc R(C \times_B A)$ we formulate the same definition using $\barsma{B}$, $\odot_B$, and $S^n_B$.
\begin{lem}\label{cw_n-dual}
	If $X$ and $Y$ are $n$-dual $f$-cofibrant spaces in this sense, then applying $F_n$ to $c$ and $e$ gives a bicategorical duality between $F_0 X$ and $F_n Y$ in $\Ex$, or in $\Ex_B$ in the fiberwise case.
\end{lem}

\begin{proof}
	Since $F_0$ and $F_n$ can be pulled out of the bicategorical products, and commute with $P$, the proof is formally identical to that of \autoref{lem:n-dual}.
\end{proof}

\begin{proof} of \autoref{thm:cw_duality_2}.
	Let $D$ be a large disc centered at the origin containing $i(X)$. By \autoref{lem:mapping_cones_equivalences}, our proposed coevaluation and evaluation maps can be rearranged slightly so that they only use unbased mapping cones:
	\[ \resizebox{\textwidth}{!}{
		\xymatrix @C=1em @R=.8em{
			C^u(\R^n,\R^n-D) \ar[r] & C^u(\R^n,\R^n-X) & \ar[l]_-\sim C^u(N,N-X) \ar@{<-}[r]^-\cong_-{\mu_\odot} & C^u_{* \times X}(N,N-X) \odot X_{+(X \times *)} \\
			X_{+(X \times *)} \odot C^u_{* \times X}(N,N-X) \ar[r] & C^u_{X \times X}(X \times N,(X \times N)-\Delta_X) & C^u_{X \times X}(\Gamma,\Gamma - \Delta_X) \ar[l]_-\sim \ar[r] & C^u(\R^n,\R^n-0) \barsmash X^I_{+(X \times X)}.
		}
	} \]
	By \autoref{cw_n-dual} it suffices to show these give an $n$-duality between $X_{+(X \times *)}$ and $C^u_{* \times X}(N,N-X)$, which are both $f$-cofibrant retractive spaces over $X$. For the first triangle identity, the composite of coevaluation and evaluation maps becomes the top, right, and bottom edges of the following diagram. Note that $C^u(\R^n,\R^n-D)$ is a space over $*$ and so it can be moved around at will.
	\[ \resizebox{\textwidth}{!}{
		\xymatrix @C=1em{
			X_{+(X \times *)} \odot C^u(\R^n,\R^n-D) \ar[d]_-\cong \ar[r] &
			X_{+(X \times *)} \odot C^u(\R^n,\R^n-X) &
			\ar[l]_-\sim X_{+(X \times *)} \odot C^u(N,N-X) \ar[d]^-\cong \\
			C^u(\R^n,\R^n-D) \barsmash X_{+(X \times *)} \ar[d]_-\cong &
			C^u(\R^n,\R^n-0) \barsmash (X^I_{+(X \times X)} \odot X_{+(X \times *)}) &
			X_{+(X \times *)} \odot C^u_{* \times X}(N,N-X) \odot X_{+(X \times *)} \ar[d] \\
			C^u(\R^n,\R^n-D) \barsmash (X_{+(X \times X)} \odot X_{+(X \times *)}) \ar[ur]^-\sim &
			C^u_{X \times X}(\Gamma_\epsilon,\Gamma_\epsilon - \Delta_X) \odot X_{+(X \times *)} \ar[u] \ar[r]_-\sim &
			C^u_{X \times X}(X \times N, (X \times N) - \Delta_X) \odot X_{+(X \times *)}
		}
	} \]
	Applying $\mu$ liberally turns this into a termwise equivalent diagram of retractive spaces over $X \times *$, where every term is a cone and every map arises as a map of pairs.\footnote{The coherence statements we made earlier about $\mu_\odot$ are used right here to check that this transformation extrudes each arrow into a commuting square. This is necessary to conclude that one diagram commutes iff the other one does.} To check it commutes in the homotopy category, it suffices to give homotopies of pairs as indicated below.
	\[ \resizebox{\textwidth}{!}{
		\xymatrix @C=5em{
		X \times (\R^n,\R^n-D) \ar[d]_-{(x,v) \mapsto (v,x)} \ar[r] &
		X \times (\R^n,\R^n-X) \ar[d]_-{(x,v) \mapsto (v-x,c_x)} &
		\ar[l]_-\sim X \times (N,N-X) \ar[d]^-{(x,v) \mapsto (x,v,p(v))} \\
		(\R^n,\R^n-D) \times X \ar[d]_-\sim &
		(\R^n,\R^n-0) \times X^I \times_X X &
		X \times (N,N-X) \times_X X \ar[d] \\
		(\R^n,\R^n-D) \times X \times_X X \ar[ur]^-{(v,x) \mapsto (v,c_x)} &
		(\Gamma_\epsilon,\Gamma_\epsilon - \Delta_X) \times_X X \ar[u]_-{(x,v) \mapsto (v-x,\gamma_{x,v})} \ar[r]_-\sim &
		(X \times N, (X \times N) - \Delta_X) \times_X X \ar[ul]_-{(x,v) \mapsto (v-x,c_x)}
	}
	} \]
	The unlabeled maps are all inclusions and collapses. The top-left region commutes by the homotopy $(x,v-tx,c_x)$. This gives a map of pairs because $v - tx = 0 \Rightarrow v \in D$. The next region commutes strictly. The final region at the bottom commutes by the homotopy $(v-x,\gamma_{x,p((1-t)x+tv)})$. This gives a map of pairs because $v - x = 0 \Rightarrow (x,v) \in \Delta_X$. In summary, the first triangle identity follows from the homotopies $$(v-x,\gamma_{x,v}) \sim (v-x,c_x) \sim (v,c_x).$$
	
	Note that all maps and homotopies respect the projection back to $X \times *$, taking all of the above points to $x \in X$. Furthermore, note that all spaces are $f$-cofibrant and the map from each of the composition products to the derived composition product, $Y \odot Z \to PZ \odot PZ$, is an equivalence, using \autoref{lem:derived_circle_product}. This could also be observed directly because $- \odot X_{+(X \times *)}$ and $X_{+(X \times *)} \odot -$ correspond on the point-set level to pushforward and pullback, respectively, along the fibration $X \times *$. We therefore get the first triangle identity in the homotopy category.
	
	For the second triangle identity, the composite of the point-set versions of the coevaluation and evaluation maps becomes the top, right, and bottom edges of the following diagram.
	\[ \resizebox{\textwidth}{!}{
		\xymatrix @C=2em @R=2em{
		C^u(\R^n,\R^n-D) \odot C^u_{* \times X}(N,N-X) \ar[d]_-\sim \ar[r] & C^u(\R^n,\R^n-X) \odot C^u_{* \times X}(N,N-X) \\
		C^u(\R^n,\R^n-0) \barsmash C^u_{* \times X}(N,N-X) \ar[d]_-\sim & C^u(N,N-X) \odot C^u_{* \times X}(N,N-X) \ar[u]_-\sim \ar[d] \\
		C^u_{* \times X}(N,N-X) \barsmash C^u(\R^n,\R^n-0) \ar[d]_-\sim & C^u_{* \times X}(N,N-X) \odot C^u_{X \times X}(X \times N,(X \times N) - \Delta_X) \\
		C^u_{* \times X}(N,N-X) \barsmash C^u(\R^n,\R^n-0) \odot X^I_{+(X \times X)} & C^u_{* \times X}(N,N-X) \odot C^u_{X \times X}(\Gamma_\epsilon,\Gamma_\epsilon - \Delta_X) \ar[l] \ar[u]_-\sim
	}
	} \]
	Again, applying $\mu$ to everything gives a termwise homotopy equivalence to a new diagram of spaces over $* \times X$, where every term arises from a single pair and the maps arise as maps of pairs, indicated below. Using excision on this diagram allows us to conclude that the backwards maps in the diagram above are weak equivalences, even though the composition products have not been derived.
	\[ \xymatrix @C=4em @R=4em{
		(\R^n,\R^n-D) \square (N,N-X) \ar[d]_-\sim \ar[r] & (\R^n,\R^n-X) \square (N,N-X) \ar[ldd]_-{(v,w) \mapsto (w,w-v)} \\
		(\R^n,\R^n-0) \square (N,N-X) \ar[d]_-{(v,w) \mapsto (w,-v)} & (N,N-X) \square (N,N-X) \ar[ld]_-{(v,w) \mapsto (w,w-v)} \ar[u]_-\sim \ar[d]^-{(v,w) \mapsto (v,p(v),w)} \\
		(N,N-X) \square (\R^n,\R^n-0) \ar[d]_-{(w,u) \mapsto (w,u,c_{p(w)})} & (N,N-X) \square_X (X \times N,(X \times N) - \Delta_X) \ar[l]^-{(v,p(v),w) \mapsto (w,w-v)} \\
		(N,N-X) \square (\R^n,\R^n-0) \times_X X^I & (N,N-X) \square_X (\Gamma_\epsilon,\Gamma_\epsilon - \Delta_X) \ar[l]^-{(v,p(v),w) \mapsto (v,w-p(v),\gamma_{p(v),w})} \ar[lu]_-{(v,p(v),w) \mapsto (w,w-v)} \ar[u]_-\sim \\
		& (N_\delta,N_\delta-X) \square_X (\Gamma_{\epsilon'},\Gamma_{\epsilon'} - \Delta_X) \ar[u]_-\sim
	} \]
	The unlabeled maps are all inclusions and collapses. The top-left region commutes by the homotopy $(w,tw-v)$. This gives a map of pairs because if $w \in X$ and $tw - v = 0$ then $(v,w) \in D \times X$. The remaining regions commute except for the bottom triangle, which commutes up to homotopy after we pre-compose with the map at the bottom right of the diagram for $\delta$ and $\epsilon'$ sufficiently small.
	
	Specifically, we use uniform continuity to find $\delta > 0$ such that if $d(v,i(X)) < \delta$ then $d(v,p(v)) \leq \frac{\epsilon}{2}$. To find $\epsilon'$, by \autoref{lem:awesome} $X \to \Lambda X$ is an $f$-cofibration over $X$, so $\Lambda X$ contains a small neighborhood of $X$ that fiberwise deformation retracts onto $X$. The map $N_\epsilon \to \Lambda X$ taking $w$ to $\gamma_{p(w),w}$ sends $X \subseteq N_\epsilon$ to the constant loops, so we can find an $\epsilon' > 0$, $\epsilon' \leq \frac{\epsilon}{2}$, such that $N_{\epsilon'} \to \Lambda X$ lands in a neighborhood that deformation retracts through $LX$ to $X$, through maps that are over $X$. This gives a homotopy of based loops $\gamma_{p(w),w} \sim c_{p(w)}$, fixing the basepoint, that is continuous in $w$.
	
	With these bounds on $v$ and $w$, we therefore get a homotopy from $(w,w-v,c_{p(w)})$ to $(w,w-v,\gamma_{p(w),w})$, as maps into the fiber product over $X$. Then we apply the homotopy
	\[ ((1-t)w+tv,w-v,\gamma_{p((1-t)w+tv),w}). \]
	This is well-defined because $d(v,p(v)) \leq \frac{\epsilon}{2}$ and $d(p(v),w) \leq \epsilon' \leq \frac{\epsilon}{2}$ guarantees that every point on the line segment from $v$ to $w$ is within $\epsilon$ of $p(v)$ and is therefore in $N$. It gives a map of pairs because if $w = v$ and $(1-t)w+tv = v \in X$ then $(v,w) \in X \times X$.
	Finally we apply the homotopy
	\[ (v,w-(1-t)v-tp(v),\gamma_{p(v),w}). \]
	It gives a map of pairs because if $v \in X$ and $w - (1-t)v - tp(v) = w-v = 0$ then $(v,w) \in X \times X$. In summary, the second triangle identity follows from the homotopies $$(v,w-p(v),\gamma_{p(v),w}) \sim (v,w-v,\gamma_{p(v),w}) \sim (w,w-v,\gamma_{p(w),w}) \sim (w,w-v,c_{p(w)}) \sim (w,-v,c_{p(w)}).$$
	
	Note that all maps and homotopies respect the projection back to $* \times X$, taking all of the above points to $p(w) \in X$. Using \autoref{lem:derived_circle_product}, the map to the derived circle product induces an equivalence on all the terms in the left-hand column, which is enough to deduce the second triangle identity in the homotopy category.
\end{proof}
\begin{proof} of \autoref{thm:fiberwise_cw_duality_2}.
	We perform the same rearrangement as above, then prove the point-set version of both triangle identities by the same formulas. Again, the backwards maps in the second triangle identity are equivalences by excision after applying $\mu_{\odot_B}$. By \autoref{lem:derived_circle_product}, all the circle products $\odot_B$ in the first triangle identity agree with the right-derived circle product. In the second triangle identity, this is at least true in the left-hand column. This is enough to deduce the triangle identities in the homotopy category.
\end{proof}

\newpage
\section{Modules over group rings}\label{sec:models}

In this section we show how parametrized spectra over $BG$ are equivalent to spectra with a $G$-action, for any topological group $G$. In other words, when $X$ is a parametrized spectrum over $B$, its derived fiber $\R b^* X$ inherits an action by a group equivalent to $\Omega B$. When $B$ is connected, this gives an equivalence between parametrized spectra over $B$ and module spectra over the ring spectrum $\Sigma^\infty_+ \Omega B$.

Along this equivalence, the three basic operations $f^*$, $f_!$, $\barsmash$ turn into restriction and extension of scalars, and the ordinary smash product $\sma$. We can phrase the result as a fiberwise equivalence of SMBFs (\autoref{prop:SMBF_equivalence}) and an associated equivalence of shadowed bicategories (\autoref{prop:all_groups}). As a result, most of the fixed-point theory done using parametrized spectra can be re-interpreted entirely with ring and module spectra.

It is usually easier to prove things for module spectra than for parametrized spectra -- for instance we use this equivalence to characterize the fiberwise dualizable spectra without having to wrangle with the functor $(\Delta_B)_*$. On the other hand, passing to module spectra loses geometry, and the geometric model of the bicategorical trace can be useful for computations, as in \cite[\S 9]{lind2019transfer} or \autoref{ex:lefschetz_fiberwise}.

\beforesubsection
\subsection{Comparison for a single group}\label{sec:equivalent_to_modules}\aftersubsection

Let $G$ be a topological group. We assume at least that $G$ is well-based and homotopy equivalent to a cell complex, but for the main theorem we will make the stronger assumption that the underlying space of $G$ is $q$-cofibrant.\footnote{Equivalently, the unit $\{1\} \to G$ is a $q$-cofibration.}

Let $\mathscr M(G)$ denote the category of orthogonal spectra with an action of $G$. Equivalently, the objects are orthogonal module spectra over the group ring $\Sph[G] = \Sigma^\infty_+ G$ whose multiplication comes from $G$. These have a cofibrantly generated model structure in which the stable equivalences and fibrations are measured on the underlying orthogonal spectrum, and the generating cofibrations and acyclic cofibrations are the ``free $G$-cell spectra''
	\[ \begin{array}{ccrl}
	I &=& \{ \ F_k\left[ (G \times S^{n-1})_{+} \to (G \times D^n)_{+} \right] & : n,k \geq 0 \ \} \\
	J &=& \{ \ F_k\left[ (G \times D^n)_{+} \to (G \times D^n \times I)_{+} \right] &: n,k \geq 0 \ \} \\
	& \cup & \{ \ k_{i,j} \ \square \ \left[ (G \times S^{n-1})_{+} \to (G \times D^n)_{+} \right] & : i,j,n \geq 0 \ \}.
	\end{array} \]
Define $EG = B(*,G,G)$ and $BG = B(*,G,*)$ using the unbased bar construction, in other words the realization of the simplicial unbased spaces
\[ [n] \mapsto * \times G^{\times n} \times G, \qquad [n] \mapsto * \times G^{\times n} \times *. \]
In particular, $EG$ has a free right $G$-action by multiplication on the last coordinate, and the quotient $EG/G$ is $BG$. For an unbased space $X$ with continuous left $G$-action, we call
\[ B(*,G,X) \cong EG \times_G X \]
the \textbf{Borel construction}\index{Borel construction $EG \times_G (-)$} on $X$. This defines a functor $EG \times_G -$ from $G$-spaces to spaces over $BG$. The map $EG \to BG$ is a principal $G$-bundle and therefore $EG \times_G X \to BG$ is a fiber bundle with fiber $X$. Therefore the Borel construction is a homotopical functor, i.e. it sends every weak equivalence of spaces with $G$-action to a weak equivalence of spaces over $BG$.

If $X$ is based, its Borel construction is a retractive space over $BG$. Furthermore, when $X = K \sma Y$ and $K$ has a trivial $G$-action, we get a canonical isomorphism
\[ EG \times_G (K \sma Y) \cong K \barsmash (EG \times_G Y) \]
that is associative and unital in the $K$ variable. Therefore if $X$ is a spectrum with $G$-action, the retractive spaces $EG \times_G X_n$ assemble in a canonical way into a parametrized spectrum over $BG$. Therefore the Borel construction defines a functor $\mathscr M(G) \to \Osp(BG)$.

\begin{prop}\label{prop:one_group}
	When $G$ is $q$-cofibrant, $EG \times_G -$ is a left Quillen equivalence from $\Sph[G]$-modules with the above model structure, to $\Osp(BG)$ with the stable model structure.
\end{prop}

\begin{proof}
	By induction up the skeleton, $EG$ is cofibrant as a free $G$-space and $EG \times_G -$ sends generating (acyclic) cofibrations to (acyclic) cofibrations. Its right adjoint is the mapping spectrum $\barF_{BG}(EG_{+BG},-)$, in other words $\Map_{BG}(EG,-)$ applied to each spectrum level. This proves that $EG \times_G -$ is left Quillen.
	
	To prove it induces an equivalence, take a cofibrant $G$-spectrum $X$ and a stably fibrant parametrized spectrum $Y$ over $BG$, and any map $EG \times_G X \to Y$. This map is a stable equivalence of parametrized spectra iff at each spectrum level, over the basepoint of $BG$, the resulting map of spaces $X_n \to (Y_n)|_{*}$ is a weak equivalence. This map is also obtained by composing the adjunct with restriction from $EG$ to a point $*$:
	\[ \xymatrix{ X_n \ar[r] & \Map_{BG}(EG,Y_n) \ar[r]^-\sim & \Map_{BG}(*,Y_n) \cong (Y_n)|_{*}. } \]
	The marked equivalence is by \autoref{ex:smashing_spaces_is_left_Quillen}. So $X_n \to (Y_n)|_{*}$ is an equivalence iff $X \to \barF_{BG}(EG,Y)$ is a level equivalence, iff it is a stable equivalence because $\barF_{BG}(EG_{+BG},Y)$ is already a fibrant $\Sph[G]$-module.
\end{proof}
\begin{rmk}
	By the same argument, if $G$ is only well-based and homotopy equivalent to a cell complex, then $EG \times_G -$ and $\barF_{BG}(EG_{+BG},-)$ form a deformable adjoint pair, which give an equivalence of homotopy categories. We can also weaken further to the case that $G$ is a grouplike topological monoid, but then the argument is a bit longer because $EG \times_G X \to BG$ is no longer a fiber bundle, only a quasifibration. See \cite{vegetables} for details. 
\end{rmk}
\begin{rmk}
	This proof also shows that the inverse functor $\barF_{BG}(EG_{+BG},Y)$, if we forget its $G$-action, is equivalent to the derived fiber of $Y$. If we remember the $G$-action, it is also equivalent to the functor that pulls back $Y$ to $EG$ then pushes forward to a point, which we could imprecisely call $EG \times_{BG} Y$.
	
	One way to conceptualize these different constructions is to build a mixed bicategory where the 0-cells are pairs $(G,A)$ where $G$ is a group and $A$ is a space, and the 1-cells from $(G,A)$ to $(H,B)$ are retractive spaces over $A \times B$ with fiberwise $G \times H^\op$-actions. The composition is just like for $\mc R/\cat{Top}$ and $\mathscr Act$. In this setting, $EG_{+BG}$ is an invertible 1-cell from $(G,*)$ to $(*,BG)$, essentially because $EG \times_{BG} EG \simeq G$ and $EG \times_G EG \simeq BG$. Therefore tensoring with it on either side gives inverse equivalences of homotopy categories, between based $G$-spaces and retractive spaces over $BG$, or $G$-spectra and parametrized spectra over $BG$.
\end{rmk}

\begin{ex}
	Parametrized spectra over $S^1$ are equivalent to spectra with an action by $\Z$. Since $E\Z \simeq \R$, the equivalence takes a $\Z$-spectrum to its mapping torus.
\end{ex}

If $\phi\colon H \to G$ is a group homomorphism then the functor $\phi^*$ that restricts $G$-actions on spectra to $H$-actions is a Quillen right adjoint, with left adjoint $\phi_!X = G_+ \sma_H X$. It is a Quillen equivalence when $\phi$ is a weak equivalence as a map of topological spaces.

The functor $\phi^*$ also has a right adjoint on the homotopy category. To define the right adjoint we either assume that $G$ is a free $H$-cell complex (for instance because $G$ is a compact Lie group and $H$ is a closed subgroup) or we replace $G$ by $B(H,G,G)$. Then we fibrantly replace $X$ and take $\phi_* X = F^H(G_+,X)$.

The homomorphism $\phi$ induces a map $\ti f\colon EH \to EG$ over $f\colon BH \to BG$, giving a map
\[ EH \times_H X \congar f^*(EG \times_G X). \]
This last map is actually a homeomorphism because geometric realization preserves all finite limits 
(e.g. \cite[11.6]{may_iterated_loop_spaces}), so in particular preserves the pullback of simplicial spaces
\[ \xymatrix @R=1.7em{
	{*} \times H^{\times n} \times X \ar[d] \ar[r] & {*} \times G^{\times n} \times X \ar[d] \\
	{*} \times H^{\times n} \times {*} \ar[r] & {*} \times G^{\times n} \times {*}.
} \]

This homeomorphism tells us that the pullback functors $\phi^*$ and $f^*$ correspond to each other along the Borel construction:
\[ \xymatrix @R=1.7em @C=6em{
	\mathscr M(G) \ar[d]_-{\phi^*} \ar[r]^-{EG \times_G -} & \Osp(BG) \ar[d]^-{f^*} \\
	\mathscr M(H) \ar[r]_-{EH \times_H -} & \Osp(BH)
} \]
On fibrant $G$-spectra $X$, the relevant functors are all equivalent to their derived functors. Therefore on the homotopy category, $\R \phi^*$ corresponds to $\R f^*$. As a formal consequence, $\L \phi_!$ corresponds to $\L f_!$ and $\R\phi_*$ corresponds to $\R f_*$.

\begin{ex}
	Taking $G = 1$, this proves that if the parametrized spectrum $Y$ over $BH$ corresponds to the $H$-spectrum $X$, collapsing the basepoint section $r_!Y$ corresponds to (based) homotopy orbits $X_{hH}$, and sections $r_*Y = \Gamma(Y)$ correspond to homotopy fixed points $X^{hH}$. This is consistent with the fact that at each spectrum level $Y$ is the unbased homotopy orbits of $X$, and we get the based orbits from the unbased ones by quotienting out $BH$.
\end{ex}

\beforesubsection
\subsection{Comparison for all groups}\aftersubsection

Now we upgrade \autoref{prop:one_group} to a map of symmetric monoidal fibrations (SMBFs) $\mathscr M \to \Osp$ that is an equivalence on each fiber category.

First we build the SMBF $\mathscr M$ of module spectra over varying groups $G$.\index{SMBF!of module spectra $\mathscr M$} The base category is the category of $q$-cofibrant topological groups $G$,\footnote{This is to ensure that the functors $- \sma_G -$ will be coherently deformable. Though, by a variant of \cite[8.7]{mp1}, we could argue that $- \sma_G -$ preserves equivalences when $G$ is well-based and the spectra are cell complexes built out of $F_k G_+$ smashed with $h$-cofibrations of unbased spaces. So $q$-cofibrations are not really necessary for any of this to work, they are merely convenient because it takes less time to prove that they work.} and the fiber category is $\mathscr M(G)$. A map $(H,X) \to (G,Y)$ is a homomorphism $\phi\colon H \to G$ and an $H$-equivariant map of spectra $X \to \phi^* Y$. It is elementary to check that restricting a group action gives a cartesian arrow, and arrows of the form $X \cong H_+ \sma_H X \to G_+ \sma_H X$ are cocartesian. Therefore $\phi^*$ and $\phi_!$ define pullback and pushforward functors in $\mathscr M$.

If $(G,X)$ and $(G',X')$ are two objects of $\mathscr M$, the smash product spectrum $X \sma X'$ inherits a $G \times G'$-action. This uniquely extends to a symmetric monoidal structure on $\mathscr M$ lying over the cartesian monoidal structure on cofibrant topological groups $\cat{TopGrp}$. It clearly respects cartesian arrows, and for cocartesian arrows we construct a natural isomorphism
\[ (G_+ \sma_{H} X) \sma (G'_+ \sma_{H'} X') \cong (G \times G')_+ \sma_{H \times H'} (X \sma X') \]
that agrees with the maps in from $X \sma X'$.

\begin{lem}\label{prop:bc_groups}
	For a commuting square of groups
\[ \xymatrix @R=1.7em{
	A \ar[r] \ar[d] & B \ar[d] \\
	C \ar[r] & D
} \]
	the Beck-Chevalley map of functors $\mathscr M(C) \to \mathscr M(B)$ is an isomorphism iff the induced map of $(B,C)$ spaces $B \times_A C \to D$ is a homeomorphism. We call such a square of groups a \textbf{compositional pushout square}.
\end{lem}
\begin{proof}
	To see this is necessary, restrict to the free $\Sph[C]$-module of rank one. To see it is sufficient, rewrite the Beck-Chevalley map for a general $C$-module $X$ as $- \sma_C X$ applied to the above map,
	\[ (B \times_A C)_+ \sma_C X \longrightarrow D_+ \sma_C X. \]
\end{proof}

\begin{rmk}
	A compositional pushout square is not a pushout square. If all groups are discrete and $A = 1$, then the above square is a pushout in (non-abelian) groups if $D = B * C$ and a compositional pushout if $D = B \times C$.
\end{rmk}

This finishes the construction of the SMBF structure on $\mathscr M$. Next we invert the weak equivalences on each fiber category $\mathscr M(G)$. Recall from the definitions before \autoref{prop:one_group} that we are taking these to be the maps of $G$-spectra $X \to Y$ that are stable equivalences on the underlying spectra.
\begin{prop}\label{prop:modules_SMBF}
	The resulting category $\ho\mathscr M$ forms a symmetric monoidal bifibration over cofibrant groups $\cat{TopGrp}$. It has Beck-Chevalley for any compositional pushout square in which $\Sph[A] \to \Sph[C]$ is a cofibration of $A$-modules.
\end{prop}

Note that $\Sph[A] \to \Sph[C]$ is a cofibration iff $C$ is a retract of a free $A$-cell complex.

\begin{proof}
	We check the seven conditions from \autoref{prop:deform_SMBF}.
	\begin{enumerate}
		\item $\barsmash$ sends cofibrant $G$-spectra and $H$-spectra to cofibrant $G \times H$-spectra, and preserves weak equivalences between them because cofibrant $G$-spectra are also cofibrant over $\Sph$. The small-object argument allows us to define a cofibrant replacement on all $G$-spectra for all $G$ at the same time.
		\item $\Sph$ is cofibrant as a spectrum with an action of $G = 1$.
		\item The pushforwards $\phi_!$ are coherently left-deformable, using the cofibrant $G$-modules.
		\item The pullbacks $f^*$ are coherently right-deformable, in fact no deformation is necessary.
		\item For a compositional pushout square along a cofibration, the Beck-Chevalley map is between coherently deformable functors, because a cofibrant $C$-module is also cofibrant as an $A$-module.
		\item $\sma$ and all of the pushforwards $\phi_!$ are coherently deformable, because we can check that $\sma\colon \mathscr M(G) \times \mathscr M(G') \to \mathscr M(G \times G')$ is a left Quillen bifunctor.
		\item $\barsmash$ and all of the pullbacks $\phi^*$ are coherently deformable, using the $G$-spectra whose underlying spectra are cofibrant.
	\end{enumerate}
\end{proof}
As in \autoref{expand_beck_chevalley}, we can expand the Beck-Chevalley squares to include any square of groups weakly equivalent to one as above. We call such squares the \textbf{compositional homotopy pushout squares}.

\begin{lem}
	The associated bicategory $\mathscr M/\cat{TopGrp}$ is canonically isomorphic to the bicategory of spectra with group actions $\mathscr Act_\Sph$ from \autoref{ex:ring_bicategories}.
\end{lem}\index{bicategory!of group actions on spectra $\mathscr Act_\Sph$}
\begin{proof}
	They both have the same 0-, 1-, and 2-cells. In $\mathscr Act_\Sph$, the product is left-derived from a point-set product that is obtained by applying \autoref{thm:SMBF_to_bicategory} to the point-set SMBF $\mathscr M$
	\[ X \sma_H Y = (r_H)_!(\Delta_H)^*(X \sma Y) \]
	and similarly for the unit and shadow. We are therefore left with showing that passage to the homotopy category commutes with passage from an SMBF to a bicategory. At this point we can restrict to the 1-cells that are cofibrant as bimodules, and follow the proof of \autoref{thm:four_bicategories_of_spectra}. The proof in this case is actually a bit simpler because the components of the product, shadow, and unit are coherently left-deformable, using the cofibrant bimodules.
\end{proof}

Next we define a functor $\mathscr M \to \Osp$ on the point-set level by taking $(G,X)$ to the Borel construction $EG \times_G X$, and a map $(H,X) \to (G,Y)$ to the map
\[ EH \times_H X \to EG \times_G Y \]
induced by $H \to G$ and $X \to Y$. We check this is indeed a map of parametrized spectra over the map $BH \to BG$, and this rule respects identities and compositions. By the discussion in the previous section, it also preserves cartesian arrows, so it is a map of Grothendieck fibrations. It does not strictly preserve co-cartesian arrows, but it preserves them up to equivalence. To see this, factor $X \to G_+ \sma_H X$ into a map of spectra that at each level fits into a composite of squares
\[ \xymatrix @R=1.7em{
	EH \times_H {*} \ar[d] \ar[r]^-\sim & EG \times_H {*} \ar[d] \ar[r]^-\cong & EG \times_G (G/H) \ar[d] \ar[r] & EG \times_G {*} \ar[d] \\
	EH \times_H X_n \ar[r]^-\sim & EG \times_H X_n \ar[r]^-\cong & EG \times_G (G \times_H X_n) \ar[r] & EG \times_G (G_+ \sma_H X_n).
} \]
Note that $\phi$ may not be injective, so $G/H$ means the orbits of $G$ under the right action of $H$. The middle and right-hand squares are strict pushouts, while the square on the left is a homotopy pushout. Therefore $EH \times_H X \to EG \times_G (G_+ \sma_H X)$ is weakly equivalent to a cocartesian arrow, whose source is cofibrant if $X$ is cofibrant.

We define a symmetric monoidal structure on the Borel construction by choosing an isomorphism natural in $X$ and $X'$ for each pair of groups $G$,$G'$
\[ (EG \times_G X) \barsmash (EG' \times_{G'} X') \cong E(G \times G') \times_{G \times G'} (X \barsmash X') \]
and similarly
\[ E1 \times_1 \Sph \cong \Sph. \]
By rigidity applied to each fiber, these natural isomorphisms are unique and coherent with each other, making the Borel construction into a symmetric monoidal functor. Since it preserves cofibrant objects, as a functor on the homotopy category it also inherits a symmetric monoidal structure.

Finally, given a compositional pushout square of groups, the resulting square of bar constructions is homotopy pullback by comparing homotopy fibers in either direction:
\[ \xymatrix @R=1.7em{
	B/A \ar[r] \ar[d]_-\cong & BA \ar[r] \ar[d] & BB \ar[d] \\
	D/C \ar[r] & BC \ar[r] & BD
} \]

We conclude
\begin{prop}\label{prop:SMBF_equivalence}
	The Borel construction induces a map of SMBFs
	\[ \xymatrix{ \ho\mathscr M \ar[r] & \ho \Osp } \]
	that is an equivalence on each fiber category.
\end{prop} 
Because SMBFs are equivalent to indexed symmetric monoidal categories, we get as a corollary
\begin{cor}
	The equivalence of homotopy categories $\ho\mathscr M(G) \simeq \ho\Osp(BG)$ given by the Borel construction is symmetric monoidal with respect to $\sma$ on the left and $\sma_{BG}$ on the right.
\end{cor}
\begin{cor}\label{fiberwise_dualizability}
	A spectrum over $BG$ is dualizable (resp. invertible) with respect to $\sma_{BG}$ iff its derived fiber is dualizable (resp. invertible) with respect to $\sma$.
\end{cor}\index{Spanier-Whitehead duality!fiberwise}

By \autoref{prop:little_bit_of_functoriality}, the equivalence of \autoref{prop:SMBF_equivalence} also induces a map of shadowed bicategories
\[ \xymatrix{ \mathscr Act_\Sph = \ho\mathscr M/\cat{TopGrp} \ar[r] & \ho \Osp/\cat{Top} = \Ex } \]
that is an equivalence on each of the categories of 1-cells and 2-cells
\[ \mathscr Act_\Sph(G,H) \simeq \Ex(BG,BH). \]
Since every path-connected base space is weakly equivalent to $BG$ for some $q$-cofibrant topological group $G$, we conclude
\begin{thm}\label{prop:all_groups}
	The Borel construction gives an equivalence of shadowed bicategories between $\mathscr Act_\Sph$ and the subcategory $\Ex_{conn} \subset \Ex$ on those spaces that are path-connected.
\end{thm}

As a result, the shadow of a 1-cell in $\mathscr Act_\Sph$ and its image in $\Ex$ are canonically weakly equivalent. Applying this to the units, we get an elegant proof that
\[ \THH(\Sigma^\infty_+ G) \simeq \Sigma^\infty_+ \Lambda BG. \]
Unwinding the proof slightly, this is because $E(G \times G) \times_{G \times G} G \simeq BG$ and then the three steps that construct the shadow correspond to each other up to weak equivalence:
\[ \xymatrix{
	_G(\Sigma^\infty_+ G)_G \ar@{~>}[d]_-{\textup{left-deform}}
	& \Sigma^\infty_{+(BG \times BG)} BG \ar@{~>}[d]^-{\textup{right-deform}} \\
	_G(\Sigma^\infty_+ B(G;G;G))_G \ar@{~>}[d]_-{\textup{restrict to diagonal action}}
	& \Sigma^\infty_{+(BG \times BG)} BG^I \ar@{~>}[d]^-{\textup{restrict to diagonal}} \\
	_G(\Sigma^\infty_+ B(G;G;G)) \ar@{~>}[d]_-{\textup{quotient out $G$}}
	& \Sigma^\infty_{+(BG)} \Lambda BG \ar@{~>}[d]^-{\textup{push forward to $*$}} \\
	\Sigma^\infty_+ B^{\cyc}(G;G)
	& \Sigma^\infty_{+} \Lambda BG
} \]
	
Along the above equivalence of shadows, the trace of any 1-cell in $\mathscr Act_\Sph$ agrees with the trace of its image in $\Ex$, by \autoref{bicategory_map_preserves_traces}. In other words, we can directly compare traces in the setting of parametrized spectra with traces in bimodules. This idea has proven to be important in recent work on traces \cite{lind2019transfer,campbell_ponto}, and most likely a good deal of future work as well. \autoref{sec:applications} uses this result to give a revisionist account of \cite{ponto_asterisque}.

\begin{rmk}
	If one wants to hit all the base spaces in $\Ex$ and not just the connected ones, the strategy is to re-define $\cat{TopGrp}$ to consist of topological groupoids $G$ whose mapping spaces are $q$-cofibrant. Then $BG$ is the categorical classifying space, and $EG$ is the two-sided bar construction, which gives a $G$-diagram of spaces over $BG$. Each of the spaces in this diagram projects to just one of the path components of $BG$. The rest of the treatment is largely unchanged. See also \cite{ponto_asterisque,vegetables}.
\end{rmk}

\beforesubsection
\subsection{Applications}\label{sec:applications}\aftersubsection

We finish with three more applications of the preceding discussion. The first is a minimalist sketch of the approach to Thom spectra pioneered in \cite{units_of_ring_spectra}. See \cite{hebestreit_sagave_schlichtkrull} for more recent technical results in this area.

Let $R$ be any associative orthogonal ring spectrum, and $R_0$ the space at spectrum level 0 of $R$. The product on $R$ makes $R_0$ into a topological monoid. The set of connected components $\pi_0(R)$ is therefore also a monoid, and we let $\pi_0(R_0)^\times$ denote its group of invertible elements.

A standard example is $R = \ti\Sph$, a fibrant replacement of the sphere spectrum $\Sph$ in the model category of associative orthogonal ring spectra from \cite[12.1(iv)]{mmss}. Then $R_0$ is the space of based maps $S^n \to S^n$, stabilized by letting $n \rightarrow\infty$, and the product is the one that smashes maps of spheres. In this example, $\pi_0(R) = (\Z,\cdot)$ and $\pi_0(R)^\times = \{\pm 1\}$.

Let $GL_1(R)$ be the union of invertible components of $R_0$, or in other words the pullback
\[ \xymatrix @R=1.7em{
	GL_1(R) \ar[d] \ar[r] & R_0 \ar[d] \\
	\pi_0(R_0)^\times \ar[r] & \pi_0(R_0).
} \]
The operation $GL_1(-)$ preserves equivalences when $R$ is fibrant, and outputs a $q$-cofibrant space when $R$ is cofibrant. So when $R$ is both cofibrant and fibrant, the bar construction $BGL_1(-)$ also preserves equivalences. The product on $R$ restricts to a left action of the grouplike monoid $GL_1(R)$ on the orthogonal spectrum $R$.

Then we form a parametrized spectrum over $BGL_1(R)$ by the Borel construction:
\[ EGL_1(R) \times_{GL_1(R)} R \to BGL_1(R) \]
this produces a levelwise-quasifibrant spectrum whose fibers are $R$. Since we used a left action of the units, there is a residual right $R$-action on this bundle. Because $GL_1(R)$ is precisely the monoid of maps $R \to R$ that commute with this right $R$-action, we might call this parametrized spectrum the \textbf{universal $R$-line bundle}.

We then define a \textbf{generalized Thom spectrum}\index{fiberwise!Thom spectrum} $Mf$ for any map of spaces $f\colon X \to BGL_1(R)$ by pulling back to $X$ and pushing forward to a point,
\[ Mf := \L(r_X)_!\R f^*(EGL_1(R) \times_{GL_1(R)} R). \]
Equivalently, if $P$ denotes the homotopy pullback of $EGL_1(R)$ to $X$, and $P'$ is a cofibrant replacement, the Thom spectrum is
\[ Mf = \L(r_X)_! (P \times_{GL_1(R)} R) \simeq P'_+ \sma_{GL_1(R)} R,  \ \]
compare \cite[3.13]{units_rigid}. If we forego pushing forward to a point, we get a fiberwise Thom spectrum instead.

The reason for the name is that the classical Thom spectra are a special case. Taking $R$ to be the fibrant sphere $\ti\Sph$, the space $BGL_1(\ti\Sph)$ receives a homotopy class of maps from $BO$ arising from the action of $O(n)$ on $\R^n$, which passes to its one-point compactification $S^n$. Since pullback $f^*$ commutes with products and geometric realization, pulling back the bar construction
\[ EGL_1(\ti\Sph) \times_{GL_1(\ti\Sph)} \ti\Sph \cong B(*,GL_1(\ti\Sph),\ti\Sph) \]
along $BO(n) \to BGL_1(\ti\Sph)$ gives the bar construction $B(*,O(n),\ti\Sph)$. The $O(n)$-action on $\ti\Sph$ is equivalent to the $O(n)$-action on $F_n S^n$ that acts on the $S^n$, so this is equivalent to the bundle of spectra
\[ B(*,O(n),O(n)) \times_{O(n)} F_n S^n. \]
We can pull the $F_n$ out front. So this is just the fiberwise Thom spectrum of the canonical bundle $\gamma^n = EO(n) \times_{O(n)} \R^n \to BO(n)$, de-suspended $n$ times.

Given a virtual bundle $\xi$ of formal dimension $d$ over a compact space $X$, we can represent it by a difference of an $n$-dimensional vector bundle $V$ and an integer $n-d$. The vector bundle $V$ gives a map $X \to BO(n) \to BO \to BGL_1(\ti\Sph)$. Its fiberwise Thom spectrum is equivalent to the pullback we described just above, pulled back from $BO(n)$ to $X$. Since $V$ is the pullback of the canonical bundle $\gamma^n$, the result is therefore $F_n \Th_X(V)$. At any point in this process, we can suspend $d$ times, so that the end result is instead
\[ \Th_X(\xi) = F_{n-d} \Th_X(V). \]
This is the fiberwise Thom spectrum of $\xi$ from \autoref{ex:thom_spectra}. Pushing forward to a point gives the classical Thom spectrum of $V$.

We could even get the cobordism spectrum $MO$ by taking $X = BO(n)$ and taking the colimit of the results as $n \rightarrow\infty$. The most elegant way to do this is to let the base space $B$ vary as the levels of our parametrized spectra vary, so that level $n$ of the Thom spectrum is parametrized over $BO(n)$, for every value of $n$. See \cite[\S 9]{hebestreit_sagave_schlichtkrull} for more details, including the multiplicative properties of this construction.

The next application is a characterization of the dualizable 1-cells in $\Ex$, see also \cite[7.1]{vegetables}.
\begin{thm}
	A parametrized spectrum $X$ over $A \times C$ is dualizable over $A$ precisely when for every $c \in C$, the derived fiber $c^*X$ over $A$ is a retract in the homotopy category of a spectrum built from finitely many of the cells in the stable model structure, \autoref{thm:stable_model_structure}.
\end{thm}\index{Costenoble-Waner duality}

\begin{proof}
	If $A \simeq BG$ and $C \simeq BH$ are path-connected, then by the equivalence of bicategories from \autoref{prop:all_groups}, $X$ is dualizable over $A$ iff its derived fiber over a point $(a,c) \in A \times C$ is dualizable over $\Sigma^\infty_+ G$. As a $G$-module this is determined by the fiber $c^*X$ as a spectrum over $A$. The dualizable $G$-modules are precisely the homotopy retracts of the finite cell $G$-spectra, which along the Borel construction are turned into finite cell spectra over $BG$.\footnote{There is a small amount of extra argument because $EG \times_G (G \times D_n)$ is homeomorphic to $EG \times D^n$, not $D^n$. But because $EG$ is homotopy equivalent to a point, a cell complex built out of cells of the form $EG \times D^n$ is homotopy equivalent to one built out of cells $D^n$; see \cite[7.1]{vegetables} for more details.}

	When $C$ has multiple components, we reduce to one component using \cite[4.3]{ponto_shulman_mult}, or we replace $H$ by a groupoid and run the same argument. When $A$ has multiple components, we either do the same to $G$, or we observe that a spectrum $X$ over $A \times *$ that is dualizable over $A$ must be compact in the homotopy category (meaning the derived version of $\bar F_A(X,-)\colon \Osp(A) \to \Osp(*)$ commutes with arbitrary coproducts), so it can only be nontrivial on finitely many of the components of $A$. On each of these components, $X$ must be a homotopy retract of a finite complex, hence the entire spectrum is as well.
\end{proof}

The final application is a comparison of definitions of the Reidemeister trace. Suppose that $X$ is a path-connected compact ENR. We defined its Reidemeister trace in two ways, as the trace of the map
\[ \xymatrix{ C_*(\ti X;\Z) \ar[r]^-{C_*(\ti f)} & \Z[\pi]^f \otimes_{\Z[\pi]} C_*(\ti X;\Z) } \]
in the bicategory of rings and bimodule chain complexes, and as the trace of the isomorphism of base-change objects
\[ \xymatrix{ \bcr{X}{}{*} \ar[r]^-\cong & \bcr{X}{f}{X} \odot \bcr{X}{}{*} } \]
in the bicategory of parametrized spectra.

By the proof of \autoref{prop:all_groups}, we have a map of symmetric monoidal bifibrations, and by \autoref{prop:little_bit_of_functoriality} this preserves traces of maps of base-change objects that arise from the canonical isomorphisms between their compositions. So, if we modify $f$ up to homotopy so that it preserves the basepoint, the second definition gives the same map of spectra as the trace of the isomorphism of base-change objects
\[ \xymatrix @R=0.5em{
	\bcr{\Omega X}{}{*} \ar[r]^-\cong & \bcr{\Omega X}{\Omega f}{\Omega X} \odot \bcr{\Omega X}{}{*} \\
	\Sph \ar[r]^-\cong & \Sigma^\infty_+ \Omega X^f \times_{\Omega X} *
} \]
which gives a map of spectra $\Sph \to \Sigma^\infty_+ \Lambda^f X$.\index{Reidemeister trace $R(f)$} This proves:
\begin{thm}\label{thm:reidemeister_traces_agree}
	The Reidemeister trace as computed in parametrized spectra is equivalent to the Reidemeister trace as computed in ring and bimodule spectra.
\end{thm}
This fact is of foundational importance to the work in \cite{lind2019transfer,campbell_ponto} and will likely prove useful in further investigations of topological fixed-point theory.

It is possible to go from here all the way back to chain complexes, but we will only do the first half of this argument here. Take $\pi = \pi_1(X)$ and observe that the map $X \to B\pi$ is an isomorphism on $\pi_0$ and $\pi_1$. Therefore the map $\Lambda^f X \to \Lambda^{f_*} B\pi$ is an isomorphism on $\pi_0$. To compute the Reidemeister trace it therefore suffices to evaluate the composite
\[ \xymatrix{ \Sph \ar[r] & \Sigma^\infty_+ \Lambda^f X \ar[r] & \Sigma^\infty_+ \Lambda^{f_*} B\pi } \]
The second map is in fact also a trace of
\[ \xymatrix @R=0.5em{
	\bcl{\Omega X}{}{\pi} \odot \bcr{\Omega X}{\Omega f}{\Omega X} \ar[r] & \bcr{\pi}{f_*}{\pi} \odot \bcl{\Omega X}{}{\pi} \\
	\Sigma^\infty_+ B(\pi;\Omega X;\Omega X^f) \ar[r]^-{f_*} & \Sigma^\infty_+ \pi^{f_*} \times_\pi \pi.
} \]
So by the multiplicativity property from \autoref{prop:bicat_trace_properties}, the composite of these two traces is the trace of the product map
\[ \resizebox{\textwidth}{!}{
	\xymatrix @R=0.5em{
	\bcl{\Omega X}{}{\pi} \odot \bcr{\Omega X}{}{*} \ar[r]
	& \bcl{\Omega X}{}{\pi} \odot \bcr{\Omega X}{\Omega f}{\Omega X} \odot \bcr{\Omega X}{}{*} \ar[r]
	& \bcr{\pi}{f_*}{\pi} \odot \bcl{\Omega X}{}{\pi} \odot \bcr{\Omega X}{}{*} \\
	\Sigma^\infty_+ B(\pi;\Omega X;*) \ar[r]
	& \Sigma^\infty_+ B(\pi;\Omega X;\Omega X^f;\Omega X;*) \ar[r]
	& \Sigma^\infty_+ \pi^{f_*} \times_\pi B(\pi;\Omega X;*) \\
	\Sigma^\infty_+ \ti X \ar[rr]^-{\ti f}
	&
	& \Sigma^\infty_+ \pi^{f_*} \times_\pi \ti X.
}
} \]
This is the trace of the $\pi$-equivariant map $\ti f\colon \ti X \to \pi^{f_*} \times_\pi \ti X$ in $\mathscr Act_\Sph$. This is very close to the algebraic trace, it only remains to pass between spectra and chain complexes. The argument for this half is more purely algebraic and has little to do with parametrized spectra, so we will not develop it here.\footnote{Note that a direct comparison of the trace on $\ti X$ with the trace on $C_*(\ti X)$ already appears in \cite[\S 6]{ponto_asterisque}, and that there are more classical arguments comparing the formula in \autoref{reidemeister_trace_is_weighted_sum_of_fixed_points} with the trace in chain complexes, see e.g. \cite{wecken}, \cite{husseini}.}

\begin{rmk}
	We don't have a similar comparison theorem for the fiberwise Reidemeister trace $R_B(f)$, because this argument does not appear to generalize well. Though when $\dim B = 1$, the work of Geoghegan-Nicas gives chain-level descriptions of $R_B(f)$ \cite{gn_one_parameter}, which we can compare to ours by going through the geometric characterization of a count of the fixed points as a 1-manifold with a framing \cite{geoghegan_forum}. When $\dim B \geq 2$, we expect that an algebraic characterization of $R_B(f)$ would be much more difficult.
\end{rmk}

\newpage
\section{Genuine equivariance}\label{sec:G}

In this final section we describe how to insert a compact Lie group $G$ of equivariance everywhere. Almost every theorem remains true, with the same proof, but there are a few places where an extra argument is needed. For the most part, it is a simple matter of getting the definitions right.

We say very little about genuine fixed points, geometric fixed points, or multiplicative norm constructions, because these require additional arguments to work that go beyond the scope of this text. The inspired reader might take the results on geometric fixed points from \cite{mp1} as a starting point and attempt to go further.

\beforesubsection
\subsection{Parametrized spaces}\aftersubsection

Let $G$ be a compact Lie group. A $G$-space is a space $B$ with a left action by $G$, and a map $f\colon A \to B$ of $G$-spaces is equivariant when it commutes with the action of each $g \in G$. We say the map is non-equivariant when the condition does not necessarily hold. For any $G$-space $B$, a retractive $G$-space over $B$ is a $G$-space $X$ such that the inclusion and projection are both equivariant. We keep the convention that $B$ is always CGWH, and that our convention can either ask for $X$ to be CG or CGWH.\index{retractive space!equivariant}

Let $G\mc R(B)$ be the category of such spaces, where the morphisms are equivariant maps commuting with the inclusion and projection. We also have a larger category $G\mc R(B)^\non$, equivalent to $\mc R(B)$, with the same objects but where the morphisms are non-equivariant maps. Then $G$ acts on the mapping spaces of $G\mc R(B)^\non$ by conjugation,
\[ g(f) := g \circ f \circ g^{-1}, \]
and the $G$-fixed maps are precisely the equivariant ones. Unless otherwise noted, all results use the category $G\mc R(B)$ of equivariant maps.

By convention, we will only consider closed subgroups of $G$. If $H \leq G$, the \textbf{$H$-fixed points} $X^H$ is the closed subspace of $X$ on which $hx = x$ for all $h \in H$. This is a retractive space over $B^H$, with an action by the \textbf{Weyl group} $WH = NH/H$. This defines a functor $G\mc R(B) \to WH\mc R(B^H)$.

A \textbf{weak equivalence}\index{weak equivalence!equivariant} $f\colon X \to Y$ is an (equivariant) map such that for every (closed) $H \leq G$ the map $f^H\colon X^H \to Y^H$ is a weak homotopy equivalence. $f$ is an \textbf{$h$-fibration}\index{$h$-fibration!equivariant} if the projection map $X^I \to X \times_Y Y^I$ has a $G$-equivariant section, and a \textbf{$q$-fibration}\index{$q$-fibration!equivariant} if each $X^H \to Y^H$ is a Serre fibration, equivalently the map has lifts with respect to $G/H \times (D^n \to D^n \times I)$. The cofibrations are all defined as before, but in an $h$-cofibration or $f$-cofibration the retract $(Y \times I) \to (X \times I) \cup_{(X \times \{0\})} (Y \times \{0\})$ has to be $G$-equivariant\index{$h$-cofibration!equivariant}\index{$f$-cofibration!equivariant}. The following lemma is one indication that these are good definitions.

\begin{lem}
	The $H$-fixed point functor preserves every notion of equivalence, fibration, and cofibration.
\end{lem}

We have the same technical lemmas as before, because the formulas we used in the proofs all respect the $G$-action. In particular, the result of Clapp holds, so equivariant $h$-fibrations are preserved by pushout along an equivariant $f$-cofibration.

For an equivariant map $f\colon A \to B$ of base spaces, the pullback $f^*$, pushforward $f_!$, and sheafy pushforward $f_*$ are defined as before. They also define functors on the larger category of nonequivariant maps. On the category of equivariant maps, the Beck-Chevalley isomorphism is equivariant.

These functors interact with cofibrations, fibrations, and weak equivalences in the same way as before. But the proof of this requires the following additional lemma:
\begin{lem}\label{lem:commute_with_fixed_points}
	$f^*$ commutes with $H$-fixed points for $H \leq G$. $f_!$ commutes with $H$-fixed points (CGWH) always (CG) on $i$-cofibrant spaces (basepoint section is a closed inclusion).
\end{lem}
The issue is that $H$-fixed points don't commute with all pushouts, only pushouts along closed inclusions. This is not something that switching between (CG) and (CGWH) solves; it happens even in simple examples like the pushout of
\[ \xymatrix{
	{*} & \ar[l] \Z/2 \ar[r] & {*}.
} \]
However, working in (CGWH) makes all of the basepoint sections $B \to X$ into closed inclusions, so that $f_!$ is always a pushout along a closed inclusion. For this reason and others, working in (CGWH) becomes more convenient in the equivariant theory than it was in the non-equivariant theory.

We define $WX$ and $PX$ by the same formulas, and they have the same properties. The $H$-fixed points commute with $WX$ always, and with $PX$ under the same assumptions as in \autoref{lem:commute_with_fixed_points}.

We define $X \barsmash Y$ as before, and get an isomorphism of spaces over $A^H \times B^H$
\begin{equation}\label{eq:smash_fixed_points}
(X \barsmash Y)^H \cong X^H \barsmash Y^H
\end{equation}
for $i$-cofibrant $X$ and $Y$. In the definition of $\barmap_B(Y,Z)$ we use the space of \emph{non-equivariant} maps, with $G$ acting by conjugation, so that $\barmap_B(Y,Z)$ is a $G$-equivariant retractive space over $A$ when $Y \in G\mc R(B)$ and $Z \in G\mc R(A \times B)$. The proofs are the same as before, taking care to remember that most of the maps appearing in the argument are equivariant, only the ones representing points in a mapping space $\Map(-,-)$ are not. In particular, equivariant maps of retractive spaces $Y' \to Y$ and $Z \to Z'$ induce an equivariant map
\[ \barmap_B(Y,Z) \to \barmap_B(Y',Z'). \]
We get a bijection both on equivariant and on non-equivariant maps
\[ X \barsmash Y \to Z \quad \textup{ over }A \times B \quad \longleftrightarrow \quad X \to \barmap_B(Y,Z) \quad \textup{ over }A. \]
Letting $S^V$ denote the one-point compactification of an orthogonal $G$-representation $V$, this gives us fiberwise reduced suspension and based loop functors
\[ \Sigma^V_B X := S^V \barsmash X, \quad \Omega^V_B X := \barmap_*(S^V,X) \]
for $X \in G\mc R(B)$. Note that $\barmap_B(Y,Z)$ does not commute with fixed points, there is instead a restriction map
\[ \barmap_B(Y,Z)^H \to \barmap_{B^H}(Y^H,Z^H). \]

The pushout-product and pullback-hom interact with cofibrations, fibrations, and weak equivalences in the same way as before. When working with the pullback-hom, the relevant variant of \autoref{lem:pullback_hom_adjunction} asks for all the dotted maps to be non-equivariant and all the solid maps to be equivariant. Though, when proving that the pullback-hom preserves weak equivalences, without loss of generality we only have to consider its $G$-fixed points, and then all the maps become $G$-equivariant and the cells do not require an extra $G/H$. The external smash product also commutes with pullbacks and pushforwards by the same proof as before. So we still get a canonical isomorphism $PX \barsmash PY \cong P(X \barsmash Y)$.

The first significant change in the equivariant theory is that the functor $f_!g^*(X_1 \barsmash \ldots \barsmash X_n)$ is not rigid. A counterexample is given by any span of the form
\[ \xymatrix{ \Z/2 & \ar[l] (\Z/2) \amalg (\Z/2) \ar[r]^-\cong & \Z/2 \times \Z/2. } \]
However, since the non-equivariant rigidity theorem holds, we can still say that among all the natural automorphisms of $f_!g^*(X_1 \barsmash \ldots \barsmash X_n)$ there is only one that is also natural with respect to the non-equivariant maps. So whenever we switch between different models for this functor, all we have to do is insist that we use the unique isomorphism that is natural in non-equivariant maps.\index{rigidity!for $G$-spaces}

\begin{rmk}
	Suppose that we also take $G$-fixed points. Then the proof of \autoref{prop:spaces_rigidity} shows that the functor $f_!g^*(X_1 \barsmash \ldots \barsmash X_n)^G$ is rigid. In (CG) we restrict to $i$-cofibrant spaces so that we don't have to worry about whether we should take $(-)^G$ or $f_!$ first. This and the previous result can be generalized: for any $H \leq G$, the functor $f_!g^*(X_1 \barsmash \ldots \barsmash X_n)^H$ has only one automorphism that is natural with respect to the $H$-equivariant maps.
\end{rmk} 

The non-equivariant rigidity theorem makes $G\mc R$ into a symmetric monoidal category under $\barsmash$, not necessarily in a unique way, but in a canonical way. This makes $G \mc R$ into an SMBF over the category $G\cat{Top}$ of unbased $G$-spaces, with Beck-Chevalley for the strict pullback squares of $G$-spaces.

We define the internal smash product $X \sma_B Y$ as before, making $G\mc R(B)$ into a symmetric monoidal category. The operation $\sma_B$ preserves equivalences of $f$-cofibrant spaces if one of them is $q$-fibrant. The proof is by reduction to the non-equivariant case, because all the constructions commute up to homeomorphism with the fixed points $(-)^H$.

The pushforward $f_!$ and external smash product $\barsmash$ are left-deformable, while the pullback $f^*$ and external mapping space $\barmap_B(-,-)$ are right-deformable. We can write down the same composites of coherently deformable functors as in the nonequivariant case.

The \textbf{Quillen model structure} on $G\mc R(B)$ has the weak equivalences and $q$-fibrations that we defined earlier (namely the maps that are weak equivalences or Serre fibrations on the $H$-fixed points for all closed subgroups $H \leq G$). These fit into a proper model structure that is cofibrantly generated by
	\[ \begin{array}{ccrlll}
	I &=& \{ \ (G/H \times S^{n-1})_{+B} \to (G/H \times D^n)_{+B} & : n,k \geq 0, & H \leq G, & D^n \to B \ \} \\
	J &=& \{ \ (G/H \times D^n)_{+B} \to (G/H \times D^n \times I)_{+B} &: n,k \geq 0, & H \leq G, & (D^n \times I) \to B \ \}.
	\end{array} \]
The adjunction $(f_! \adj f^*)$ is Quillen, and is a Quillen equivalence whenever $f\colon A \to B$ is a weak equivalence of $G$-spaces (i.e. $f^H$ is an equivalence for all $H$).

\beforesubsection
\subsection{Parametrized spectra}\aftersubsection

In the interest of brevity we only discuss orthogonal spectra. A parametrized orthogonal $G$-spectrum is an object with left $G$-action, in the category of parametrized orthogonal spectra $\Osp$. Equivalently, it consists of a $G$-space $B$, regarded as a $G \times O(n)$-space for every $n$ by having $O(n)$ act trivially, a sequence of equivariant retractive spaces $X_n \in (G \times O(n))\mc R(B)$, and $G$-equivariant bonding maps
\[ \sigma\colon \Sigma_B X_n \to X_{1+n}, \]
such that each composite
\[ \sigma^p\colon S^p \barsmash X_q \to \ldots \to X_{p + q} \]
is $O(p) \times O(q)$-equivariant. Note that the $G$ and $O(n)$ actions on $X_n$ commute. For each $b \in B$ the fiber spectrum $X_b$ is usually not a $G$-spectrum, but rather a spectrum with an action of the stabilizer subgroup $G_b = \{ g \in G : gb = b \}$.\index{orthogonal spectra!equivariant category $G\Osp(B)$}

It is sometimes convenient to have a level $X_V$ for every finite-dimensional $G$-representation $V$. One might think that this requires a change in our definition of orthogonal $G$-spectra, but it does not, essentially because we already have $G \times O(n)$-actions on $X_n$.

By convention, we fix a complete $G$-universe $\mc U$\footnote{Recall this means that $\mc U \cong \R^\infty$ with an orthogonal $G$-action, and as a representation it is isomorphic to a countable direct sum of every irreducible representation of $G$.} and only take those $V$ that are finite-dimensional $G$-invariant linear subspaces of $\mc U$. Recall the category $\mathscr J_G$ that has one object for every such $V \subset \mc U$. Even though $V$ and $W$ have nontrivial $G$-actions, we define $\mathscr J_G(V,W)$ just as $\mathscr J(V,W)$, using the \emph{non-equivariant} linear isometric maps $V \to W$. Then it inherits a $G$-action by conjugation, making $\mathscr J_G$ into a category enriched in $G$-spaces.\index{orthogonal spectra!equivariant indexing category $\mathscr J_G$}

Recall that $G\mc R(B)^\non$ refers to equivariant retractive spaces over $B$ with a $G$-action, and non-equivariant morphisms between them.
\begin{df}
	A parametrized $\mathscr J_G$-space is a $\mathscr J_G$-diagram in $G\mc R(B)^\non$ enriched in $G$-spaces. Concretely, this means for each object $V$ an object $X(V) \in G\mc R(B)$, for each pair $V,W$ a $G$-equivariant map over $B$
	\[ \mathscr J(V,W) \barsmash X(V) \to X(W) \]
	with the same associativity and unit as before.
\end{df}

Then as in the non-parametrized case \cite[V.1.5]{mandell2002equivariant}, the category of parametrized $\mathscr J_G$-spaces $G\Osp(B)$ is equivalent to the category of parametrized orthogonal $G$-spectra, defined using \emph{only trivial representations}. This is somewhat surprising when one first learns about it, but it is true essentially because $\mathscr J$ and $\mathscr J_G$ are equivalent categories when you ignore $G$-actions, and that is enough for a diagram on $\mathscr J_G$ to be determined up to isomorphism by its behavior on $\mathscr J$.

More concretely, $X(V)$ is isomorphic to $X_n$ with $G$ acting through a homomorphism $G \leq G \times O(n)$ that is the graph of some (smooth) homomorphism $\rho\colon G \to O(n)$. So $X(V)^H \cong X_n^\Gamma$ for some (closed) subgroup $\Gamma$ of such a graph. We call such a $\Gamma$ a \textbf{graph subgroup} of $G \times O(n)$. In other words, it is a subgroup whose elements are of the form $(g,\rho(g))$ for some fixed homomorphism $\rho\colon G \to O(n)$.\footnote{If we consider instead the larger class of subgroups that come from homomorphisms $H \to O(n)$, we are keeping track of representations that have an $H$-action that might not extend to a $G$-action. This eventually leads to the complete model structure on $G$-spectra from \cite[App B]{hhr}.}

We define free spectra as before. Note that with orthogonal $G$-spectra we can take a free spectrum $F_V A$ on a retractive $G$-space $A$ and a nontrivial $G$-representation $V$. At level $n$ this gives the retractive space $\mathscr J_G(V,\R^n) \barsmash A$. The space $\mathscr J_G(V,\R^n)$ is a $G \times O(n)$-CW complex by an application of Illman's triangulation theorem \cite{illman}.  

A \textbf{level equivalence}\index{level equivalence!equivariant} is a map of spectra $X \to Y$ inducing an equivalence $X(V)^H \to Y(V)^H$ for all $G$-representations $V$ and $H \leq G$. Again, our convention is to only consider the closed subgroups. Similarly a \textbf{level $q$-fibration}\index{$q$-fibration!equivariant level} is a map for which $X(V)^H \to Y(V)^H$ is a Serre fibration for all $V$ and $H$. Note that $X \to Y$ is a level equivalence or $q$-fibration precisely when $X_n^\Gamma \to Y_n^\Gamma$ is an equivalence or $q$-fibration for every $n \geq 0$ and every graph subgroup $\Gamma \leq G \times O(n)$.

A \textbf{level $h$-fibration}\index{$h$-fibration!equivariant level} is a map for which $X_n \to Y_n$ is a $G \times O(n)$-equivariant $h$-fibration for all $n$. This is strictly stronger than asking for $X(V) \to Y(V)$ to be a $G$-equivariant $h$-fibration for all $V$. In a \textbf{level $f$-cofibration}\index{$f$-cofibration!equivariant level} or \textbf{level $h$-cofibration}\index{$h$-cofibration!equivariant level} we ask that $X_n \to Y_n$ is a $G \times O(n)$-equivariant $f$-cofibration or $h$-cofibration, respectively. Again, this implies that $X(V) \to Y(V)$ is a $G$-equivariant cofibration for all $V$.

Consider all maps of spectra over $B$ of the form $F_V(K \to X)$, where $V$ may have nontrivial $G$-action and $K \to X$ is a $G$-equivariant $f$-cofibration of equivariantly $f$-cofibrant retractive $G$-spaces over $B$. The class of \textbf{free $f$-cofibrations}\index{$f$-cofibration!equivariant free} is the smallest class of maps of spectra containing the above, and closed under pushout, transfinite composition, and retracts. The \textbf{free $h$-cofibrations}\index{$h$-cofibration!equivariant free} are defined similarly, and for \textbf{free $q$-cofibrations}\index{$q$-cofibration!equivariant free} we ask for $K \to X$ to be a relative $G$-cell complex, so that the free $q$-cofibrations are generated by the maps of the form
\[ F_V((G/H \times S^{n-1})_{+B}) \to F_V((G/H \times D^n)_{+B}). \]

When proving the following, we think of $K \to X$ as a $G \times O(n)$-space with trivial $O(n)$-action, and recall that $\mathscr J_G(V,\R^n)$ is a $G \times O(n)$-cell complex.

\begin{lem}
	Every free cofibration is a level cofibration. The free spectrum functor $F_V$ sends cofibrant spaces to freely cofibrant spectra. It sends $h$-cofibrant, $h$-fibrant spaces to level $h$-fibrant spectra.
\end{lem}
	
\begin{prop}[Level model structure]
	The level equivalences, free $q$-cofibrations, and level $q$-fibrations define a proper model structure on $G\Osp(B)$. It is cofibrantly generated by
\[ \begin{array}{ccrllll}
I &=& \{ \ F_V\left[ (G/H \times S^{n-1})_{+B} \to (G/H \times D^n)_{+B} \right] & : n \geq 0, & V \subset \mc U, & H \leq G, & D^n \to B^H \ \} \\
J &=& \{ \ F_V\left[ (G/H \times D^n)_{+B} \to (G/H \times D^n \times I)_{+B} \right] &: n \geq 0, & V \subset \mc U, & H \leq G, & (D^n \times I) \to B^H \ \}.
\end{array} \]
\end{prop}
	
The pullback $f^*$, pushforward $f_!$, and monoidal fibrant replacement functor $P$ lift to spectra as before by applying them to each level separately. They have the same properties as before. In particular, $f_!$ preserves level $h$-cofibrant spectra and equivalences between them, $f^*$ preserves level $h$-fibrant spectra and equivalences between them, $(f_! \adj f^*)$ form a Quillen pair, and $X$ is freely $h$-cofibrant then $PX$ is freely $f$-cofibrant and level $h$-fibrant.
	
We define the functor that smashes a spectrum $X$ with a space $K$,
\[ \barsmash\colon G\mc R(A) \times G\Osp(B) \to G\Osp(A \times B). \]
by applying $K \barsmash -$ at each spectrum level of $X$. This is unique up to canonical (not unique) isomorphism, so its right adjoints $\barF_A(K,-)$ and $\barmap_B(X,-)$ and the two adjunctions themselves are also defined up to canonical (not unique) isomorphism.

We define the smash product of spectra
\[ \barsmash\colon G\Osp(A) \times G\Osp(B) \to G\Osp(A \times B) \]
as before, by left Kan extending along the direct sum map $\mathscr J_G \sma \mathscr J_G \to \mathscr J_G$. Equivalently, we work with only trivial representations and use the direct sum map $\mathscr J \sma \mathscr J \to \mathscr J$. So we can think of it is the same external smash product from before, with the diagonal $G$-action, and its right adjoint in each variable is also the same right adjoint from before, with a conjugation $G$-action. The description using $\mathscr J_G$ is useful however for verifying that
\begin{equation}\label{eq:smash_free}
	F_V X \barsmash F_W Y \cong F_{V \oplus W} (X \barsmash Y)
\end{equation}
for nontrivial $G$-representations $V$ and $W$.\footnote{We insist that $V \oplus W$ is a model of the direct sum of $V$ and $W$ that is a subspace of $\mc U$. The choice of subspace does not change the answer, up to canonical isomorphism.} As before, the external smash product always commutes with pushforward, and commutes with pullback either in (CG) or in (CGWH) when the spectra are freely $i$-cofibrant.
%
%

Once again the functor $f_!g^*(X_1 \barsmash \ldots \barsmash X_n)$ is not rigid, but on any full subcategory of $\Osp(C_1) \times \ldots \times \Osp(C_n)$ containing all $n$-tuples of the form $(F_{V_1} *_{+C_1},\ldots, F_{V_n} *_{+C_n})$, any functor isomorphic to $\Phi$ is admits a canonical isomorphism to $\Phi$, the one that is natural with respect to all non-equivariant maps.\index{rigidity!for $G$-spectra}
	
\beforesubsection
\subsection{Skeleta and semifree cofibrations}\label{sec:G_reedy}\aftersubsection

The equivariant version of this section requires some additional lemmas, primarily because the level equivalences are sensitive to the entire $G \times O(n)$-action at level $n$. As a result, we must make the equivariant version of a Reedy cofibration more sensitive to the $O(n)$-action.

Let $O$ be a compact Lie group, such as $O = O(n)$, and let $H$ be a (closed) subgroup. We give $O$ a left $H$-action with $h$ acting by $o \mapsto oh^{-1}$. For a left $H$-space $X$ we let $O \times_H X$ denote the quotient of $O \times X$ by the diagonal left $H$-action.
\begin{lem}\label{lem:induce_equivalences}
	The functor $O \times_H -$ sends $H$-equivalences of unbased spaces to $O$-equivalences.
\end{lem}
\begin{proof}
	It suffices to recall the formula for the $K$-fixed points of $O \times_H X$ for a closed subgroup $K \leq O$. By \cite[B.17]{schwede_global} with $G = H$ and $K = K$ applied to $O \times X$,
	\[ (O \times_H X)^K \cong \coprod_{(\alpha)} (O \times X)^\alpha/C(\alpha). \]
	Here the sum runs over conjugacy classes of homomorphisms $\alpha\colon K \to H$, $C(\alpha)$ is the centralizer of $\alpha(K)$ in $H$, and the $(-)^\alpha$ denotes points that are fixed by the subgroup of $K \times H$ that is the graph of $\alpha$. This subgroup is isomorphic to $K$, but acts on $O \times X$ by
	\[ k(o,x) = (ko\alpha(k)^{-1},\alpha(k)x). \]
	Therefore its fixed points are $Q_\alpha \times X^{\alpha(K)}$ where $Q_\alpha \leq O$ is the subgroup of $O$ fixed by the $K$-action $o \mapsto ko\alpha(k)^{-1}$. Since $C(\alpha) \leq H$ acted freely on $O$, its action on the subset $Q_\alpha$ is necessarily still free, so we get
	\[ (O \times_H X)^K \cong \coprod_{(\alpha)} Q_\alpha \times_{C(\alpha)} X^{\alpha(K)}. \]
	An $H$-equivalence in the $X$ variable induces an equivalence on all the subspaces $X^{\alpha(K)}$, which then become equivalences on the above terms by a simple induction on the free $C(\alpha)$-cells of $Q_\alpha$.
\end{proof}

\begin{lem}\label{lem:induce_fibrations}
	The functor $O \times_H -$ sends $H$-equivariant Hurewicz fibrations to $O$-equivariant Hurewicz fibrations.
\end{lem}
\begin{proof}
	Let $L$ be the tangent space at $1H \in O/H$ with the conjugation $H$-action, and $D(L)$ its unit disc. As in \cite[3.2.1]{schwede_global}, we can pick a smooth embedding $s\colon D(L) \to O$ satisfying $s(hl) = hs(l)h^{-1}$ for all $h \in H$, such that the map $D(L) \times H \to O$, $(l,h) \mapsto s(l)h$, is an embedding whose image is a tubular neighborhood of $H$ in $O$. Note that such a map respects the left and right $H$-multiplications in the following way.
	\[ \begin{array}{rcl}
	l,h &\mapsto & s(l)h \\
	h_1l,h_1h &\mapsto & h_1s(l)h \\
	l,hh_2 &\mapsto & s(l)hh_2
	\end{array} \]
	The map $(D(L) \times H) \times_H X \to O \times_H X$ is a closed inclusion by directly checking the topology. (It is also an open inclusion on the interior of $D(L)$.) This gives a subspace that is $H$-equivariantly homeomorphic to the product $D(L) \times X$ with the diagonal $H$-action.
	
	Since $D(L) \times -$ clearly preserves $H$-equivariant fibrations, this means that for an $H$-equivariant fibration $X \to Y$, $O \times_H Y$ can be covered by $H$-equivariant open sets on which $O \times_H X \to O \times_H Y$ is an $H$-equivariant fibration. Furthermore these sets are equivariantly numerable, meaning that the continuous function to $[0,1]$ that picks them out factors through the $H$-orbits. By \cite[3.2.4]{waner_fibrations}, or the $H$-equivariant version of the usual argument from e.g. \cite[7.4]{concise}, this means that the entire map $O \times_H X \to O \times_H Y$ is an $H$-equivariant fibration .
	
	To prove that it is an $O$-equivariant fibration, consider the path space $(O \times_H Y)^I$, and let $(O \times_H Y,Y)^{(I,0)}$ denote the $H$-invariant subspace of those paths starting in $H \times_H Y$. The $O$-action gives a continuous bijection
	\begin{equation}\label{eq:induce_paths}
		O \times_H (X \times_Y (O \times_H Y,Y)^{(I,0)}) \to (O \times_H X) \times_{(O \times_H Y)} (O \times_H Y)^I.
	\end{equation}
	If we restrict to $D(L) \times H$, the resulting map is isomorphic to
	\[ D(L) \times (X \times_Y (O \times_H Y,Y)^{(I,0)}) \to (D(L) \times X) \times_{(D(L) \times Y)} (O \times_H Y,D(L) \times Y)^{(I,0)} \]
	and given by the formula
	\[ (l,x,\gamma) \mapsto (l,x),(s(l)\gamma), \]
	so it has a continuous inverse $(l,x),\gamma \mapsto (l,x,s(l)^{-1}\gamma)$. Therefore the map \eqref{eq:induce_paths} is open in the interior of $D(L) \times H$. Translating by elements of $O$, we get an open cover on which the map is open, therefore \eqref{eq:induce_paths} is a homeomorphism.
	
	Now take an $H$-equivariant path-lifting function for $O \times_H X \to O \times_H Y$, restrict it to $X \times_Y (O \times_H Y,Y)^{(I,0)}$, and extend to an $O$-equivariant map
	\[ O \times_H (X \times_Y (O \times_H Y,Y)^{(I,0)}) \to (O \times_H X)^I. \]
	Along the homeomorphism \eqref{eq:induce_paths}, this gives an $O$-equivariant path-lifting function for $O \times_H X \to O \times_H Y$.
\end{proof}

The retractive space version of these lemmas follows almost immediately.
\begin{cor}\label{lem:induce_all}
	The functor $O_+ \barsmash_H (-)\colon H\mc R(B) \to O\mc R(B)$ preserves equivariant $h$- and $f$-cofibrations, and on equivariantly $h$-cofibrant spaces it preserves $h$-fibrations and equivalences.
\end{cor}
\begin{proof}
	Recall the pushout square
	\[ \xymatrix @R=1.7em{
		O/H \times B \ar[d] \ar[r] & O \times_H X \ar[d] \\
		B \ar[r] & O_+ \barsmash_H X.
	} \]
	This construction preserves cofibrations because the orbits commute with $- \times I$, so they preserve the retract that defines an $h$- or $f$-cofibration. On an $h$-cofibrant space, the horizontal maps of the above square are equivariant $h$-cofibrations. Since the term $O \times_H X$ preserves equivariant equivalences and fibrations by the previous two results, for $h$-cofibrant $X$ the pushout $O_+ \barsmash_H X$ does as well.
\end{proof}

From here we can follow the outline of \autoref{sec:reedy}. Only the first result has a substantially different proof. Let $O$ and $O'$ be compact Lie groups, such as $O = O(n)$ and $O' = O(m)$ with $m \geq n$. They will always act trivially on the base spaces.

\begin{prop}\label{eq_smashing_with_free_complex}\hfill
	\vspace{-1em}
	
	\begin{itemize}
		\item If $f: K \to L$ is an $O$-free $O \times O'$-cell complex of based spaces and $g: X \to Y$ is a $G \times O$-equivariant $h$-cofibration of spaces over $B$ then $f \square g$ constructed by $(- \barsmash -)_{O}$ is a $G \times O'$-equivariant $h$-cofibration.
		\item The same is true for $f$-cofibrations or closed inclusions.
		\item If $L$ is any finite $O$-free based $O \times O'$-cell complex, and $X$ is any $G \times O$-equivariantly $h$-cofibrant and $h$-fibrant retractive space, then $(L \barsmash X)_O$ is $G \times O'$-equivariantly $h$-fibrant.
		\item If $L$ is any finite $O$-free based $O \times O'$-cell complex, and $g: X \to Y$ is any $G \times O$-equivariant equivalence of $G \times O$-$h$-cofibrant spaces over $B$, then $(\id_L \barsmash g)_{O}$ is a $G \times O'$-equivariant weak equivalence.
	\end{itemize}
\end{prop}

\begin{proof}
	An $O$-free $O \times O'$-cell complex is built out of cells of the form $(O \times O')/\rho_H \times D^n$, where $H \leq O'$ is closed, $\rho\colon H \to O$ is a homomorphism, and $\rho_H \leq O \times O'$ is the image of $H$ under the homomorphism $h \mapsto (\rho(h),h)$. So if $D$ is an unbased space with no group action, and $Y$ is a $G \times O$-equivariant $h$-cofibrant retractive space, we form the pushout square
	\[ \xymatrix @R=1.7em{
		O \times O' \times D \times B \ar[d] \ar[r] & O \times O' \times D \times Y \ar[d] \\
		B \ar[r] & (O \times O' \times D)_+ \barsmash Y.
	} \]
	in which the top horizontal is a $G \times O \times O'$-equivariant $h$-cofibration. Quotienting by the diagonal left $O$-action on the $O$ and $Y$, and the diagonal right $H$-action on $O$ (through $\rho$) and $O'$, gives a new pushout square along a $G \times O'$-equivariant $h$-cofibration
	\[ \xymatrix @R=1.7em{
		O'/H \times D \times B \ar[d] \ar[r] & D \times (O' \times_H Y) \ar[d] \\
		B \ar[r] & ((O \times O')/\rho_H \times D)_+ \barsmash_O Y
	} \]
	where the $H$-action on $Y$ in the upper-right term is through the homomorphism $\rho$. Comparing to the pushout squares
	\[ \xymatrix @R=1.7em{
		D \times O'/H \times B \ar[d] \ar[r] & D \times (O' \times_H Y) \ar[d] \\
		D \times B \ar[d] \ar[r] & D \times (O' \barsmash_H Y) \ar[d] \\
		B \ar[r] & D_+ \barsmash (O' \barsmash_H Y)
	} \]
	gives a natural isomorphism of $G \times O'$-equivariantly $h$-cofibrant spaces
	\[ ((O \times O')/\rho_H \times D)_+ \barsmash_O Y \cong D_+ \barsmash (O' \barsmash_H Y). \]
	We note that any $G \times O$-equivariant property of $Y$ becomes a $G \times H$-equivariant property along $\rho\colon H \to O$, and that
	\[ (G \times O') \barsmash_{G \times H} Y \cong O' \barsmash_H Y \]
	as $G \times O'$-spaces, hence this becomes a $G \times O'$-equivariant property of $O' \barsmash_H Y$ by \autoref{lem:induce_all}.
	
	The first two claims follow immediately from this analysis because as in \autoref{smashing_with_free_complex} if suffices to assume that $f$ is a single cell
	\[ [(O \times O')/\rho_H \times S^{n-1}]_+ \to [(O \times O')/\rho_H \times D^n]_+. \]
	For the fourth statement we induct on the skeleta of $L$. If $L$ is obtained from $K$ by pushout along a single cell, then $g$ gives a map of two pushout squares of the form
	\[ \xymatrix @R=1.7em{
		S^{d-1}_+ \barsmash O' \barsmash_H (-) \ar[r] \ar[d] & D^d_+ \barsmash O' \barsmash_H (-) \ar[d] \\
		(K \barsmash (-))_O \ar[r] & (L \barsmash (-))_O.
	} \]
	The horizontal maps are $G \times O'$-cofibrations by the first part, so these squares are $G \times O'$-equivariantly homotopy pushout squares. The map $g$ induces $G \times O'$-equivalences on the top-left and top-right by the above analysis, so we can do the induction and conclude $(\id_{L'} \barsmash f)_O$ is an equivalence for all skeleta $L'$ of $L$, and therefore for $L$ itself. The third statement is proven by the same induction, with just one pushout square of the above form for each cell of $L$.
\end{proof}

We define semifree spectra $\mc G_n A$ as before, where $A$ has a $G \times O(n)$-action. It turns out that a semifree spectrum on a nontrivial representation $V$ is always isomorphic to a semifree spectrum on a trivial representation, so it is unnecessary to define $\mc G_V A$ for general $V$. We then define skeleta and latching maps just as in \autoref{sec:reedy}. The $n$th latching map $L_n X \to X_n$ is now a $G \times O(n)$-equivariant map.

\begin{df}
	A map of spectra $X \to Y$ over $B$ is a \textbf{semifree} or \textbf{Reedy $h$-cofibration} if each relative latching map $L_n Y \cup_{L_n X} X_n \to Y_n$ is a $G \times O(n)$-equivariant $h$-cofibration. A semifree $f$-cofibration is defined similarly.\index{$f$-cofibration!equivariant semifree}\index{$h$-cofibration!equivariant semifree}
\end{df}

As before, the semifree cofibrations are closed under pushouts, compositions, sequential colimits, and retracts. Using the formula
	\[ F_V A \cong \mc G_n (\mathscr J_G(V,\R^n) \barsmash A) \]
for $n = \dim V$, we check that every generating free cofibration is a semifree cofibration. Therefore every free cofibration is also a semifree cofibration.

The same arguments as in \autoref{sec:reedy} prove that semifree cofibrations are level cofibrations, a level equivalence of semifreely $h$-cofibrant spectra has all its relative latching maps equivalences, and a level $h$-fibrant semifreely $h$-cofibrant spectrum has all its latching objects fibrant. In all of these claims the notions of cofibration, fibration, and weak equivalence at spectrum level $n$ is always the $G \times O(n)$-equivariant one. The equivariant version of \autoref{prop:reedy_pushout_product} has the same argument as before. We will only record its specialization back to free spectra.

\begin{thm}\label{eq_prop:spectra_pushout_product}\hfill
	\vspace{-1em}
	
	\begin{itemize}
		\item Let $f: K \to X$ and $g: L \to Y$ be free $h$-cofibrations of orthogonal $G$-spectra over $A$ and $B$, respectively. Then $f \square g$, constructed using $\barsmash$, is a free $h$-cofibration.
		\item The same is true for free $f$-cofibrations, free $q$-cofibrations, and free closed inclusions.
		\item If $X$ and $Y$ are freely $h$-cofibrant and level $h$-fibrant then $X \barsmash Y$ is level $h$-fibrant.
		\item If $X$ is freely $h$-cofibrant and $g: Y \to Y'$ is a level equivalence of freely $h$-cofibrant spectra then $\id_X \barsmash g$ is a level equivalence.
	\end{itemize}
\end{thm}
	
\beforesubsection
\subsection{Stable equivalences}\aftersubsection
	
For $X \in G\Osp(B)$ the \textbf{homotopy groups}\index{homotopy!groups} are the genuine stable homotopy groups of the $G_b$-equivariant fiber spectrum $X_b$, for $b \in B$ and $n \in \Z$. For each (closed) subgroup $H \leq G_b$ these are defined as
\[ \pi_{k,b}^H(X) := \begin{cases}\quad
\underset{V \subset \mathcal U}\colim\, \pi_k([\Omega^V X_b(V)]^H)  &\text{if } k \geq 0 \\
\underset{V \subset \mathcal U,\ \R^{|k|} \subset V}\colim\, \pi_0([\Omega^{V-\R^{|k|}} X_b(V)]^H)  &\text{if } k < 0.\end{cases} \]
where $V$ runs over all $G_b$-representations in $\mc U$ and their inclusions in $\mc U$. On level $q$-fibrant spectra, every level equivalence induces an equivariant equivalence on the spaces $X_b(V)$, hence an isomorphism on these groups. So the \textbf{right-derived homotopy groups} $\R \pi_{n,b}(X)$ can be defined from these by replacing the spectrum $X$ by one that is level $q$-fibrant.

The stable equivalences are the maps inducing isomorphisms on these homotopy groups; they are generated under 2-out-of-3 by the level equivalences and the maps of level $q$-fibrant spectra that are stable equivalences on each fiber.\index{stable equivalence} The examples we gave earlier all have the same properties in the equivariant case, using \cite[III.3]{mandell2002equivariant} in the place of \cite[\S 7]{mmss}.

Most of the proofs in this section use the formal behavior of the various kinds of cofibrations and fibrations, and so the equivariant versions have the same proofs, verbatim. One might worry about the $i$-cofibrancy hypothesis on commuting $f_!$, $\barsmash$, or $P$ with $H$-fixed points, but that was already used to establish the earlier lemmas concerning weak equivalences and cofibrant spaces, so it does not need to be invoked again.

\begin{thm}[Stable model structure]\label{eq_thm:stable_model_structure}
	There is a proper model structure on $G\Osp(B)$ whose weak equivalences are the stable equivalences, and is cofibrantly generated by the sets of maps
	\[ \resizebox{\textwidth}{!}{$
		\begin{array}{ccrllll}
	I &=& \{ \ F_V\left[ (G/H \times S^{n-1})_{+B} \to (G/H \times D^n)_{+B} \right] & : n \geq 0, & V \subset \mc U, & H \leq G, & D^n \to B^H \ \} \\
	J &=& \{ \ F_V\left[ (G/H \times D^n)_{+B} \to (G/H \times D^n \times I)_{+B} \right] &: n \geq 0, & V \subset \mc U,  & H \leq G, & (D^n \times I) \to B^H \ \} \\
	& \cup & \{ \ k_{V,W} \ \square \ \left[ (G/H \times S^{n-1})_{+B} \to (G/H \times D^n)_{+B} \right] & : n \geq 0, & V,W \subset \mc U, & H \leq G, & D^n \to B^H \ \}.
	\end{array}
	$} \]
\end{thm}\index{stable model structure}
The map $k_{V,W}$ is the inclusion of the front end of the mapping cylinder $\Cyl_{V,W}$ for the map
\[ \lambda_{V,W}\colon F_{V \oplus W} S^W \ra F_V S^0 \]
adjoint to the equivariant map of spaces $S^W \to \mathscr J_G(V,V \oplus W)$ that identifies $S^W$ with the fiber over the standard embedding $V \to V \oplus W$. As in \autoref{thm:stable_model_structure} this is a free $q$-cofibration, and it is a stable equivalence by \cite[4.5]{mandell2002equivariant}.

The proof of \autoref{eq_thm:stable_model_structure} is the same with the following modifications. When proving the smallness condition, we use the fact that equivariant maps out of $G/H \times D$ correspond to maps from $D$ into the $H$-fixed points, and that $H$-fixed points preserve pushouts and sequential colimits along levelwise $h$-cofibrations. The fibrations ($J$-injective maps) rearrange into the maps that are (equivariant) $q$-fibrations at every level $V$ and for which every one of the squares
\begin{equation}\label{eq_eq:fibration_of_spectra_means_this_is_a_pullback}\index{stably fibrant}
\xymatrix @R=1.7em{
	X_V \ar[r] \ar[d]^-{p_V} & \Omega^W_B X_{V\oplus W} \ar[d]^-{\Omega_B^W p_{V\oplus W}} \\
	Y_V \ar[r] & \Omega^W_B Y_{V\oplus W}
}
\end{equation}
\noindent
is an equivariant homotopy pullback square (i.e. a homotopy pullback on the $H$-fixed points for all $H \leq G$). When proving the last property, we cannot restrict to the $H$-fixed points too early in the proof; we show that the maps of the colimit system defining the homotopy groups are isomorphisms by checking that we get an equivariant equivalence of spaces, before taking $H$-fixed points.

We get the same results on the derived smash product and base change functors, in particular:
\begin{thm}\label{eq_thm:spectra_SMBF}
	The homotopy category of parametrized orthogonal $G$-spectra $\ho G\Osp$ forms a symmetric monoidal bifibration over $G\cat{Top}$, with Beck-Chevalley squares the equivariant homotopy pullback squares of spaces.
\end{thm}\index{SMBF!of parametrized $G$-spectra $\ho G\Osp$}

The rest of the proofs concerning derived smash products over $B$ are formal, so there is no change in the equivariant case. For simplicity we often restrict the category $G\Osp_{(B)}$ to base spaces $B$ with trivial $G$-action, but the proofs still work even if the $G$-action is nontrivial.
	
\beforesubsection
\subsection{Duality and traces}\aftersubsection

For simplicity, in this last section we take $G$ to be finite. We do still get the Spanier-Whitehead and Costenoble-Waner duality theorems when $G$ is a compact Lie group, it is just a bit more work in that case to check the assumptions and to do calculations.

Let $G$ be a finite group. In the equivariant stable homotopy category $\ho G\Osp(*)$, a spectrum $X$ is dualizable iff it is equivalent to $F_V K$ where $K$ is a finite $G$-CW complex. In the homotopy category of spectra over $B$, $\ho G\Osp(B)$, a spectrum $X$ is known to be dualizable iff for every $b \in B$ its derived fiber spectrum $(\R b^*)X$ is a dualizable $G_b$-spectrum \cite[15.1.1]{ms}, though we have not done the equivariant version of our argument in this document.

If $f\colon X \to X$ is an equivariant map of finite $G$-cell complexes, its \textbf{equivariant Lefschetz number}\index{Lefschetz number $L(f)$!equivariant} is the trace of $\Sigma^\infty_+ f$ in $\ho G\Osp$. It is a map $\Sph \to \Sph$ in the equivariant stable homotopy category, in other words an element of the Burnside ring $L(f) \in A(G)$.

More generally, suppose $B$ is a space with trivial $G$-action and $f\colon E \to E$ is an equivariant self-map of an equivariant fibration over $B$ with finite or finitely dominated fibers. Then the \textbf{equivariant fiberwise Lefschetz number} $L_B(f)$ is the trace of $\Sigma^\infty_{+B} f$ in $\ho G\Osp(B)$, which gives an equivariant map of trivial $\Sph$-bundles over $B$. The tom Dieck splitting turns this into an element of the group
\[ L_B(f) \in \bigoplus_{(H) \leq G} [B,\Sigma^\infty_+ BH] \]
where the sum is over the conjugacy classes of subgroups of $G$. Informally, the summand for $H$ tracks the $f$-fixed points whose stabilizer group is conjugate to $H$.

We define a $G$-ENR to be a $G$-space $X$ that can be equivariantly embedded in a finite-dimensional $G$-representation $V$ with an equivariant retract $p$ of a $G$-invariant neighborhood $N$. In particular, this is possible when $X$ is a finite $G$-CW complex. A fiberwise $G$-ENR is a bundle with $G$-trivial base that can be equivariantly fiberwise embedded into $B \times V$ with fiberwise equivariant retract. This includes the smooth fiber bundles of closed $G$-manifolds.

When $X$ is a compact $G$-ENR, we pick such an embedding into $V$ with equivariant neighborhood retract $p$. The formulas from \autoref{thm:sw_duality_1} and \autoref{thm:sw_duality_2}, with $n$ replaced by $V$, give the dual of $\Sigma^\infty_+ X$ in $\ho G\Osp(*)$, under the equivariant versions of the same assumptions.\index{Spanier-Whitehead duality!equivariant}

Therefore \autoref{cor:lefschetz_formula} is also the formula for the Lefschetz number of an equivariant map $f\colon X \to X$, for any compact $G$-ENR. The geometric interpretation is slightly trickier because we have to control how the index changes when we pass from the $K$-fixed points to the $H$-fixed points for $H \leq K$, and we don't discuss it further here.

When $E \to B$ is a fiberwise $G$-ENR with compact ENR base space $B$, the fiberwise version of the discussion runs as before. When the base is a manifold, the equivariant fiberwise Lefschetz number is counting the fixed points as an equivariant manifold with framing, but we defer a detailed discussion of this to future work.

The proof of Spanier-Whitehead duality is identical, with $n$ replaced by $V$, and $G$ acting trivially on the $I$ coordinate of the mapping cone. All the maps and homotopies we used before become equivariant in this setting. We note the importance of making $V$ an orthogonal $G$-representation so that the $\epsilon$-tubes are invariant under the $G$-action. The formation of the mapping cones commutes with $H$-fixed points, so we can think of them as the same cones as before but with a $G$-action.

We form the point-set bicategory of parametrized spectra $G\Osp/G\cat{Top}$, its homotopy bicategory $G\Ex$, and its base-change objects in the same way as before. The proof that the four different definitions are canonically isomorphic is formal, so it does not need to be repeated. The technical lemmas \autoref{lem:derived_circle_product}, \autoref{lem:derived_shadow} concerning when $\odot$ and $\shad{}$ are derived have the same statements and proofs as before.\index{bicategory!of parametrized $G$-spectra $G\Ex$}

We define the \textbf{equivariant Reidemeister trace}\index{Reidemeister trace $R(f)$!equivariant} of an equivariant map $f\colon X \to X$ as the trace in $G\Ex$ of the isomorphism of base-change objects
\[ \xymatrix{ \bcr{X}{}{*} \ar[r]^-\cong & \bcr{X}{f}{X} \odot \bcr{X}{}{*} } \]
and similarly for the fiberwise equivariant Reidemeister trace in $G\Ex_B$.

When $X$ is a compact $G$-ENR, the formulas from \autoref{thm:cw_duality_1} and \autoref{thm:cw_duality_2} with $n$ replaced by $V$, give the Costenoble-Waner dual of $\Sigma^\infty_{+(X \times *)} X$ in $G\Ex$, again under the equivariant versions of the same assumptions.\index{Costenoble-Waner duality!equivariant} We note that $X^I$ has a $G$-action by post-composition -- we do not take only the $G$-fixed paths. The formula from \autoref{cor:reidemeister_formula} therefore describes the equivariant Reidemeister trace of an equivariant map $f\colon X \to X$, for $X$ any compact $G$-ENR. We run the fiberwise discussion as before, and conclude that \autoref{cor:fiberwise_reidemeister_formula} describes the equivariant fiberwise Reidemeister trace for a fiberwise $G$-ENR $E \to B$ with compact ENR base.

In the proof of Costenoble-Waner duality we replace $n$ and $\R^n$ by $V$. We check that all the subspaces are $G$-invariant and the point-set maps and homotopies are equivariant. In particular, the retract of the neighborhood of $X$ in $LX$ is equivariant, because the neighborhood deformation retract for an equivariant cofibration is always equivariant. When we apply uniform continuity, we ignore the $G$-action because the $\delta$-, and $\epsilon$-tubes about $X$ in $V$ and the $\epsilon'$-tube about $X$ in $X \times N$ are automatically $G$-invariant subspaces. Since the equivariant version of \autoref{lem:derived_circle_product} and \autoref{lem:derived_shadow} hold, we can pass to the homotopy category as before.

We leave the equivariant generalization of \autoref{sec:models} to the interested reader. It appears important in this case to work with groupoids with $G$-action whose realization is equivariantly equivalent to the base space $B$.

\printindex

\bibliographystyle{amsalpha}
\bibliography{streamlined}{}

\end{document}